\documentclass[a4paper, reqno, 10pt]{amsart}

\usepackage[utf8]{inputenc}
\usepackage{amsthm, amsfonts, amssymb, amsmath}
\usepackage{enumitem}
\usepackage{eqparbox}
\usepackage{etoolbox}
\usepackage{pdfsync}
\usepackage{stmaryrd}
\usepackage{mathtools}
\usepackage{mathrsfs}
\usepackage{bbm}
\usepackage{tikz-cd}
\usepackage[colorlinks=citecolor=Black, linkcolor=Red, urlcolor=Blue]{hyperref}
\usepackage{cleveref}
\usepackage[backend=biber, style=numeric, giveninits=true]{biblatex}
\DeclareSourcemap{
	\maps[datatype=bibtex]{
		\map{
			\step[fieldsource=doi, final]
			\step[fieldset=isbn, null]
			\step[fieldset=issn, null]
			\step[fieldset = url, null]
		}
	}
}
\DeclareFieldFormat{postnote}{#1}
\DeclareFieldFormat{multipostnote}{#1}

\DeclareFieldFormat
[article,inbook,incollection,inproceedings,patent,thesis,unpublished]
{title}{\mkbibemph{#1\isdot}}

\DeclareFieldFormat
[article,inbook,incollection,inproceedings,patent,thesis,unpublished]
{volume}{\mkbibbold{#1\isdot}}

\newtheorem{theoremalphabetical}{Theorem}[section]
\newtheorem{propositionalphabetical}[theoremalphabetical]{Proposition}

\newtheorem{theorem}{Theorem}[subsection]
\newtheorem{proposition}[theorem]{Proposition}
\newtheorem{lemma}[theorem]{Lemma}

\newtheorem{corollary}[theorem]{Corollary}

\theoremstyle{definition}
\newtheorem{definition}[theorem]{Definition}
\newtheorem{definitionproposition}[theorem]{Definition and Proposition}
\newtheorem{example}[theorem]{Example}

\newtheorem{remark}[theorem]{Remark}
\newtheorem{remarks}[theorem]{Remarks}

\newenvironment{altenumerate}
	{\begin{list}
		{\textup{(\theenumi)} }
		{\usecounter{enumi}
			\setlength{\labelwidth}{0pt}
			\setlength{\labelsep}{0pt}
			\setlength{\leftmargin}{0pt}
			\setlength{\itemsep}{0pt}
			\setlength{\topsep}{0pt}
		\renewcommand{\theenumi}{\roman{enumi}}
		}}
	{\end{list}}

\newenvironment{altenumeratelevel2}
{\begin{list}
		{\textup{(\theenumi)} }
		{\usecounter{enumii}
			\setlength{\labelwidth}{2em}
			\setlength{\labelsep}{0pt}
			\setlength{\leftmargin}{2em}
			\setlength{\itemsep}{2pt}
			\setlength{\topsep}{2pt}
			\setlength{\itemindent}{0pt}
			\renewcommand{\theenumi}{\arabic{enumii}}
	}}
	{\end{list}}

\newcommand{\eqmathbox}[2][N]{\eqmakebox[#1]{$\displaystyle#2$}}

\numberwithin{equation}{section}

\makeatletter
\def\@seccntformat#1{%
	\protect\textup{\protect\@secnumfont
		\ifnum\pdfstrcmp{subsection}{#1}=0 \bfseries\fi% subsection # in \bfseries
		\csname the#1\endcsname
		\protect\@secnumpunct
	}%
}  
\makeatother

\makeatletter
\def\@tocline#1#2#3#4#5#6#7{\relax
	\ifnum #1>\c@tocdepth % then omit
	\else
	\par \addpenalty\@secpenalty\addvspace{#2}%
	\begingroup \hyphenpenalty\@M
	\@ifempty{#4}{%
		\@tempdima\csname r@tocindent\number#1\endcsname\relax
	}{%
	\@tempdima#4\relax
}%
\parindent\z@ \leftskip#3\relax \advance\leftskip\@tempdima\relax
\rightskip\@pnumwidth plus4em \parfillskip-\@pnumwidth
#5\leavevmode\hskip-\@tempdima
\ifcase #1
\or\or \hskip 1em \or \hskip 2em \else \hskip 3em \fi%
#6\nobreak\relax
\hfill\hbox to\@pnumwidth{\@tocpagenum{#7}}\par% <---- \dotfill -> \hfill
\nobreak
\endgroup
\fi}
\makeatother

\newcommand{\Ad}{{\mathrm{Ad}}}

\newcommand{\abs}[1]{{\lvert{#1}\rvert}}

\newcommand{\ad}{{\mathrm{ad}}}

\newcommand{\blank}{{\,\_\,}}

\newcommand{\botimes}[1]{{\,\otimes_{#1}\,}}

\newcommand{\Coker}{{\mathrm{Coker}}}
\newcommand{\cotimes}[1]{{\,\widehat{\otimes}_{#1}\,}}
\newcommand{\cont}{{\mathrm{cts}}}
\newcommand{\cpltn}{{\,\widehat{}}}

\renewcommand{\dim}{{\mathrm{dim}}}
\newcommand{\defeq}{\vcentcolon=}

\newcommand{\End}{{\mathrm{End}}}

\newcommand{\eqdef}{=\vcentcolon}
\newcommand{\ev}{{\mathrm{ev}}}
\newcommand{\Ext}{{\mathrm{Ext}}}

\newcommand{\GL}{{\mathrm{GL}}}

\newcommand{\Hom}{{\mathrm{Hom}}}

\newcommand{\hy}{{\mathrm{hy}}}

\renewcommand{\Im}{{\mathrm{Im}}}

\newcommand{\id}{{\mathrm{id}}}
\newcommand{\Ind}{{\mathrm{Ind}}}
\newcommand{\inv}{{\mathrm{inv}}}
\newcommand{\indotimes}{{\,\otimes_{K,\iota}\,}}
\newcommand{\indcotimes}{{\,\widehat{\otimes}_{K,\iota}\,}}

\newcommand{\Ker}{{\mathrm{Ker}}}

\newcommand{\la}{{\mathrm{la}}}

\newcommand{\lra}{\longrightarrow}

\newcommand{\lto}{\longmapsto}

\renewcommand{\mod}{{\,\,\mathrm{mod}\,\,}}

\newcommand{\norm}[1]{\lVert{#1}\rVert}

\newcommand{\mto}{\mapsto}

\newcommand{\pr}{{\mathrm{pr}}}
\newcommand{\Proj}{{\mathrm{Proj}}}

\newcommand{\projotimes}{{\,\otimes_{K,\pi}\,}}
\newcommand{\projcotimes}{{\,\widehat{\otimes}_{K,\pi}\,}}

\newcommand{\ra}{\rightarrow}

\newcommand{\res}[1]{{\!\,\mid_{#1}}}
\newcommand{\rig}{{\mathrm{rig}}}

\newcommand{\Sp}{{\mathrm{Sp}}}
\newcommand{\Spec}{{\mathrm{Spec}}}

\newcommand{\nsubset}{{\not\subset}}
\newcommand{\sm}{{\mathrm{sm}}}

\newcommand{\unts}{^{\times}}
\newcommand{\ul}[1]{{\underline{#1}}}

\let\originalmiddle\middle
\renewcommand{\middle}[1]{\,\originalmiddle#1\,}

\newcommand{\llrrbracket}[1]{\mkern-3mu\left[\mkern-2.5mu\left[#1\right]\mkern-2.5mu\right]}

\newcommand{\llrrparen}[1]{\mkern-3mu\left(\mkern-3mu\left(#1\right)\mkern-3mu\right)}

\makeatletter
\let\save@mathaccent\mathaccent
\newcommand*\if@single[3]{%
	\setbox0\hbox{${\mathaccent"0362{#1}}^H$}%
	\setbox2\hbox{${\mathaccent"0362{\kern0pt#1}}^H$}%
	\ifdim\ht0=\ht2 #3\else #2\fi
}
\newcommand*\rel@kern[1]{\kern#1\dimexpr\macc@kerna}
\newcommand*\widebar[1]{\@ifnextchar^{{\wide@bar{#1}{0}}}{\wide@bar{#1}{1}}}
\newcommand*\wide@bar[2]{\if@single{#1}{\wide@bar@{#1}{#2}{1}}{\wide@bar@{#1}{#2}{2}}}
\newcommand*\wide@bar@[3]{%
	\begingroup
	\def\mathaccent##1##2{%
		\let\mathaccent\save@mathaccent
		\if#32 \let\macc@nucleus\first@char \fi
		\setbox\z@\hbox{$\macc@style{\macc@nucleus}_{}$}%
		\setbox\tw@\hbox{$\macc@style{\macc@nucleus}{}_{}$}%
		\dimen@\wd\tw@
		\advance\dimen@-\wd\z@
		\divide\dimen@ 3
		\@tempdima\wd\tw@
		\advance\@tempdima-\scriptspace
		\divide\@tempdima 10
		\advance\dimen@-\@tempdima
		\ifdim\dimen@>\z@ \dimen@0pt\fi
		\rel@kern{0.6}\kern-\dimen@
		\if#31
		\overline{\rel@kern{-0.6}\kern\dimen@\macc@nucleus\rel@kern{0.4}\kern\dimen@}%
		\advance\dimen@0.4\dimexpr\macc@kerna
		\let\final@kern#2%
		\ifdim\dimen@<\z@ \let\final@kern1\fi
		\if\final@kern1 \kern-\dimen@\fi
		\else
		\overline{\rel@kern{-0.6}\kern\dimen@#1}%
		\fi
	}%
	\macc@depth\@ne
	\let\math@bgroup\@empty \let\math@egroup\macc@set@skewchar
	\mathsurround\z@ \frozen@everymath{\mathgroup\macc@group\relax}%
	\macc@set@skewchar\relax
	\let\mathaccentV\macc@nested@a
	\if#31
	\macc@nested@a\relax111{#1}%
	\else
	\def\gobble@till@marker##1\endmarker{}%
	\futurelet\first@char\gobble@till@marker#1\endmarker
	\ifcat\noexpand\first@char A\else
	\def\first@char{}%
	\fi
	\macc@nested@a\relax111{\first@char}%
	\fi
	\endgroup
}

\newcommand{\BB}{{\mathbb {B}}}
\newcommand{\BC}{{\mathbb {C}}}

\newcommand{\BF}{{\mathbb {F}}}
\newcommand{\BG}{{\mathbb {G}}}

\newcommand{\BN}{{\mathbb {N}}}

\newcommand{\BP}{{\mathbb {P}}}
\newcommand{\BQ}{{\mathbb {Q}}}
\newcommand{\BR}{{\mathbb {R}}}

\newcommand{\BZ}{{\mathbb {Z}}}

\newcommand{\CA}{{\mathcal {A}}}
\newcommand{\CB}{{\mathcal {B}}}
\newcommand{\CC}{{\mathcal {C}}}

\newcommand{\CE}{{\mathcal {E}}}
\newcommand{\CF}{{\mathcal {F}}}

\newcommand{\CH}{{\mathcal {H}}}
\newcommand{\CI}{{\mathcal {I}}}
\newcommand{\CJ}{{\mathcal {J}}}
\newcommand{\CK}{{\mathcal {K}}}
\newcommand{\CL}{{\mathcal {L}}}
\newcommand{\CM}{{\mathcal {M}}}

\newcommand{\CO}{{\mathcal {O}}}

\newcommand{\CU}{{\mathcal {U}}}

\newcommand{\CX}{{\mathcal {X}}}
\newcommand{\CY}{{\mathcal {Y}}}

\newcommand{\bB}{{\mathrm{\bf B}}}

\newcommand{\bG}{{\mathrm{\bf G}}}

\newcommand{\bL}{{\mathrm{\bf L}}}

\newcommand{\bP}{{\mathrm{\bf P}}}
\newcommand{\bQ}{{\mathrm{\bf Q}}}

\newcommand{\bT}{{\mathrm{\bf T}}}
\newcommand{\bU}{{\mathrm{\bf U}}}

\newcommand{\Fa}{{\mathfrak {a}}}
\newcommand{\Fb}{{\mathfrak {b}}}

\newcommand{\Fd}{{\mathfrak {d}}}

\newcommand{\Fg}{{\mathfrak {g}}}

\newcommand{\Fl}{{\mathfrak {l}}}
\newcommand{\Fm}{{\mathfrak {m}}}

\newcommand{\Fp}{{\mathfrak {p}}}

\newcommand{\Ft}{{\mathfrak {t}}}
\newcommand{\Fu}{{\mathfrak {u}}}

\newcommand{\Fx}{{\mathfrak {x}}}

\newcommand{\Fz}{{\mathfrak {z}}}

\newcommand{\FM}{{\mathfrak {M}}}

\newcommand{\FX}{{\mathfrak {X}}}
\newcommand{\FY}{{\mathfrak {Y}}}

\title[Equivariant Vector Bundles on the Drinfeld Upper Half Space]{Equivariant Vector Bundles on the Drinfeld Upper Half Space over a Local Field of Positive Characteristic}
\author{Georg Linden}
\date{\today}
\DeclareRobustCommand\longtwoheadrightarrow
{\relbar\joinrel\twoheadrightarrow}
\DeclareRobustCommand\longhookrightarrow
{\lhook\joinrel\longrightarrow}

\newcommand{\unif}{{\pi}}
\newcommand{\aan}{{\mathrm{rig}}}

\newcommand{\alg}{{\mathrm{alg}}}

\begin{document}

\begin{abstract}
	We describe the locally analytic $\GL_d(K)$-representations which arise as the global sections of homogeneous vector bundles on the projective space restricted to the Drinfeld upper half space over a non-archimedean local field $K$. 
	We thereby generalize work of Orlik \cite{Orlik08EquivVBDrinfeldUpHalfSp} for $p$-adic fields to the effect that it becomes applicable to local fields of positive characteristic.	
	Our description of this space of global sections is in terms of a filtration by subrepresentations, and a characterization of the resulting subquotients via  adaptations of the functors $\CF^G_P$ considered by Orlik--Strauch \cite{OrlikStrauch15JordanHoelderSerLocAnRep} and Agrawal--Strauch \cite{AgrawalStrauch22FromCatOLocAnRep}.

	For a local field $K$ of positive characteristic, we also determine the locally analytic (resp.\ continuous) characters of $K^\times$ with values in $K$-Banach algebras which are integral domains (resp.\ with values in finite extensions of $K$) in an appendix.
\end{abstract}

\maketitle

\tableofcontents

\section*{Introduction}
%\addcontentsline{toc}{section}{\protect\numberline{}Introduction}

Let $K$ be a non-archimedean local field of residue characteristic $p$.
The Drinfeld upper half space of dimension $d\in \BN$ over $K$ is defined as the complement
\begin{equation*}
	\CX \defeq \BP_K^d \Big\backslash \bigcup_{H \subsetneq K^{d+1}} \BP(H) 
\end{equation*}
of all $K$-rational hyperplanes in $d$-dimensional projective space, and naturally carries a rigid-analytic structure.
It is of interest in arithmetic geometry for a number of reasons, one of which is the study of its cohomology.
This aspect originates from Drinfeld conjecture \cite{Drinfeld76CovpAdicSymReg} (specified by Carayol \cite{Carayol90NonAbLubinTateTh}) that the (compactly supported) $\ell$-adic cohomology of the \'etale coverings of $\CX$ realises the supercuspidal part of the local Langlands and Jacquet--Langlands correspondences which by now has been proven \cite{Faltings94TraceFormulaDrinfeld, Harris97SupercuspRepCohomDrinfeld, HarrisTaylor01GeomCohomSimpleShimuraVar}.
Results in this direction include the computation of the \'etale cohomology with torsion coefficients prime to $p$, and, for $p$-adic $K$, the de Rham cohomology of $\CX$ by Schneider and Stuhler \cite{SchneiderStuhler91CohompAdicSymSp}.
The compactly supported $\ell$-adic cohomology, with $\ell \neq p$, has been determined by Dat \cite{Dat06EspSymDrinfeldCorrespLanglandsLoc}.
More recently Colmez, Dospinescu and Nizio{\l} computed the $p$-adic \'etale and pro-\'etale cohomology of $\CX$, for $p$-adic $K$ \cite{ColmezDospinescuNiziol20CohompAdicSteinSp}.
They also show that, for $d=1$ and $K=\BQ_p$, the $p$-adic \'etale cohomology of the \'etale coverings of $\CX$ encodes the $p$-adic local Langlands correspondence for $2$-dimensional de Rham representations (of weight $0$ and $1$) \cite{ColmezDospinescuNiziol20CohompAdicTourDrinfeldCasDim1}.

On a slightly different note in \cite{Schneider92CohomLocSystemspAdicUnifVar}, for $p$-adic $K$, Schneider introduced the notion of ($p$-adic) holomorphic discrete series representations of $\GL_{d+1}(K)$, when studying the cohomology of local systems on certain projective varieties uniformized by $\CX$.
These representations occur as the space of global rigid analytic sections $H^0 (\CX,\CE)$ of $\GL_{d+1,K}$-equivariant vector bundles $\CE$ on $\BP_K^d$ restricted to $\CX$.
Their strong dual spaces are locally analytic representations as introduced by F\'eaux de Lacroix and Schneider--Teitelbaum.
Extending previous work by Y.\ Morita for the ${\rm SL}_2$-case, descriptions of these resulting locally analytic $\GL_{d+1}(K)$-representations were given by Schneider and Teitelbaum \cite{SchneiderTeitelbaum02pAdicBoundVal} for the canonical bundle $\Omega_{\BP_K^d}^d$, by Pohlkamp \cite{Pohlkamp02RandwerteHolomFunkt} for the structure sheaf $\CO_{\BP_K^d}$, and by Orlik \cite{Orlik08EquivVBDrinfeldUpHalfSp} for general $\CE$.
\\

In this work our goal is to describe the global rigid analytic sections of homogeneous vector bundles on $\BP_K^d$ restricted to the Drinfeld upper half space $\CX$ over a general non-archimedean local field $K$.
We thereby adapt Orlik's methods from \cite{Orlik08EquivVBDrinfeldUpHalfSp} in a way that they are applicable in the case when $K$ has positive characteristic as well.
We note that the coherent cohomology of such vector bundles is solely concentrated in the global sections because $\CX$ is a Stein space.

The basic definitions and results for locally analytic representations transfer from the $p$-adic case to the setting over a non-archimedean field of positive characteristic.
This was already remarked by Gräf in \cite{Graef21BoundDistr}.
Even the anti-equivalence between locally analytic representations of a locally analytic Lie group $G$ and modules over algebras $D(G)$ of locally analytic distributions realised by passing to the strong dual spaces is still valid, and we make frequent use of it.

Thus, for a homogeneous vector bundle $\CE$ on $\BP_K^d$, the strong dual $H^0(\CX,\CE)'_b$ of the global sections on $\CX$ continues to be a locally analytic $\GL_{d+1}(K)$-representation.
Also Orlik's technique from \cite{Orlik08EquivVBDrinfeldUpHalfSp} which takes advantage of the geometric structure of the divisor at infinity $\BP_K^d \setminus \CX$ via a certain spectral sequence is still applicable.
%This spectral sequence involves the induction of local cohomology groups of Schubert varieties $\BP_K^j \subset \BP_K^d$ as representations of parabolic subgroups of $\GL_{d+1}(K)$.
The result is a filtration of $H^0 (\CX,\CE)$ by closed $\GL_{d+1}(K)$-invariant subspaces.
Moreover, the strong duals of the subquotients of this filtration can be described as extensions of certain locally analytic $\GL_{d+1}(K)$-representations.
In analysing the locally analytic representations which arise here however, we have to take an approach different from the one for a $p$-adic field in \cite{Orlik08EquivVBDrinfeldUpHalfSp}.
Our main result is the following description.

\begin{theoremalphabetical}[\Cref{Thm 3 - Orliks theorem on global sections of an equivariant vector bundle on the DHS}, \Cref{Thm 3 - Main theorem}]\label{Thm 0 - Main Theorem}
	Let $K$ be a non-archimedean local field and $\CE$ a $\GL_{d+1,K}$-equivariant vector bundle on $\BP_K^d$.
	Then there exists a filtration by closed $\GL_{d+1}(K)$-invariant subspaces
	\begin{equation*}
		H^0 (\CX,\CE)  = V^{d} \supset V^{d-1} \supset \ldots \supset V^1 \supset V^0 = H^0(\BP_K^d, \CE),
	\end{equation*}
	and, for $j = 1,\ldots, d$, there are extensions of locally analytic $\GL_{d+1}(K)$-representations
	%	\begin{equation*}
		%	\begin{aligned}
			%		0 \lra  D\big(\GL_{d+1}(K) \big) \cotimes{D(\mathfrak{gl}_{d+1},P_{(d-j+1,j)})} \bigg( \widetilde{H}^{j}_{(\BP_K^{d-j})^\rig} (\BP_K^d, \CE) \cotimes{K} \Big( v^{\GL_{j}(K)}_{B_{j}} \Big)'_b \bigg) \quad& \\
			%		\lra V^j/V^{j-1} \lra \Big( H^j (\BP_K^d, \CE)' \botimes{K} v^{\GL_{d+1}(K)}_{P_{(d-j+1,1,\ldots,1)}} \Big)'_b &\lra 0 ,
			%	\end{aligned}
		%	\end{equation*}
	\begin{equation*}
		\begin{aligned}
			0 \lra &{\hspace{2pt}} H^j (\BP_K^d, \CE)' \botimes{K} v^{\GL_{d+1}(K)}_{P_{(d-j+1,1,\ldots,1)}}  \lra \big( V^j/V^{j-1} \big)'_b\\
			&\lra  \bigg( D\big(\GL_{d+1}(K) \big) \cotimes{D(\mathfrak{gl}_{d+1},P_{(d-j+1,j)}), \iota} \Big( \widetilde{H}^{j}_{(\BP_K^{d-j})^\rig} (\BP_K^d, \CE) \cotimes{K} \big( v^{\GL_{j}(K)}_{B_{j}} \big)'_b \Big) \bigg)'_b \lra 0 .
		\end{aligned}
	\end{equation*}
\end{theoremalphabetical}

We explain the objects which occur in this theorem.
We let $P_{(d-j+1,j)}$ and $P_{(d-j+1,1,\ldots,1)}$ denote the standard parabolic subgroups of $\GL_{d+1}(K)$ corresponding to the decompositions in their respective index.
Moreover, $\GL_j(K)$ is considered as a subgroup of the standard Levi factor $\GL_{d-j+1}(K) \times \GL_j(K)$ of $P_{(d-j+1,j)}$, and $B_j$ denotes the standard Borel subgroup of $\GL_j(K)$.
The representations $v^{\GL_{d+1}(K)}_{P_{(d-j+1,1,\ldots,1)}}$ and $v^{\GL_{j}(K)}_{B_{j}}$ are smooth (generalized) Steinberg representations with coefficients in $K$, and $P_{(d-j+1,j)}$ acts on $v^{\GL_{j}(K)}_{B_{j}}$ via inflation.

Furthermore, $D\big(\Fg\Fl_{d+1},P_{(d-j+1,j)}\big)$ is a certain subalgebra of the locally analytic distribution algebra $D\big(\GL_{d+1}(K)\big)$:
For any non-archimedean Lie group $G$, we define the hyperalgebra $\hy(G)$ of $G$ which embeds into $D(G)$ as a subalgebra.
For a locally analytic subgroup $H \subset G$, the subalgebra $D(\Fg, H)$ is then defined to be generated by $\hy(G)$ and $D(H)$. 
The definition of this hyperalgebra $\hy(G)$ is inspired by the distribution algebra of an algebraic group $\bG$ as treated for example in \cite{Jantzen03RepAlgGrp}.
It can be canonically identified with the latter when $G$ arises as the $K$-valued points of such $\bG$ which is smooth.
In particular if ${\rm char}(K)=0$, $\hy(G)$ agrees with the universal enveloping algebra of the Lie algebra $\Fg$ of $\bG$.
Therefore $D(\Fg, H)$ generalizes a construction of Orlik and Strauch \cite[\S 3.4]{OrlikStrauch15JordanHoelderSerLocAnRep} for $p$-adic $K$.
The value of this hyperalgebra to us lies in the fact that there is a non-degenerate pairing between $\hy(G)$ and the space of germs of locally analytic functions on $G$ at the identity element even when ${\rm char}(K)>0$ (\Cref{Prop 1 - Pairing for germs of locally analytic functions}).
Hence, one might informally say that the algebra $D(\Fg, H)$ incorporates an infinitesimal neighbourhood around $H$.

Finally, there is the subspace
\begin{equation*}
	\widetilde{H}^{j}_{(\BP_K^{d-j})^\rig} (\BP_K^d, \CE) \defeq \Ker \Big( H^{j}_{(\BP_K^{d-j})^\rig} (\BP_K^d, \CE) \ra H^j (\BP_K^d, \CE) \Big)
\end{equation*}
of the local cohomology with respect to the Schubert variety $\big(\BP_K^{d-j}\big)^\rig $ viewed as a rigid-analytic subvariety of $\BP_K^d$.
We show that this subspace is canonically equipped with the structure of a $D\big(\Fg\Fl_{d+1}, P_{(d-j+1,j)}\big)$-module.
Taking the completed inductive tensor product of it with $D\big(\GL_{d+1}(K) \big)$ yields the $D\big(\GL_{d+1}(K) \big)$-module
\begin{equation}\label{Eq 0 - Tensored up representation}
	D\big(\GL_{d+1}(K) \big) \cotimes{D(\Fg\Fl_{d+1},P_{(d-j+1,j)}), \iota} \Big( \widetilde{H}^{j}_{(\BP_K^{d-j})^\rig} (\BP_K^d, \CE) \cotimes{K} \big( v^{\GL_{j}(K)}_{B_{j}} \big)'_b \Big) . \tag{$*$}
\end{equation}
Here $v^{\GL_{j}(K)}_{B_{j}}$ carries the finest locally convex topology.
The strong dual of the $D\big(\GL_{d+1}(K)\big)$-module \eqref{Eq 0 - Tensored up representation} is a locally analytic $\GL_{d+1}(K)$-representation by the aforementioned anti-equi\-va\-lence for locally analytic representations.

For a $p$-adic field, this relates to the description of \cite[Thm.\ 1]{Orlik08EquivVBDrinfeldUpHalfSp} for $H^0(\CX,\CE)'_b$ as follows:
With the filtration of $H^0(\CX,\CE)$ being the same, Orlik there obtains a certain subspace of a locally analytic induced representation in place of the strong dual space of \eqref{Eq 0 - Tensored up representation}; the two other terms of the short strictly exact sequence in \Cref{Thm 0 - Main Theorem} remain unchanged.
However, besides the isomorphism induced a posteriori in this way, there is a more intrinsic connection between \eqref{Eq 0 - Tensored up representation} and the representation Orlik arrives at.
Indeed, the subspace he obtains can be characterized as $\CF^{\GL_{d+1}(K)}_{P_{(d-j+1,j)}} \big( \widetilde{H}^{j}_{\BP_K^{d-j}} (\BP_K^d, \CE) , v^{\GL_{j}(K)}_{B_{j}} \big)$.
The functors $\CF^G_P$ used here were introduced by Orlik and Strauch \cite{OrlikStrauch15JordanHoelderSerLocAnRep} for the more general setup of a split connected reductive group $\bG$ over $K$ and a standard parabolic subgroup $\bP \subset \bG$.
Let $G = \bG(K)$, $P = \bP(K)$, and let $\Fg$, $\Fp$ be the respective Lie algebras.
For a $U(\Fg)$-module $M \in \CO_\alg^\Fp$ (where $\CO_\alg^\Fp$ is a certain subcategory of the BGG category $\CO$ for $\Fg$) and an admissible smooth representation $V$ of the standard Levi subgroup $L_\bP \subset P$, this functor yields an admissible locally analytic $G$-representation $\CF^G_P(M,V)$.
It is expected that their duals can be described as 
\[\CF^G_P(M,V)' \cong D(G) \botimes{D(\Fg,P)} \big( M \botimes{K} V' \big) \]
(for trivial $V=K$ this is \cite[Prop.\ 3.7]{OrlikStrauch15JordanHoelderSerLocAnRep}). 
Furthermore, in \cite{AgrawalStrauch22FromCatOLocAnRep} Agrawal and Strauch constructed functors which expand the functors $\CF^G_P$ and are defined via taking a tensor product with $D(G)$ over $D(\Fg,P)$ in a similar way.

A $U(\Fg)$-module $M$ in the category $\CO_\alg^\Fp$ necessarily is finitely generated.
Thus it can be endowed with a locally convex topology via some epimorphism $U(\Fg)^{\oplus n} \twoheadrightarrow M$ using the subspace topology $U(\Fg) \subset D(G)$.
The admissible smooth representation $V$ in turn can be considered with the finest locally convex topology.
To compare Orlik's description to \eqref{Eq 0 - Tensored up representation} we show:

\begin{propositionalphabetical}[{\Cref{Lemma 3 - Comparision between Orlik-Strauch functors and topological tensor product}}]
	There is a topological isomorphism of $D(G)$-modules
	\begin{equation*}
		D(G) \cotimes{D(\Fg,P),\iota} \big( M \projcotimes V'_b \big) \cong D(G) \botimes{D(\Fg,P)} \big( M \botimes{K} V' \big) 
	\end{equation*}
	in the sense that $D(G) \botimes{D(\Fg,P),\iota} \big( M \projotimes V'_b \big)$ topologized via the above already is complete, and this topology agrees with the canonical Fr\'echet topology induced from it being a coadmissible (abstract) $D(G)$-module, cf.\ \cite{AgrawalStrauch22FromCatOLocAnRep}.
\end{propositionalphabetical}

Moreover, (the kernel of) the algebraic local cohomology group $\widetilde{H}^{j}_{\BP_K^{d-j}} (\BP_K^d, \CE)$ with respect to the Schubert variety is an element of $\CO_\alg^{\Fp_{(d-j+1,j)}}$.
On the other hand, one can consider it as a subspace of $\widetilde{H}^{j}_{(\BP_K^{d-j})^\rig} (\BP_K^d, \CE)$ (see \Cref{Cor 2 - Projective limit description of rigid local cohomology of Schubert varieties}), and we prove in \Cref{Cor 3 - The two topologies on the local cohomology groups agree} that the subspace topology agrees with the locally convex topology induced via some epimorphism $U(\Fg)^{\oplus n} \twoheadrightarrow \widetilde{H}^{j}_{\BP_K^{d-j}} (\BP_K^d, \CE)$.
This yields a canonical topological isomorphism of $D(G)$-modules between \eqref{Eq 0 - Tensored up representation} and 
\begin{equation*}
	D\big(\GL_{d+1}(K) \big) \botimes{D(\Fg\Fl_{d+1},P_{(d-j+1,j)})} \Big( \widetilde{H}^{j}_{\BP_K^{d-j}} (\BP_K^d, \CE) \botimes{K} \big( v^{\GL_{j}(K)}_{B_{j}} \big)' \Big)
\end{equation*}
endowed with its canonical Fr\'echet topology.

Orlik's proof in \cite{Orlik08EquivVBDrinfeldUpHalfSp} uses that the algebraic local cohomology groups $H^j_{\BP_K^{d-j}}(\BP_K^d,\CE)$ are finitely generated over the universal enveloping algebra $U(\Fg\Fl_{d+1})$.
Since for a field of positive characteristic this has an analogue only in exceptional cases, our strategy is to employ the non-degenerate pairing between $\hy\big(\GL_{d+1}(K)\big)$ and the germs of locally analytic functions on $\GL_{d+1}(K)$ in a more direct manner instead.
This comes at the cost that the necessary arguments from functional analysis are more involved.
\\

The multiplicative group $K\unts$ is among the most basic examples of a locally $K$-analytic Lie group.
We include an \Cref{Sect - Continuous and Locally Analytic Characters} where, for a local field $K$ of positive characteristic $p$, we investigate the one-dimensional continuous and locally analytic representations of $K\unts$ (i.e.\ characters) which take values in a non-archimedean field of the same characteristic $p$.

Compared with the $p$-adic situation, we find that there are significantly less locally analytic characters in relation to continuous ones (\Cref{Cor A2 - Ring of locally analytic characters in equal characteristic}).
Moreover, the locally analytic characters of $K\unts$ behave rigidly in a sense.
It suffices here to focus on the subgroup of principal units $1 +\Fm_K \subset K\unts$ where $\Fm_K$ is the maximal ideal of the ring of integers of $K$.
This subgroup constitutes the non-discrete part under the usual decomposition of $K\unts$, and for its characters we obtain:

\begin{theoremalphabetical}[\Cref{Thm A2 - Locally analytic characters in equal characteristic}, \Cref{Cor A2 - Ring of locally analytic characters in equal characteristic}]
	Let $K$ be local field of characteristic $p>0$, and let $C$ be a complete extension of $K$.
	Then every locally analytic character
	\begin{equation*}
		\chi \colon 1 + \Fm_K \lra C\unts 
	\end{equation*}
	factors over $1+ \Fm_K \subset C\unts$, and there exists $c\in \BZ_p$ such that $\chi = \chi_c$ where
	\begin{equation*}
		\chi_c (z) = z^c \defeq \sum_{n=0}^\infty \binom{c}{n} (z-1)^n \quad\text{, for all $z\in 1 +\Fm_K $.}
	\end{equation*}
	Moreover, the values of all $p^i$-th hyperderivatives $D^{(p^i)}\chi$ at $1$ are contained in $\BF_p \subset K$, and $c$ is uniquely determined by $c_i \equiv D^{(p^i)} \chi(1) \mod (p)$, for the $p$-adic expansion $c=\sum_{i= 0}^\infty c_i p^i$.

	This yields an isomorphism $\End_{\la}(1+\Fm_K) \cong \BZ_p$ of topological rings where the former is the ring of locally analytic endomorphisms with multiplication given by composition and carrying the compact-open topology.
\end{theoremalphabetical}

We want to add more details about the content of some of the individual sections.
The first chapter covers the theory of locally analytic representations necessary for our goal.
For the convenience of the reader, we decide to recapitulate the foundational theory in detail and for the most part with proofs in the first five sections there.

The sixth section treats the space $C^\la_x(X,V)$ of germs of locally analytic functions at $x \in X$ with values in a locally convex vector space $V$, for a locally analytic manifold $X$.
For a locally analytic Lie group $G$, the strong dual $D_e(G)$ of this space $C^\la_e(G,K)$ at the identity element $e$ embeds into the algebra of locally analytic distributions $D(G)$.
The hyperalgebra $\hy(G)$ is then defined to consists of those elements of $D_e(G)$ which vanish on some power of the unique maximal ideal of $C^\la_e(G,K)$.
Moreover, following Orlik--Strauch \cite{OrlikStrauch15JordanHoelderSerLocAnRep} we consider subalgebras $D(\Fg,H) \subset D(G)$ generated by $\hy(G)$ and $D(H)$, for locally analytic subgroups $H\subset G$.
Analogously to Agrawal--Strauch \cite{AgrawalStrauch22FromCatOLocAnRep}, modules over these subalgebras correspond to so called locally analytic $(\hy(G), H)$-modules.

The final section concerns endowing the $K$-rational points $X(K)$ of a smooth, separated rigid analytic $K$-space $X$ of countable type with the structure of a locally $K$-analytic manifold.
There we also show that $\hy\big(\bG(K)\big)$ and ${\rm Dist}(\bG)$ agree, for a smooth algebraic group $\bG$ over $K$.
\\

The second chapter is devoted to showing that the strong dual spaces of $H^0(\CX,\CE)$ and $\widetilde{H}^{d-j}_{(\BP_K^{j})^\rig}(\BP_K^d, \CE)$ are locally analytic representations of compact type.

While for $H^0(\CX,\CE)$ this is done like in the $p$-adic case for $\CE = \Omega_{\BP_K^d}$ considered by Schneider--Teitelbaum \cite{SchneiderTeitelbaum02pAdicBoundVal}, the local cohomology groups require some preparation.
The main step there is to give a description 
\begin{equation*}
	\widetilde{H}^k_{(\BP_K^{j})^\rig}(\BP_K^d, \CE) \cong \varprojlim_{n\in \BN} \widetilde{H}^k_{\BP_K^{j}(\varepsilon_n)^-}(\BP_K^d, \CE)
\end{equation*}
where $\varepsilon_n \defeq \abs{\unif}^n $, for a uniformizer $\unif$ of $K$.
In the limit on the right hand side, the local cohomology groups with respect to $\varepsilon_n$-neigh\-bour\-hoods $\BP_K^{j}(\varepsilon_n)^-$ around the Schubert variety are Banach spaces.
To take this limit in a controlled way we show that the differentials of a certain \v{C}ech complex which computes the cohomology of $\CE$ on the complement $\BP_K^d \setminus \BP_K^j$ are strict homomorphisms.
The topology on this \v{C}ech complex comes from certain affinoid subdomains of the principal open subsets $D_+(X_i) \subset \BP_K^d$.
Thereby we correct a flaw in the proof of \cite[Lemma 1.3.1]{Orlik08EquivVBDrinfeldUpHalfSp}.

Then $\widetilde{H}^k_{\BP_K^{d-j}(\varepsilon_n)^-}(\BP_K^d, \CE)'_b$ is a locally analytic $P^{n+1}_{(d-j+1,j)}$-representation where $P^{n+1}_{(d-j+1,j)}$ is a certain open subgroup of $\GL_{d+1}(\CO_K)$ which stabilizes $\BP_K^{d-j}(\varepsilon_n)^-$.
Ultimately we can conclude that $\widetilde{H}^{j}_{(\BP_K^{d-j})^\rig}(\BP_K^d, \CE)'_b$ is a locally analytic $\big(\hy\big( \GL_{d+1}(K) \big), P_{(d-j+1,j)} \big)$-module.
\\

The last chapter includes the proof of \Cref{Thm 0 - Main Theorem}.
In the first section we recall Orlik's method of using a certain acyclic ``fundamental complex'' of \'etale sheaves on the complement $\BP_K^d \setminus \CX$ considered as a closed pseudo-adic subspace.
This complex captures the combinatorial geometry of the complement and is available for period domains more generally, cf.\ \cite{Orlik05CohomPerDomRedGrpLocF}.
As mentioned, a spectral sequence associated with it yields the filtration in \Cref{Thm 0 - Main Theorem} and extensions
\begin{equation}\label{Eq 0 - Extension for first description}
	\begin{aligned}
		0 \lra &{\hspace{2pt}} H^j (\BP_K^d, \CE)' \botimes{K} v^{\GL_{d+1}(K)}_{P_{(d-j+1,1,\ldots,1)}}  \lra \big( V^j/V^{j-1} \big)'_b\\
		&\lra  \varinjlim_{n\in \BN} \Ind^{\GL_{d+1}(\CO_K)}_{P^n_{(d-j+1,j)}} \Big( \widetilde{H}^j_{\BP_K^{d-j} (\varepsilon_n)} (\BP_K^d, \CE)'_b \botimes{K} v^{P^n_{(d-j+1,j)}}_{P^n_{(d-j+1,1,\ldots,1)}} \Big) \lra 0
	\end{aligned} \tag{$\ast \ast$}
\end{equation}
of locally analytic $\GL_{d+1}(K)$-representations.

The next two sections then contain further analysis of the last term occurring in \eqref{Eq 0 - Extension for first description}.
Roughly outlined our approach is to embed this term into $C^\la \big(\GL_{d+1}(\CO_K), W_j \big)$ where 
\begin{equation*}
	W_j \defeq \varinjlim_{n\in \BN} \Big( \widetilde{H}^j_{\BP_K^{d-j} (\varepsilon_n)^-} (\BP_K^d, \CE)'_b \botimes{K} v^{P^n_{(d-j+1,j)}}_{P^n_{(d-j+1,1,\ldots,1)}} \Big).
\end{equation*}
For elements of $C^\la \big(\GL_{d+1}(\CO_K), W_j \big)$ the property of being invariant, for some $n\in \BN$, under the subgroup $P^n_{(d-j+1,j)}$ transfers to being invariant under the action of $\bP_{(d-j+1,j)}(\CO_K)$ and the ``infinitesimal'' action of $\hy\big( \GL_{d+1}(K)\big)$.
Dualizing then eventually results in an isomorphism between the last term of \eqref{Eq 0 - Extension for first description} and the strong dual of \eqref{Eq 0 - Tensored up representation}.

In the last section we compare our description in the case of a $p$-adic field $K$ to the one given by Orlik \cite{Orlik08EquivVBDrinfeldUpHalfSp} and the functors $\CF^G_P$ due to Orlik and Strauch \cite{OrlikStrauch15JordanHoelderSerLocAnRep}.
This comparison and the generalization of the $\CF^G_P$ due to Agrawal and Strauch \cite{AgrawalStrauch22FromCatOLocAnRep} then motivates the definition of an analogue of the functors $\CF^G_P$ for a general non-archimedean local field.

%Consider the more general situation of a connected split reductive group $\bG$ over $K$ and a (standard) parabolic subgroup $\bP \subset \bG$.
%For a $U(\Fg)$-module $M$ in the category $\CO^\Fp_\alg$ (see \cite{OrlikStrauch15JordanHoelderSerLocAnRep}) and an admissible smooth representation $V$ over $K$ of the standard Levi subgroup $L_\bP$, we show in \Cref{Lemma 3 - Comparision between Orlik-Strauch functors and topological tensor product} that there is a canonical topological isomorphism of $D(G)$-modules
%\begin{equation*}
%	D(G) \cotimes{D(\Fg,P),\iota} \big( M \projcotimes V'_b \big) \cong D(G) \botimes{D(\Fg,P)} \big( M \botimes{K} V'_b \big) .
%\end{equation*}
%Here the right hand side carries its canonical Fr\'echet topology induced from being a coadmissible $D(G)$-module, cf.\ \cite{AgrawalStrauch22FromCatOLocAnRep}.
%On the left hand side, $M$ acquires a topology via some epimorphism $U(\Fg)^{\oplus n} \twoheadrightarrow M$ of $U(\Fg)$-modules where $U(\Fg)\subset D(G)$ is endowed with the subspace topology.
%On $V$ we consider the finest locally convex topology.

\subsection*{Acknowledgements}

This paper comes from the author's doctoral thesis with slight corrections and additions.
I want to thank my advisor Sascha Orlik for introducing me to this diverse and interesting topic.
I am grateful to him for his valuable advise and strong support.
Moreover I want to thank Oliver Fürst, Eugen Hellmann, Roland Huber, Christoph Spenke, Matthias Strauch, and Yingying Wang for helpful comments and discussions.

A substantial part of this project was done while the author was a member of the research training group \textit{GRK 2240: Algebro-Geometric Methods in Algebra, Arithmetic and Topology} which is funded by the Deutsche Forschungsgemeinschaft.

\subsection*{Notation and Conventions}

We write $\BN = \{1,2,\ldots \}$ and $\BN_0 = \{0,1,\ldots\}$.
For multiindices $\ul{i}= (i_1,\ldots,i_n) \in \BN^n_0$, with $n\in \BN$, we set $\abs{\ul{i}}\defeq i_1+\ldots +i_n$.
For $\ul{r} = (r_1,\ldots,r_n) \in \BR^n$, we write $\ul{r}^\ul{i} \defeq r_1^{i_1} \cdots r_n^{i_n}$.

Let $K$ be a non-archimedean field with non-trivial absolute value $\abs{\blank}\colon K \ra \BR_{\geq 0}$.
We let $\CO_K \defeq \{ x \in K \mid \abs{x}\leq 1 \}$ denote its ring of integers.
For $n\in \BN$, $\ul{r}\in \BR_{>0}^n$, and $a \in K^n$, we denote by
\begin{align*}
	B^{n}_{\ul{r}}(a) &\defeq \big\{ x\in K^n \,\big\vert\, \forall i=1,\ldots,n: \abs{x_i - a_i} \leq r_i  \big\} 
\end{align*}
the ``closed'' ball of multiradius $\ul{r}$ around $a$; it is open and closed.

In this work, locally convex $K$-vector spaces play a central role.
These are topological $K$-vector spaces which have a neighbourhood basis of the origin consisting of $\CO_K$-submodules.
We will frequently refer to \cite{Emerton17LocAnVect}, \cite{PerezGarciaSchikhof10LocConvSpNonArch} and \cite{Schneider02NonArchFunctAna} for the theory of this non-archimedean functional analysis.

For locally convex $K$-vector spaces $V$ and $W$, we denote by $\CL(V,W)$ the $K$-vector space of continuous homomorphisms from $V$ to $W$.
With the strong topology of bounded convergence (respectively, the weak topology of pointwise convergence) this space becomes a locally convex $K$-vector space itself denoted by $\CL_b(V,W)$ (respectively, $\CL_s(V,W)$), see \cite[Examples p.\ 35]{Schneider02NonArchFunctAna}.
We note that, for continuous homomorphisms of locally convex $K$-vector spaces $f\colon V'\ra V$ and $h\colon W\ra W'$, the homomorphisms
\begin{alignat*}{3}
	\CL_b(V,W) &\lra \CL_b(V',W) \,,\quad g &&\lto g \circ f ,\quad \text{and}\\
	\CL_b(V,W) &\lra \CL_b(V,W') \,,\quad g &&\lto h \circ g ,
\end{alignat*}
are continuous \cite[\S 18, p.\ 113]{Schneider02NonArchFunctAna}.

Moreover, we denote the dual space of a $K$-vector space $V$ by $V^\ast \defeq \Hom_K (V,K)$.
When $V$ is a locally convex, we write $V'\defeq \CL(V,K) \subset V^\ast$ for the subspace of continuous linear forms, as well as $V'_b$ and $V'_s$ for the strong and weak dual spaces accordingly.
However, when $E$ is a $K$-Banach space, we occasionally simplify the notation by letting $E'$ denote its strong dual space.
For locally convex $K$-vector spaces $V$ and $W$, taking the transpose yields homomorphisms (see \cite[\S 0.3.8]{Emerton17LocAnVect})
\begin{equation*}
	\CL (V,W) \lra \CL(W'_b, V'_b) \qquad \text{ and } \qquad \CL(V,W) \lra \CL(W'_s, V'_s) .
\end{equation*}

On the tensor product of locally convex $K$-vector spaces $V$ and $W$, we denote the projective (respectively, inductive) tensor product topology by $V \projotimes W$ (respectively, $V \indotimes W$), cf.\ \cite[\S 17]{Schneider02NonArchFunctAna}.
We write $V \projcotimes W$ and $V \indcotimes W$ for the Hausdorff completions of the respective locally convex $K$-vector spaces.
If $V$ and $W$ both are $K$-Fr\'echet spaces or if both are semi-complete LB-spaces, the projective and inductive tensor product topology agree, see \cite[Prop.\ 17.6]{Schneider02NonArchFunctAna} and \cite[Prop.\ 1.1.31]{Emerton17LocAnVect}.
In these cases, we unambiguously write $V \botimes{K} W$ and $V\cotimes{K} W$.

The category of locally convex $K$-vector spaces with continuous homomorphisms is an example of a quasi-abelian category in the sense of \cite{Schneiders98QuasiAbCat}.
The strict morphisms are precisely the homomorphisms $f\colon V \ra W$ which are strict in the conventional sense, i.e.\ for which the induced $V/\Ker(f) \ra \Im(f)$ is a topological isomorphism.

For subgroups $H$ and $H'$ of a group $G$, we use the notation $H \cdot H'$ to denote the subset of $G$ of all elements of the form $h h'$, for $h\in H$, $h' \in H'$.

Finally, for a scheme or a rigid analytic space, we denote its structure sheaf by $\CO$ when the considered scheme or rigid analytic space is apparent from the context.

\setcounter{section}{0}

\section{Locally Analytic Representation Theory}

%This first chapter covers the theory of locally analytic representations necessary to describe the cohomology groups of equivariant vector bundles over the Drinfeld upper half space.
%The fundamental definitions and the duality theory with distribution algebras developed by F\'eaux de Lacroix and Schneider--Teitelbaum for the case of $p$-adic fields readily generalize to local fields of positive characteristic, cf.\ \cite[Part I, App.\ A]{Graef21BoundDistr}.
%Still we want to recall these results for the convenience of the reader in the first five sections.

%Afterwards we define a certain subalgebra of the algebra of locally analytic distributions of a non-archimedean Lie group.
%We call this the hyperalgebra of that Lie group as it generalizes the hyperalgebra (also called distribution algebra) of an algebraic group.
%We show that there is a non-degenerate pairing of this hyperalgebra with the space of germs of locally analytic functions.
%Moreover, we define subalgebras of the locally analytic distribution algebra generated by it and the distributions on a subgroup similarly to \cite[\S 3.4]{OrlikStrauch15JordanHoelderSerLocAnRep}.

%In the final section, we consider locally analytic manifolds associated with smooth, separated rigid analytic spaces of countable type.
%In particular, for a smooth algebraic group, we show that the hyperalgebra of the associated locally analytic Lie group agrees with the hyperalgebra of the algebraic group.

\subsection{Non-Archimedean Manifolds}

For the basics on manifolds over non-archimedean fields we follow \cite[\S 4,5]{Bourbaki07VarDiffAnFasciDeResult} and \cite[Ch.\ II]{Schneider11pAdicLieGrps}.
Let $L$ be a complete non-archimedean field with non-trivial absolute value $\abs{\blank}$.

Let $E=(E,\lVert \blank\rVert_E)$ be an $L$-Banach space, and denote by $E \llrrbracket{X_1,\ldots, X_n}$ the space of formal power series in $n$ variables with values in $E$, for $n\in \BN$.
For $\ul{r}\in \BR_{>0}^n$, we define the subspace of all \textit{power series strictly convergent on $B_{\ul{r}}^{n}(0)$ with values in $E$}
\begin{equation*}
	\CA_\ul{r} (L^n,E) \defeq \bigg\{ \sum_{\ul{i}\in \BN^n_0} v_{\ul{i}} \, X_1^{i_1}\cdots X_n^{i_n} \,\bigg|\, \ul{r}^{\ul{i}} \, \lVert v_{\ul{i}} \rVert_E  \ra 0 \text{ as } \abs{\ul{i}} \ra \infty \bigg\}
	\subset E \llrrbracket{X_1,\ldots, X_n} .
\end{equation*}
The $L$-vector space $\CA_\ul{r}(L^n,E)$ is an $L$-Banach space with respect to the norm
\begin{equation*}
	\bigg\lVert \sum_{\ul{i}\in \BN^n_0} v_{\ul{i}} \, X_1^{i_1}\cdots X_n^{i_n} \bigg\rVert_r \defeq \sup_{\ul{i}\in\BN^n_0} \ul{r}^{\ul{i}} \, \lVert v_{\ul{i}} \rVert_E .
\end{equation*}
Note that, for $\ul{r}\geq \ul{r}'$ (i.e.\ $r_j \geq r'_j$, for all $j=1,\ldots,n$), the inclusion $\CA_\ul{r}(L^n,E) \subset \CA_{\ul{r}'}(L^n,E)$ is a continuous homomorphism. 
Hence we define the space of \textit{power series convergent at $0$ with values in $E$}
\begin{equation*}
	\CA(L^n, E) \defeq \bigcup_{\ul{r} \in \BR_{>0}^n} \CA_\ul{r}(L^n,E) ,
\end{equation*}
and endow it with the inductive limit topology, i.e.\ with the finest locally convex topology such that all inclusions $\CA_\ul{r}(L^n,E) \hookrightarrow \CA(L^n,E)$ are continuous.

Moreover, every $f = \sum_{\ul{i}\in \BN^n_0} v_{\ul{i}} \,X_1^{i_1}\cdots X_n^{i_n}\in \CA_\ul{r}(L^n,E)$ defines a continuous function
\begin{equation}\label{Eq 1 - Function associated to a power series}
	B_\ul{r}^n(0) \lra E \,,\quad (x_1,\ldots,x_n) \lto f(x_1,\ldots,x_n) \defeq  \sum_{\ul{i}\in \BN^n_0} v_{\ul{i}} \, x_1^{i_1}\cdots x_n^{i_n} .
\end{equation}

\begin{proposition}[{Identity theorem for power series \cite[4.1.4]{Bourbaki07VarDiffAnFasciDeResult}}]\label{Prop 1 - Identity theorem for power series}
	Let $n\in \BN$, $\ul{r} \in \BR_{>0}^n$, and let $E$ be an $L$-Banach space.
	The homomorphism from $\CA_\ul{r}(L^n,E)$ to the $L$-vector space of continuous functions on $B_\ul{r}^{n}(0)$ given by associating to $f \in \CA_\ul{r}(L^n, E)$ the function \eqref{Eq 1 - Function associated to a power series} is injective.
\end{proposition}

Therefore we will denote both the power series as well as the induced function by $f$.

\begin{proposition}[{\cite[4.1.5]{Bourbaki07VarDiffAnFasciDeResult} or \cite[Prop.\ 5.4]{Schneider11pAdicLieGrps}}]\label{Prop 1 - Composition of power series}
	For $m,n \in \BN$, $\ul{r} \in \BR_{>0}^m$, $\ul{s}\in \BR_{>0}^n$, let $f\in \CA_\ul{r}(L^m,L^n)$ be written as $f = (f_j)_{j=1,\ldots,n}$, for $f_j \in \CA_{\ul{r}}(L^m, L)$.
	Moreover, assume that $\norm{f_j}_\ul{r} \leq s_j$, for all $j=1,\ldots,n$, and let $E$ be an $L$-Banach space.
	Then the map
	\begin{equation*}
		\begin{tikzcd}[row sep = 0ex,	/tikz/column 1/.append style={anchor=base east},	/tikz/column 2/.append style={anchor=base west}]
			\CA_\ul{s}(L^n, E) \ar[r] &\CA_\ul{r}(L^m,E) \\
			g(\ul{Y}) = \sum_{\ul{i} \in \BN^n_0} v_{\ul{i}} \,\ul{Y}^{\ul{i}} \ar[r, mapsto] &(g\circ f)(\ul{X}) \defeq  \sum_{\ul{i}\in \BN^n_0} v_{\ul{i}}\, f_1(\ul{X})^{i_1} \cdots f_n(\ul{X})^{i_n}
		\end{tikzcd}
	\end{equation*}
	is a well-defined continuous homomorphism of operator norm $\leq 1$, and the associated functions satisfy $(g\circ f)(x) =g(f(x))$, for all $x\in B_\ul{r}^{m}(0)$. 
\end{proposition}

\begin{corollary}[{\cite[Cor.\ 5.5]{Schneider11pAdicLieGrps}}]\label{Cor 1 - Redevelopment of power series}
	Let $f\in \CA_\ul{r}(L^m,E)$, and $y\in B_\ul{r}^{m}(0)$.
	Then there exists $f_y \in \CA_\ul{r}(L^m, E)$ such that $\norm{f_y}_\ul{r} = \norm{f}_\ul{r}$ and the associated functions satisfy
	\begin{equation*}
		f(x) = f_y(x-y) \quad \text{, for all $x\in B_\ul{r}^{m}(0) = B_\ul{r}^{m}(y)$.}
	\end{equation*}
\end{corollary}

\begin{definition}\label{Def 1 - Locally analytic functions with values in Banach space}
	Let $U\subset L^m$ be an open subset, for $m\in \BN$, and $E$ an $L$-Banach space.
	We call a function $f \colon U \ra E$ \textit{locally $L$-analytic} if, for every $a \in U$, there exists a power series $f_{a}\in \CA_\ul{r}(L^m,E)$, for some $\ul{r}\in \BR_{>0}^m$, such that $f(x) = f_{a}(x-a)$, for all $x\in B_\ul{r}^{m}(a)$.
	We denote the $L$-vector space of locally $L$-analytic functions on $U$ with values in $E$ by $C^\la(U, E)$.
\end{definition}

\begin{remark}
	In particular, such a locally $L$-analytic function is continuous.
\end{remark}

\begin{lemma}[{\cite[Lemma 6.3]{Schneider11pAdicLieGrps}}]\label{Lemma 1 - Composition of locally analytic functions}
	Let $U\subset L^m$ and $U'\subset L^n$ be open subsets.
	Moreover, let $f\in C^\la(U,L^n)$ such that $f(U)\subset U'$, and let $E$ be an $L$-Banach space.
	Then the map
	\begin{equation*}
		C^\la(U',E) \lra C^\la(U,E) \,,\quad g \lto g\circ f ,
	\end{equation*}
	is well-defined and $L$-linear.
\end{lemma}

\begin{definition}
	Let $X$ be a topological space.
	\begin{altenumerate}
		\item
		A \textit{chart} of $X$ consists of an open subset $U\subset X$ and a map $\varphi \colon U \ra L^m$, for some $m\in\BN$, which is a homeomorphism onto an open subset of $L^m$.
		We will occasionally refer to a chart simply by $\varphi$ or $U$ if the context allows it. 
		For $x\in X$, we say that $\varphi$ is a chart \textit{around $x$} if $x\in U$.
		We call $\varphi$ \textit{centred at $x$} if $\varphi(x) =0$.
		\item 
		Two charts $\varphi\colon U \ra L^{m}$ and $\psi\colon W\ra L^{n}$ of $X$ are \textit{compatible} if the functions
		\begin{equation*}
			\psi \circ \varphi^{-1} \colon \varphi(U \cap W) \lra \psi(U\cap W) \qquad\text{and}\qquad \varphi \circ \psi^{-1} \colon \psi(U\cap W) \lra \varphi(U\cap W)
		\end{equation*}
		are locally $L$-analytic.
		\item
		An \textit{atlas} $\CA$ of $X$ is a set of pairwise compatible charts whose domains cover $X$.
		Two atlases $\CA$ and $\CB$ of $X$ are \textit{equivalent} if $\CA \cup \CB$ is an atlas as well.
		An atlas $\CA$ is \textit{maximal} if any equivalent atlas $\CB$ satisfies $\CB \subset \CA$.
	\end{altenumerate}
\end{definition}

\begin{remarks}
	\begin{altenumerate}
		\item
		Equivalence of atlases indeed is an equivalence relation, and every equivalence class contains a unique maximal atlas, see \cite[Rmk.\ 7.2]{Schneider11pAdicLieGrps}.
		\item
		Given $x\in X$ and a maximal atlas $\CA$ of $X$, there is a chart in $\CA$ that is centred at $x$:
		Let $\varphi\colon U \ra L^m$ be any chart with $x\in U$.
		Then $\varphi'\colon U \ra L^m, y \mto \varphi(y)-\varphi(x)$ is compatible with the charts of $\CA$ by \Cref{Lemma 1 - Composition of locally analytic functions}.		
	\end{altenumerate}
\end{remarks}

We want to consider manifolds with the following good properties:

\begin{definition}
	A \textit{(finite-dimensional) locally $L$-analytic manifold} is a Hausdorff, paracompact, second-countable topological space $X$ together with a maximal atlas $\CA$.
	In the following, when we speak of a chart of a locally $L$-analytic manifold, we mean a chart of its maximal atlas.

	For a point $x\in X$ with a chart $\varphi \colon U \ra L^m$ around $x$, we call $m$ the \textit{dimension of $X$ at $x$}.
	By \cite[Lemma 7.1]{Schneider11pAdicLieGrps}, this dimension is independent of the chart around $x$.
\end{definition}

\begin{remarks}\label{Rmk 1 - Locally analytic manifold has disjoint countable covering by compact open subsets}
	\begin{altenumerate}
		\item
		Any locally $L$-analytic manifold $X$ is strictly paracompact, i.e.\ any open covering of $X$ admits a refinement by pairwise disjoint open subsets (\cite[5.3.7]{Bourbaki07VarDiffAnFasciDeResult} or \cite[Prop.\ 8.7]{Schneider11pAdicLieGrps}).
		\item
		Let $X$ be a locally $L$-analytic manifold. Then $X$ is locally compact if and only if $L$ is locally compact (i.e.\ a local field) or $X$ is a discrete topological space.
		\item
		Any disjoint open covering of $X$ is countable.
		Moreover, if $L$ is locally compact, then, for any locally $L$-analytic manifold $X$, there exists a disjoint countable covering of $X$ by compact open subsets.
	\end{altenumerate}
\end{remarks}
\begin{proof}[Proof of (ii) and (iii)]
	These statements are probably well-known, but we still want to include proofs here.

	For (ii), first assume that $X$ is locally compact.
	If, for all $x\in X$, we can find a charts $\varphi\colon U \ra L^0=\{0\}$ with $x\in U$, it follows that $X$ is discrete.
	On the other hand, consider the situation that there exists a chart $\varphi\colon U \ra L^n$ with $n>0$.
	We then find a compact subset $C \subset U$, and after shrinking we may assume that $\varphi(C) = B^{n}_{\ul{r}}(a)$, for $\ul{r} \in \BR^n_{>0}$, $a=(a_1,\ldots,a_n)\in \varphi(C)$.
	This implies that $B^{1}_{r_1} (a_1) \subset L$ is compact, too.
	But this is equivalent to $L$ being locally compact.
	For the reverse implication see \cite[5.1.9]{Bourbaki07VarDiffAnFasciDeResult}.

	In (iii), because $X$ is second countable, for every open covering $X= \bigcup_{i\in I} U_i$, there exists a countable subset $J\subset I$ such that $X = \bigcup_{i\in J} U_i$ is a covering, see \cite[Ch.\ IX.\ \S 2.8 Prop.\ 13]{Bourbaki66GenTop2}.
	If the covering $X= \bigcup_{i\in I} U_i$ is disjoint, we necessarily have $J = I$.

	Furthermore, the topology of $X$ can be defined by a metric which satisfies the strict triangle inequality because $X$ is paracompact, see \cite[Prop.\ 8.7]{Schneider11pAdicLieGrps}.
	Hence there exists a base $\CB$ for the topology of $X$ that consists of subsets which are open and closed \cite[Ch.\ IX.\ Ex.\ for \S 6, Ex.\ 2a)]{Bourbaki66GenTop2}.
	As we have seen in (ii), the assumption that $L$ is locally compact implies that $X$ is locally compact, i.e.\ for any $x\in X$, there exists a compact neighbourhood $C_x$ of $x$.
	Then we find an open and closed subset $B_x \in \CB$ such that $B_x\subset C_x$ and which therefore is compact itself.
	In conclusion, we see that the set of compact open subsets constitutes a covering of $X$.
	Hence there exists a countable collection $\{C_n\}_{n\in \BN}$ of compact open subsets which already covers $X$.	
	Setting $W_n \defeq C_n \setminus (C_0 \cup \ldots \cup C_{n-1})$ now yields the sought disjoint countable covering $X=\bigcup_{n\in\BN} W_n$ by compact open subsets.
\end{proof}

\begin{definition}[{\cite[5.8.3]{Bourbaki07VarDiffAnFasciDeResult}}]\label{Def 1 - Submanifolds}
	A subset $Y \subset X$ of a locally $L$-analytic manifold $X$ is called a \textit{locally $L$-analytic submanifold} if, for every $y\in Y$, there exist a chart $\varphi \colon U \ra L^m$ around $x$ and a linear subspace $F \subset L^m$ such that $\varphi$ induces a homeomorphism 
	\begin{equation*}
		\varphi\res{U\cap Y} \colon U \cap Y \lra \varphi(U) \cap F.
	\end{equation*}
	Taking isomorphisms $F \cong L^k$, for some $k\leq m$, the charts $\varphi\res{U\cap Y} \colon U \cap Y \ra L^k$ equip $Y$ with the structure of a locally $L$-analytic manifold, see \cite[5.8.1]{Bourbaki07VarDiffAnFasciDeResult}.
	Indeed, $Y$ also is paracompact because $X$ is metrizable by \cite[Prop.\ 8.7]{Schneider11pAdicLieGrps}.
	When $Y \subset X$ is open, a maximal atlas of $Y$ is given by the charts $U$ of $X$ such that $U \subset Y$, see \cite[p.\ 48]{Schneider11pAdicLieGrps}.
\end{definition}

\begin{remark}
	The product of two locally $L$-analytic manifolds $X$ and $Y$ becomes a locally $L$-analytic manifold when endowed with the product topology and the atlas given by $\varphi \times \psi \colon U \times V \ra L^{m+n}$, for charts $\varphi\colon U \ra L^m$ and $\psi \colon V \ra L^n$ of $X$ and $Y$ respectively.
\end{remark}

\begin{definition}\label{Def 1 - Locally analytic maps between locally analytic manifolds}
	\begin{altenumerate}
		\item
		Let $X$ be a locally $L$-analytic manifold and $E$ an $L$-Banach space.
		A function $f\colon X \ra E$ is \textit{locally $L$-analytic} if $f\circ \varphi^{-1}\colon \varphi(U)\ra E$ is locally $L$-analytic, for every chart $\varphi\colon U \ra L^m$ of $X$.
		We denote the $L$-vector space of these functions by $C^\la(X,E)$.
		\item
		A map $f\colon X \ra Y$ between two locally $L$-analytic manifolds is \textit{locally $L$-analytic} if $f$ is continuous and, for all charts $\psi\colon V \ra L^n$ of $Y$, the function $\psi\circ f$ from the open locally $L$-analytic submanifold $f^{-1}(V)$ to the $L$-Banach space $L^n$ is locally $L$-analytic.
		% Continuity of $f$ in this first definition is necessary to talk about the open submanifold $f^{-1}(V)$.
		
		Equivalently, such $f\colon X \ra Y$ is locally $L$-analytic if, for every point $x\in X$, there exist a chart $\varphi \colon U \ra L^m$ around $x$ and a chart $\psi \colon V \ra L^n$ around $f(x)$ such that $f(U) \subset V$ and $\psi \circ f \circ \varphi^{-1} \in C^\la(\varphi(U),L^n)$, see \cite[Lemma 8.3]{Schneider11pAdicLieGrps}.
	\end{altenumerate}
\end{definition}

\begin{remark}
	In the case that $Y \subset L^n$ is an open subsets with the canonical structure of locally $L$-analytic manifolds, (i) and (ii) in the above definition are compatible.
	If in turn $X \subset L^m$ is an open subset, (i) is compatible with \Cref{Def 1 - Locally analytic functions with values in Banach space}, see \cite[5.3.1,2]{Bourbaki07VarDiffAnFasciDeResult}.
\end{remark}

\subsection{Locally Analytic Functions}

Let $K$ be a complete non-archimedean field with non-trivial absolute value $\abs{\blank}$, and $L\subset K$ a complete subfield.

Let $X$ be a locally $L$-analytic manifold and $V$ a Hausdorff locally convex $K$-vector space.
For the case of ${\rm char}(K)=0$, F\'eaux de Lacroix \cite{FeauxdeLacroix99TopDarstpAdischLieGrp} defined locally analytic functions on $X$ which take values in $V$, and endowed the space $C^\la(X,V)$ of such functions with the structure of a locally convex $K$-vector space.
As remarked by Gräf, this carries over to the case of a general complete non-archimedean field $K$ verbatim \cite[Part I, App.\ A]{Graef21BoundDistr}.
Nevertheless we want to recapitulate the reasoning for the construction of $C^\la(X,V)$ as well as some properties of it.

Recall that, for a locally convex $K$-vector space $V$, a \textit{BH-subspace} of $V$ is an (algebraic) subspace $E\subset V$ which admits the structure of a $K$-Banach space (with underlying $K$-vector space structure coming from $V$) such that the associated topology is finer than its subspace topology.
We denote $E$ carrying its Banach space structure by $\widebar{E}$ so that we have a continuous injection $\widebar{E} \hookrightarrow V$.
Note that the topologies from any two Banach space structures of a BH-subspace $E\subset V$ are the same by the open mapping theorem \cite[Prop.\ 8.6]{Schneider02NonArchFunctAna}.

If $E$ is a $K$-Banach space, it also carries the structure of an $L$-Banach space by restriction of scalars.
We will use this identification freely, for example to consider power series $\CA_\ul{r}(L^m, E)$ on $B^{m}_\ul{r}(0)$ with values in $E$.

\begin{definition}\label{Def 1 - Definition locally analytic function}
	A function $f \colon X \ra V$ is called \textit{locally analytic} if, for every $a\in X$, there exists a BH-subspace $E\subset V$, a chart $\varphi\colon U \ra B_{\ul{r}}^{n}(0)$ of $X$, for some $\ul{r}\in \BR_{>0}^n$, with $a\in U$, and a power series $f_a \in \CA_\ul{r}(L^n, \widebar{E})$ such that $f(x) = f_a(\varphi(x)- \varphi(a))$, for all $x$ in some neighbourhood of $a$.
	Here we consider $f_a(\varphi(\blank)- \varphi(a))$ as a function taking values in $V$ via $\widebar{E} \hookrightarrow V$.
	We denote the $K$-vector space of locally analytic functions on $X$ with values in $V$ by $C^\la (X,V)$.
\end{definition}

\begin{remark}
	In particular a locally analytic function $f \colon X \ra V$ is continuous.
\end{remark}

To topologize $C^\la(X,V)$ one expresses this space as the inductive limit of spaces of functions which are locally analytic with respect to certain indices.

\begin{definition}
	\begin{altenumerate}
		\item
		A \textit{$V$-index} $\CI$ of $X$ is a family $\big(\varphi_i\colon U_i \ra L^{m_i}, \ul{r}_i, E_i\big)_{i\in I}$ where the $\varphi_i$ are charts of $X$, $\ul{r}_i\in \BR_{>0}^{m_i}$, and the $E_i \subset V$ are BH-subspaces such that
		\begin{altenumeratelevel2}
			\item
			$X = \bigcup_{i\in I} U_i$ is a disjoint open covering,
			\item
			$\varphi_i (U_i) = B_{\ul{r}_i}^{m_i}(a_i)$, for some (or any) $a_i\in \varphi_i(U_i)$.
		\end{altenumeratelevel2}
		\item
		Given two $V$-indices
		\begin{equation*}
			\CI = \big(\varphi_i\colon U_i \ra L^{m_i},\ul{r}_i, E_i \big)_{i\in I} \qquad \text{and} \qquad \CJ = \big( \psi_j \colon W_j \ra L^{n_j}, \ul{s}_j, F_j\big)_{j\in J}
		\end{equation*}
		of $X$, we call $\CI$ \textit{finer} than $\CJ$, if, for every $i \in I$, there exists $j\in J$ such that
		\begin{altenumeratelevel2}
			\item
			$U_i \subset W_j$ (i.e.\ the covering of $\CI$ is a refinement of the one of $\CJ$), 
			\item
			there exist $a \in \varphi_i(U_i)$ and $g_{i,j}=(g_{i,j,k})_{k=1,\ldots,n_j} \in \CA_{\ul{r}_i}(L^{m_i},L^{n_j})$ such that
			\begin{equation*}
				\lVert g_{i,j,k} - g_{i,j,k}(0) \rVert_{\ul{r}_i} \leq s_{j,k} \quad \text{ , for all $k=1,\ldots,n_j$},
			\end{equation*}
			and $\psi_j \circ \varphi_i^{-1}(x) = g_{i,j}(x-a)$, for all $x \in \varphi_i(U_i)$,
			%where $\ul{s}_j = (s_{j,1},\ldots,s_{j,n_j})$.
			\item
			$F_j \subset E_i$ (which implies that $\widebar{F_j} \hookrightarrow \widebar{E_i}$ is continuous).
		\end{altenumeratelevel2}
	\end{altenumerate}
\end{definition}

\begin{remark}
	Using \Cref{Cor 1 - Redevelopment of power series} one sees that condition (2) in (ii) is independent of the choice of $a\in \varphi_i(U_i)$, cf.\ \cite[p.\ 76]{Schneider11pAdicLieGrps}.
\end{remark}

\begin{lemma}[{\cite[Bem.\ 2.1.9]{FeauxdeLacroix99TopDarstpAdischLieGrp}, cf.\ \cite[Lemma 10.2]{Schneider11pAdicLieGrps}}]
	The set of $V$-indices of $X$ is a directed set with respect to the relation of being finer.
\end{lemma}

Let $\varphi \colon U \ra L^m$ be a chart of $X$.
If there exist $\ul{r}\in \BR_{>0}^m$ and $a\in L^m$ such that \mbox{$\varphi(U) = B_\ul{r}^{m}(a)$}, we call $\varphi$ an \textit{analytic chart}.
For such a chart and a $K$-Banach space $E$, we set
\begin{equation*}
	C^\aan (\varphi,E) \defeq \left\{ f\colon U \ra E \middle{|} \exists g \in \CA_\ul{r}(L^m,E) ,  \forall x \in U : f(x) = g(\varphi(x)-a)\right\}.
\end{equation*}
Using the identity theorem for power series (\Cref{Prop 1 - Identity theorem for power series}), we immediately see that there is an isomorphism
\begin{equation*}
	\CA_\ul{r}(L^m, E) \overset{\cong}{\lra} C^\aan(\varphi,E) \,,\quad g \lto g(\varphi(\blank)-a) .
\end{equation*}
In this way, we consider $C^\aan(\varphi,E)$ as a $K$-Banach space with norm given by $\norm{f}\defeq \norm{g}_\ul{r}$ when $f = g(\varphi(\blank)-a)$.
If the analytic chart $\varphi\colon U \ra B^{m}_{\ul{r}}(a)$ is understood, we also write
\begin{equation*}
	C^\aan (U,E) \defeq C^\aan (\varphi, E).
\end{equation*}

\begin{remark}
	If there exists some $a\in L^m$ such that $\varphi(U)= B_\ul{r}^{m} (a)$, then $\varphi(U)= B_\ul{r}^{m}(b)$, for all $b\in \varphi(U)$.
	However, the existence of $g\in \CA_\ul{r}(L^m,E)$ in the definition and $\norm{f}$ do not depend on the choice of $a\in \varphi(U)$ by \Cref{Cor 1 - Redevelopment of power series}.
\end{remark}

\begin{definition}
	Let $\CI = \big(\varphi_i\colon U_i \ra L^{m_i}, \ul{r}_i, E_i\big)_{i\in I}$ be a $V$-index of $X$.
	\begin{altenumerate}
		\item
		A function $f\colon X \ra V$ is \textit{subordinate} to $\CI$ if $f\res{U_i} \in C^\aan (\varphi_i, \widebar{E_i})$, for all $i\in I$.
		Spelled out, this means that, for all $i\in I$, there exist $g_i \in \CA_{\ul{r}_i} (L^{m_i},\widebar{E_i})$ and some $a_i \in \varphi_i(U_i)$ such that $(f\circ \varphi_i^{-1})(x)= g_i(x-a_i)$, for all $x\in \varphi_i(U_i)$.
		\item
		We denote the $K$-vector space of all functions $f \colon X\ra V$ which are subordinate to $\CI$ by $C^\la_\CI(X,V)$.
		As $(U_i)_{i\in I}$ is a disjoint covering of $X$, the map
		\begin{equation*}
			C^\la_\CI (X,V) \lra \prod_{i\in I} C^\aan (\varphi_i, \widebar{E_i}) \,,\quad f \lto (f\res{U_i})_{i\in I},
		\end{equation*}
		is an isomorphism of $K$-vector spaces.
		Via this isomorphism, we endow $C^\la_\CI(X,V)$ with a locally convex topology coming from the product topology of the right hand side.
	\end{altenumerate}
\end{definition}

\begin{remark}
	If the index set $I$ is finite, then $C^\la_\CI(X,V)$ itself is a $K$-Banach space.
	In any case, $C^\la_\CI(X,V)$ is a $K$-Fr\'echet space since $I$ is necessarily countable, see \Cref{Rmk 1 - Locally analytic manifold has disjoint countable covering by compact open subsets} (iii).
\end{remark}

\begin{lemma}[{\cite[Bem.\ 2.1.9]{FeauxdeLacroix99TopDarstpAdischLieGrp}, cf.\ \cite[Lemma 10.3]{Schneider11pAdicLieGrps}}]
	If the $V$-index $\CI$ is finer than the $V$-index $\CJ$, then $C^\la_\CJ(X,V) \subset C^\la_\CI(X,V)$ and this inclusion map is continuous.
\end{lemma}

\begin{proposition}[{\cite[Bem.\ 2.1.9]{FeauxdeLacroix99TopDarstpAdischLieGrp}, cf.\ \cite[p.\ 75]{Schneider11pAdicLieGrps}}]
	For any locally analytic function $f \colon X \ra V$, there exists a $V$-index $\CI$ of $X$ such that $f$ is subordinate to $\CI$.
	In other words 
	\begin{equation*}
		C^\la(X,V) = \bigcup_{\CI} C^\la_{\CI}(X,V)
	\end{equation*}
	where the union is taken over all $V$-indices $\CI$ of $X$.
\end{proposition}

Hence we can and will endow $C^\la(X,V)$ with the locally convex inductive limit topology with respect to the $C^\la_\CI (X,V)$, i.e.\ the finest locally convex topology such that the inclusions $C^\la_\CI (X,V) \hookrightarrow C^\la(X,V)$ are continuous.
This finishes the construction of the locally convex $K$-vector space $C^\la(X,V)$.

\begin{proposition}[{\cite[Satz 2.1.10]{FeauxdeLacroix99TopDarstpAdischLieGrp}, cf.\ \cite[Prop.\ 12.1]{Schneider11pAdicLieGrps}}]\label{Prop 1 - Evaluation maps are continuous}
	Let $X$ be a locally $L$-analytic manifold, and $V$ a Hausdorff locally convex $K$-vector space.
	For any $x\in X$, the evaluation homomorphism
	\begin{equation*}
		\ev_x \colon C^\la(X,V) \lra V \,,\quad f \lto f(x),
	\end{equation*}
	is continuous.
\end{proposition}

\begin{corollary}[{\cite[Satz 2.1.10]{FeauxdeLacroix99TopDarstpAdischLieGrp}, cf.\ \cite[Cor.\ 12.2]{Schneider11pAdicLieGrps}}]\label{Cor 1 - Space of locally analytic functions is Hausdorff and barrelled}
	Let $X$ be a locally $L$-analytic manifold, and $V$ a Hausdorff locally convex $K$-vector space.
	Then $C^\la(X,V)$ is Hausdorff and barrelled.
\end{corollary}
\begin{proof}
	Let $f, f' \in C^\la(X,V)$ with $f\neq f'$, and let $x\in X$ such that $f(x) \neq f'(x)$.
	Because $V$ is Hausdorff, there exist open neighbourhoods $U, U'\subset V$ of $f(x)$ resp.\ $f'(x)$ such that $U\cap U' = \emptyset$.
	As $\ev_x$ is continuous by \Cref{Prop 1 - Evaluation maps are continuous}, $\ev_x^{-1}(U)$ and $\ev_x^{-1}(U')$ are open subsets of $C^\la(X,V)$ that separate $f$ and $f'$. 
	Therefore $C^\la(X,V)$ is Hausdorff.

	Since $K$-Banach spaces are barrelled \cite[Expl.\ 2) after Cor.\ 6.16]{Schneider02NonArchFunctAna}, the direct product $C^\la_\CI (X,V)$ is barrelled, for every $V$-index $\CI$ of $X$, by \cite[Prop.\ 14.3]{Schneider02NonArchFunctAna}.
	Moreover the inductive limit of barrelled locally convex $K$-vector spaces is barrelled again \cite[Expl.\ 3) after Cor.\ 6.16]{Schneider02NonArchFunctAna}, and we conclude that $C^\la(X,V)$ is barrelled.	
\end{proof}

\begin{proposition}[{\cite[Kor.\ 2.2.4]{FeauxdeLacroix99TopDarstpAdischLieGrp}, cf.\ \cite[Prop.\ 12.5]{Schneider11pAdicLieGrps}}]\label{Prop 1 - Disjoint coverings and locally analytic functions}
	Let $X$ be a locally $L$-analytic manifold, and $V$ a Hausdorff locally convex $K$-vector space.
	Then, for any disjoint covering $X = \bigcup_{i\in I} X_i$ by open subsets $X_i$, there is a topological isomorphism
	\begin{equation*}
		C^\la(X,V) \overset{\cong}{\lra} \prod_{i\in I} C^\la(X_i , V) \,,\quad f \lto (f\res{X_i})_{i\in I} .
	\end{equation*}
\end{proposition}
\begin{proof}
	By applying the statement of \cite[Lemma 11.7]{Schneider11pAdicLieGrps} it suffices to show that, for any given $V$-index \mbox{$\CJ = \big(\varphi_j\colon U_j \ra L^{m_j}, \ul{r}_j, E_j\big)_{j\in J}$} of $X$, there exists a $V$-index $\CI$ of $X$ which is finer than $\CJ$ and whose covering is a refinement of $X= \bigcup_{i\in I} X_i$.
	Using \cite[Lemma 1.4]{Schneider11pAdicLieGrps}, we find, for each open subset $\varphi_j(U_j \cap X_i) \subset L^{m_j}$ with $i\in I, j\in J$, a disjoint covering of the form
	\begin{equation*}
		\varphi_j(U_j \cap X_i) = \bigcup_{k\in J_{i,j}} B^{m_j}_{\ul{s}_{i,j,k}}(a_{i,j,k}) ,
	\end{equation*}
	for certain index sets $J_{i,j}$, and $\ul{s}_{i,j,k}\in \BR_{>0}^{m_j}$, $a_{i,j,k} \in L^{m_j}$.
	Now we define the index set $A \defeq \left\{ (i,j,k)\middle{|} i\in I, j\in J, k\in J_{i,j}\right\}$, and set $W_{i,j,k} \defeq \varphi_j^{-1} \big(B^{m_j}_{\ul{s}_{i,j,k}}(a_{i,j,k})\big)$, for $(i,j,k)\in A$.
	Then the $W_{i,j,k}$ constitute a disjoint open covering of $X$ by charts.
	Moreover, 
	\begin{equation*}
		\CI \defeq  \big( \varphi_j\res{W_{i,j,k}} \colon W_{i,j,k} \ra L^{m_j}, \ul{s}_{i,j,k}, E_j \big)_{(i,j,k)\in A} 
	\end{equation*}
	is a $V$-index which is finer than $\CJ$,
	%For condition (2) one sees that $\varphi_j \circ (\varphi_j\res{W_{i,j,k}})^{-1} = \id_{B...}$ so take $g=X$ for the power series.
	and its covering is a refinement of $X= \bigcup_{i\in I} X_i$ by construction.
\end{proof}

\begin{proposition}[{\cite[Bem.\ 2.1.11]{FeauxdeLacroix99TopDarstpAdischLieGrp}, \cite[p.\ 40]{Emerton17LocAnVect}}]\label{Prop 1 - Functorialities for the space of locally analytic functions}
	Let $X$ be a locally $L$-analytic manifold, and $V$ a Hausdorff locally convex $K$-vector space.
	\begin{altenumerate}
		\item
		If $W$ is a Hausdorff locally convex $K$-vector space and $\lambda \colon V \ra W$ a continuous homomorphism, then $\lambda$ induces a continuous homomorphism
		\begin{equation*}
			\lambda_\ast \colon C^\la(X,V) \lra C^\la(X,W) \,,\quad f \lto \lambda \circ f .
		\end{equation*}
		\item
		If $Y$ is a locally $L$-analytic manifold and $h \colon X \ra Y$ a locally $L$-analytic map, then $h$ induces a continuous homomorphism
		\begin{equation*}
			h^\ast \colon C^\la(Y,V) \lra C^\la(X,V) \,,\quad f \lto f \circ h .
		\end{equation*}
	\end{altenumerate}
\end{proposition}
\begin{proof}
	The statement of (i) follows from \cite[Prop.\ 1.1.7]{Emerton17LocAnVect}, see ibid.\ p.\ 40.

	For (ii), we adapt the argument outlined in the proof of \cite[Prop.\ 12.4 (ii)]{Schneider11pAdicLieGrps}.
	First we construct, for each fine enough $V$-index $\CJ=\big(\psi_j\colon W_j \ra L^{n_j},\ul{s}_j,F_j \big)_{j\in J}$ of $Y$, a $V$-index $\CI=\big(\varphi_i \colon U_i \ra L^{m_i},\ul{r}_i,E_i\big)_{i\in I}$ of $X$ which satisfies:
	For all $i \in I$, there exists $j\in J$ such that
	\begin{altenumeratelevel2}
		\item
		$U_i \subset h^{-1}(W_j)$, i.e.\ the covering of $\CI$ is a refinement of the covering $X= \bigcup_{j \in J} h^{-1}(W_j)$.
		\item
		there exist $a_i \in \varphi_i(U_i)$ and $g_{i,j}=(g_{i,j,k})_{k=1,\ldots,n_j} \in \CA_{\ul{r}_i} (L^{m_i}, L^{n_j})$ such that
		\begin{equation*}
			\norm{ g_{i,j,k}- g_{i,j,k}(0)}_{\ul{r}_i} \leq s_{j,k} \quad \text{, for all $k=1,\ldots,n_j$,}
		\end{equation*}
		and $\psi_j \circ h \circ \varphi_i^{-1} (x) = g_{i,j} (x-a_i)$, for all $x\in \varphi_i(U_i)= B_{\ul{r}_i}^{m_i}(a_i)$,
		\item
		$ F_j \subset E_i$.
	\end{altenumeratelevel2}
	Indeed, for a covering $Y = \bigcup_{j \in J} W_j$ of a given $V$-index $\CJ$, we may take $X = \bigcup_{i \in I} U_i$ to be a disjoint refinement by analytic charts of $X= \bigcup_{j \in J} h^{-1}(W_j)$.
	By the \Cref{Def 1 - Locally analytic maps between locally analytic manifolds} (ii) of $h$ being a locally analytic map, we may assume that $\psi_j \circ h \circ \varphi_i^{-1} \in C^\la(\varphi_i(U_i), L^{n_j})$, for all $i\in I$, $j\in J$, with $U_i \subset h^{-1}(W_j)$, after passing to fine enough $\CJ$ and $X = \bigcup_{i \in I} U_i$.
	Therefore the property (2) is satisfied for $X = \bigcup_{i \in I} U_i$ after further refining.
	For $i\in I$ with $j\in J$ such that $U_i \subset h^{-1}(W_j)$, we then set $E_i \defeq F_j$, and obtain the sought $V$-index $\CI$ of $X$.

	For such $\CI$ and $j \in J$, $i\in I$ with $U_i \subset h^{-1}(W_j)$, we can use the identifications
	\begin{alignat*}{3}
		\CA_{\ul{s}_j} (L^{n_j},\widebar{F_j}) &\lra C^\rig(\psi_j, \widebar{F_j}) \,,\quad g &&\lto g(\psi_j (\blank) - g_{i,j}(0)) ,\\
		\CA_{\ul{r}_i} (L^{m_i},\widebar{E_i}) &\lra C^\rig(\varphi_i, \widebar{E_i}) \,,\quad g &&\lto g(\varphi_i (\blank) - a_i),
	\end{alignat*}
	to obtain the commutative diagram
	\begin{equation*}
		\begin{tikzcd}
			C^\rig(\psi_j, \widebar{F_j}) \ar[r] &C^\rig(\varphi_i, \widebar{E_i})  \\
			\CA_{\ul{s}_j} (L^{n_j},\widebar{F_j}) \ar[u, "\cong"]\ar[r] & \CA_{\ul{r}_i} (L^{m_i},\widebar{E_i}) \ar[u, "\cong"'] 
		\end{tikzcd}
	\end{equation*}
	where the upper map is given by $f \mto f \circ h $ and the lower one by $g \mto g \circ (g_{i,j} - g_{i,j}(0))$.
	As this latter map is continuous by \Cref{Prop 1 - Composition of power series}, the homomorphism $C^\la_\CJ (Y,V) \ra C^\la_\CI(X,V)$, $f \mto f\circ h$, induced by the upper homomorphisms, for all such $i$ and $j$, is continuous.
	Because the sufficiently fine $V$-index $\CJ$ of $Y$ was arbitrary, this shows that $h^\ast \colon C^\la(Y,V) \ra C^\la(X,V)$ is continuous.
\end{proof}

\begin{proposition}\label{Prop 1 - Direct limit description of locally analytic functions for compact manifold}
	Let $X$ be a compact locally $L$-analytic manifold, and $V$ a Hausdorff locally convex $K$-vector space.
	\begin{altenumerate}
		\item
		\textnormal{(cf.\ \cite[p.\ 40]{Emerton17LocAnVect})}
		Taking the inductive limit of the homomorphisms $C^\la(X,\widebar{E}) \ra C^\la(X,V)$, for all BH-subspaces $E \subset X$, yields a topological isomorphism
		\begin{equation*}
			\varinjlim_{E\subset V} C^\la(X,\widebar{E}) \overset{\cong}{\lra} C^\la(X,V).
		\end{equation*}
		\item
		\textnormal{(cf.\ \cite[Prop.\ 2.1.30]{Emerton17LocAnVect})}
		If $V$ is of LF-type, i.e.\ can be written as the increasing union $V = \bigcup_{n\in \BN} \iota_n(V_n)$, for $K$-Fr\'echet spaces $V_n$ with continuous injections $\iota_n \colon V_n \hookrightarrow V$, then $C^\la(X,V)$ is an LF-space, i.e.\ topologically isomorphic to the inductive limit of a sequence of $K$-Fr\'echet spaces.
		
		If $V$ even is of LB-type, i.e.\ the increasing union $V = \bigcup_{n\in \BN} V_n$ of BH-subspaces $V_n$, then $C^\la(X,V)$ is an LB-space, i.e.\ topologically isomorphic to the inductive limit of a sequence of $K$-Banach spaces.
		\item
		\textnormal{(\cite[Satz 2.3.2]{FeauxdeLacroix99TopDarstpAdischLieGrp})}
		If $V$ is of compact type, then $C^\la(X,V)$ is of compact type and 
		\begin{equation}\label{Eq 1 - Tensor product decomposition for locally analytic functions with values in space of compact type}
			C^\la(X,K) \cotimes{K} V \overset{\cong}{\lra } C^\la(X,V) \,,\quad f\otimes v \lto f(\blank) v ,
		\end{equation}
		is a topological isomorphism.
		In particular, $C^\la(X,K)$ is of compact type in this case.
	\end{altenumerate}
\end{proposition}
\begin{proof}
	First note that every disjoint open covering of $X$ necessarily is finite by compactness. 
	Because the finite sum of BH-subspaces is again a BH-subspace \cite[Prop.\ 1.1.5]{Emerton17LocAnVect}, the set of $V$-indices of $X$ which have the same BH-subspace for all charts is cofinal in the set of all $V$-indices of $X$.
	This shows the topological isomorphism in (i).

	For (ii), if $V$ is of LF-type, then \cite[I.\ \S 3.3 Prop.\ 1]{Bourbaki87TopVectSp1to5} implies that, for every BH-subspace $E$ of $V$, the injection $\widebar{E} \hookrightarrow V$ factors over some $\iota_n$ via a continuous injection into $V_n$.
	The inductive limit over these yields a continuous injection $\varinjlim_{E\subset V} C^\la(X,\widebar{E}) \hookrightarrow \varinjlim_{n\in \BN} C^\la(X,V_n)$.
	Moreover, the $\iota_n$ give rise to a continuous injection $ \iota \colon \varinjlim_{n\in \BN} C^\la(X,V_n) \hookrightarrow C^\la(X,V)$.
	The composition
	\begin{equation*}
		\varinjlim_{E\subset V} C^\la(X,\widebar{E}) \longhookrightarrow \varinjlim_{n\in \BN} C^\la(X,V_n) \longhookrightarrow C^\la(X,V)
	\end{equation*}
	then agrees with the topological isomorphism from (i).
	Therefore $\iota$ itself is a topological isomorphism.

	Let $(\CU_{n})_{n\in \BN}$ be a cofinal sequence of disjoint open coverings of $X$, say $\CU_{n} = \{ U_{n,i} \}_{i\in I_n}$ with finite index sets $I_n$.
	Taking the inductive limit with respect to the BH-subspaces of $V_n$ first, we obtain a topological isomorphism
	\begin{equation}\label{Eq 1 - Direct limit description of locally analytic functions}
		C^\la(X,V_n) \cong \varinjlim_{n \in \BN} \,\, \prod_{i\in I_n} \varinjlim_{E \subset V_n}  C^\aan (U_{n,i},\widebar{E}) .
	\end{equation}
	Since $V_n$ is a $K$-Fr\'echet space, there is a topological isomorphism 
	\begin{equation}\label{Eq 1 - Tensor product for some Frechet spaces}
		\varinjlim_{E \subset V_n} C^\aan (U_{n,i},\widebar{E}) \cong C^\aan (U_{n,i}, K) \cotimes{K} V_n ,
	\end{equation}
	cf.\ \cite[Prop.\ 2.1.13 (ii)]{Emerton17LocAnVect}.
	Because the latter is a $K$-Fr\'echet space (see the discussion after \cite[Prop.\ 17.6]{Schneider02NonArchFunctAna}), we have exhibited $C^\la(X,V)$ as an inductive limit of $K$-Fr\'echet spaces.
	Furthermore, if $V$ is of LB-type, we may assume that the $V_n$ are $K$-Banach spaces.
	Since \eqref{Eq 1 - Tensor product for some Frechet spaces} is a $K$-Banach space then and the products in \eqref{Eq 1 - Direct limit description of locally analytic functions} are finite, the above also shows that $C^\la(X,V)$ is an LB-space in this case.

	For (iii), in regard of Remark \ref{Rmk 1 - Locally analytic manifold has disjoint countable covering by compact open subsets} (ii), we may distinguish the cases that $X$ is discrete or that $L$ is locally compact.
	In the first case, we have $C^\la(X,V) \cong V^n$, for some $n\in \BN$.
	Let us now assume that $L$ is locally compact.
	Here we want to find a sequence $(\CI_n)_{n\in \BN}$ of $V$-indices of $X$, with $\CI_{n+1}$ finer than $\CI_n$, which is cofinal and such that the transition maps $C^\la_{\CI_n}(X,V) \hookrightarrow C^\la_{\CI_{n+1}}(X,V)$ are compact.
	Applying \Cref{Prop 1 - Disjoint coverings and locally analytic functions} to some finite disjoint covering of $X$ by charts, it suffices to consider $X=B_{r}^{m}(0)\subset L^m$, for $r \defeq (r,\ldots,r) \in \BR_{>0}^m$.
	We fix $\varepsilon \in \abs{L}$ with $0<\varepsilon <1$.
	Since $L$ is locally compact, for each $n\in \BN$, we find a finite family of closed balls $(B^{m}_{\varepsilon^n r}(a_{n,i}))_{i=1,\ldots,d_n}$ of radius $\varepsilon^n r$ that constitute a disjoint covering of $X$.
	Then, for $B_{\varepsilon^{n+1} r}^{m}(a)  \subset B_{\varepsilon^{n} r}^{m}(b)$, the induced homomorphism
	\begin{equation*}
		C^\aan \big(B_{\varepsilon^{n} r}^{m}(b), K \big) \lra C^\aan \big(B_{\varepsilon^{n+1} r}^{m}(a), K \big)
	\end{equation*}
	of $K$-Banach spaces is compact, cf.\ \cite[\S 16 Claim, p.\ 98]{Schneider02NonArchFunctAna}.

	If $V$ is of compact type, let $(V_n)_{n\in \BN}$ be an inductive sequence of $K$-Banach spaces with injective and compact transition maps such that $V=  \varinjlim_{n\in \BN} V_n$.
	We define the $V$-indices $\CI_n \defeq \big(B^{m}_{\varepsilon^n r}(a_{n,i}), \varepsilon^n r, V_n \big)_{i=1,\ldots,d_n}$ of $X$ which form a cofinal sequence.
	Moreover the homomorphism
	\begin{equation*}
		\begin{tikzcd}
			C^\aan \big(B_{\varepsilon^{n} r}^{m}(b), V_n \big) \ar[r] \ar[d, "\cong"] &C^\aan \big(B_{\varepsilon^{n+1} r}^{m}(a), V_{n+1} \big) \ar[d, "\cong"] \\
			C^\aan \big(B_{\varepsilon^{n} r}^{m}(b) , K \big) \cotimes{K} V_n \ar[r] &C^\aan \big(B_{\varepsilon^{n+1} r}^{m}(a), K \big) \cotimes{K} V_{n+1}
		\end{tikzcd}
	\end{equation*}
	is compact by \cite[Lemma 18.12]{Schneider02NonArchFunctAna}.
	It follows from \Cref{Lemma A1 - Generalities on compact maps} (iii) that the homomorphism $C^\la_{\CI_n} (X,V) \hookrightarrow C^\la_{\CI_{n+1}}(X,V)$ given by the sum of these is compact, for all $n\in \BN$, so that $C^\la(X,V)$ is of compact type.

	Finally, \cite[Prop.\ 1.1.32 (i)]{Emerton17LocAnVect} shows that \eqref{Eq 1 - Tensor product decomposition for locally analytic functions with values in space of compact type} is a topological isomorphism.
\end{proof}

\begin{corollary}[{cf.\ \cite[p.\ 40]{Emerton17LocAnVect}}]\label{Cor 1 - Space of K-valued locally analytic functions is reflexive and barrelled}
	Let $X$ be a locally compact locally $L$-analytic manifold.
	Then $C^\la(X,K)$ is reflexive and complete.
\end{corollary}
\begin{proof}
	By Remark \ref{Rmk 1 - Locally analytic manifold has disjoint countable covering by compact open subsets} (iii), we find a covering $X = \bigcup_{i\in I} X_i$ by compact open subsets.
	Then \Cref{Prop 1 - Direct limit description of locally analytic functions for compact manifold} (iii) implies that $C^\la(X_i,K)$ is reflexive and complete \cite[Prop.\ 16.10]{Schneider02NonArchFunctAna}, for all $i\in I$.
	As both properties are preserved under taking products (\cite[Prop.\ 9.10 and 9.11]{Schneider02NonArchFunctAna} resp.\ \cite[Comment before Lemma 7.8]{Schneider02NonArchFunctAna}), the claim follows from \Cref{Prop 1 - Disjoint coverings and locally analytic functions}.
\end{proof}

\begin{proposition}[{cf.\ \cite[Lemma A.1]{SchneiderTeitelbaum05DualAdmLocAnRep} and the discussion after \cite[Thm.\ 12.2]{SchneiderTeitelbaum04ContLocAnRepThLectHangzhou}}]\label{Prop 1 - Locally analytic functions on product of manifolds}
	Let $X$ and $Y$ be compact locally $L$-analytic manifolds.
	Then the map
	\begin{equation}\label{Eq 1 - Locally analytic functions on product of manifolds}
		C^\la \big(X , C^\la(Y,K) \big) \overset{\cong}{\lra} C^\la(X\times Y,K) \,, \quad f \lto \left[(x,y) \mto f(x)(y)\right]
	\end{equation}
	is a well-defined topological isomorphism.
\end{proposition}
\begin{proof}
	It suffices to define \eqref{Eq 1 - Locally analytic functions on product of manifolds} on all $C^\la(Y,K)$-indices of $X$.
	To this end, consider a $K$-index $\CJ= \big(\psi_j \colon W_j \ra L^{n_j}, \ul{s}_j, K \big)_{j=1,\ldots,e}$ of $Y$ with necessarily finite index set as $Y$ is compact.
	Then $C^\la_\CJ (Y,K) \hookrightarrow C^\la(Y,K)$ is a BH-subspace, and BH-subspaces of this form exhaust $C^\la(Y,K)$.
	Hence it suffices to consider the $C^\la(Y,K)$-indices of $X$ of the form $\CK = \big(\varphi_i \colon U_i \ra L^{m_i},\ul{r}_i, C^\la_\CJ(Y,K) \big)_{i=1,\ldots,d}$, for $\CI = \big(\varphi_i\colon U_i \ra L^{n_i},\ul{r}_i, K \big)_{i=1,\ldots,d}$ a $K$-index of $X$ and $\CJ$ a $K$-index of $Y$.
	But using 
	\begin{equation}\label{Eq 1 - Power series on product of domains}
		\CA_{\ul{r}_i} \big(L^{m_i}, \CA_{\ul{s}_j}(L^{n_j},K) \big) \cong \CA_{(\ul{r}_i,\ul{s}_j)}(L^{m_i+n_j},K)
	\end{equation}
	we find that
	\begin{align*}
		C^\la_\CK \big(X, C^\la(Y,K) \big) &= \prod_{i=1}^d C^\aan \bigg(\varphi_i, \prod_{j=1}^e C^\aan (\psi_j , K) \bigg) \\
		&\cong \prod_{i,j=1}^{d,e} C^\aan (\varphi_i \times \psi_j ,K)  = C^\la_{\CI \times \CJ}(X \times Y, K)
	\end{align*}
	where $\CI \times \CJ$ is the obvious $K$-index of $X\times Y$.
	This way, we obtain the continuous homomorphism \eqref{Eq 1 - Locally analytic functions on product of manifolds}.
	Furthermore, by applying \eqref{Eq 1 - Power series on product of domains} one sees that
	\begin{equation*}
		C^\la(X\times Y, K) \lra C^\la \big(X,C^\la(Y,K) \big) \,,\quad f \lto \big[ x \mto \left[y \mto f(x,y) \right] \big] ,
	\end{equation*}
	defines a $K$-linear map which is inverse to \eqref{Eq 1 - Locally analytic functions on product of manifolds}.
	It follows from the open mapping theorem \cite[Thm.\ 1.1.17]{Emerton17LocAnVect} that \eqref{Eq 1 - Locally analytic functions on product of manifolds} is even a topological isomorphism.
\end{proof}

\begin{corollary}\label{Cor 1 - Locally analytic functions on product of manifolds}
	Let $X$ and $Y$ be compact locally $L$-analytic manifolds.
	Then there is a topological isomorphism
	\begin{equation*}
		C^\la(X,K) \cotimes{K} C^\la(Y,K)\overset{\cong}{\lra} C^\la(X\times Y, K)  \,,\quad f \otimes g \lto \left[ (x,y) \mto f(x) g(y) \right] .
	\end{equation*}
\end{corollary}
\begin{proof}
	This follows from combining \Cref{Prop 1 - Locally analytic functions on product of manifolds} with \Cref{Prop 1 - Direct limit description of locally analytic functions for compact manifold} (iii).
\end{proof}

\subsection{Locally Analytic Representations}

In this section $L \subset K$ continues to be a complete subfield of a non-archimedean field $K$ with non-trivial absolute value $\abs{\blank}$.
We will recall the notion of locally analytic representations of locally $L$-analytic Lie groups from \cite{FeauxdeLacroix99TopDarstpAdischLieGrp} which also readily generalizes to our case (cf.\ \cite[Part I, App.\ A]{Graef21BoundDistr}).

\begin{definition}\label{Def 1 - Definition Lie group}
	A \textit{locally $L$-analytic Lie group} (or \textit{non-archimedean Lie group}) is a locally $L$-analytic manifold $G$ which carries the structure of a group such that the multiplication and inversion maps
	\begin{equation*}
		m \colon G \times G \lra G \,,\qquad \inv\colon G \lra G
	\end{equation*}
	are locally $L$-analytic.
\end{definition}

\begin{remarks}
	\begin{altenumerate}
		\item
		In the above definition it suffices to assume that the multiplication map is locally $L$-analytic because this already implies that the inversion map is locally $L$-analytic as well, see \cite[5.12.1]{Bourbaki07VarDiffAnFasciDeResult}, \cite[Prop.\ 13.6]{Schneider11pAdicLieGrps}.
		\item
		If $L$ is locally compact, then in particular every non-archimedean Lie group $G$ is a topological group which is Hausdorff, totally disconnected, and locally compact.
		Therefore, each neighbourhood of the identity element $e$ in $G$ contains an open subgroup of $G$ \cite[Ch.\ III.\ \S 4.6, Cor.\ 1]{Bourbaki66GenTop1}.
		This implies that each neighbourhood of $e$ in $G$ also contains a compact open subgroup.
	\end{altenumerate}	
\end{remarks}

\begin{definition}
	A \textit{locally $L$-analytic subgroup} $H$ of a locally $L$-analytic Lie group $G$ is a subgroup $H \subset G$ which is a locally $L$-analytic submanifold.
	Such a subgroup naturally acquires the structure of a locally $L$-analytic Lie group itself, and is closed in $G$ necessarily, see \cite[5.12.3]{Bourbaki07VarDiffAnFasciDeResult}.
\end{definition}

\begin{example}
	Assume that $L$ is locally compact with uniformizer $\unif$, and let $d\in \BN$.
	The group $\GL_d(L)$ is an example of a locally $L$-analytic Lie group.
	A family of charts centred at the identity is given by 
	\begin{equation*}
		1 + \unif^n M_d(\CO_L) \lra B^{d^2}_{\abs{\unif}^n} (0) \,,\quad 1 +(a_{ij}) \lto (a_{ij}).
	\end{equation*}
\end{example}

\begin{definition}
	A \textit{(left) locally analytic $G$-representation} of a locally $L$-analytic Lie group $G$ is a barrelled, Hausdorff locally convex $K$-vector space $V$ with a $G$-action by continuous endomorphisms such that the orbit maps $G \ra V$, $g \mto g.v$, are locally analytic in the sense of \Cref{Def 1 - Definition locally analytic function}, for all $v\in V$.
	A \textit{homomorphism of locally analytic $G$-representations} between $V$ and $W$ is a $G$-equivariant continuous homomorphism $V \ra W$.
\end{definition}

\begin{remark}
	By definition the map $G\times V \ra V$, $(g,v)\mto g.v$, of a locally analytic $G$-representation is separately continuous.
	But if $G$ is locally compact (e.g.\ if $L$ is locally compact) this already is equivalent to being jointly continuous by \Cref{Lemma A1 - Separately continuous is jointly continuous under maps by locally compact space on barrelled vector space}.
\end{remark}

\begin{example}\label{Expl 1 - Examples of locally analytic representations}
	\begin{altenumerate}
		\item
		We refer to \Cref{Sect - Continuous and Locally Analytic Characters} for a discussion of locally analytic characters $\psi\colon L \ra K\unts$ and $\chi \colon L\unts \ra K\unts$ when $L$ is a local non-archimedean field.
		In particular in \Cref{Thm A2 - Locally analytic characters in equal characteristic} we show that, for $L$ a local field of ${\rm char}(L)=p>0$, every locally $L$-analytic character $\chi\colon 1+\Fm_L \ra K\unts$ is of the form $\chi(z)= z^c$, for some $c\in \BZ_p$.
		Here $1+\Fm_L \subset L\unts$ denotes the subgroup of principal units satisfying $\abs{z-1}< 1$.
		\item
		Let $G$ be a compact locally $L$-analytic Lie group, and $V$ a Hausdorff locally convex $K$-vector space.
		Then
		\begin{equation*}
			G \times C^\la(G,V) \lra C^\la(G,V) \,,\quad (g,f) \lto f(g^{-1}  \blank) \defeq \big[h \mto f(g^{-1}h) \big],
		\end{equation*}
		defines a locally analytic $G$-representation called the \textit{left regular $G$-representation with coefficients in $C^\la(G,V)$}.
		Indeed, $C^\la(G,V)$ is barrelled and Hausdorff by \Cref{Cor 1 - Space of locally analytic functions is Hausdorff and barrelled}, and the $G$-action is via continuous endomorphisms by \Cref{Prop 1 - Functorialities for the space of locally analytic functions} (ii).
		To see that the orbit maps are locally analytic, consider the locally analytic map of locally $L$-analytic manifolds
		\begin{equation*}
			G \times G \lra G \,,\quad (g,h) \lto g^{-1} h.
		\end{equation*}
		Using \Cref{Prop 1 - Locally analytic functions on product of manifolds} and functoriality, this map induces the homomorphism
		\begin{equation*}
			C^\la(G,V) \lra C^\la(G\times G, V) \cong C^\la( G, C^\la(G,V)) \,,\quad f \lto [g \mto [h \mto f(g^{-1}h)]],
		\end{equation*}
		whose image precisely consists of the orbit maps.

		Similarly, one shows that the \textit{right regular $G$-representation with coefficients in $C^\la(G,V)$}
		\begin{equation*}
			G \times C^\la(G,V) \lra C^\la(G,V) \,,\quad (g,f) \lto f(\blank g) \defeq [h \mto f(hg)],
		\end{equation*}
		and the \textit{$G$-representation by conjugation with coefficients in $C^\la(G,V)$}
		\begin{equation*}
			G \times C^\la(G,V) \lra C^\la(G,V) \,,\quad (g,f) \lto f(g^{-1} \blank g),
		\end{equation*}
		are locally analytic $G$-representations.
	\end{altenumerate}
\end{example}

\begin{proposition}[{cf.\ \cite[Satz 3.1.7]{FeauxdeLacroix99TopDarstpAdischLieGrp}, \cite[Prop.\ 3.6.14]{Emerton17LocAnVect}}]\label{Prop 1 - Subrepresentations and quotientrepresentations of locally analytic representations}
	Let $G$ be a locally $L$-analytic Lie group, and $V$ a locally analytic $G$-representation.
	Let $W\subset V$ be a $G$-invariant closed subspace.
	\begin{altenumerate}
		\item
		Then $V/W$ is a locally analytic $G$-representation with respect to the induced $G$-action.
		\item
		If $W$ is barrelled, then $W$ is a locally analytic $G$-representation with respect to the induced $G$-action.	
	\end{altenumerate}
\end{proposition}
\begin{proof}
	The quotient space $V/W$ is barrelled \cite[Expl.\ 4) after Cor.\ 6.16]{Schneider02NonArchFunctAna}, and Hausdorff because $W\subset V$ is closed.
	Moreover, the $G$-invariance of $W$ ensures that $G$ acts by continuous endomorphisms on $W$ and $V/W$.
	To show that the orbit maps are locally analytic, consider a BH-subspace $E\subset V$.
	As $W\subset V$ is closed, $E\cap W \subset W$ is a BH-subspace.
	Therefore the orbit maps of $W$ are locally analytic.
	Furthermore, by the functoriality of \Cref{Prop 1 - Functorialities for the space of locally analytic functions} (i), we have a continuous homomorphism $C^\la(G,V) \ra C^\la(G,V/W)$.
	For $v\in V$, the image of its orbit map under this homomorphism is the orbit map of the residue class $v+W$.	
\end{proof}

\begin{proposition}[{cf.\ \cite[Lemma 3.2.4]{FeauxdeLacroix99TopDarstpAdischLieGrp}, \cite[Prop. 3.6.11]{Emerton17LocAnVect}}]\label{Prop 1 - Locally analytic representations and open subgroups}
	Let $H$ be an open subgroup of a locally $L$-analytic Lie group $G$, and $V$ a locally convex $K$-vector space on which $G$ acts by continuous endomorphisms.
	Then $V$ is a locally analytic $G$-representation if and only if $V$ is a locally analytic $H$-representation with respect to the induced $H$-action.
\end{proposition}
\begin{proof}
	If $V$ is a locally analytic $H$-representation, consider $v\in V$ and $g\in G$.
	Then there exists an analytic chart $U\subset H$ around the identity element $e$ such that the orbit map $\rho_{g.v}$ is given by a convergent power series there.
	Hence the orbit map $\rho_v$ is given by a convergent power series on the analytic chart $Ug$ around $g$.
	This shows that $V$ is a locally analytic $G$-representation.
	The converse implication is clear.
\end{proof}

\begin{proposition}[{For ${\rm char}(L)=0$, cf.\ \cite[Kor.\ 3.1.9]{FeauxdeLacroix99TopDarstpAdischLieGrp}}]\label{Prop 1 - Equivalent characterization for locally analytic representations on Banach spaces}
	Let $G$ be a locally $L$-analytic Lie group, and $E$ a $K$-Banach space with an abstract $G$-action.
	Then $E$ is a locally analytic representation with respect to this $G$-action if and only if the $G$-action on $E$ is given by an analytic linear representation in the sense of Bourbaki \cite[III.\ \S 1.2 Expl.\ (3)]{Bourbaki89LieGrpLieAlg1to3}, i.e.\ a locally $L$-analytic homomorphism $\rho\colon G \ra \GL(E)\subset \CL(E,E)$ of Lie groups\footnote{For the definition of locally $L$-analytic manifolds with charts taking values in $L$-Banach spaces and locally $L$-analytic maps thereof, see \cite[\S 5.1]{Bourbaki07VarDiffAnFasciDeResult}}.
	Here we view $E$ as an $L$-Banach space via restriction of scalars.

	In particular, a locally analytic $G$-representation on a $K$-Banach space $E$ is \textit{uniformly locally analytic}: For every $g\in G$, there exists a neighbourhood $U\subset G$ of $g$ such that on $U$ all orbit maps $\rho_v\res{U}$, $v\in E$, are given by convergent power series.
\end{proposition}
\begin{proof}
	Our proof differs from the one in \cite{FeauxdeLacroix99TopDarstpAdischLieGrp} where differentiation with respect to elements of the Lie algebra of $G$ is used.

	First assume that $E$ is a locally analytic $G$-representation.
	This implies that $G$ acts by continuous endomorphisms, i.e.\ that there is a homomorphism $\rho \colon G \ra \GL(E)$ of (abstract) groups.
	To show that $\rho$ is a locally $L$-analytic map, it suffices to consider a fixed $g\in G$ and a chart $\varphi\colon U \ra \varphi(U)$ centred at $g$.
	We may now apply \Cref{Prop A1 - Strong and weak analytic functions on a Banach space} to the function $\rho \circ \varphi^{-1} \colon \varphi(U)\ra \CL_b(E,E)$ and the continuous $L$-bilinear pairing
	\begin{equation}\label{Eq 1 - Concrete pairing of Banach spaces}
		\CL_b(E,E) \times E \lra E \,,\quad (\lambda, v)  \lto \lambda(v),
	\end{equation}
	of $L$-Banach spaces.
	The continuity of \eqref{Eq 1 - Concrete pairing of Banach spaces} follows from the fact that the topology of $\CL_b(E,E)$ is induced by the operator norm \cite[Rmk.\ 6.7]{Schneider02NonArchFunctAna}.
	This proposition then says that $\rho \circ \varphi^{-1}$ is analytic in some open neighbourhood of $0$ since the orbit maps of $\rho$ are locally analytic.
	%We consider $\GL(E)$ as a locally $L$-analytic Lie group
	
	Conversely, if $\rho\colon G \ra \GL(E)$ is a locally $L$-analytic homomorphism of Lie groups, the opposite implication of \Cref{Prop A1 - Strong and weak analytic functions on a Banach space} implies that the orbit maps $G \ra E, g\mto g.v$, for $v\in E$, are locally analytic.
	Moreover, $G$ clearly acts by continuous endomorphisms because $\rho(G) \subset \CL(E,E)$.
\end{proof}

\subsection{Modules over Locally Analytic Distribution Algebras}

Let $K$ be a complete non-ar\-chi\-me\-de\-an field which is spherically complete, and let $L\subset K$ a locally compact complete subfield.
Note that with these assumptions every locally $L$-analytic manifold admits a disjoint countable covering by compact open subsets (see  Remark \ref{Rmk 1 - Locally analytic manifold has disjoint countable covering by compact open subsets} (iii)), and the Hahn--Banach theorem for locally convex $K$-vector spaces applies \cite[Prop.\ 9.2]{Schneider02NonArchFunctAna}.

In this section, we want to review locally analytic distributions and their interplay with locally analytic representations.
We follow \cite{SchneiderTeitelbaum02LocAnDistApplToGL2} as the characteristic of $L$ again makes no difference.
However occasionally we will need and prove slightly stronger statements.

\begin{definition}
	\begin{altenumerate}
		\item
		Let $X$ be a locally $L$-analytic manifold.
		The space of \textit{locally analytic distributions on $X$} is defined as the strong dual space
		\[D(X,K) \defeq C^\la (X,K)'_b .\]
		
		\item
		For $x\in X$, the homomorphism
		\[\delta_x \colon C^\la (X,K) \lra K\,,\quad f \lto f(x),\]
		is continuous by \Cref{Prop 1 - Evaluation maps are continuous}.
		This element $\delta_x \in D(X,K)$ is called the \textit{Dirac distribution supported at $x$}.
	\end{altenumerate}
\end{definition}

\begin{proposition}[{cf.\ \cite{SchneiderTeitelbaum02LocAnDistApplToGL2}}]\label{Prop 1 - Properties of the distribution algebra}
	Let $X$ be a locally $L$-analytic manifold.
	\begin{altenumerate}
		\item
		The locally convex $K$-vector space $D(X,K)$ is reflexive.
		\item
		If $X$ is compact, $D(X,K)$ is a nuclear $K$-Fr\'echet space.
		\item
		Given a disjoint covering $X= \bigcup_{i\in I} X_i$ by open subsets $X_i$, there is a natural topological isomorphism
		\[D(X,K) \cong \bigoplus_{i\in I} D(X_i,K) .\]
		In particular, for every $\mu \in D(X,K)$, there exists some compact open subset $Y \subset X$ on which it is supported, i.e.\ for which $\mu \in D(Y,K)\subset D(X,K)$.
		\item
		The subspace of $D(X,K)$ generated by all Dirac distributions $\delta_x$, $x\in X$, is dense.			
	\end{altenumerate}
\end{proposition}
\begin{proof}
	If $X$ is compact, then $C^\la (X,K)$ is of compact type by \Cref{Prop 1 - Direct limit description of locally analytic functions for compact manifold} (iii).
	Therefore its strong dual $D(X,K)$ is a nuclear $K$-Fr\'echet space (see \cite[Prop.\ 16.10]{Schneider02NonArchFunctAna} and \cite[Prop.\ 19.9]{Schneider02NonArchFunctAna}) showing (ii).

	The statement of (iii) follows from \Cref{Prop 1 - Disjoint coverings and locally analytic functions} and \cite[Prop.\ 9.11]{Schneider02NonArchFunctAna}.

	Note that a nuclear Fr\'echet space is reflexive \cite[Cor.\ 19.3 (ii)]{Schneider02NonArchFunctAna} which settles (i) if $X$ is compact.
	In the general case, we find a countable disjoint covering $X=\bigcup_{i\in I} X_i$ by open compact subsets using Remark \ref{Rmk 1 - Locally analytic manifold has disjoint countable covering by compact open subsets} (iii).
	It follows from \cite[Prop.\ 9.10]{Schneider02NonArchFunctAna} and \cite[Prop.\ 9.11]{Schneider02NonArchFunctAna} that the direct sum of reflexive locally convex $K$-vector spaces is reflexive again.	
	Applying this to the direct sum of (iii), for $X=\bigcup_{i\in I} X_i$, we see that $D(X,K)$ is reflexive.

	Finally, let $\Delta \subset D(X,K)$ denote the closure of the subspace generated by the Dirac distributions of $X$ and assume that $\Delta \subsetneq D(X,K)$.
	By the Hahn-Banach theorem \cite[Cor.\ 9.3]{Schneider02NonArchFunctAna}, there is a continuous linear functional $\widebar{\ell}\colon D(X,K)/\Delta \ra K$ which is non-zero.
	This induces a non-zero continuous functional $\ell \colon D(X,K)\ra K$ that vanishes on $\Delta$.
	Because $D(X,K)$ is reflexive, $\ell$ corresponds to a locally analytic function $f $ on $X$.
	But the vanishing of $\ell$ on $\Delta$ implies that $f=0$ which is a contradiction.
\end{proof}

\begin{proposition}[{cf.\ \cite[Prop.\ 2.3]{SchneiderTeitelbaum02LocAnDistApplToGL2}\footnote{For the proof of this proposition, Schneider and Teitelbaum refer to the diploma thesis \cite{FeauxdeLacroix92pAdischDist} here which was not available to me.} and \cite[App.]{SchneiderTeitelbaum05DualAdmLocAnRep}}]\label{Prop 1 - Convolution product}
	Let $G$ be a locally $L$-analytic Lie group.
	There exists a separately continuous $K$-bilinear map
	\begin{equation}\label{Eq 1 - Definition of convolution product}
		D(G,K) \times D(G,K) \lra D(G,K) \,,\quad (\mu, \nu) \lto \mu \ast \nu ,
	\end{equation}
	such that $\delta_g \ast \delta_{g'} = \delta_{gg'}$, for $g, g'\in G$.
	Concretely, for $\mu, \nu \in D(G,K)$ supported on compact open subsets $H$ respectively $H'$ of $G$, $\mu \ast \nu$ factors over $C^\la(H \cdot H', K)$ and is given by\footnote{Recall that $H\cdot H'$ denotes the set $\left\{hh'\middle{|}h\in H, h'\in H'\right\}\subset G$.}
	\begin{equation}\label{Eq 1 - Explicit definition of convolution product}
		\begin{tikzcd}[row sep = 0ex ,	/tikz/column 1/.append style={anchor=base east}	,	/tikz/column 2/.append style={anchor=base west}]
			C^\la(H \cdot H',K) \ar[r, "m^\ast"] & C^\la(H\times H',K) \cong C^\la(H,K) \cotimes{K} C^\la (H',K) \ar[r, "\mu \otimes \nu"] & K . \\
			f \ar[r, mapsto] & \left[ (h,h') \mto f(h h') \right] & 
		\end{tikzcd}
	\end{equation}
	If $G$ is compact then \eqref{Eq 1 - Definition of convolution product} is even jointly continuous.
\end{proposition}
\begin{proof}
	Let $G = \bigcup_{i \in I} H_i$ be a countable disjoint covering by compact open subsets. 
	Note that there is a topological isomorphism \cite[Cor.\ 1.2.14]{Kohlhaase05InvDistpAdicAnGrp}
	\begin{equation*}
		D(G,K) \indcotimes D(G,K) \cong \bigoplus_{i,j \in I} \left( D(H_i,K) \indcotimes D(H_j,K) \right) .
	\end{equation*}	
	Hence it suffices to define continuous homomorphisms 
	\begin{equation*}
		D(H_i,K) \indcotimes D(H_j,K) \lra D(H_i \cdot H_j, K) \subset D(G,K)
	\end{equation*}
	to define \eqref{Eq 1 - Definition of convolution product}.
	The multiplication map of $G$ induces a continuous homomorphism 
	\begin{equation*}
		C^\la(H_i \cdot  H_j,K) \overset{m^\ast}{\lra} C^\la(H_i\times H_j, K) \cong  C^\la(H_i,K) \projcotimes C^\la(H_j,K)
	\end{equation*}
	using \Cref{Prop 1 - Functorialities for the space of locally analytic functions} (ii) and \Cref{Cor 1 - Locally analytic functions on product of manifolds}.
	Taking the transpose of this homomorphism yields the continuous homomorphism
	\begin{equation*}
		D(H_i,K) \indcotimes D(H_j,K) \cong D(H_i,K) \projcotimes D(H_j,K)  \lra D(H_i \cdot H_j, K)
	\end{equation*}
	by \cite[Prop.\ 17.6]{Schneider02NonArchFunctAna} and \cite[Prop.\ 20.13]{Schneider02NonArchFunctAna}.
	This also shows that the convolution product is given by \eqref{Eq 1 - Explicit definition of convolution product}, and that it is jointly continuous if $G$ is compact.
	Moreover, for $g \in H_i$, $g'\in H_j$, the linear form $\delta_{(g,g')}$ agrees with $\delta_g \otimes \delta_{g'}$ on $C^\la(H_i,K) \botimes{K} C^\la(H_j,K) \subset C^\la(H_i \times H_j,K)$. 
	Using that $C^\la(H_i,K)\botimes{K} C^\la(H_j,K)$ is a dense subspace, this implies that $\delta_g \ast \delta_{g'} = \delta_{gg'}$.
\end{proof}

\begin{definitionproposition}\label{Def 1 - Definition of distribution algebra}
	For a locally $L$-analytic Lie group $G$, the convolution product \eqref{Eq 1 - Definition of convolution product} endows $D(G,K)$ with the structure of an associative, unital $K$-algebra called the \textit{(locally analytic) distribution algebra} of $G$.
	Its unit element is $\delta_e$ where $e$ is the identity element of $G$.
	We also write $D(G) \defeq D(G,K)$ when the coefficient field $K$ is clear from the context.
\end{definitionproposition}
\begin{proof}	
	By the separate continuity of \eqref{Eq 1 - Definition of convolution product} it suffices to check the necessary properties of $D(G)$ only for Dirac distributions.
	But for those they directly follow from the respective properties of $G$ as a group due to $\delta_g \ast \delta_{g'} = \delta_{gg'}$, for $g, g'\in G$.
\end{proof}

\begin{corollary}[{cf.\ \cite[Proof of Prop.\ 3.5]{OrlikStrauch15JordanHoelderSerLocAnRep}}]\label{Cor 1 - Fubini theorem}
	Let $G$ be a locally $L$-analytic Lie group. 
	For $f \in C^\la(G,K)$ and distributions $\mu , \nu \in D(G)$, the functions $G \ra K$ given by
	\begin{align*}
		g \lto \nu \big[ g' \mto f(gg') \big] \quad\text{and}\quad g' \lto \mu \big[ g \mto f(gg') \big]
	\end{align*}
	are locally analytic, and we have the following identities reminiscent of Fubini's theorem:
	\begin{equation}\label{Eq 1 - Compute the convolution product}
		\begin{aligned}
			(\mu \ast \nu )(f) = \mu \big[g \mto \nu\big[g'\mto f(gg') \big]\big]
			= \nu\big[ g'\mto \mu\big[g \mto f(gg') \big]\big].
		\end{aligned}
	\end{equation}
\end{corollary}
\begin{proof}
	Assuming $\nu$ is supported on the compact open subsets $H'\subset G$, the first function restricted to some compact open subset $H\subset G$ is the image of $f$ under
	\begin{equation*}
		\begin{tikzcd}
			C^\la(H \cdot H',K) \ar[r, "m^\ast"] &C^\la(H \times H', K) \cong C^\la(H, C^\la(H',K)) \ar[r, "\nu_\ast"] &C^\la(H,K) .
		\end{tikzcd}
	\end{equation*}
	Analogously, one shows that the second function is locally analytic.
	Now let $\mu$ be supported on the compact open subset $H \subset G$.
	Then the statement of \eqref{Eq 1 - Compute the convolution product} follows from the commutativity of 
	\begin{equation*}
		\begin{tikzcd}[row sep = 0ex]
			&&C^\la(H,K) \cotimes{K} C^\la(H',K) \ar[rd, start anchor=east, end anchor=145, "\mu \otimes \nu"] & \\
			C^\la(H \cdot H',K) \ar[r, "m^\ast"] &C^\la(H \times H', K) \ar[ru, end anchor=west, "\cong"] \ar[rd, start anchor=south east, end anchor = 178, "\cong"] &&K \\
			&&C^\la(H, C^\la(H',K)) \overset{\nu_\ast}{\lra}  C^\la(H,K) \ar[ru, start anchor=east, "\mu"'] &
		\end{tikzcd}
	\end{equation*}
	and the analogous diagram for $\nu \circ \mu_\ast$.
\end{proof}

The locally $L$-analytic anti-automorphism $\inv \colon G \ra G, \, g \mto g^{-1},$ induces by functoriality an automorphism of locally convex $K$-vector spaces
\begin{equation*}
	\inv^\ast \colon C^\la (G,K) \lra C^\la (G,K) \,,\quad f \lto f\circ \inv.
\end{equation*}
Hence we obtain an automorphism of locally convex $K$-vector spaces
\begin{equation}\label{Eq 1 - Definition of involution of distribution algebra}
	D(G) \lra D(G) \,,\quad \mu \lto \dot{\mu} \defeq \mu \circ \inv^\ast = \big[f \mto \mu(f \circ \inv)\big] .
\end{equation}

\begin{lemma}\label{Prop 1 - Involution of distribution algebra and convolution}
	For $\mu, \nu \in D(G)$, we have that
	\begin{equation}
		\left( \mu \ast \nu \right)\dot{} = \dot{\nu} \ast \dot{\mu} .
	\end{equation}
\end{lemma}
\begin{proof}
	We may assume that $\mu$ and $\nu$ are supported on compact open subsets $H$ and $H'$ of $G$ respectively.
	Then the claim follows from the commutativity of the following diagram:
	\begin{equation*}
		\begin{tikzcd}
			C^\la(H'^{-1},K) \cotimes{K} C^\la(H^{-1},K) \ar[r, "\inv^\ast \botimes{} \inv^\ast"] &[+30pt] C^\la(H',K) \cotimes{K} C^\la(H,K) \ar[rd, start anchor=355, "\nu \otimes \mu"]& \\
			C^\la\big((H\cdot H')^{-1},K \big) \ar[u] \ar[d, "\inv^\ast"'] && K \\
			C^\la(H\cdot H',K) \ar[r] &C^\la(H,K) \cotimes{K} C^\la(H',K) \ar[uu, "\mathrm{swap}"] \ar[ru, start anchor= 5, "\mu \otimes \nu"'] & 
		\end{tikzcd}
	\end{equation*}
\end{proof}

\begin{proposition}[{cf.\ \cite[Thm.\ 2.2]{SchneiderTeitelbaum02LocAnDistApplToGL2}}]\label{Prop 1 - Integration map}
	Let $X$ be a locally $L$-analytic manifold and $V$ a Hausdorff locally convex $K$-vector space.
	\begin{altenumerate}
		\item
		There exists a unique continuous $K$-linear \textit{integration map}
		\begin{equation}\label{Eq 1 - Integration map}
			I \colon C^\la (X,V) \lra \CL_b \big(D(X,K),V \big)
		\end{equation}
		such that $I(f)(\delta_x) = f(x)$, for all $f\in C^\la (X,V)$ and $x \in X$.
		Moreover, this map is natural in $X$ and $V$, and injective.
		\item
		If $V$ is of LB-type, i.e.\ $V = \bigcup_{n\in \BN} V_n$, for a sequence $V_0 \subset V_1 \subset \ldots \subset V$ of BH-subspaces, then \eqref{Eq 1 - Integration map} is an isomorphism of $K$-vector spaces with inverse 
		\begin{equation*}
			I^{-1} \colon \CL \big(D(X,K),V \big) \overset{\cong}{\lra} C^\la (X,V) \,,\quad T \lto \big[x \mto T(\delta_x) \big] .
		\end{equation*}
	\end{altenumerate}
\end{proposition}
\begin{proof}
	First note that, for $f\in C^\la(X,V)$, the condition $I(f)(\delta_x) = f(x)$, for all $x\in X$, determines $I(f)$ uniquely by the density of the subspace generated by the Dirac distributions of $X$.
	
	For the existence of $I$ take a countable disjoint covering $X = \bigcup_{i\in I} X_i$ by open compact subsets.
	Then \Cref{Prop 1 - Disjoint coverings and locally analytic functions} and \Cref{Prop 1 - Properties of the distribution algebra} (iii) give topological isomorphisms $C^\la(X,V) \cong \prod_{i\in I} C^\la (X_i,V)$ and $D(X,K) \cong \bigoplus_{i\in I} D(X_i,K)$ respectively, with the $D(X_i,K)$ being $K$-Fr\'echet spaces.
	By \Cref{Lemma A1 - Isomorphism for space of linear maps on direct sum with strong topology} we have the topological isomorphism
	\begin{equation*}
		\CL_b \big(D(X,K),V \big) \overset{\cong}{\lra} \prod_{i\in I} \CL_b \big(D(X_i,K),V \big) \,,\quad F \lto \big(F\res{D(X_i,K)}\big)_{i\in I} .
	\end{equation*}
	This way, we may reduce to the case that $X$ is compact.

	Here we first assume that $V$ is a $K$-Banach space.
	As $C^\la(X,K)$ is of compact type, we have a continuous linear bijection $C^\la (X,V) \ra C^\la(X,K) \cotimes{K} V $ by \Cref{Prop A1 - Continuous bijection for tensor product of space of compact type and Banach space}.
	Together with \cite[Cor.\ 18.8]{Schneider02NonArchFunctAna} this gives the continuous linear bijection
	%as $D(X,K)$ reflexive Fr\'echet implies bornological and semi-reflexive
	\begin{equation}\label{Eq 1 - Integration map for Banach space}
		\begin{tikzcd}[row sep = 0ex	,/tikz/column 1/.append style={anchor=base east}	,/tikz/column 3/.append style={anchor=base west}]
			I \colon C^\la(X,V) \ar[r] & \eqmathbox[A]{C^\la(X,K) \cotimes{K} V} \ar[r, "\cong"] & \CL_b \big(D(X,K),V \big)  \\
			f(\blank) \, v  \ar[r, mapsfrom]& \eqmathbox[A]{f\otimes v} \ar[r, mapsto]& \big[\mu \mto \mu(f) \,v\big]
		\end{tikzcd}
	\end{equation}
	which satisfies $I(f)(\delta_x) = f(x)$.

	If $V=\bigcup_{n\in \BN} V_n$ is of LB-type then every BH-subspace of $V$ factors over some $\widebar{V_n}$ by \cite[I.\ \S 3.3 Prop.\ 1]{Bourbaki87TopVectSp1to5}.
	Hence the compactness of $X$ implies by \Cref{Prop 1 - Direct limit description of locally analytic functions for compact manifold} (i) that
	\[C^\la (X,V) \cong \varinjlim_{n \in \BN} C^\la(X,\widebar{V_n}) .\]
	Furthermore, the continuous injections $\widebar{V_n} \hookrightarrow V$ induce continuous homomorphisms 
	\[\CL_b \big(D(X,K),\widebar{V_n} \big) \lra \CL_b \big(D(X,K),V \big) . \]
	These in turn give rise to a continuous homomorphism
	\[\varinjlim_{n \in \BN} \CL_b \big(D(X,K), \widebar{V_n} \big) \lra \CL_b \big(D(X,K),V \big) \]
	which is bijective by \cite[I.\ \S 3.3 Prop.\ 1]{Bourbaki87TopVectSp1to5}.
	Taking the direct limit over the continuous linear bijections \eqref{Eq 1 - Integration map for Banach space}, for all $\widebar{V_n}$, we now arrive at the continuous linear bijection
	\[I \colon C^\la (X,V) \cong \varinjlim_{n \in \BN} C^\la(X,\widebar{V_n}) \lra \varinjlim_{n\in\BN} \CL_b \big(D(X,K),\widebar{V_n} \big) \lra \CL_b \big(D(X,K),V \big).\]

	For the case of a general Hausdorff locally convex $K$-vector space $V$, we observe that by the compactness of $X$, every locally analytic $V$-valued function on $X$ factors over some BH-subspace of $V$.
	Therefore even in this case, we obtain the injective $K$-linear map $I$.
\end{proof}

\begin{corollary}\label{Prop 1 - Pairing of D(X,K) and Cla(X,V)}
	Let $X$ be a locally $L$-analytic manifold and $V \neq \{0\}$ a Hausdorff locally convex $K$-vector space.
	Then there is a natural, separately continuous, non-degenerate $K$-bilinear pairing
	\begin{equation}\label{Eq 1 - Pairing for distribution algebra}
		D(X,K) \times C^\la(X,V) \lra V \,,\quad (\mu, f) \lto \mu(f) \defeq I(f)(\mu).
	\end{equation}
	This pairing is induced by the duality between $D(X,K)$ and $C^\la(X,K)$ in the sense that, for compact open $U\subset X$ and a BH-subspace $E \subset V$, the restriction of the pairing \eqref{Eq 1 - Pairing for distribution algebra} to $D(U,K) \times C^\la(U,\widebar{E})$ is given by tensoring the duality pairing $D(U,K) \times C^\la(U,K)\ra K$ with $\widebar{E}$:
	\begin{equation}\label{Eq 1 - Pairing for distribution algebra on compact subset and for BH-subspace}
		D(U,K) \times C^\la(U,K) \cotimes{K} \widebar{E} \lra \widebar{E} \,,\quad (\mu, f\otimes v) \lto \mu(f) \, v ,
	\end{equation}
	using \Cref{Prop A1 - Continuous bijection for tensor product of space of compact type and Banach space}.
\end{corollary}
\begin{proof}
	It is clear that the pairing \eqref{Eq 1 - Pairing for distribution algebra} defined by $I$ is natural, $K$-bilinear, and separately continuous.
	The claimed compatibility with the duality pairing between $D(X,K)$ and $C^\la(X,K)$ follows from the construction of $I$ in \eqref{Eq 1 - Integration map for Banach space}.

	To show the non-degeneracy, fix a distribution $\mu \in D(X,K)$ and assume that $\mu(f)=0$, for all $f\in C^\la(X,V)$.
	Let $\mu$ be supported on a compact open subset $U\subset X$.
	Moreover, we find $v \in V$, $v\neq 0$, contained in some BH-subspace $E\subset V$.
	For all $h \in C^\la(U,K)$, we then have $\mu(h)\, v = \mu(h \otimes v)=0$.
	Therefore, the non-degeneracy of the duality pairing between $D(U,K)$ and $C^\la(U,K)$ implies that $\mu = 0$.
	On the other hand, consider $f \in C^\la(X,V)$ such that $\mu(f)=0$, for all $\mu \in D(X,K)$.
	Then we have $f(x)= I(f)(\delta_x) =0$, for all Dirac distributions $\delta_x$, and hence $f=0$.
\end{proof}

\begin{proposition}[{\cite[Prop.\ 3.2]{SchneiderTeitelbaum02LocAnDistApplToGL2}}]\label{Prop 1 - Module structures over the distribution algebra}
	Let $G$ be a locally $L$-analytic Lie group, and $V$ a locally analytic $G$-re\-pre\-sen\-tation with orbit maps $\rho_v \colon G \ra V$, for $v\in V$.
	\begin{altenumerate}
		\item
		The $K$-bilinear map
		\begin{equation}\label{Eq 1 - Map for D(G,K)-module structure}
			D(G) \times V \lra V \,,\quad (\mu,v) \lto \mu \ast v \defeq I(\rho_v)(\mu),
		\end{equation}
		is separately continuous, and $V$ becomes a $D(G)$-module this way.
		If $G$ is compact and $V$ a $K$-Fr\'echet space, \eqref{Eq 1 - Map for D(G,K)-module structure} even is jointly continuous.
		\item
		The $K$-bilinear map
		\begin{equation}\label{Eq 1 - Map for contragredient D(G,K)-module structure}
			D(G) \times V'_b \lra V'_b \,,\quad (\mu,\ell) \lto \mu \ast \ell \defeq \big[v\mto \ell(\dot{\mu} \ast v) \big],
		\end{equation}
		given by the $D(G)$-action contragredient to \eqref{Eq 1 - Map for D(G,K)-module structure} is separately continuous, and $V'_b$ becomes a $D(G)$-module this way.
		If $G$ is compact and $V'_b$ a $K$-Fr\'echet space, e.g.\ if $V$ is of compact type, then \eqref{Eq 1 - Map for contragredient D(G,K)-module structure} even is jointly continuous.
	\end{altenumerate}
\end{proposition}

In particular, we have $\delta_g \ast v = g.v$, and $\delta_g \ast \ell = g.\ell$, for all $g\in G$, $v\in V$, $\ell \in V'_b$, where $g.\ell = \ell( g^{-1}.\blank)$ denotes the contragredient $G$-action.

\begin{proof}
	The $K$-bilinearity of \eqref{Eq 1 - Map for D(G,K)-module structure} and its continuity in $D(G)$ are clear.
	Now fix a distribution $\mu \in D(G)$. 
	We may assume that $\mu \in D(H,K)$, for some compact open subset $H \subset G$.
	By \Cref{Prop 1 - Properties of the distribution algebra} (iv), as $D(H,K)$ is metrizable, $\mu$ is the limit of a sequence $(\mu_n)_{n\in \BN}$ in $D(H,K)$ where the $\mu_n$ are linear combinations of Dirac distributions.
	But a Dirac distribution $\delta_g$, $g\in H$, acts by the continuous endomorphism $v\mto g.v$ on $V$.
	Hence the $\mu_n$ act by continuous endomorphisms as well.
	As these continuous endomorphisms of $V$ converge to the endomorphism induced by $\mu$ pointwise, it follows from a version of the Banach-Steinhaus theorem \cite[III.\ \S 4.2, Cor.\ 2]{Bourbaki87TopVectSp1to5} that the endomorphism induced by $\mu$ is continuous.
	
	To show that \eqref{Eq 1 - Map for D(G,K)-module structure} endows $V$ with the structure of a $D(G)$-module, we have to see that $\delta_e \ast v = v$, for the identity element $e$ of $G$, and $\mu \ast (\nu \ast v) = (\mu \ast \nu) \ast v$, for all $v\in V$, $\mu, \nu \in D(G)$.
	But this holds for Dirac distributions, and hence for general elements of $D(G)$ by continuity.
	
	For (ii), note that the homomorphism $D(G) \ra \CL_b(V,V)$ induced from \eqref{Eq 1 - Map for D(G,K)-module structure} is continuous by \cite[III.\ \S 5.3, Prop.\ 6]{Bourbaki87TopVectSp1to5} as $D(G)$ is reflexive and therefore barrelled \cite[Lemma 15.4]{Schneider02NonArchFunctAna}.
	Moreover, taking the transpose gives a topological embedding $\CL_b (V,V) \hookrightarrow \CL_b(V'_b,V'_b) $ by \cite[Prop.\ 1.1.36]{Emerton17LocAnVect}.
	Combining this with \eqref{Eq 1 - Definition of involution of distribution algebra} yields
	\begin{equation*}
		\begin{tikzcd}[
			,row sep = 0ex
			,/tikz/column 1/.append style={anchor=base east}
			,/tikz/column 2/.append style={anchor=base west}
			]
			D(G) \ar[r] &D(G) \ar[r] &\CL_b(V,V) \ar[r, hookrightarrow] & \CL_b(V'_b,V'_b) \\
			\mu \ar[r, mapsto] & \dot{\mu} &&
		\end{tikzcd}
	\end{equation*}
	which gives the separately continuous $K$-bilinear pairing \eqref{Eq 1 - Map for contragredient D(G,K)-module structure}.
	To see that this defines a $D(G)$-module structure on $V'_b$, one again considers Dirac distributions first and then extends to general elements of $D(G)$ by continuity.
	
	If $G$ is compact and $V$ or $V'_b$ is a Fr\'echet space, the joint continuity of \eqref{Eq 1 - Map for D(G,K)-module structure} and \eqref{Eq 1 - Map for contragredient D(G,K)-module structure} follows from \cite[III.\ \S 5.2 Cor.\ 1]{Bourbaki87TopVectSp1to5} because $D(G)$ is a $K$-Fr\'echet space in this case.
\end{proof}

\begin{proposition}[{cf.\ \cite[\S 3]{SchneiderTeitelbaum02LocAnDistApplToGL2}}]\label{Prop 1 - Equivalences for categories of locally analytic representations}
	Let $G$ be a locally $L$-analytic Lie group.
	\begin{altenumerate}
		\item
		Associating a $D(G)$-module structure via \eqref{Eq 1 - Map for D(G,K)-module structure} gives an equivalence of categories
		\begin{equation}\label{Eq 1 - Equivalence of categories between representations and modules over the distribution algebra}
			\left(\substack{\text{\small locally analytic $G$-representations} \\ \text{\small on locally convex $K$-vector spaces} \\ \text{\small of LB-type with continuous} \\ \text{\small  $G$-equivariant homomorphisms}}\right) \lra 
			\left(\substack{\text{\small separately continuous $D(G)$-modules} \\ \text{\small on locally convex $K$-vector spaces} \\ \text{\small of LB-type with continuous} \\ \text{\small $D(G)$-module maps}}\right).
		\end{equation}
		\item
		Passing to the strong dual and associating the $D(G)$-module structure of \eqref{Eq 1 - Map for contragredient D(G,K)-module structure} gives an anti-equivalence of categories
		\begin{align}\label{Eq 1 - Antiequivalence between locally analytic representations and modules over the distribution algebra}
			\begin{split}
				\left(\substack{\text{\small locally analytic $G$-representations} \\ \text{\small on locally convex $K$-vector spaces}\\ \text{\small of compact type with continuous} \\ \text{\small $G$-equivariant homomorphisms}}\right) &\lra \left(\substack{\text{\small separately continuous $D(G)$-modules} \\ \text{\small on nuclear $K$-Fr\'echet spaces} \\ \text{\small with continuous $D(G)$-module maps}}\right) \\
				\left[f\colon V \ra W \right] &\lto \left[f^t \colon W'_b \ra V'_b \right] . 
			\end{split}
		\end{align}
		If $G$ is compact, the latter category already is equal to the category of continuous $D(G)$-modules on nuclear $K$-Fr\'echet spaces with continuous $D(G)$-module homomorphisms.
	\end{altenumerate}
\end{proposition}
\begin{proof}
	For a continuous $G$-equivariant homomorphism $f \colon V \ra W$, we immediately have $\delta_g \ast f(v) = f(\delta_g \ast v)$, for all $g \in G$, $v\in V$.
	Hence it follows by continuity and density of the space of Dirac distributions that $f$ is a $D(G)$-module homomorphism.
	On the other hand, if $V$ is a separately continuous $D(G)$-module and of LB-type, we can define a $G$-action with locally analytic orbit maps $\rho_v \defeq I^{-1}(\mu \mto \mu \ast v)$, for $v \in V$, using \Cref{Prop 1 - Integration map} (ii).
	Then $g \in G$ acts by the endomorphism $V \ra V$, $v \mto \delta_g \ast v$, which therefore is continuous.
	One readily checks that the functor defined this way is a quasi-inverse to \eqref{Eq 1 - Equivalence of categories between representations and modules over the distribution algebra}.
	
	For the statement of (ii), note that by (i) the first category is equivalent to the category of separately continuous $D(G)$-modules on locally convex $K$-vector spaces of compact type with continuous $D(G)$-module maps.
	Now \Cref{Prop A1 - Duality of spaces of compact type and nuclear Frechet spaces} and \Cref{Prop 1 - Module structures over the distribution algebra} (ii) show that the functor \eqref{Eq 1 - Antiequivalence between locally analytic representations and modules over the distribution algebra} is well defined and essentially surjective. 
	Moreover, for locally convex $K$-vector spaces $V$ and $W$ of compact type, one readily checks that the homomorphism
	\begin{equation*}
		\begin{tikzcd}[row sep = 0ex, /tikz/column 1/.append style={anchor=base east},	/tikz/column 3/.append style={anchor=base west}]
			\CL_b(V,W) \ar[r]&\eqmathbox[B]{\CL_b(W'_b,V'_b)} \ar[r]&\CL_b\big( (V'_b)'_b, (W'_b)'_b \big) \cong \CL_b(V,W) ,\\
			f \ar[r, mapsto] &\eqmathbox[B]{f^t} \ar[r, mapsto] &(f^t)^t ,
		\end{tikzcd}
	\end{equation*}
	using the reflexivity of $V$ and $W$ \cite[Prop.\ 16.10 (i)]{Schneider02NonArchFunctAna}, is in fact the identity.
	Combined with a similar argument for $\CL_b(W'_b, V'_b)$ we conclude that the natural map induced by taking the transpose is a topological isomorphism
	\begin{equation*}
		\CL_b(V,W) \lra \CL_b(W'_b, V'_b) \,,\quad f \lto f^t.
	\end{equation*}
	Furthermore, one computes that $f$ is a homomorphism of $D(G)$-modules if and only if $f^t$ is.
	Therefore, \eqref{Eq 1 - Antiequivalence between locally analytic representations and modules over the distribution algebra} is an anti-equivalence of categories.
	Finally, if $G$ is compact, a separately continuous $D(G)$-module structure on a $K$-Fr\'echet space is jointly continuous by \cite[III.\ \S 5.2 Cor.\ 1]{Bourbaki87TopVectSp1to5}, like before.
\end{proof}

\subsection{Locally Analytic Induction}

We keep the setting that $K$ is a spherically complete non-archimedean field and $L\subset K$ a locally compact complete subfield.
We will now recall the notion of locally analytic induction from \cite[Kap.\ 4]{FeauxdeLacroix99TopDarstpAdischLieGrp} and make some easy comparisons to the ``finite'' induction of locally analytic representation.

\begin{definition}[{\cite[\S 4.1]{FeauxdeLacroix99TopDarstpAdischLieGrp}}]\label{Def 1 - Definition of locally analytic induction}
	Let $G$ be a locally $L$-analytic Lie group, $H\subset G$ a locally $L$-analytic subgroup, and $V$ a locally analytic $H$-representation.
	We define the subspace
	\begin{equation*}
		\Ind^{\la,G}_H (V) \defeq \left\{ f \in C^\la(G,V) \middle{|} \forall g\in G, h \in H: f(gh)= h^{-1}.f(g) \right\} \subset C^\la(G,V)
	\end{equation*}
	and consider it with the left regular $G$-action.
	We note that the continuity of the evaluation homomorphisms and the action of $H$ on $V$ imply that this subspace is closed.
\end{definition}

\begin{proposition}[{cf.\ \cite[Satz 4.1.5]{FeauxdeLacroix99TopDarstpAdischLieGrp}}]\label{Prop 1 - When locally analytic induction is locally analytic representation}
	Let $G$ be a locally $L$-analytic Lie group and $H\subset G$ a locally $L$-analytic subgroup.
	Moreover, let $V$ be a locally analytic $H$-representation.
	\begin{altenumerate}
		\item
		\textnormal{(cf.\ \cite[\S 2.1]{Emerton07JacquetModLocAnRep2})}
		If $V$ is of compact type and there exists a compact open subgroup $G_0 \subset G$ such that $G = G_0 \cdot H$, then $\Ind^{\la,G}_H(V)$ is a locally analytic $G$-representation of compact type.
		\item
		\textnormal{(cf.\ \cite[Satz 4.3.1]{FeauxdeLacroix99TopDarstpAdischLieGrp})}
		If $V$ is a $K$-Banach space and $G/H$ is compact, then $\Ind^{\la,G}_H(V)$ is a locally analytic $G$-representation and any section of the projection map $G \ra G/H$ induces a topological isomorphism $\Ind^{\la,G}_H(V) \cong C^\la(G/H,V)$.
	\end{altenumerate}
\end{proposition}
\begin{proof}
	In both cases it follows from the functoriality of \Cref{Prop 1 - Functorialities for the space of locally analytic functions} (ii) that $G$ acts on $\Ind_H^{\la,G}(V)$ by continuous endomorphisms.
	In the first case, we set $H_0 \defeq G_0 \cap H$, for such $G_0\subset G$.
	Arguing analogously to \cite[\S 2.1]{Emerton07JacquetModLocAnRep2} we may view $V$ as a locally analytic $H_0$-representation and have an identification
	\begin{align*}
		\Ind^{\la,G}_H(V) &\overset{\cong}{\lra} \Ind^{\la,G_0}_{H_0}(V) \subset C^\la(G_0,V) , \\
		f &\lto f\res{G_0}
	\end{align*}	
	of (abstract) $G_0$-representations on locally convex $K$-vector spaces.
	Because $C^\la(G_0,V)$ is of compact type (\Cref{Prop 1 - Direct limit description of locally analytic functions for compact manifold} (iii)), the closed subspace $\Ind^{\la,G}_H(V)$ is so as well by \Cref{Lemma A1 - Closed subspaces and quotients of spaces of compact type}.	
	Moreover, the orbit maps of the $G_0$-action on $\Ind^{\la,G_0}_{H_0}(V)$ are locally analytic by \cite[Satz 4.1.5]{FeauxdeLacroix99TopDarstpAdischLieGrp} since $G_0/H_0$ is compact.
	It follows from \Cref{Prop 1 - Locally analytic representations and open subgroups} that $\Ind^{\la,G}_H(V)$ is a locally analytic $G$-representation.

	In the second case, the orbit maps of the $G$-action on $\Ind^{\la,G}_H(V)$ again are locally analytic by \cite[Satz 4.1.5]{FeauxdeLacroix99TopDarstpAdischLieGrp} as $G/H$ is assumed to be compact.
	The topological isomorphism $\Ind^{\la,G}_H(V) \cong C^\la(G/H,V)$, for any section of $G \ra G/H$, is the content of \cite[Satz 4.3.1]{FeauxdeLacroix99TopDarstpAdischLieGrp}.
	It follows that $\Ind^{\la,G}_H(V)$ is barrelled (see \Cref{Cor 1 - Space of locally analytic functions is Hausdorff and barrelled}) and therefore a locally analytic $G$-representation.	
\end{proof}

\begin{lemma}[{see discussion after \cite[Thm.\ 3.6.12]{Emerton17LocAnVect}}]\label{Lemma 1 - Continuity of the orbit homomorphism}
	Let $V$ be a locally analytic representation of a locally $L$-analytic Lie group $G$ with orbit maps $\rho_v \colon G \ra V$, for $v\in V$.
	If $V$ is an LF-space (i.e.\ $V$ is topologically isomorphic to the inductive limit of a sequence of $K$-Fr\'echet spaces) then the orbit homomorphism
	\begin{equation*}
		o\colon V \lra C^\la(G,V) \,,\quad v \lto \rho_v , 
	\end{equation*}
	is continuous.
\end{lemma}
\begin{proof}
	Let $H\subset G$ be a compact open subgroup so that $G = \bigcup_{i\in I} H g_i$ is a disjoint covering.
	Under the topological isomorphism from \Cref{Prop 1 - Disjoint coverings and locally analytic functions}, the homomorphism $o$ coincides with the map
	\begin{align*}
		V \lra \prod_{i\in I} C^\la(H g_i, V) \cong C^\la (G,V) \,,\quad
		v \lto \big(\rho_{v} \res{H g_i} \big)_{i\in I} .
	\end{align*}
	Hence it suffices to show that $o_i \colon V \ra C^\la(H g_i, V)$, $v \mto \rho_v\res{H g_i}$, is continuous, for all $i\in I$.

	We fix $i\in I$ and consider the graph $\Gamma_{o'} \subset V \times C^\la(H',V)$, for $H' \defeq H g_i$.
	One readily computes that $\Gamma_{o'}$ is precisely the kernel of the continuous homomorphism
	\begin{equation*}
		V \times C^\la(H',V) \lra \prod_{h\in H} V \,,\quad (v,f) \lto \big((hg_i).v - f(hg_i)\big)_{h\in H} .
	\end{equation*}
	Therefore $\Gamma_{o'} \subset V \times C^\la(H',V)$ is closed, and we conclude by a version of the closed graph theorem \cite[II.\ \S 4.6 Prop.\ 10]{Bourbaki87TopVectSp1to5} that $o'$ is continuous.
	For this we remark that in particular $V$ is of LF-type (see \cite[p.\ 15]{Emerton17LocAnVect}) so that $C^\la(H',V)$ is of LF-type by \Cref{Prop 1 - Direct limit description of locally analytic functions for compact manifold} (ii).
\end{proof}

When $H$ is a subgroup of finite index in $G$, it suggest itself to consider the ``finite'' induction of a $H$-representation $V$
\begin{equation*}
	\Ind^G_H (V) \defeq \bigoplus_{i=1}^n g_i \bullet V
\end{equation*}
where $g_1,\ldots,g_n$ are coset representatives of $G/H$.
Here the $G$-action on $\Ind^G_H(V)$ is defined via
\begin{equation*}
	g. \bigg(\sum_{i=1}^{n} g_i \bullet v_i \bigg) = \sum_{i=1}^{n} g_{j(i)}\bullet h_i.v_i ,
\end{equation*}
for $g\in G$ with $g g_i = g_{j(i)}h_i$, $j(i)\in \{1,\ldots,n\}$, $h_i \in H$.
We can compare this to the locally analytic induction.
%If a profinite group $G$ is topologically finitely generated, then every subgroup of finite index already is closed. But the locally analytic Lie groups over $\BF_p\llrrparen{t}$ rarely are topologically finitely generated.

\begin{proposition}\label{Prop 1 - Isomorphism between locally analytic and finite induction}
	Let $G$ be a locally $L$-analytic Lie group and $H\subset G$ a locally $L$-analytic subgroup of finite index.
	Let $V$ be a locally analytic $H$-representation which is an LF-space.
	Then there is a $G$-equivariant, topological isomorphism
	\begin{equation}\label{Eq 1 - Homomorphism for comparison between locally analytic and finite induction}
		\Ind_{H}^{\la,G} (V) \overset{\cong}{\lra} \Ind_H^{G} (V) \,,\quad f \lto \sum_{i=1}^{n} g_i \bullet f(g_i) ,	
	\end{equation}
	where $g_1,\ldots,g_n$ are coset representatives of $G/H$.
\end{proposition}
\begin{proof}
	We first note that the homomorphism \eqref{Eq 1 - Homomorphism for comparison between locally analytic and finite induction} is $G$-equivariant; it is continuous because the evaluation homomorphisms $\ev_{g_i} \colon C^\la(G,V) \ra V$ are.
	Then we consider	
	\begin{equation*}
		\Ind_{H}^{G} (V) \lra \Ind_H^{\la,G} (V) \,,\quad \sum_{i=1}^{n} g_i \bullet v_i \lto \big[ g \mto h^{-1}.v_i \quad\text{, for $g=g_i h$ with $h\in H$} \big] ,
	\end{equation*}
	which is an inverse to \eqref{Eq 1 - Homomorphism for comparison between locally analytic and finite induction}. 
	Using $C^\la(G,V) = \bigoplus_{i=1}^n C^\la(g_i H,V)$ this homomorphism is the direct sum of the homomorphisms
	\begin{equation*}
		g_i \bullet V \lra C^\la (g_i H, V) \cap \Ind^{\la, G}_H(V) \,,\quad v \lto \big[ g_i h \mto h^{-1}.v \quad \text{, for $h \in H$} \big] .
	\end{equation*}
	These in turn each can be identified with $\inv^\ast \circ o \colon V \ra C^\la(H,V)$ which is continuous by \Cref{Lemma 1 - Continuity of the orbit homomorphism}.
\end{proof}

\begin{remarks}\label{Rmk 1 - Dual of finite induction}
	Let $G$ be a group and $H\subset G$ a subgroup of finite index.
	\begin{altenumerate}
		\item
		Let $V$ be a locally convex $K$-vector space which also is an (abstract) $H$-representation.
		Then we have a $G$-equivariant, topological isomorphism $\Ind_H^G(V)'_b \cong \Ind_H^G(V'_b)$ via
		\begin{equation*}
			\bigg( \bigoplus_{i=1}^n g_i \bullet V \bigg)'_b \overset{\cong}{\lra} \bigoplus_{i=1}^n g_i \bullet V'_b \,,\quad \ell \lto \sum_{i=1}^n g_i \bullet \big[ v \mto \ell(g_i \bullet v) \big]  ,
		\end{equation*}
		where $g_1,\ldots,g_n$ are coset representatives of $G/H$.
		\item
		We also have the following version of a push-pull formula (projection formula):
		Let $V$ and $W$ be locally convex $K$-vector space such that $V$ is an (abstract) $G$-representation and $W$ an (abstract) $H$-representation.
		Then there exists a $G$-equivariant, topological isomorphism
		\begin{equation*}
			V \botimes{K} \Ind^G_H (W) \overset{\cong}{\lra} \Ind^G_H \big( V\res{H} \botimes{K} W \big) \,,\quad \sum_{i=1}^n v_i \otimes  g_i \bullet w_i \lto \sum_{i=1}^n g_i \bullet (g_i^{-1}. v_i \otimes w_i)
		\end{equation*}
		when the tensor products either both carry the projective or inductive tensor product topology.
		Again $g_1,\ldots,g_n$ denote coset representatives of $G/H$.
	\end{altenumerate}
\end{remarks}
%\begin{proof}
%	The isomorphism in (ii) is given by
%	\begin{align*}
	%		V \botimes{K} \Ind^G_H (W) &\overset{\cong}{\lra} \Ind^G_H \big( V\res{H} \botimes{K} W \big) \\
	%		v \otimes \sum_{i=1}^n g_i \bullet w_i &\lto \sum_{i=1}^n g_i \bullet (g_i^{-1}. v \otimes w_i) \\
	%		\sum_{i=1}^n g_i.v_i. \otimes g_i \bullet w_i &\mapsfrom \sum_{i=1}^n g_i \bullet (v_i \otimes w_i).
	%	\end{align*}
%\end{proof}

In view of the anti-equivalence from \Cref{Prop 1 - Equivalences for categories of locally analytic representations} (ii), the locally analytic induction can also be expressed in terms of taking tensor products with locally analytic distribution algebras.
To this end, we need the following from \cite[Rmk.\ 1.2.11]{Kohlhaase05InvDistpAdicAnGrp} and \cite[\S 1.2]{Emerton17LocAnVect}.

\begin{definition}\label{Def 1 - Topological algebras and modules}
	\begin{altenumerate}
		\item
		By a \textit{separately continuous locally convex $K$-algebra} we mean a locally convex $K$-vector space $A$ which carries the structure of a $K$-algebra such that the multiplication map $A\times A \ra A$ is separately continuous.
		If the multiplication map even is jointly continuous, we simply call $A$ a \textit{locally convex $K$-algebra}.	
		
		If in addition $A$ is a $K$-Fr\'echet space, we call $A$ a \textit{$K$-Fr\'echet algebra}.
		We remark that the multiplication map of such an algebra is jointly continuous automatically \cite[III.\ \S 5.2 Cor.\ 1]{Bourbaki87TopVectSp1to5}.
		\item
		Let $A$ be a separately continuous locally convex $K$-algebra and $M$ a locally convex $K$-vector space.
		We call $M$ a \textit{separately continuous locally convex (left) $A$-module} if $M$ is a left $A$-module and the scalar multiplication map is separately continuous.
		If $A$ is a locally convex $K$-algebra and the scalar multiplication map is jointly continuous, we call $M$ a \textit{locally convex (left) $A$-module}.
		
		If $M$ is a separately continuous locally convex $A$-module, for a $K$-Fr\'echet algebra $A$, and a $K$-Fr\'echet space itself, we call $M$ an \textit{$A$-Fr\'echet module}.
		Again the scalar multiplication of such $M$ is jointly continuous automatically.
	\end{altenumerate}
\end{definition}

\begin{lemma}[{\cite[Lemma 1.2.3]{Emerton17LocAnVect}}]\label{DefProp 1 - Projective tensor product over algebras}
	Let $A \ra B$ be a continuous homomorphism of locally convex $K$-algebras, and $M$ a locally convex $A$-module.
	Then there is an isomorphism of  $B$-modules
	\begin{equation*}
		B \botimes{A} M \cong \big( B \projotimes M \big) \big/ M' ,
	\end{equation*}
	where $M'$ is the $B$-submodule generated by $ba \otimes m - b \otimes a m$, for $b\in B$, $a\in A$, $m\in M$.
	With the induced quotient topology $B \botimes{A} M$ becomes a locally convex $B$-module.
\end{lemma}

\begin{definition}\label{Def 1 - Completed projective tensor product over algebra}
	In the situation of the above lemma, we let $B \botimes{A,\pi} M$ denote the $B$-module $B \botimes{A} M$ with this quotient topology.
	Moreover, we write $B \cotimes{A,\pi} M$ for its Hausdorff completion.
\end{definition}

\begin{remark}\label{Rmk 1 - Completion of projective tensor product over algebras}
	Note that $B \cotimes{A,\pi} M$ again is a locally convex $B$-module.
	%Via multiplication that factors as $B \times (B \cotimes{A,\pi} M) \ra \widehat{B} \times (B \cotimes{A,\pi} M) \ra B \cotimes{A,\pi} M$.
	If $B \cotimes{K,\pi} M$ is hereditarily complete, i.e.\ all its Hausdorff quotients are complete \cite[Def.\ 1.1.39]{Emerton17LocAnVect}, then by \cite[Cor.\ 2.2]{BreuilHerzig18TowardsFinSlopePartGLn} completing preserves the strict exactness of 
	\begin{equation*}
		0 \lra M' \lra B \projotimes M \lra B \botimes{A,\pi} M \lra 0 .
	\end{equation*}
	We thus have $B \cotimes{A,\pi} M \cong  \big( B \projcotimes M \big) \big/ \widebar{M'}$ where $\widebar{M'}$ denotes the closure of $M'$ in $ B \projcotimes M$.
	The above condition on $B \cotimes{K,\pi} M$ is fulfilled for example when $B$ and $M$ both are $K$-Fr\'echet spaces (see discussion after \cite[Prop.\ 17.6]{Schneider02NonArchFunctAna}) or both are of compact type (see \cite[Prop.\ 1.1.32 (i)]{Emerton17LocAnVect}) by the comment after \cite[Def.\ 1.1.39]{Emerton17LocAnVect}.
\end{remark}

Similarly if $A \ra B$ is a continuous homomorphism of separately continuous locally convex $K$-algebras and $M$ is a separately continuous locally convex $A$-module, then $B\botimes{A}M$ becomes a separately continuous locally convex $B$-module when given the quotient topology of $B \indotimes M$ (cf.\ \cite[Rmk.\ 1.2.11]{Kohlhaase05InvDistpAdicAnGrp}).
We write $B \botimes{A,\iota} M$ in this case.
Furthermore, we let $B \cotimes{A,\iota} M$ denote its Hausdorff completion.

\begin{remark}
	In the case that $A\ra B$ is a continuous homomorphism of $K$-Fr\'echet algebras and $M$ is an $A$-Fr\'echet module, the projective and inductive tensor product topology on $B \botimes{K} M$ agree \cite[Prop.\ 17.6]{Schneider02NonArchFunctAna}.
	Consequently we then simply write $B \cotimes{A} M$ (and $B \botimes{A} M$) to denote the $B$-Fr\'echet module (respectively, the locally convex $B$-module).
\end{remark}

\begin{lemma}\label{Lemma 1 - Tensor identities for modules over locally convex algebras}
	\begin{altenumerate}
		\item
		For a locally convex (respectively, separately continuous locally convex) unital $K$-algebra $A$ and a locally convex (respectively, separately continuous locally convex) $A$-module $M$, there is a topological isomorphism of locally convex $A$-modules
		\begin{equation*}
			A \botimes{A,\pi} M \overset{\cong}{\lra} M \,,\quad a \otimes m \lto am ,
		\end{equation*}
		(respectively, of separately continuous locally convex $A$-modules $A \botimes{A,\iota} M \cong M$).
		\item
		Let $A$ and $B$ be locally convex $K$-algebras, and $L$, $M$ and $N$ a locally convex right $A$-module, a locally convex $A$-$B$-bi-module and a locally convex left $B$-module respectively.
		Then there is a canonical topological isomorphism
		\begin{equation*}
			\big( L \botimes{A,\pi} M \big) \botimes{B,\pi} N \cong L \botimes{A,\pi} \big( M \botimes{B,\pi} N \big) .
		\end{equation*}
		For separately continuous locally convex $K$-algebras and separately continuous locally convex modules, the analogous assertion holds with respect to the inductive tensor product topologies instead.
	\end{altenumerate}
\end{lemma}
\begin{proof}
	In (i), the homomorphisms
	\begin{equation*}
		\begin{tikzcd}
			M \ar[r] & A \indotimes M \ar[d, two heads] \ar[r] & A \projotimes M \ar[d, two heads] \\
			& A \botimes{A,\iota} M  & A \botimes{A,\pi} M 
		\end{tikzcd}
	\end{equation*}
	given by $m \mto 1 \otimes m$ are continuous by the definition of the inductive tensor product topology (see \cite[\S 17 A.]{Schneider02NonArchFunctAna}), and constitute inverses to the respective claimed isomorphisms.

	For (ii), it is a classical result that the tensor product over a (commutative) ring is associative.
	In particular this holds for lattices in $L$, $M$ and $N$ over $\CO_K$ so that 
	\begin{equation*}
		\varphi \colon \big( L \botimes{K} M \big) \botimes{K} N \lra L \botimes{K} \big( M \botimes{K} N \big) 
	\end{equation*}
	even is a topological isomorphism with respect to the projective tensor product topologies.
	(These are defined by the tensor product of such lattices over $\CO_K$, see \cite[\S 17 B.]{Schneider02NonArchFunctAna}).
	We can now pass to the quotients
	\begin{align*}\label{Eq 1 - Quotient map of tensor products 1}
		\pi_\ell \colon &\big(L \botimes{K,\pi} M \big) \botimes{K,\pi} N \lra \big(L \botimes{A,\pi} M \big)\botimes{K,\pi} N  \lra \big(L \botimes{A,\pi}  M \big)\botimes{B,\pi} N  , \\
		\pi_r \colon &L \botimes{K,\pi} \big(M \botimes{K,\pi} N \big) \lra L \botimes{K,\pi} \big(M \botimes{B,\pi} N \big) \lra L \botimes{A,\pi} \big( M \botimes{B,\pi} N \big) .
	\end{align*}
	Then $\pi_r \circ \varphi$ factors over $\pi_\ell$, and $\pi_\ell \circ \varphi^{-1} $ factors over $\pi_r$ yielding the sought topological isomorphism.

	In the case of only separately continuous locally convex $K$-algebras and modules, we recall that the inductive tensor product topology on the tensor product $V \botimes{K} W$ of two locally convex $K$-vector spaces $V$ and $W$ is the final locally convex topology with respect to the homomorphisms
	\begin{alignat*}{3}
		\blank \otimes w  &\colon V &&\lra V \botimes{K} W \,,\quad  v' &&\lto v' \otimes w , \\
		v \otimes \blank &\colon W &&\lra V \botimes{K} W \,,\quad w' &&\lto v \otimes w' ,
	\end{alignat*}
	for all $v \in V$, $w\in W$.
	For fixed $v = \sum_{i=1}^r \ell_i \otimes m_i \in L \indotimes M$, it then follows from the commutative diagrams
	\begin{equation*}
		\begin{tikzcd}
			N \ar[r, "m_i \otimes \blank"] \ar[d, "(\ell_i \otimes m_i ) \otimes \blank"'] & M \indotimes N \ar[d, "\ell_i \otimes \blank"] \\
			\big( L \indotimes M \big) \indotimes N \ar[r, "\varphi"] & L \indotimes \big( M \indotimes N \big)
		\end{tikzcd}
	\end{equation*}
	that $\varphi \circ  ( v \otimes \blank ) $ is continuous.
	In turn for fixed $n\in N$, the commutativity of 
	\begin{equation*}
		\begin{tikzcd}
			L \indotimes M \ar[rd, "L \indotimes (\blank \otimes n)", start anchor = 355, end anchor = 140] \ar[d, "\blank \otimes n"'] & \\
			\big( L \indotimes M \big) \indotimes N \ar[r, "\varphi"] & L \indotimes \big( M \indotimes N \big)
		\end{tikzcd}
	\end{equation*}
	shows that $\varphi \circ ( \blank \otimes n )$ is continuous as well.
	Therefore \cite[Lemma 5.1 (i)]{Schneider02NonArchFunctAna} implies that $\varphi$ is continuous.
	Similarly one deduces that $\varphi^{-1}$ is continuous so that $\varphi$ is a topological isomorphism with respect to the inductive tensor product topologies.
	One then argues like in the preceding case.
\end{proof}

\begin{proposition}[{cf.\ \cite[\S 5]{Kohlhaase11CohomLocAnRep}}]
	Let $G$ be a compact locally $L$-analytic Lie group with a locally $L$-analytic subgroup $H\subset G$, and let $V$ be a locally analytic $H$-representation of compact type.
	Then there is a canonical topological isomorphism of $D(G)$-modules
	\begin{equation*}
		\big( \Ind_H^{\la,G} (V) \big)'_b\cong D(G) \cotimes{D(H)} V'_b  .
	\end{equation*}
\end{proposition}
\begin{proof}
	By the definition of $\Ind^{\la, G}_H (V)$ and using $C^\la(G,V) \cong C^\la(G,K) \cotimes{K} V$, we have the exact sequence
	\begin{alignat}{3}\label{Eq 1 - Strictly short exact sequence for locally analytic induction}
		0 \lra \Ind^{\la,G}_H(V) \overset{\iota}{\lra} C^\la(G,K) &\,\widehat{\otimes}_K\, V &&\overset{\psi}{\lra} \prod_{g \in G, h\in H} V  \\
		f &\otimes v &&\lto \big( f(gh) \, v - f(g)\, h^{-1}.v \big)_{g, h}  \nonumber
	\end{alignat}
	where $\iota$ is strict.
	We want to consider the complex obtained by taking the strong dual of \eqref{Eq 1 - Strictly short exact sequence for locally analytic induction}.
	The homomorphisms of this dual complex are continuous by \cite[Rmk.\ 16.1]{Schneider02NonArchFunctAna}.

	By \cite[Prop.\ 9.11]{Schneider02NonArchFunctAna} we have a topological isomorphism
	\begin{equation*}
		\bigg( \prod_{g\in G, h\in H} V \bigg)'_b \overset{\cong}{\lra} \bigoplus_{g\in G, h\in H} V'_b .
	\end{equation*}
	Under this isomorphism the transpose of $\psi$ is given by
	\begin{align*}
		\psi^t \colon \bigoplus_{g\in G, h\in H} V'_b &\lra \big( C^\la(G,K) \cotimes{K} V \big)'_b , \\
		\sum \ell_{g, h} &\lto \bigg[ f\otimes v \mto \sum f(gh)\, \ell_{g,h}(v) - f(g) \, \ell_{g,h} (h^{-1}.v) \bigg] .
	\end{align*}
	There also is the topological isomorphism
	\begin{equation*}
		D(G) \cotimes{K} V'_b \overset{\cong}{\lra} \big( C^\la(G,K) \cotimes{K} V \big)'_b \,,\quad \delta \otimes \ell \lto \big[ f\otimes v \mto \delta(f) \, \ell(v) \big] ,
	\end{equation*}
	by \cite[Prop.\ 20.13]{Schneider02NonArchFunctAna} and \cite[Cor.\ 20.14]{Schneider02NonArchFunctAna}.
	Because we have $f(g) \, \ell(h^{-1}.v) = \delta_g(f) \, (h. \ell)(v)$ and $f(gh) = (\delta_g \ast \delta_h)(f) $, the complex of the strong duals of \eqref{Eq 1 - Strictly short exact sequence for locally analytic induction} is
	\begin{align*}
		\bigoplus_{g\in G, h\in H} V'_b &\overset{\psi^t}{\lra} D(G) \cotimes{K} V'_b \overset{\iota^t}{\lra} \big( \Ind_H^{\la,G} (V) \big)'_b \lra 0 \\
		\sum \ell_{g,h} &\lto \sum \delta_g \ast \delta_h \otimes \ell_{g,h} - \delta_g \otimes h. \ell_{g,h}.
	\end{align*}
	Since $\iota$ is a closed embedding it follows from the Hahn-Banach theorem \cite[Cor.\ 9.4]{Schneider02NonArchFunctAna} that $\iota^t$ is surjective, and from the open mapping theorem \cite[Prop.\ 8.6]{Schneider02NonArchFunctAna} that $\iota^t$ is strict. 
	Furthermore, we have $\Ker(\iota^t) = \Im(\iota)^\perp$ by \cite[IV. \S 4.1 Prop.\ 2]{Bourbaki87TopVectSp1to5} where
	\begin{equation*}
		\Im(\iota)^\perp \defeq \left\{ \ell \in D(G) \cotimes{K} V'_b \middle{|} \forall v \in \Im(\iota): \ell(v) = 0 \right\} .
	\end{equation*}
	Then it follows from the algebraic exactness of \eqref{Eq 1 - Strictly short exact sequence for locally analytic induction} that $ \Im(\iota)^\perp = \Ker(\psi)^\perp$ which implies that $\Ker(\iota^t) = \Ker(\psi)^\perp \subset \widebar{\Im(\psi^t)}$ by \Cref{Lemma A1 - Annihilator of kernel is weak closure of image of the transpose}.
	Because $\Im(\psi^t) \subset \Ker(\iota^t)$ and $\Ker(\iota^t)$ is closed, we have $\Ker(\iota^t) = \widebar{\Im(\psi^t)}$.
	But $\Im(\psi^t)$ is generated by the elements
	\begin{equation*}
		\delta_g \ast \delta_h \otimes \ell - \delta_g \ast h.\ell \quad\text{, for $g \in G$, $h\in H$, $\ell \in V'_b$.}
	\end{equation*}
	Therefore \Cref{Rmk 1 - Completion of projective tensor product over algebras} together with the density of the Dirac distributions yields the topological isomorphism
	\begin{equation*}
		\big( \Ind_H^{\la, G} (V) \big)'_b \cong \big( D(G) \cotimes{K} V'_b \big)/ \widebar{\Im(\psi^t)} \cong D(G) \cotimes{D(H)} V'_b .
	\end{equation*}
	As $\iota^t$ is $D(G)$-linear with respect to the $D(G)$-action on the first component of $D(G) \cotimes{K} V'_b$ via left multiplication, we see that the above isomorphism is an isomorphism of $D(G)$-modules.	
\end{proof}

\subsection{The Hyperalgebra}

In this section $K$ continues to be a spherically complete non-archimedean field with a locally compact complete subfield $L \subset K$.
We recapitulate the concept of germs of locally analytic functions and investigate certain subalgebras of the dual space of these following \cite[\S 2.3]{FeauxdeLacroix99TopDarstpAdischLieGrp} and \cite[\S 1.2]{Kohlhaase05InvDistpAdicAnGrp}.

\begin{definition}
	Let $X$ be a locally $L$-analytic manifold, and $V$ a Hausdorff locally convex $K$-vector space.
	For $x\in X$, we define the \textit{space of germs of locally analytic functions on $X$ with values in $V$ at $x$} as the inductive limit over all open neighbourhoods $U\subset X$ of $x$
	\begin{equation*}
		C^\la_x (X,V) \defeq \varinjlim_{x\in U \subset X} C^\la (U,V)
	\end{equation*}
	with respect to the canonical restriction homomorphisms and endowed with the inductive limit topology.
\end{definition}

\begin{lemma}[{cf.\ \cite[\S 2.3.1]{FeauxdeLacroix99TopDarstpAdischLieGrp}}]\label{Lemma 1 - Easier description of the stalk of locally analytic functions}
	Let $X$ be a locally $L$-analytic manifold, and $V$ a Hausdorff locally convex $K$-vector space.
	\begin{altenumerate}
		\item
		For every open subset $U\subset X$ with $x\in U$, the canonical map
		\begin{equation*}
			C^\la(U,V) \lra C^\la_x(X,V)
		\end{equation*}
		is a strict epimorphism.
		\item
		We have a canonical topological isomorphism
		\begin{equation*}
			C^\la_x(X,V) \cong \varinjlim_{(U,E)} C^\aan(U,\widebar{E})
		\end{equation*}
		where the latter inductive limit is taken over all pairs of analytic charts $\varphi\colon U\ra L^m$ and BH-subspaces $E\subset V$ such that $x\in U$.
		These are partially ordered via
		\begin{equation*}
			(U,E) \geq (W,F)\,\vcentcolon\Longleftrightarrow\, \text{$U\subset W$ and $F \subset E$.} 
		\end{equation*}
		In particular, the $C^\aan(U,\widebar{E}) \subset C^\la_x(X,V)$ are BH-subspaces.
	\end{altenumerate}
\end{lemma}
\begin{proof}
	For (i) note that the restriction map $C^\la(U,V) \ra C^\la(U',V)$, for open subsets $U'\subset U$ of $X$, is a strict epimorphism as it is given by the projection 
	\begin{equation*}
		C^\la(U,V) \cong \prod_{i\in I} C^\la(U_i,V) \lra C^\la(U',V),
	\end{equation*}
	for a suitable disjoint covering $U = \bigcup_{i\in I} U_i$ by open subsets such that $U'\in \left\{U_i\right\}_{i\in I}$.
	For a fixed open subset $U\subset X$ with $x\in U$, the canonical homomorphism $C^\la(U,V) \ra C^\la_x(X,V)$ is the inductive limit over these restriction maps, for $U'\subset U$ with $x\in U'$.
	Because it is the colimit over the respective cokernels, $C^\la(U,V) \ra C^\la_x(X,V)$ is a strict epimorphism itself.
	
	For (ii), we have by definition 
	\begin{equation*}
		C^\la_x(X,V) = \varinjlim_{x\in U \subset X} \Big( \varinjlim_{\CI} C^\la_\CI (U,V) \Big)
	\end{equation*}
	where the ``inner'' inductive limit is taken over all $V$-indices $\CI$ of $U$.
	The latter is topologically isomorphic to the inductive limit indexed by the directed set 
	\begin{equation*}
		\Phi \defeq \big\{(U,\CI) \,\big\vert\, \text{$U\subset X$ open with $x\in U$, $\CI$ a $V$-index of $U$}\big\}
	\end{equation*}
	endowed with the following preorder:
	Let $(U,\CI)$, $(W,\CJ)$ be elements of $\Phi$, with
	\begin{equation*}
		\CI= \big(\varphi_i\colon U_i \ra L^{m_i},\ul{r}_i,E_i \big)_{i\in I} \qquad \text{and} \qquad \CJ = \big(\psi_j\colon W_j \ra L^{n_j},\ul{s}_j,F_j \big)_{j\in J} .
	\end{equation*}
	Then we set $(U,\CI) \geq (W,\CJ)$ if $U \subset W$ and the relation from the proof of \Cref{Prop 1 - Functorialities for the space of locally analytic functions} (ii) between $\CI$ and $\CJ$ for the embedding $U \hookrightarrow W$ holds:
	For every $i\in I$, there exists $j \in J$ such that
	\begin{altenumeratelevel2}
		\item
		$U_i \subset W_j \cap U$, i.e.\ the covering of $\CI$ is a refinement of the covering $U= \bigcup_{j \in J} W_j \cap U$,
		\item
		there exist $a_i \in \varphi_i(U_i)$ and $g_{i,j}=(g_{i,j,k})_{k=1,\ldots,n_j} \in \CA_{\ul{r}_i} (L^{m_i}, L^{n_j})$ such that 
		\[\norm{ g_{i,j,k}- g_{i,j,k}(0)}_{\ul{r}_i} \leq s_{j,k} \quad\text{, for all $k=1,\ldots,n_j$,} \]
		and $\psi_j \circ \varphi_i^{-1} (x) = g_{i,j} (x-a_i)$, for all $x\in \varphi_i(U_i)= B_{\ul{r}_i}^{m_i}(a_i)$,
		\item
		$F_j \subset E_i$.
	\end{altenumeratelevel2}
	Now consider the subset $\Psi \subset \Phi$ of those $(U,\CI)$ for which the covering of $\CI$ only consists of $U$ itself.
	This subset is cofinal in $\Phi$:
	For $(U,\CI)\in \Phi$, let $i_0\in I$ such that $x\in U_{i_0}$.
	Then 
	\begin{equation*}
		(U_{i_0},\CI_0) \defeq \Big( U_{i_0}, \big(\varphi_{i_0}\colon U_{i_0} \ra L^{m_{i_0}}, \ul{r}_{i_0},E_{i_0} \big) \Big) \in \Psi ,
	\end{equation*}
	and $(U_{i_0},\CI_0) \geq (U,\CI)$.
	Hence we conclude that
	\begin{equation*}
		C^\la_x (X,V) \cong \varinjlim_{(U,\CI)\in \Psi} C^\la_{\CI}(U,V).
	\end{equation*}
	But the directed set $\Psi$ is precisely the directed set of pairs $(U,E)$ where $\varphi\colon U\ra L^m$ is an analytic chart around $x$ and $E\subset V$ a BH-subspaces.
	For such $\CI=\big(\varphi \colon U\ra L^m, \ul{r}, E \big) $ with $(U,\CI) \in \Psi$, we moreover have $C^\la_\CI(U,V)= C^\aan(U,\widebar{E})$.
\end{proof}

\begin{proposition}\label{Prop 1 - Properties of the germs of locally analytic functions}
	Let $X$ be a locally $L$-analytic manifold, $V$ be a Hausdorff locally convex $K$-vector space, and $x\in X$.
	\begin{altenumerate}
		\item
		\textnormal{(cf.\ \cite[Satz 2.3.1]{FeauxdeLacroix99TopDarstpAdischLieGrp})}
		The locally convex $K$-vector space $C^\la_x(X,V)$ is Hausdorff and barrelled.
		\item
		\textnormal{(cf.\ \cite[Satz 2.3.1]{FeauxdeLacroix99TopDarstpAdischLieGrp})}
		If $V$ is of LB-type, then $C^\la_x(X,V)$ is an LB-space.
		\item
		\textnormal{(cf.\ \cite[Satz 2.3.2]{FeauxdeLacroix99TopDarstpAdischLieGrp})}
		If $V$ is of compact type, then $C^\la_x(X,V)$ is of compact type and there is a natural topological isomorphism $C^\la_x(X,V) \cong C^\la_x(X,K) \cotimes{K} V$.
		\item
		If $E$ is a $K$-Banach space, then every analytic chart $\varphi$ centred at $x$ induces a topological isomorphism $C^\la_x (X,E) \cong \CA(L^m,E)$ where $m$ is the dimension of $X$ at $x$.
	\end{altenumerate}
\end{proposition}
\begin{proof}
	Let $\varphi\colon U\ra B^{m}_r (0)$ be an analytic chart centred at $x$, for some $r\defeq(r,\ldots,r)\in \BR_{>0}^m$.
	Let $\varepsilon \in \abs{L}$ with $0<\varepsilon < 1$.
	Then the analytic charts $U_n \defeq \varphi^{-1}\big(B^{m}_{\varepsilon^n r}(0) \big) \ra B^{m}_{\varepsilon^n r}(0)$ form a neighbourhood basis of $x$, and in view of \Cref{Lemma 1 - Easier description of the stalk of locally analytic functions} (ii) we have
	\begin{equation}\label{Eq 1 - Direct limit description of stalk of locally analytic functions}
		C^\la_x(X,V) \cong \varinjlim_{(n,E)} C^\aan (U_n,\widebar{E})
	\end{equation}
	where inductive limit is taken over pairs of $n\in \BN$ and BH-subspaces $E$ of $V$.
	As the $C^\aan (U_n, \widebar{E})$ are $K$-Banach spaces, $C^\la_x(X,V)$ is barrelled by \cite[Expl.\ 3) after Cor.\ 6.16]{Schneider02NonArchFunctAna}.
	Moreover, we have continuous injections
	\begin{equation*}
		C^\aan(U_n, \widebar{E}) \cong \CA_{\varepsilon^n r} (L^m, \widebar{E}) \lra \prod_{\ul{i}\in \BN^m_0} \widebar{E} \lra \prod_{\ul{i}\in \BN^m_0} V
	\end{equation*}
	by mapping a power series to the tuple of its coefficients.
	Taking the inductive limit over these, we obtain a continuous injection $C^\la_x(X,V) \hookrightarrow \prod_{\ul{i}\in \BN^m_0} V$.
	It follows that $C^\la_x(X,V)$ is Hausdorff.
	
	If $V$ is of LB-type, write $V = \bigcup_{n\in\BN} V_n$ for an increasing sequence of BH-subspaces $(V_n)_{n\in \BN}$.	
	Then the set of pairs $\{(n,V_n)\mid n\in\BN \}$ is cofinal in the directed set of \eqref{Eq 1 - Direct limit description of stalk of locally analytic functions}.
	Hence $C^\la_x(X,V)$ even is an LB-space in this case.
	
	Now let $V$ be of compact type, say $V \cong \varinjlim_{n\in \BN} V_n$, for a sequence of $K$-Banach spaces $(V_n)_{n\in \BN}$ with injective compact transition homomorphisms.
	Analogous to the proof of \Cref{Prop 1 - Direct limit description of locally analytic functions for compact manifold} (iii), the transition maps
	\begin{equation*}
		C^\aan(U_n, \widebar{V_n}) \lra C^\aan (U_{n+1} , \widebar{V_{n+1}})
	\end{equation*}
	are compact and injective, showing that $C^\la_x(X,V)$ is of compact type.
	The topological isomorphism $C^\la_x(X,V) \cong C^\la_x(X,K) \cotimes{K} V$ now follows from \cite[Prop.\ 1.1.32 (i)]{Emerton17LocAnVect}.
	
	For a $K$-Banach space $E$, the claim of (iv) directly follows from the definition
	\begin{equation*}
		\CA(L^m, E) = \varinjlim_{n\in \BN} \CA_{\varepsilon^n r}(L^m, E).
	\end{equation*}
\end{proof}

\begin{proposition}\label{Prop 1 - Functorialities for germs of locally analytic functions}
	Let $\varphi\colon X \ra Y$ be a locally $L$-analytic map between locally $L$-analytic manifolds, and let $V$ be a Hausdorff locally convex $K$-vector space.
	For $x\in X$, the map $\varphi$ induces a continuous homomorphism
	\begin{equation*}
		\varphi^\ast \colon C^\la_{\varphi(x)}(Y,V) \lra C^\la_x(X,V) \,,\quad f \lto f \circ \varphi .
	\end{equation*}
\end{proposition}
\begin{proof}
	Let $U\subset Y$ be an open neighbourhood of $\varphi(x)$.
	Then $\varphi^{-1}(U)$ is an open neighbourhood of $x$, and the locally $L$-analytic map $\varphi^{-1}(U) \ra U$ induces a continuous homomorphism $C^\la(U,V) \ra C^\la(\varphi^{-1}(U),V)$ by \Cref{Prop 1 - Functorialities for the space of locally analytic functions} (ii).
	Via the universal property of the inductive limit, these induce the desired $\varphi^\ast$.
\end{proof}

\begin{proposition}\label{Prop 1 - Properties of the algebra of germs of locally analytic functions}
	Let $X$ be a locally $L$-analytic manifold and $x\in X$.
	\begin{altenumerate}
		\item
		With respect to pointwise multiplication $C^\la_x(X,K)$ is a local $K$-algebra with maximal ideal
		\begin{equation*}
			\Fm_x \defeq \Ker (\mathrm{ev}_x) \subset C^\la_x(X,K).
		\end{equation*}
		Here $\mathrm{ev}_x \colon C^\la_x(X,K) \ra K$ denotes the continuous evaluation homomorphism induced by $\mathrm{ev}_x\colon C^\la(U,K)\ra K$, for all open neighbourhoods $U\subset X$ of $x$.
		\item
		Let $m$ be the dimension of $X$ at $x$.
		The choice of an analytic chart centred at $x$ yields a topological isomorphism $C^\la_x(X,K) \cong \CA(L^m,K)$ of $K$-algebras with the ring of convergent power series.
		\item
		For all $n\in \BN$, $\Fm_x^n \subset C^\la_x(X,K)$ is a closed subspace of finite codimension.
	\end{altenumerate}
\end{proposition}
\begin{proof}
	Clearly, $C^\la_x(X,K)$ is a $K$-algebra and $\Fm_x$ a maximal ideal.
	Via the usual inverse function theorem for power series, one shows that every $f\in C^\la_x(X,K)\setminus \Fm_x$ has a multiplicative inverse, i.e.\ that $C^\la_x(X,K)$ is local.

	%	, there is some analytic chart $\varphi\colon U \ra L^m$ centered at $x$ and $r\defeq (r,\ldots,r)\in \BR_{>0}^m$ such that $(f \circ \varphi^{-1})\res{B_r^{m}(0)}$ is given by a convergent power series.
	%	As $f \notin \Fm_x$, we may assume that $(f\circ \varphi^{-1})\res{B_r^{m}(0)}(0) = 1$ after rescaling.
	%	Now, the function 
	%	\begin{equation*}
		%		\iota\colon B_{\frac{1}{2}}^{1}(1) \lra K \,,\quad y \lto \sum_{n=0}^{\infty} (-1)^n (y-1)^n = \frac{1}{y},
		%	\end{equation*}
	%	is given by a convergent power series.
	%	%Any $\varepsilon$ with $0<\varepsilon<1$ does the trick, here we take $\varepsilon = $\frac{1}{2}$.
	%	Let $0<r'< r$ such that $\norm{f \circ \varphi^{-1}}_{r'} \leq \frac{1}{2}$.
	%	Then the composition $\iota \circ (f \circ \varphi^{-1})\res{B_{r'}^{m}(0)}$ is given by a convergent power series on $B_{r'}^{m}(0)$ by \Cref{Prop 1 - Composition of power series}.
	%	Therefore, $\iota \circ f \in C^\aan\big(\varphi^{-1}\big(B_{r'}^{m}(0)\big),K \big)$ constitutes a multiplicative inverse to $f$ which shows that the $K$-algebra $C^\la_x(X,K)$ is local.
	
	For (ii), note that the topological isomorphism in \Cref{Prop 1 - Properties of the germs of locally analytic functions} (iv) is an isomorphism of $K$-algebras.
	
	To show that $\Fm_x^n$ is a closed subspace of finite codimension, we may use (ii) to work with $\Ker(\mathrm{ev}_0)^n \subset \CA(L^m,K)$ instead.
	There we have the strict epimorphism
	\begin{equation*}
		\CA_r(L^m,K) \lra K^{\binom{m+n-1}{n-1}} \,,\quad \sum_{\ul{i}\in \BN^m_0} a_{\ul{i}} \, X^{i_1}\cdots X^{i_m} \lto (a_{\ul{i}})_{\abs{\ul{i}} \leq n-1},
	\end{equation*}
	for every $r>0$.
	%These are continuous because they are sequentially continuous and $\CA_r(L^m,K)$ in particular is a metric space.
	These induce a strict epimorphism $\CA (L^m,K) \ra K^{\binom{m+n-1}{n-1}}$ whose kernel precisely is $\Ker(\mathrm{ev}_0)^n$.
	This shows the claim of (iii).
\end{proof}

\begin{lemma}\label{Lemma 1 - Comultiplication map for germs of locally analytic functions}
	Let $G$ be a locally $L$-analytic Lie group with identity element $e$.
	The multiplication $m\colon G\times G \ra G$ induces a continuous homomorphism of $K$-algebras 
	\begin{equation}\label{Eq 1 - Comultiplication map for germs of locally analytic functions}
		\Delta \colon C^\la_e(G,K) \lra C^\la_e (G,K) \cotimes{K} C^\la_e(G,K)
	\end{equation}
	which is compatible with 
	\begin{equation*}
		C^\la(H,K) \overset{m^\ast}{\lra} C^\la(H\times H,K) \cong C^\la(H,K) \cotimes{K} C^\la(H,K) ,
	\end{equation*}
	for every compact open subgroup $H\subset G$.
	Moreover, for all $n\in \BN$, we have
	\begin{equation*}
		\Delta(\Fm_e^n) \subset \sum_{i=0}^n \Fm_e^i \cotimes{K} \Fm_e^{n-i}.
	\end{equation*}
\end{lemma}
\begin{proof}
	Let $(H_n)_{n\in \BN}$ be a family of compact open subgroups of $G$ such that the restriction homomorphisms $C^\aan (H_n,K)\ra C^\aan(H_{n+1},K)$ are compact.
	Then the inductive limit of the corresponding homomorphisms $m^\ast$ yields a continuous homomorphism
	\begin{equation*}
		C^\la_e(G,K) \cong \varinjlim_{n\in \BN} C^\aan (H_n,K) \overset{m^\ast}{\lra} \varinjlim_{n\in \BN} C^\aan(H_n \times H_n, K) 
	\end{equation*}
	via \Cref{Lemma 1 - Easier description of the stalk of locally analytic functions} (ii).
	Moreover, 
	\begin{equation*}
		\varinjlim_{n \in \BN} C^\aan(H_n\times H_n, K) \cong \varinjlim_{n\in \BN} \big( C^\aan(H_n,K) \cotimes{K} C^\aan(H_n,K) \big) \cong C^\la_e(G,K) \cotimes{K} C^\la_e(G,K)
	\end{equation*}
	by \cite[Expl.\ after Prop.\ 17.10]{Schneider02NonArchFunctAna} and \cite[Prop.\ 1.1.32 (i)]{Emerton17LocAnVect}.
	The resulting continuous map \eqref{Eq 1 - Comultiplication map for germs of locally analytic functions} is a homomorphism of $K$-algebras.
	Now consider the continuous homomorphism
	\begin{equation*}
		\mathrm{ev}_e \otimes \mathrm{ev}_e \colon C^\la_e(G,K) \botimes{K} C^\la_e(G,K) \lra K \,,\quad f \otimes f' \lto f(e) f'(e).
	\end{equation*}
	and the induced $\mathrm{ev}_e \cotimes{} \mathrm{ev}_e \colon C^\la_e(G,K) \cotimes{K} C^\la_e(G,K) \ra K$.
	We claim that 
	\begin{equation}\label{Eq 1 - Equation for comultiplication}
		\big(\mathrm{ev}_e \cotimes{} \mathrm{ev}_e \big) \circ \Delta = \mathrm{ev}_e .
	\end{equation}
	Indeed, it suffices to consider an compact open subgroup $H\subset G$ and show the statement for
	\begin{equation*}
		C^\aan(H,K) \overset{\Delta}{\lra} C^\aan(H,K) \cotimes{K} C^\aan(H,K) \xrightarrow{\mathrm{ev}_e \cotimes{} \mathrm{ev}_e} K.
	\end{equation*}
	Note that by density and metrizability of the completed tensor product, we can express every element of $C^\aan (H,K) \cotimes{K} C^\aan (H,K)$ as a convergent sum $\sum_{n\geq 0} f_n \otimes f'_n$, for some $f_n,f'_n \in C^\aan(H,K)$.
	Consequently, for $f\in C^\aan(H,K)$, we may write $\Delta(f) = \sum_{n\geq 0} f_n \otimes f'_n$, so that we have $f(gg')= \sum_{n\geq 0} f_n(g)f'_n(g')$, for all $g,g'\in H$.
	We compute that
	\begin{align*}
		\big((\mathrm{ev}_e \cotimes{} \mathrm{ev}_e)\circ \Delta\big)(f) = (\mathrm{ev}_e \cotimes{} \mathrm{ev}_e)\bigg(\sum_{n\geq 0} f_n \otimes f'_n\bigg) = \sum_{n\geq 0} f_n(e) f'_n(e) = f(e) = \ev_e(f).
	\end{align*}
	It follows from \eqref{Eq 1 - Equation for comultiplication} that $\Delta(\Fm_e) \subset \Ker \big(\mathrm{ev}_e \cotimes{K} \mathrm{ev}_e \big)$.
	We want to show that this latter subspace of $C^\la_e(G,K) \cotimes{K} C^\la_e(G,K)$ equals $C^\la_e(G,K) \cotimes{K} \Fm_e + \Fm_e \cotimes{K} C^\la_e(G,K)$.
	As the short strictly exact sequence
	\begin{equation*}
		0 \lra  \Fm_e  \lra  C^\la_e(G,K) \lra K  \lra 0
	\end{equation*}	
	consists of locally convex $K$-vector spaces of compact type (see \Cref{Lemma A1 - Closed subspaces and quotients of spaces of compact type}), its completed tensor product with $C^\la_e(G,K)$ remains strictly exact by \cite[Cor.\ 2.2]{BreuilHerzig18TowardsFinSlopePartGLn}:
	\begin{equation*}
		0 \lra C^\la_e(G,K) \cotimes{K} \Fm_e  \lra  C^\la_e(G,K) \cotimes{K} C^\la_e(G,K) \xrightarrow{\id \cotimes{} \ev_e} C^\la_e(G,K)  \lra 0 .
	\end{equation*}
	Similarly, the map $\id \cotimes{} \ev_e $ restricts to a strict epimorphism $ \Fm_e \cotimes{K} C^\la_e(G,K) \ra \Fm_e$.
	Since $\mathrm{ev}_e \cotimes{} \mathrm{ev}_e = \mathrm{ev}_e \circ (\id \cotimes{} \ev_e)$, we conclude that
	\begin{align*}
		\Delta(\Fm_e) \subset \Ker\big(\mathrm{ev}_e \cotimes{} \mathrm{ev}_e \big) = \big(\id \cotimes{} \ev_e \big)^{-1} (\Fm_e) 
		&= \Ker \big( \id \cotimes{} \ev_e \big) + \Fm_e \cotimes{K} C^\la_e(G,K) \\
		&= C^\la_e(G,K) \cotimes{K} \Fm_e + \Fm_e \cotimes{K} C^\la_e(G,K) .
	\end{align*}
	The claim for $\Delta(\Fm_e^n)$, with $n\in \BN$, now follows because $\Delta$ is a $K$-algebra homomorphism.
	
	%Under the multiplication map
	%\begin{equation*}
	%	C^\la_e(G,K)\projotimes C^\la_e(G,K) \times C^\la_e(G,K)\projotimes C^\la_e(G,K) \lra C^\la_e(G,K)\projotimes C^\la_e(G,K) \,,\quad (f\otimes g, f'\otimes g') \lto ff'\otimes gg'
	%\end{equation*}
	%induces the multiplication of $C^\la_e(G,K)\projcotimes C^\la_e(G,K)$ by completion, and induces the map
	%\begin{equation*}
	%	\Fm_e^i \projotimes \Fm_e^j \times \Fm_e^{i'} \projotimes \Fm_e^{j'} \lra \Fm_e^{i+i'} \projotimes \Fm_e^{j+j'}
	%\end{equation*}
	%Completing this yields that the product of $\Fm_e^i \projcotimes \Fm_e^j$ and $\Fm_e^{i'} \projcotimes \Fm_e^{j'}$ gets mapped into $\Fm_e^{i+i'} \projcotimes \Fm_e^{j+j'}$.
\end{proof}

For a locally $L$-analytic Lie group $G$, we now want to study the strong dual of these spaces of locally analytic germs supported at $e$.
We continue to let $e$ denote the identity element of $G$, and we write $D_e(G) = D_e(G,K) \defeq C^\la_e(G,K)'_b$.

\begin{remark}
	The $K$-algebra $C^\la_e(G,K)$ constitutes an example of a CT-Hopf $\widehat{\otimes}$-algebra as considered by Lyubinin in \cite[Ch.\ 3.1.2]{Lyubinin10ModComodNonArchHopfAlg} and \cite[Ch.\ 3.2]{Lyubinin14NonArchCoAlgCoadmMod}.

	Indeed, let $H_n \subset G$, for $n\in \BN$, be a family of compact open subgroups such that the restriction homomorphisms $r_n \colon C^\aan (H_n,K)\ra C^\aan(H_{n+1},K)$ are compact.
	Then each $C^\aan (H_n,K)$ is a Banach Hopf $\widehat{\otimes}$-algebra (with comultiplication $\Delta$, counit $\ev_e$, and antipode $\inv^\ast$), and the transition homomorphisms $r_n$ are homomorphisms of Banach Hopf $\widehat{\otimes}$-algebras.

	Moreover, the dual $D_e(G)$ is an NF-Hopf $\widehat{\otimes}$-algebra \cite[Prop. 3.13]{Lyubinin14NonArchCoAlgCoadmMod}.
	\qed
\end{remark}

Based on the hyperalgebra\footnote{Often this object is called the ``distribution algebra'', cf.\ \cite[I.\ Ch.7]{Jantzen03RepAlgGrp}. But to avoid confusion we prefer the name ``hyperalgebra'' here.} classically associated with algebraic groups, we make the following definition.

\begin{definitionproposition}\label{Def 1 - Hyperalgebra}
	Let $G$ be a locally $L$-analytic Lie group.
	We define
	\begin{align*}
		\hy(G,K)_n &\defeq \left\{ \mu \in D_e(G,K) \middle{|} \mu (\Fm_e^{n+1}) =0\right\} \quad \text{, for $n\in \BN_0$,} \\
		\hy(G,K) &\defeq \bigcup_{n\in \BN_0} \hy(G,K)_n ,
	\end{align*}
	and call $\hy(G,K)$ the \textit{hyperalgebra} of $G$.
	
	Then $\hy(G,K)_n$ is a finite-dimensional closed subspace of $D(G,K)$ via the strict epimorphism $C^\la(G,K) \ra C^\la_e(G,K)$.
	Moreover, $\hy(G,K)$ is a $K$-subalgebra of $D(G,K)$ with 
	\begin{equation*}
		\hy(G,K)_n \ast \hy(G,K)_m \subset \hy(G,K)_{n+m} \quad \text{, for all $n,m\in \BN_0$.}
	\end{equation*}
	When the coefficient field $K$ is understood implicitly, we write $\hy(G)\defeq \hy(G,K)$.
\end{definitionproposition}
\begin{proof}
	We may suppose that $G$ is compact, and consider the transpose $D_e(G) \ra D(G)$ of the epimorphism $C^\la(G,K) \ra C^\la_e(G,K)$.
	As $C^\la(G,K)$ and $C^\la_e(G,K)$ are reflexive locally convex $K$-vector spaces, we can apply \cite[IV.\ \S 4.2 Cor.\ 1]{Bourbaki87TopVectSp1to5} to conclude that $D_e(G) \ra D(G)$ is a strict injective homomorphism. We therefore may view $D_e(G) \subset D(G)$ as a closed subspace.
	We write $\iota \colon \Fm_e^{n+1} \hookrightarrow C^\la_e(G,K)$ for the closed embedding of \Cref{Prop 1 - Properties of the algebra of germs of locally analytic functions} (iii).
	By \cite[IV.\ \S 4.1 Prop.\ 2]{Bourbaki87TopVectSp1to5}, we have
	\begin{equation*}
		D_e(G) \supset \hy(G)_n = \Ker(\iota^t) \cong \Coker(\iota)'
	\end{equation*}
	which in particular is a finite-dimensional subspace.
	Thus we have realised $\hy(G)_n$ as a closed finite-dimensional subspace of $D(G)$.
	
	Now let $\mu \in \hy(G)_n$ and $\nu \in \hy(G)_m$.
	As $C^\la(G,K)\ra C^\la_e(G,K)$ factors over $C^\la(H,K)$, for any compact open subgroup $H\subset G$, we may assume that $G$ is compact.
	Then the distribution $\mu \ast \nu \colon C^\la(G,K) \ra K$ factors as
	\begin{equation*}
		\begin{tikzcd}[row sep = 0ex]
			C^\la(G,K) \ar[r, "\Delta"] \ar[dd] &C^\la(G,K) \projcotimes C^\la(G,K) \ar[rd, start anchor=358, end anchor = 155, "\mu \otimes \nu"] \ar[dd] & \\
			&&K \\
			C^\la_e(G,K) \ar[r, "\Delta"] &C^\la_e(G,K) \projcotimes C^\la_e(G,K) \ar[ru, start anchor = 2, end anchor=205, "\mu \otimes \nu"'] &.
		\end{tikzcd}
	\end{equation*}
	We know from \Cref{Lemma 1 - Comultiplication map for germs of locally analytic functions} that $\Delta(\Fm_e^{n+m+1}) \subset \sum_{i=0}^{n+m+1} \Fm_e^i \cotimes{K} \Fm_e^{n+m+1-i}$.
	Therefore we conclude that $(\mu \otimes \nu)\big(\Delta(\Fm_e^{n+m+1}) \big) = 0$ which shows $\mu \ast \nu \in \hy(G)_{n+m}$.
\end{proof}

\begin{example}
	Assume that $K$ is a finite field extension of $L$, and let $\bG$ be a smooth algebraic group over $L$.
	In Remark \ref{Rmks 1 - Associating a locally analytic manifold with a scheme} (ii) we will endow the group of $L$-valued points $\bG(L)$ with the structure of a locally $L$-analytic Lie group.
	Furthermore, we will see in \Cref{Cor 1 - Hyperalgebra agrees with algebraic distribution algebra} that $\hy\big(\bG(L) \big)$ canonically agrees with ${\rm Dist}(\bG) \botimes{L} K$.
	Here ${\rm Dist}(\bG)$ denotes the hyperalgebra (algebraic distribution algebra) as treated in \cite[I.\ Ch.7]{Jantzen03RepAlgGrp}.
	In particular if ${\rm char}(L)=0$, this is an isomorphism $\hy\big(\bG(L) \big) \cong U(\Fg)\otimes_L K$ where $\Fg$ is the Lie algebra of $\bG$ and $U(\Fg)$ its universal enveloping algebra \cite[I.\ \S 7.10]{Jantzen03RepAlgGrp}.
\end{example}

\begin{lemma}\label{Lemma 1 - Involution of hyperalgebra}
	Let $G$ be a locally $L$-analytic Lie group.
	Then the topological automorphism \eqref{Eq 1 - Definition of involution of distribution algebra} preserves $\hy(G)$, i.e.\ it induces an automorphism
	\begin{equation*}
		\hy(G) \lra \hy(G) \,,\quad \mu \lto \dot{\mu}.
	\end{equation*}
\end{lemma}
\begin{proof}
	This follows from the observations that, for every open neighbourhood $U\subset G$ of $e$, $\inv(U)$ is an open neighbourhood of $e$ as well, and that the induced topological automorphism
	\begin{equation*}
		\inv^\sharp \colon C^\la_e(G,K) \lra C^\la_e(G,K) \,,\quad f \lto f\circ \inv,
	\end{equation*}
	satisfies $\inv^\sharp(\Fm_e) = \Fm_e$.
\end{proof}

\begin{remark}
	In \cite[Ch.\ 1.2.1]{Lyubinin14NonArchCoAlgCoadmMod} the notion of a finite dual of a Banach Hopf $\widehat{\otimes}$-algebra is defined.
	With the obvious analogous definition for a CT-Hopf $\widehat{\otimes}$-algebra, $\hy(G)$ is the finite dual of $C^\la_e(G,K)$.
\end{remark}

\begin{proposition}\label{Prop 1 - Pairing for germs of locally analytic functions}
	Let $G$ be a locally $L$-analytic Lie group, and $V\neq \{0\}$ a Hausdorff locally convex $K$-vector space.
	The pairing \eqref{Eq 1 - Pairing for distribution algebra} induces a natural, separately continuous, non-degenerate $K$-bilinear pairing
	\begin{equation}\label{Eq 1 - Pairing for stalk of locally analytic functions}
		D_e(G) \times C^\la_e(G,V) \lra V \,,\quad (\mu,f) \lto \mu(f) \defeq \mu(\tilde{f}),
	\end{equation}
	where $\tilde{f} \in C^\la(G,V)$ denotes some lift of $f$.
	Restricted to $C^\aan(U,\widebar{E}) \subset C^\la_e(G,V)$, for a compact open chart $U$ around $e$ and a BH-subspace $E\subset V$, this pairing is given by
	\begin{equation}\label{Eq 1 - Pairing for stalk of locally analytic functions for BH-subspace}
		D_e(G) \times C^\aan(U,\widebar{E}) \lra D_e(G) \times C^\la_e(G,K) \cotimes{K} \widebar{E} \lra \widebar{E}.
	\end{equation}	
	Here the last map is the completed tensor product of the duality pairing between $C^\la_e(G,K)$ and $D_e(G)$ with $\widebar{E}$.
\end{proposition}
\begin{proof}
	We may assume that $G$ is compact.
	To see that \eqref{Eq 1 - Pairing for stalk of locally analytic functions} is well defined, consider $\tilde{f} \in \Ker\big( C^\la(G,V) \ra C^\la_e(G,V) \big)$.
	Let $E\subset V$ be a BH-subspace such that $\tilde{f} \in C^\la(G,\widebar{E})$.
	Moreover, denote by $\tau\colon C^\la(G,K) \ra C^\la_e(G,K)$ and $\tau_{\widebar{E}} \colon C^\la(G,\widebar{E}) \ra C^\la_e(G,\widebar{E})$ the strict epimorphisms from \Cref{Lemma 1 - Easier description of the stalk of locally analytic functions} (i).
	Then tensoring the short strictly exact sequence
	\begin{equation*}
		0 \lra \Ker(\tau) \lra C^\la(G,K) \lra C^\la_e(G,K) \lra 0
	\end{equation*}
	with $\widebar{E}$ yields by the left exactness of the completed tensor product on locally convex $K$-vector spaces \cite[Lemma 2.1 (ii)]{BreuilHerzig18TowardsFinSlopePartGLn} the following diagram with exact rows
	\begin{equation*}
		\begin{tikzcd}
			0 \ar[r] & \Ker(\tau) \cotimes{K} \widebar{E}  \ar[r] & C^\la(G,K) \cotimes{K} \widebar{E}  \ar[r] & C^\la_e(G,K) \cotimes{K} \widebar{E} \\
			0 \ar[r] & \Ker(\tau_{\widebar{E}}) \ar[r] \ar[u] & C^\la(G,\widebar{E}) \ar[r] \ar[u] & C^\la_e(G,\widebar{E}) \ar[u] 
		\end{tikzcd}.
	\end{equation*}
	Because the middle and right vertical homomorphisms are bijective by \Cref{Prop A1 - Continuous bijection for tensor product of space of compact type and Banach space}, we see that $ \Ker(\tau_{\widebar{E}}) \ra\Ker(\tau) \cotimes{K} \widebar{E}$ is bijective as well.
	If we consider $\tilde{f} $ as an element of $ \Ker(\tau) \cotimes{K} \widebar{E}$ and apply the pairing \eqref{Eq 1 - Pairing for distribution algebra on compact subset and for BH-subspace} to it, we find that $\mu(\tilde{f})=0$, for all $\mu \in D_e(G)$, as (cf.\ \cite[IV.\ \S 4.1 Prop.\ 2]{Bourbaki87TopVectSp1to5})
	\begin{equation*}
		D_e(G) = \left\{ \mu \in D(G) \middle{|} \mu(\Ker(\tau))= 0\right\} .
	\end{equation*}

	The separate continuity of \eqref{Eq 1 - Pairing for stalk of locally analytic functions} and the non-degeneracy in $D_e(G)$ follow from the respective statements for \eqref{Eq 1 - Pairing for distribution algebra}.
	Moreover, the description \eqref{Eq 1 - Pairing for stalk of locally analytic functions for BH-subspace} follows from \eqref{Eq 1 - Pairing for distribution algebra on compact subset and for BH-subspace}.
	
	It remains to show the non-degeneracy in $C^\la_e(G,V)$.
	To this end, let $f\in C^\la_e(G,V)$ such that $\mu(f) = 0$, for all $\mu \in D_e(G)$.
	We have to show that $f = 0$.
	Let $U$ be a compact open chart around $e$ and $E\subset V$ a BH-subspace such that $f \in C^\aan(U,\widebar{E})$.
	Then we have the topological isomorphism $C^\aan(U,\widebar{E}) \cong C^\aan(U,K) \cotimes{K} \widebar{E}$, and $C^\aan(U,K) \botimes{K} \widebar{E}$ is a dense subspace thereof.
	Hence we find sequences $(f_n)_{n\in \BN}\subset C^\aan(U,K)$, $(a_n)_{n\in \BN} \subset \widebar{E}$, both converging to $0$, such that $f = \sum_{n\geq 1} f_n \otimes a_n$.
	%By the density and metrizability, we find a sequence $v_n \in C^\aan(U,K) \botimes{K} \widebar{E}$ such that $v_n \ra f$. Setting $w_0 \defeq 0$ and $w_k \defeq v_k -v_{k-1}$, for $k\in \BN$, we get a series $\sum_{k=0}^n w_k \ra f$ as $n\ra \infty$. Writing each $v_n$ as a sum of pure tensors, we can write this series as a series over a sequence of pure tensors. The convergence of this series implies that this sequence of pure tensors tends to $0$. As the tensor product norm of a pure tensor is the product of the norms of its constituents, we may by $K$-linear scaling of the constitutents of each pure tensor assume that both sequences making up the pure tensors tend to $0$.
	It follows from our assumption and the description \eqref{Eq 1 - Pairing for stalk of locally analytic functions for BH-subspace} that
	\begin{equation*}
		0 = \mu(f) = \sum_{n\geq 1} \mu(f_n) a_n \quad \text{, for all $\mu \in D_e(G)$.}
	\end{equation*}
	This implies that 
	\begin{equation}\label{Eq 1 - Functionals applied to tensor series}
		\sum_{n\geq 1} \lambda(f_n) a_n = 0 \quad \text{, for all $\lambda \in C^\aan(U,K)'$,}
	\end{equation}
	since the homomorphism $D_e(G) \ra C^\aan(U,K)'$ induced by the duality pairing between $C^\la_e(G,K)$ and $D_e(G)$ is surjective.
	
	From now on, we identify $C^\aan (U, K)$ as a locally convex $K$-vector space with the space of sequences tending to $0$
	\begin{equation*}
		c_0(\BN) \defeq \left\{ (a_n)_{n\in \BN} \subset K^\BN \middle{|} a_n \ra 0 \text{ as } n\ra \infty \right\}
	\end{equation*}
	endowed with the supremum norm.
	Our further reasoning uses that $c_0(\BN)$ is a $K$-Banach space which has the so called approximation property, and essentially is a special case of the statement \cite[Prop.\ 4.6]{Ryan02IntroTensProdBanachSp} in the archimedean setting.
	Nevertheless, we want to present a streamlined but detailed account.
	To prove that $f=0$ in $c_0(\BN) \cotimes{K} \widebar{E}$, we will show that
	\begin{equation*}
		T(f) = 0 \quad \text{, for all $T \in \big(c_0(\BN) \cotimes{K} \widebar{E} \big)'$ with $T \neq 0$.}
	\end{equation*}
	We fix such $T$, and note that we have an isomorphism of $K$-vector spaces
	\begin{equation}\label{Eq 1 - Isomorphism of completed tensor product of Banach space and sequence space}
		\big(c_0(\BN) \cotimes{K} \widebar{E} \big)' \overset{\cong}{\lra} \CL \big(c_0(\BN), \widebar{E}' \big) \,,\quad S \lto \big[ (a_n)_{n\in\BN} \mto [v \mto S((a_n)_{n\in \BN} \otimes v) ] \big] ,
	\end{equation}
	by \cite[Rmk.\ 20.12]{Schneider02NonArchFunctAna}.
	Set $C \defeq \lVert T \rVert_{\CL(c_0(\BN),\widebar{E}')}^{-1}$, and let $\varepsilon >0$.
	We now consider the projections
	\[P_m \colon c_0(\BN) \lra c_0(\BN) \,,\quad (x_n)_{n\in \BN} \lto (x_1,\ldots,x_m, 0,\ldots) \quad \text{, for $m\in \BN$,}\]
	which constitute continuous endomorphisms of finite rank.
	As $f_n \ra 0$ in $c_0(\BN)$, there exists some $N\in \BN$ such that $\lVert f_n \rVert_{c_0(\BN)} < C \, \varepsilon$, for all $n \geq N$.
	After enlarging $N$, we may assume that $ \lVert f_n - P_N(f_n) \rVert_{c_0(\BN)} < C \, \varepsilon$, even for all $n\in \BN$.

	We set $S \defeq T \circ P_N \in \CL \big(c_0(\BN),\widebar{E}' \big)$.
	Because $S$ is of finite rank, there exist $r\in \BN$, $\lambda_1,\ldots,\lambda_r \in c_0(\BN)'$, and $\ell_1,\ldots,\ell_r \in \widebar{E}'$ such that $ S = \big[ x \mto \sum_{i=1}^r \lambda_i(x) \, \ell_i \big] $ (cf.\ \cite[\S 18]{Schneider02NonArchFunctAna}).
	Hence, when we apply $S\in \big(c_0(\BN) \cotimes{K} \widebar{E} \big)'$ to $f$, using \eqref{Eq 1 - Isomorphism of completed tensor product of Banach space and sequence space} we compute 	 
	\begin{align*}
		S(f) = \sum_{n\geq 1} \sum_{i=1}^r \lambda_i(f_n) \, \ell_i (a_n)  = \sum_{i=1}^r \ell_i \bigg(  \sum_{n\geq 1} \lambda_i(f_n) \, a_n \bigg) = 0 .
	\end{align*}
	For the last equality, we have used \eqref{Eq 1 - Functionals applied to tensor series} here.
	Therefore
	\begin{align*}
		\abs{T(f)} &\leq \abs{T(f) - S(f)} + \abs{S(f)} = \abs{T(f) - S(f)} = \bigg\lvert \sum_{n\geq 1} T(f_n)(a_n) - S(f_n)(a_n) \bigg\rvert \\
		&\leq \max_{n\geq 1} \Big( \lVert T(f_n) - S(f_n) \rVert_{\widebar{E}'} \cdot  \lVert a_n \rVert_{\widebar{E}} \Big)  = \max_{n\geq 1} \Big( \big\lVert T\big( f_n - P_N(f_n) \big) \big\rVert_{\widebar{E}'} \cdot  \lVert a_n \rVert_{\widebar{E}} \Big)  \\
		&\leq \max_{n\geq 1} \Big( \lVert T \rVert_{\CL(c_0(\BN),\widebar{E}')} \cdot  \lVert f_n - P_N(f_n)\rVert_{c_0(\BN)} \cdot \lVert a_n \rVert_{\widebar{E}} \Big)  < \varepsilon \cdot \max_{n\geq 1} \, \lVert a_n \rVert_{\widebar{E}} .
	\end{align*}
	For $\varepsilon \ra 0$, this shows that $T(f)=0$.
	Since $T \in \big(c_0(\BN) \cotimes{K} \widebar{E} \big)'$ with $T \neq 0$ was arbitrary, we can conclude that $f=0$.
\end{proof}

\begin{corollary}\label{Cor 1 - Pairing for hyperalgebra and germs of locally analytic functions}
	Let $G$ be a locally $L$-analytic Lie group.
	Then $\hy(G) \subset D_e(G)$ is a dense $K$-subalgebra.
	In particular, for any Hausdorff locally convex $K$-vector space $V \neq (0)$, the $K$-bilinear pairing
	\begin{equation*}
		\hy(G) \times C^\la_e(G,V) \lra V
	\end{equation*}
	induced from \eqref{Eq 1 - Pairing for stalk of locally analytic functions} is non-degenerate\footnote{In \cite[Kor.\ 4.7.4]{FeauxdeLacroix99TopDarstpAdischLieGrp} F\'eaux de Lacroix shows that the pairing $U(\Fg) \otimes_L K \times C^\la_e(G,V) \ra V$ is non-degenerate for the case ${\rm char}(L)=0$. His proof uses differentiation with respect to elements of $\Fg$ which is why we pursue a different method here.}.
\end{corollary}
\begin{proof}
	As a first step, we show that the above pairing	is non-degenerate when $V = K$.
	In $\hy(G)$ this is clear.
	On the other hand, let $f \in C^\la_e(G,K)$ such that $\mu(f) =0$, for all $\mu \in \hy(G)$.
	For any $n\in \BN_0$, it follows that $f\in \hy(G)_n^\perp$ where we consider the orthogonal under \eqref{Eq 1 - Pairing for stalk of locally analytic functions}
	\begin{equation*}
		\hy(G)_n^\perp = \left\{ f' \in C^\la_e(G,K) \middle{|} \forall \mu \in \hy(G)_n : \mu(f') = 0 \right\} .
	\end{equation*}
	But we have $\hy(G)_n = (\Fm_e^{n+1} )^\perp$ by definition, and $\big( (\Fm_e^{n+1})^\perp \big)^\perp = \Fm_e^{n+1}$ by \cite[II.\ \S 6.3 Cor.\ 3]{Bourbaki87TopVectSp1to5}.
	Hence $f\in \bigcap_{n\in \BN_0} \Fm_e^{n+1}$, and we apply Krull's intersection theorem to conclude that $f=0$.
	%Note that $C^\la_e(G,K) \cong \CA(L^{d_e},K)$ which is Noetherian. 
	
	To show the density of $\hy(G) \subset D_e(G)$, assume, for the sake of contradiction, that there is $\delta \in D_e(G) \setminus \widebar{ \hy(G)}$.
	By the Hahn--Banach theorem \cite[Cor.\ 9.3]{Schneider02NonArchFunctAna} there exists $f\in C^\la_e(G,K)$ such that $\mu(f) =0$, for all $\mu \in \widebar{\hy(G)}$, and $\delta(f) =1$.
	But as we have just seen, this implies $f=0$ which is a contradiction.

	The non-degeneracy of the pairing between $\hy(G)$ and $C^\la_e(G,V)$ now follows from $\hy(G)$ being a dense subspace of $D_e(G)$ and \Cref{Prop 1 - Pairing for germs of locally analytic functions}.
\end{proof}

\begin{proposition}\label{Prop 1 - Compatibility of convolution and pairing for left regular representation}
	Let $G$ be a compact locally $L$-analytic Lie group, and $V$ a Hausdorff locally convex $K$-vector space.
	If we consider $C^\la(G,V)$ with the left regular $G$-representation, then we have in $V$ the equality
	\begin{equation*}
		(\dot{\mu} \ast f)(e) = \mu(f) \quad\text{, for all $\mu \in D(G)$, $f \in C^\la(G,V)$.}
	\end{equation*}
\end{proposition}
\begin{proof}
	Let
	\begin{equation*}
		o_{C^\la(G,V)} \colon C^\la(G,V) \lra C^\la \big(G,C^\la(G,V) \big) \,,\quad f \lto \rho_f ,
	\end{equation*}
	denote the orbit homomorphism that sends a function $f$ to its orbit map under the left regular $G$-representation.
	Moreover, let $I_W$ denote the integration map \eqref{Eq 1 - Integration map} associated with a Hausdorff locally convex $K$-vector space $W$.
	Now fix $\mu \in D(G)$, and consider the following commutative diagram using that $I_W$ is natural:
	\begin{equation*}
		\begin{tikzcd}[]
			& C^\la \big(G,C^\la(G,V) \big) \ar[r,"I_{C^\la(G,V)}"] \ar[d, "\inv^\sharp"] &[+13pt] \CL_b \big(D(G),C^\la(G,K) \big) \ar[rd, start anchor=357, end anchor=150, "\mathrm{ev}_{\dot{\mu}}"] \ar[d, "\dot{(\blank)}^\ast"] & \\
			C^\la(G,V) \ar[ru, start anchor= 15, end anchor=183, "o_{C^\la(G,V)}"] \ar[rd, start anchor= 345, end anchor=170, "\id"'] & C^\la\big(G,C^\la(G,V)\big) \ar[r, "I_{C^\la(G,V)}"] \ar[d, "(\mathrm{ev}_e)_\ast"] & \CL_b \big(D(G), C^\la(G,V) \big) \ar[r, "\mathrm{ev}_\mu"] \ar[d, "(\mathrm{ev}_e)_\ast"] & C^\la(G,V) \ar[d, "\mathrm{ev}_e"] \\
			& C^\la(G,V) \ar[r, "I_V"] & \CL_b(D(G),V) \ar[r, "\mathrm{ev}_\mu"] & V
		\end{tikzcd}
	\end{equation*}
	Taking the image of a function $f\in C^\la(G,V)$ under the homomorphisms of the top path to $V$ then yields $ \big( I_{C^\la(G,V)}(\rho_f)(\dot{\mu})\big)(e) = (\dot{\mu}\ast f)(e)$.
	Taking the image via going the bottom path yields $I_V(f)(\mu) = \mu(f)$.
	Hence the claim follows from the commutativity of the above diagram.
\end{proof}

\begin{proposition}
	Let $G$ be a locally $L$-analytic Lie group.
	\begin{altenumerate}
		\item
		For every Hausdorff locally convex $K$-vector space $V$, the $G$-representation
		\begin{equation}\label{Eq 1 - Conjugation action on germs of locally analytic functions}
			G \times C^\la_e (G,V) \lra C^\la_e (G,V) \,,\quad (g,f) \lto f(g^{-1} \blank g),
		\end{equation}
		is locally analytic.
		\item
		The \textit{adjoint representation}, for every $n\in \BN_0$,
		\begin{equation*}
			\Ad_n \colon G \times \hy(G)_n \lra \hy(G)_n \,,\quad (g, \mu) \lto \Ad_n(g)(\mu) \defeq \left[f \mto \mu\big(f(g\blank g^{-1})\big)\right] , 
		\end{equation*}
		is locally analytic.
	\end{altenumerate}
\end{proposition}
\begin{proof}
	For (i), let $H\subset G$ be a compact open subgroup. 
	By \Cref{Expl 1 - Examples of locally analytic representations} (ii), $C^\la(H,V)$ with the $H$-action by conjugation is a locally analytic $H$-representation.
	Since the strict epimorphism $C^\la(H,V) \ra C^\la_e(G,V)$ is $H$-equivariant with respect to the action by conjugation, $C^\la_e(G,V)$ is a locally analytic $H$-representation by \Cref{Prop 1 - Subrepresentations and quotientrepresentations of locally analytic representations} (i).
	Moreover, the functoriality from \Cref{Prop 1 - Functorialities for germs of locally analytic functions} shows that $G$ acts on $C^\la_e(G,V)$ by topological endomorphisms.
	Hence \Cref{Prop 1 - Locally analytic representations and open subgroups} implies that $C^\la_e(G,V)$ is a locally analytic $G$-representation.

	For the second statement, let $V = K$, and consider the locally analytic $G$-representation $C^\la_e(G,K)$ with the $G$-action of (i).
	As this representation is of compact type, $D_e(G)$ is a separately continuous $D(G)$-module with respect to the structure induced from \eqref{Eq 1 - Conjugation action on germs of locally analytic functions} by \Cref{Prop 1 - Equivalences for categories of locally analytic representations} (ii).
	We claim that $\hy(G)_n \subset D_e(G)$, for $n\in \BN_0$, is a $D(G)$-submodule.
	Indeed, if $f\in \Fm_e\subset C^\la_e(G,K)$ then
	\begin{equation*}
		\mathrm{ev}_e \big(f(g^{-1} \blank g) \big) = f(g^{-1} e g) = f(e) = 0, 
	\end{equation*}
	which shows that $f(g^{-1}\blank g)\in \Fm_e$, for all $g\in G$.
	Moreover, in \eqref{Eq 1 - Conjugation action on germs of locally analytic functions} $G$ acts by $K$-algebra homomorphisms.
	Hence $\Fm_e^{n+1} \subset C^\la_e(G,K)$ is a locally analytic $G$-subrepresentation.
	It follows that $\hy(G)_n \subset D_e(G)$ is a $D(G)$-submodule.
	Since $\hy(G)_n$ is finite dimensional, in particular it is of LB-type.
	Therefore, the equivalence of \Cref{Prop 1 - Equivalences for categories of locally analytic representations} (i) shows that $\hy(G)_n$ carries the structure of a locally analytic $G$-representation and this is given by $\Ad_n$.
\end{proof}

\begin{definitionproposition}[{cf.\ \cite[Prop.\ 3.5]{OrlikStrauch15JordanHoelderSerLocAnRep}}]\label{Prop 1 - Description of product of hyperalgebra and distribution algebra of subgroup}
	Assume that $K$ is a finite extension of $L$.
	Let $G$ be a locally $L$-analytic Lie group, and $H\subset G$ a locally $L$-analytic subgroup.
	Then every element of the $K$-subalgebra of $D(G,K)$ generated by $\hy(G,K)$ and $D(H,K)$ is a finite sum of elements of the form $\mu \ast \delta$, for $\mu \in \hy(G,K)$, $\delta \in D(H,K)$.
	
	We denote this subalgebra by $D(\Fg,H,K) \subset D(G,K)$, or $D(\Fg,H)$ when the coefficient field is implied.
\end{definitionproposition}
\begin{proof}
	With the appropriate adjustments we proceed analogously to the proof of \cite[Prop.\ 3.5]{OrlikStrauch15JordanHoelderSerLocAnRep}.
	It suffices to show that, for all $\delta\in D(H)$, $\mu \in \hy(G)$, $\delta \ast \mu$ is a finite sum of elements $\mu'\ast \delta'$, for $\mu'\in \hy(G)$, $\delta' \in D(H)$.
	We fix such $\delta$ and $\mu$, and may assume that $\mu \in \hy(G)_n$, for some $n\in \BN_0$.
	By \Cref{Prop 1 - Equivalent characterization for locally analytic representations on Banach spaces} the adjoint representation $\Ad_n$ on $\hy(G)_n$ is given by a locally $L$-analytic map of locally $L$-analytic Lie groups $G \ra \GL\big(\hy(G)_n \big)$.
	Hence, for an $L$-Basis $\mu_1,\ldots,\mu_r$ of $\hy(G)_n$, there exist $c_1,\ldots,c_r \in C^\la(G,K)$ such that
	\begin{equation*}
		\Ad_n(g)(\mu) = \sum_{i=1}^r c_i(g) \, \mu_i \quad\text{, for all $g\in G$.}
	\end{equation*}
	%	Let $\varphi\colon G \ra \GL(\hy(G)_n)$. Then in a neighbourhood of $g_0 \in G$, we have $\varphi(g) = \sum_{J\in \BN^{\dim_{g_0}(G)}} \beta_J ("g-g_0")^J$, for some $\beta_J \in \End_L(\hy(G)_n)$.
	%	Hence 
	%	\begin{align*}
		%		\varphi(g)(\mu) &= \sum_{J} \beta_J(\mu) ("g-g_0")^J \\
		%			&= \sum_J ("g-g_0")^J \sum_{i\in I}^r c_{J,i} \mu_i \\
		%			&= \sum_{i\in I}^r \Big( \sum_{J} c_{J,i} ("g-g_0")^J \Big) \mu_i
		%	\end{align*}
	%	 for $\beta_J(\mu) = \sum_{i\in I} c_{J,i}\mu_i$.
	We define $\delta_i \in D(H)$, for $i=1,\ldots,r$, by
	\[\delta_i(f) \defeq \delta\big[ h \mto c_i(h) \, f(h) \big] \quad \text{, for $f\in C^\la(H,K)$.} \]
	Then, for $f\in C^\la(G,K)$, we compute
	\begin{equation*}
		\begin{aligned}
			(\delta \ast \mu)(f) &= \delta\Big[ h \mto \mu\big[g \mto f(hg)\big] \Big]
			= \delta\Big[ h \mto \Ad_n(h)(\mu) \big[g \mto f(gh) \big] \Big] 	\\
			&= \delta\Bigg[ h \mto \bigg( \sum_{i=1}^r c_i(h) \, \mu_i \bigg)\big[g \mto f(gh)\big] \Bigg] 
			= \sum_{i=1}^r \delta\Big[ h \mto c_i(h) \, \mu_i \big[g \mto f(gh) \big] \Big] 	\\
			&= \sum_{i=1}^r \delta \Big[ h \mto \mu_i \big[ g \mto c_i(h)\, f(gh)\big] \Big] 
			= \sum_{i=1}^r \mu_i \Big[ g \mto \delta \big[ h \mto c_i(h) \, f(gh) \big] \Big]  \\
			&= \sum_{i=1}^r \mu_i \Big[ g \mto \delta_i \big[ h \mto f(gh) \big] \Big]  
			= \sum_{i=1}^r (\mu_i \ast \delta_i) (f) .
		\end{aligned}
	\end{equation*}
	using \Cref{Cor 1 - Fubini theorem} at several instances.
	Hence, we see that $\delta \ast \mu = \sum_{i=1}^r \mu_i \ast \delta_i$.
\end{proof}

\begin{corollary}\label{Cor 1 - Locally analytic representation satisfies compatibility condition for adjoint representation}
	Suppose that $K$ is a finite extension of $L$, and let $V$ be a locally analytic representation of a locally $L$-analytic Lie group $G$.
	Then we have
	\begin{equation*}
		g.(\mu \ast v) = \Ad_n(g)(\mu) \ast (g.v) \quad\text{, for all $g\in G$, $\mu \in \hy(G)$, $v \in V$,}
	\end{equation*}
	where $n\in \BN_0$ such that $\mu \in \hy(G)_n$.
\end{corollary}
\begin{proof}
	For $g\in G$ and $\mu \in \hy(G)_n$, we find $\mu_1,\ldots,\mu_r \in \hy(G)_n$ and $c_1,\ldots,c_r \in C^\la(G,K)$ such that
	\begin{equation*}
		\Ad_n(g)(\mu) = \sum_{i=1}^r c_i(g) \, \mu_i
	\end{equation*}
	and consequently $\delta_g \ast \mu = \sum_{i=1}^r c_i(g) \, \mu_i \ast  \delta_g $ like in the proof of the above proposition.
	We then compute, for $v\in V$,
	\begin{align*}
		\Ad_n(g)(\mu) \ast (g.v) 
		&= \sum_{i=1}^r c_i(g) \, \mu_i \ast (g.v) 
		= \sum_{i=1}^r c_i(g) \, \mu_i \ast (\delta_g \ast v) \\
		&= \bigg( \sum_{i=1}^r c_i(g) \, \mu_i \ast \delta_g \bigg) \ast v 
		= ( \delta_g \ast \mu ) \ast v
		= g.(\mu \ast v) .
	\end{align*}
\end{proof}

We want to characterize modules over the $K$-algebras $D(\Fg,H)$ analogously to the $p$-adic situation considered by Agrawal and Strauch in \cite{AgrawalStrauch22FromCatOLocAnRep}.

\begin{definition}[{cf.\ \cite[Def.\ 7.4.1]{AgrawalStrauch22FromCatOLocAnRep}}]\label{Def 1 - Compatible hyperalgebra modules}
	Assume that $K$ is a finite extension of $L$, and let $G$ be a locally $L$-analytic Lie group with a locally $L$-analytic subgroup $H\subset G$.
	We call a locally analytic $H$-representation $V$ which simultaneously is a $\hy(G)$-module a \textit{locally analytic $(\hy(G),H)$-module} if the scalar multiplication map $\hy(G) \times V \ra V$ is separately continuous when $\hy(G)$ is endowed with the subspace topology of $\hy(G) \subset D(G)$ and the following two compatibly conditions hold:
	\begin{altenumeratelevel2}
		\item
		The action of $\hy(H)$ as a $K$-subalgebra of $\hy(G)$ agrees with the action induced from \Cref{Prop 1 - Module structures over the distribution algebra} (i) of $\hy(H)$ as a $K$-subalgebra of $D(H)$.
		\item
		For all $h \in H$, $\mu\in \hy(G)$, $v \in V$, and $n\in \BN_0$ with $\mu \in \hy(G)_n$, we have 
		\begin{equation*}
			h. (\mu \ast v) = \Ad_n(h)(\mu) \ast (h.v) .
		\end{equation*}
	\end{altenumeratelevel2}
\end{definition}

\begin{remark}\label{Rmk 1 - Locally analytic representation induces compatible hyperalgebra module structure}
	It follows from \Cref{Cor 1 - Locally analytic representation satisfies compatibility condition for adjoint representation} that a locally analytic $G$-representation canonically carries the structure of a locally analytic $(\hy(G),H)$-module, for any locally $L$-analytic subgroup $H\subset G$.
\end{remark}

\begin{corollary}[{cf.\ \cite[Lemma 7.4.2]{AgrawalStrauch22FromCatOLocAnRep}}]\label{Cor 1 - Equivalence between compatible hyperalgebra modules and modules over product of hyperalgebra and distribution algebra of subgroup}
	Giving a locally analytic $(\hy(G),H)$-module structure is naturally equivalent to giving a separately continuous $D(\Fg,H)$-module structure. 
	Moreover, passing to the strong dual space yields an anti-equivalence of categories
	\begin{equation*}
		\left(\substack{\text{\small locally analytic $(\hy(G),H)$-modules} \\ \text{\small on locally convex $K$-vector spaces}\\ \text{\small of compact type with continuous $H$- } \\ \text{\small and $\hy(G)$-equivariant homomorphisms}}\right) \lra \left(\substack{\text{\small separately continuous $D(\Fg,H)$-modules} \\ \text{\small on nuclear $K$-Fr\'echet spaces} \\ \text{\small with continuous $D(\Fg,H)$-module maps}}\right) .
	\end{equation*}
\end{corollary}
\begin{proof}
	A locally analytic $(\hy(G),H)$-module $V$ naturally comes with a separately continuous $D(H)$-module structure by \Cref{Prop 1 - Equivalences for categories of locally analytic representations} (i).
	Via setting
	\begin{equation*}
		(\mu \ast \delta) \ast v \defeq \mu \ast (\delta \ast v) \quad\text{, for $\mu \in \hy(G)$, $\delta \in D(H)$, $v\in V$,}
	\end{equation*}	
	and $K$-linear extension, we obtain a separately continuous homomorphism
	\begin{equation*}
		D(\Fg,H) \times V \lra V
	\end{equation*}
	which is well defined by condition (1).
	To see that this defines a $D(\Fg,H)$-module structure, the only non-trivial assertion to verify is the associativity.
	Utilizing the associativity of the $\hy(G)$- and $D(H)$-actions and the density of the Dirac distributions, it suffices to show that
	\begin{equation*}
		(\delta_h \ast \mu) \ast v = \delta_h \ast (\mu \ast v) \quad\text{, for all $h\in H$, $\mu \in \hy(G)$, and $v\in V$.}
	\end{equation*}
	But like in the proof of \Cref{Cor 1 - Locally analytic representation satisfies compatibility condition for adjoint representation} we see that
	\begin{equation*}
		(\delta_h \ast \mu) \ast v =  \Ad_n(h)(\mu) \ast (h.v) .
	\end{equation*}
	Therefore the associativity follows from condition (2).
	This also shows that conversely we obtain a locally analytic $(\hy(G),H)$-module structure on a separately continuous $D(\Fg,H)$-module.

	For the anti-equivalence of categories, one argues analogously to the proof of \Cref{Prop 1 - Equivalences for categories of locally analytic representations} (ii).
\end{proof}

\subsection{Non-Archimedean Manifolds Arising from Rigid Analytic Spaces}

Here we want to associate locally analytic manifolds to rigid analytic spaces and schemes satisfying some assumptions.
When applied to a smooth algebraic group $\bG$ this allows us to relate the hyperalgebra of the locally analytic Lie group associated with $\bG$ to the (algebraic) distribution algebra ${\rm Dist}(\bG)$ as defined in \cite[I.\ \S 7.7]{Jantzen03RepAlgGrp}.
For the moment, let $L$ be a complete non-archimedean field with non-trivial absolute value $\abs{\blank}$.

Let $X$ be a rigid analytic $L$-space and let $\Fm_x$ denote the maximal ideal of $\CO_{X,x}$, for $x\in X$.
We consider on the set of $L$-valued points of $X$
\begin{equation*}
	X(L) = \left\{ x \in X \middle{|} \CO_{X,x}/\Fm_x = L \right\}
\end{equation*}
the topology generated by $U(L)\subset X(L)$, for all affinoid subdomains $U \subset X$.

\begin{lemma}\label{Lemma 1 - Comparision of canonical topology of rigid unit ball}
	For the affinoid unit ball $\BB^n = \Sp \, K \langle T_1,\ldots,T_n \rangle$, for $n\in \BN_0$, the topology defined on $\BB^n(L) = B_1^{n}(0)$ as above agrees with the ``euclidean'' topology given via $B_1^{n}(0) \subset L^n$.
\end{lemma}
\begin{proof}
	By \cite[7.2.5 Cor.\ 4]{BoschGuentzerRemmert84NonArchAna} the affinoid subdomains $U\subset \BB^n$ form a basis of the canonical topology of $\BB^n$.
	Moreover, for any $x=(x_1,\ldots,x_n)\in \BB^n$, the Weierstra{\ss} domains
	\begin{equation*} 
		\BB^n (f_1,\ldots,f_r) \defeq \big\{ y \in \BB^n \,\big\vert\, \abs{f_1(y)} \leq 1, \ldots, \abs{f_r(y)}\leq 1 \big\} = \BB^n(f_1)\cap\ldots\cap\BB^n(f_r) ,
	\end{equation*}
	for $f_1,\ldots,f_r \in \Fm_x$, form a basis of neighbourhoods of $x$ in the canonical topology \cite[7.2.1, Prop.\ 3 (ii)]{BoschGuentzerRemmert84NonArchAna}.

	For $L$-valued $x \in \BB^n(L)$ and $f\in \Fm_x = (T_1 - x_1,\ldots,T_n - x_n)$, let $f'_i \in L\langle T_1,\ldots,T_n \rangle$ such that $f= f'_1 \, (T_1-x_1) +\ldots + f'_n \, (T_n-x_n)$.
	We moreover find $c\in L\unts$ such that $\abs{c}=\max_{i=1}^n \abs{f'_i}_\mathrm{sup}\eqdef \varepsilon^{-1}$. 
	Then, for $y \in \BB^n_{{\varepsilon},x} \defeq \BB^n\big(c\,(T_1-x_1),\ldots,c\,(T_n-x_n) \big)$, i.e.\ $y\in \BB^n$ satisfying $\max_{i=0}^n \abs{y_i - x_i} \leq {\varepsilon}$, we have
	\begin{equation*}
		\abs{f(y)}\leq \max_{i=1}^n \abs{f'_i(y)}\abs{y_i - x_i} \leq \varepsilon^{-1} \max_{i=1}^n \abs{y_i - x_i} \leq 1 ,
	\end{equation*}
	and therefore $\BB^n_{{\varepsilon},x} \subset \BB^n (f)$.
	We conclude that the $\BB^n_{{\varepsilon},x}$, for $\varepsilon \in \abs{L\unts}$, constitute a neighbourhood basis of $x$ for the canonical topology.
	But we also have $\BB^n_{{\varepsilon},x}(L) = B^{n}_{\varepsilon}(x)$ which shows that the ``euclidean'' topology on $\BB^n(L)$ is finer than the topology defined via the affinoid subdomains of $\BB^n$.

	That the two topologies agree now easily follows from $B^{n}_\varepsilon (x) = \BB^n_{\varepsilon,x} (L)$ and the fact that the $\BB^n_{\varepsilon,x}$ are affinoid subdomains, for $x\in \BB^n(L)$ and $\varepsilon \in \abs{L\unts}$.
\end{proof}

For $X$ with good properties, we now want to endow $X(L)$ with the structure of a locally $L$-analytic manifold.
To define charts, we will use the following lemma. 
Its statement is probably well-known but we include a proof since we could not find it in the literature.

\begin{lemma}\label{Lemma 1 - Existence of polydisc neighbourhood for rigid analytic variety}
	Let $X$ be a rigid analytic $L$-variety, and $x \in X$ a regular $L$-rational point of local dimension $n$.
	Then there exists an open affinoid subdomain $U\subset X$ with $x \in U$ such that $U$ is isomorphic to the $n$-dimensional unit ball $\BB^n$.
	Moreover, for a system of regular parameters $(f_1,\ldots,f_n)$ of $\Fm_x$, the isomorphism $\varphi \colon U \xrightarrow{\cong} \BB^n$ can be chosen in such a way that $\varphi(x) = 0$ and $T_i$ is mapped to $f_i$ under the induced $\CA(L^n,L) \cong \CO_{\BB^n,0} \xrightarrow{\cong} \CO_{X,x}$.
\end{lemma}
\begin{proof}
	We may assume that $X$ is affinoid, say $X = \Sp \, A$, for some affinoid $L$-algebra $A$.
	Because $x\in X$ is regular, $A_{\Fm}$ is a regular local ring (\cite[7.3.2, Prop.\ 8(i)]{BoschGuentzerRemmert84NonArchAna}) where $\Fm \subset A$ denotes the maximal ideal corresponding to $x$.

	Let $(\tilde{f}_1,\ldots,\tilde{f}_n) = \Fm A_\Fm$ denote a system of regular parameters, and let $\widebar{f}_i \in \CO_{X,x}$ be the image of $\tilde{f}_i$ under $A_\Fm \ra \CO_{X,x}$, see \cite[7.3.2, Prop.\ 3]{BoschGuentzerRemmert84NonArchAna}.
	After shrinking $X$, we may assume that the $\widebar{f}_i$ lift to $f_i \in \Fm$, so that $f_i$ is mapped to $\tilde{f}_i$ under the localization map $A \ra A_\Fm$.
	Using that $\Fm \subset A$ is finitely generated, one verifies that there exists $s \in A \setminus \Fm$ such that $s \,\Fm \subset (f_1,\ldots,f_n)$ and $\abs{s(x)} \geq 1$.
	%As $\Fm A_\Fm = \left( \frac{f_1}{1},\ldots,\frac{f_n}{1} \right)$, for generators $\Fm=(a_1,\ldots,a_r)$, there exist $s_1,\ldots,s_r \in A\setminus \Fm$ and $g_1,\ldots,g_r \in (f_1,\ldots,f_n)$ such that $a_i = \frac{g_i}{s_i}$
	%Then there exist $u_1,\ldots,u_r \in A\setminus \Fm$ such that $u_i (a_i s_i - g_i) =0$, i.e.\ $s_i u_i a_i = u_i g_i \in (f_1,\ldots,f_n)$.
	%Set $s \defeq \prod_{i=1}^r s_i u_i$. 
	%Then, for $a=\lambda_1 a_1 + \ldots \lambda_r a_r \in \Fm$, we have $sa= \lambda_1 s a_1 + \ldots + \lambda_r s a_r \in (f_1,\ldots,f_n)$.
	Via replacing $A$ by the completed localization $A \langle s^{-1} \rangle $, we may assume that $\Fm = (f_1,\ldots,f_n)$.
	Furthermore, by scaling we may assume that $\abs{f_i}_\mathrm{sup} \leq 1$ or equivalently that the $f_i$ are power-bounded, see \cite[6.2.3, Prop.\ 1]{BoschGuentzerRemmert84NonArchAna}.
	Therefore there exists a continuous homomorphism of $L$-algebras \cite[6.1.1, Prop.\ 4]{BoschGuentzerRemmert84NonArchAna} 
	\begin{equation*}
		\varphi^\flat \colon L \langle T_1,\ldots, T_n \rangle \lra A \,,\quad T_i \lto f_i ,
	\end{equation*}
	which induces a morphism $\varphi \colon X \ra \BB^n$ of affinoid $L$-varieties.
	It follows that $\varphi(x)=0 \in \BB^n$ and that $x$ is the only point of $X$ which is mapped to $0$ since $\Fm = (f_1,\ldots,f_n)$.
	%For $\Fn \in \Sp A$ with $\varphi^{\ast -1}$(\Fn) = (T_1,\ldots,T_n)$, we have
	%\begin{equation*}
	%	\{f_1,\ldots,f_n \} = \{\varphi^\ast(T_1),\ldots,\varphi^\ast(T_n)\} \subset \varphi^\ast((T_1,\ldots,T_n)) = \varphi^\ast (\varphi^{\ast -1}(\Fn)) \subset \Fn .
	%\end{equation*}
	%Hence $\Fm = (f_1,\ldots,f_n) \subset \Fn$ which implies $\Fn= \Fm$ as $\Fm$ is a maximal ideal.
	Now $\varphi^\flat$ induces a homomorphism of the completion of the local rings
	\begin{equation*}
		\widehat{\varphi}^\flat_x \colon L \llrrbracket{T_1,\ldots,T_n} \cong \widehat{\CO}_{\BB^n,0} \lra \widehat{\CO}_{X,x} \,,\quad T_i \lto \widebar{f_i}.
	\end{equation*}
	As the $\widebar{f_1},\ldots,\widebar{f_n}$ generate $\Fm \widehat{\CO}_{X,x}$ and $\widehat{\CO}_{X,x}/ \Fm \widehat{\CO}_{X,x} \cong L$, we may apply \cite[Thm.\ 7.16 b.]{Eisenbud95CommAlg} to conclude that $\widehat{\varphi}^\flat_x$ is surjective.
	By considering the dimensions of these rings it follows that $\widehat{\varphi}^\flat_x$ is in fact an isomorphism.
	Then \cite[7.3.3, Prop.\ 5]{BoschGuentzerRemmert84NonArchAna} implies that there exists an affinoid subdomain $V\subset \BB^n$ containing $0$ such that $\varphi\colon \varphi^{-1}(V) \ra V$ is an isomorphism.
	But the Weierstra{\ss} domains $\BB^n ( c \, T_1,\ldots, c \,T_n ) $, for $c \in L\unts$ with $\abs{c}\geq 1$, form a basis of neighbourhoods of $0 \in \BB^n$, cf.\ the proof of \Cref{Lemma 1 - Comparision of canonical topology of rigid unit ball}.
	In this way we obtain the sought open affinoid subdomain 
	\begin{equation*}
		U \defeq \varphi^{-1} \big(\BB^n ( c\, T_1,\ldots, c\, T_n ) \big) \overset{\varphi}{\lra} \BB^n ( c \,T_1,\ldots, c \,T_n ) \cong \BB^n .
	\end{equation*}
\end{proof}

\begin{lemma}\label{Lemma 1 - Morphism between rigid analytic unit balls is given by convergent power series}
	Let $\BB^n \ra \BB^m$ be a morphism of rigid analytic $L$-spaces, for $n,m\in \BN_0$.
	Then the induced map $B_1^{n}(0) \ra B_1^{m}(0)$ is given by convergent power series.
\end{lemma}
\begin{proof}
	By \cite[6.1.1 Prop.\ 4]{BoschGuentzerRemmert84NonArchAna} we have 
	\begin{align*}
		\Hom (\BB^{n}, \BB^{m} ) \cong \Hom_{L-alg} \big(L \langle Y_1,\ldots, Y_{m} \rangle , L \langle X_1,\ldots, X_{n} \rangle \big) 
		\cong \big(\CO_L \langle X_1,\ldots, X_{n} \rangle \big)^{m} .
	\end{align*}
	where to $(f_1,\ldots, f_{m}) \in \big( \CO_L \langle X_1,\ldots, X_{n} \rangle \big)^m$ one associates the homomorphism which on $L$-valued points is given by 
	\begin{equation*}
		\BB^{n} (L) \lra \BB^{m} (L) \,,\quad z \lto \big(f_1(z),\ldots, f_{m} (z) \big) .
	\end{equation*}
	%$(f_1,\ldots,f_m)$ corresponds to $\varphi^\flat c\olon Y_i \mto f_i$ corresponds to $\varphi \colon \Fm_x \mto (\varphi^\flat)^{-1} (\Fm_x)$.
	%For $x\in \BB^n$ we have $(X_1-x_1,\ldots,X_n-x_n) = \Ker(\ev_x)$.
	%From $\varphi^\flat(Y_j - f_j(x)) = f_j - f_j(x) \in \Ker(\ev_x)$ it follows that $\varphi^\flat ((Y_1 - f_1(x),\ldots,Y_m - f_m(x))) \subset \Ker(\ev_x)$, hence $(Y_1 - f_1(x),\ldots,Y_m - f_m(x)) \subset (\varphi^\flat)^{-1} (\Fm_x)$. As these are maximal ideals we have quality, hence $\varphi(x) = (f_1(x),\ldots,f_m(x))$ corresponding to the maximal ideal $(Y_1-f_1(x),\ldots,Y_m - f_m(x))$.
\end{proof}

\begin{definition}[{\cite[Def.\ 2.5.6]{deJongvanderPut96EtaleCohomRigidAnSp}}]\label{Def 1 - Definition of rigid analytic spaces of countable type}
	A rigid analytic $L$-space is defined to be \textit{of countable type} if there exists an admissible covering $X=\bigcup_{i\in I} U_i $ by affinoid open subdomains $U_i \subset X$ such that $I$ is at most countable.
\end{definition}

\begin{remarks}\label{Rmks 1 - About rigid analytic spaces of countable type}
	\begin{altenumerate}
		\item
		Examples of such spaces include the rigid analytic $L$-space associated to a scheme of finite type over $L$ \cite[Rmk.\ 2.5.11]{deJongvanderPut96EtaleCohomRigidAnSp}.
		\item
		When the field $L$ contains a dense countable subfield (e.g.\ when $L$ is a non-archimedean local field), any admissible open subset $U$ of a rigid analytic $L$-space $X$ of countable type is of countable type itself\footnote{We learned about this from \url{https://mathoverflow.net/q/155500} (version: 2014-01-23), and we follow the reasoning suggested there by the user ``ACL'' (\url{https://mathoverflow.net/users/10696/acl}) for a proof.}.
	\end{altenumerate}
\end{remarks}
\begin{proof}[{Proof of (ii)}]
	Taking an at most countable admissible covering $X = \bigcup_{i\in I} U_i$ by open affinoid subdomains $U_i\subset X$, and considering the admissible open $U\cap U_i \subset U_i$, we may assume that $X$ is affinoid.

	Hence, let $X = \Sp\,A$, for some affinoid $L$-algebra $A$, and let $U=\bigcup_{i\in I} U_i$ be some admissible covering by open affinoid subdomains $U_i \subset X$.
	This admissible covering is an open covering with respect to the canonical topology on $X$.
	By \cite[Prop.\ 2.1.15]{Berkovich90SpecThAnGeom}, we have a topological embedding $X \hookrightarrow \CM(A)$ where $\CM(A)$ denotes the Berkovich spectrum associated with $A$.
	Moreover, using that $L$ contains a dense countable subfield, one can find a topological embedding $\CM(A) \hookrightarrow [0,1]^\BN$, see \cite[p.\ 4]{Chambert-Loir07MesEtEquiDistEspBerkovich}.
	Therefore $U$ with its induced subspace topology has a countable basis \cite[Ch.\ IX.\ \S 2.8 Prop.\ 12]{Bourbaki66GenTop2}.
	But this implies that there already exists an at most countable subset $J\subset I$ such that $U = \bigcup_{i\in J} U_i$ is a covering \cite[Ch.\ IX.\ \S 2.8 Prop.\ 13]{Bourbaki66GenTop2}.
\end{proof}

\begin{definitionproposition}
	Let $X$ be a smooth, separated rigid analytic $L$-space of countable type.
	Then $X(L)$ with the topology generated by $U(L)\subset X(L)$, for all affinoid subdomains $U\subset X$, and the atlas with charts induced by the isomorphisms of \Cref{Lemma 1 - Existence of polydisc neighbourhood for rigid analytic variety} is a locally $L$-analytic manifold.
	We will denote it by $X^\la$.
\end{definitionproposition}
\begin{proof}
	For each $x\in X(L)$, by \Cref{Lemma 1 - Existence of polydisc neighbourhood for rigid analytic variety} there exist an affinoid subdomain $U_x \subset X$ containing $x$ and an isomorphism $\varphi_x \colon U_x \xrightarrow{\cong} \BB^{n_x}$, $x \mto 0$, where $n_x$ is the local dimension at $x$.
	These isomorphisms yield charts
	\begin{equation*}
		\varphi_x \colon U_x(L) \overset{\cong}{\lra} B_1^{n_x} (0) \subset L^{n_x}
	\end{equation*}
	which we want to show to be compatible.
	For $x,y\in X(L)$ we obtain an isomorphism of rigid analytic $L$-spaces
	\begin{equation*}
		\varphi_y \circ \varphi_x^{-1} \colon \varphi_x (U_x \cap U_y) \overset{\cong}{\lra} U_x \cap U_y \overset{\cong}{\lra } \varphi_y (U_x \cap U_y) .
	\end{equation*}
	For $z \in \varphi_x (U_x \cap U_y)(L)$, we find $c\in L\unts$ such that
	\[Y'\defeq Y\big(c\,(T_1 -z_1),\ldots,c\,(T_{n_x}-z_{n_x}) \big) \subset Y \defeq \varphi_x (U_x \cap U_y) \]
	is an open affinoid subdomain with 
	\begin{equation*}
		\psi_z \colon Y' \overset{\cong}{\lra} \BB^{n_x} \,,\quad w \lto c\,(w-z),
	\end{equation*}
	and $Y'(L) = B_{\abs{c}^{-1}}^{n_x} (z)$.
	Applying \Cref{Lemma 1 - Morphism between rigid analytic unit balls is given by convergent power series} to
	\begin{equation*}
		\varphi_y \circ \varphi^{-1}_x \circ \psi^{-1}_z \colon \BB^{n_x} \lra Y' \lra \varphi_x^{-1}(Y') \lra \varphi_y ( \varphi_x^{-1}(Y')) \longhookrightarrow \BB^{n_y}
	\end{equation*}
	we find that $f(w) \defeq (\varphi_y \circ \varphi^{-1}_x \circ \psi^{-1}_z)(w)$ is given by convergent power series, for $w\in B_{\abs{c}^{-1}}^{n_x}(z)$.
	Therefore $(\varphi_y \circ \varphi_x^{-1})(w) = f(c\,(w-z))$, for all $z \in Y(L)$ and such $w$, shows that $\varphi_y \circ \varphi_x^{-1}$ is locally $L$-analytic on $\varphi_x (U_x \cap U_y)(L)$.
	This shows that the charts $\varphi_x$, for $x\in X(L)$, are compatible, and we obtain a maximal atlas induced by them.

	To see that $X(L)$ is second countable let $X = \bigcup_{i\in I} X_i$ be an admissible covering by open affinoid subdomains.
	We may assume $I$ to be at most countable by the assumption on $X$ to be of countable type.
	But each affinoid $X_i$ is the a subspace of some $\BB^{n_i}$ with respect to the canonical topology.
	Hence $X_i(L) \subset \BB^{n_i}(L) \subset L^{n_i}$ is second countable by \Cref{Lemma 1 - Comparision of canonical topology of rigid unit ball}.

	By the assumption that $X$ is separated, the diagonal morphism $\Delta \colon X \ra X \times_L X$ is a closed immersion.
	It follows from \cite[9.3.5 Lemma 3]{BoschGuentzerRemmert84NonArchAna} that $U \times_L U \subset X \times_L X$ is an open affinoid subdomain, for every open affinoid subdomain $U \subset X$, and from \cite[9.5.3 Prop.\ 2]{BoschGuentzerRemmert84NonArchAna} that the morphism $\Delta \colon U \ra U\times_L U$ is a closed immersion. 
	Hence $\Delta(U) \subset U \times_L U$ is closed in the canonical topology.
	Therefore we can deduce that on the level of $L$-valued points
	\[ \Delta\big(U(L)\big) \subset U(L) \times U(L) = (U \times_L U)(L) \]
	is closed when $U(L)$ is endowed with the topology generated by all open affinoid subdomains of $U$.
	This shows that $X(L)$ is Hausdorff.
	Finally, since $X(L)$ clearly is locally compact in addition to being second countable and Hausdorff, we can conclude that it is paracompact.
\end{proof}

\begin{corollary}\label{Cor 1 - Local analytification functor}
	Via assigning
	\begin{align*}
		X &\lto X^\la,  \\
		[f\colon X \ra Y] &\lto f\res{X^\la}
	\end{align*}
	we obtain a functor from the full subcategory of smooth, separated rigid analytic $L$-spaces of countable type to the category of locally $L$-analytic manifolds.
\end{corollary}
\begin{proof}
	Like in the proof of the previous proposition one shows that a morphism between rigid analytic $L$-spaces induces a locally $L$-analytic map between the associated manifolds.
\end{proof}

\begin{remarks}\label{Rmks 1 - Associating a locally analytic manifold with a scheme}
	\begin{altenumerate}
		\item
		For a smooth, separated $L$-scheme $X$ of finite type, we may pass to the rigid analytification $X^\rig$ which is of countable type by Remark \ref{Rmks 1 - About rigid analytic spaces of countable type} (i).
		Since being separated and smooth carries over to $X^\rig$, we may associate with $X$ the locally $L$-manifold $(X^\rig)^\la$.
		We denote the resulting functor by $(\blank)^\la$.
		\item
		In particular, if $\bG$ is a smooth algebraic group over $L$, it is necessarily separated.
		It follows from functoriality that the algebraic group structure of $\bG$ endows $\bG^\la$ with the structure of a locally $L$-analytic Lie group.
	\end{altenumerate}
\end{remarks}

We now assume that $L$ is a locally compact complete non-archimedean field and let $K$ be a finite extension of $L$.

\begin{proposition}\label{Prop 1 - Isomorphism between rigid analytic and locally analytic stalk}
	Let $X$ be a smooth, separated rigid analytic $L$-space of countable type, $x \in X(L)$.
	Then there is a isomorphism of local $K$-algebras
	\begin{equation}\label{Eq 1 - Isomorphism of rigid analytic stalk and locally analytic stalk}
		\CO_{X,x} \botimes{L} K \overset{\cong}{\lra} C^\la_x(X^\la, K) \,,\quad f \otimes \lambda \lto \lambda f\res{X^\la} .
	\end{equation}
\end{proposition}
\begin{proof}
	Employing \Cref{Lemma 1 - Existence of polydisc neighbourhood for rigid analytic variety}, we find an affinoid subdomain $U \subset X$ containing $x$ and an isomorphism $\varphi \colon U \xrightarrow{\cong} \BB^n$, for some $n \in \BN_0$, with $\varphi(x) = 0$.
	For $\varepsilon \in \abs{\CO_L \setminus \{0\}}$ we consider the Weierstra{\ss} domain $\BB^n_\varepsilon \defeq \BB^n (c\, T_1,\ldots, c\, T_n) \subset \BB^n$ where $c \in L$ such that $\abs{c} = \varepsilon^{-1}$.
	Then the set of affinoid subdomains $U_\varepsilon \defeq \varphi^{-1}(\BB^n_\varepsilon)$ is cofinal in the family of affinoid subdomains of $X$ containing $x$.
	Therefore we find that
	\begin{equation*}
		\CO_{X,x} \botimes{L} K \cong \Big( \varinjlim_{\varepsilon \in \abs{\CO_L \setminus \{0\}}} \CO(U_\varepsilon) \Big) \botimes{L} K \cong \varinjlim_{\varepsilon \in \abs{\CO_L \setminus \{0\}}} \CO(U_\varepsilon) \botimes{L} K .
	\end{equation*}
	On the other hand the induced charts $U_\varepsilon(L) \ra B_\varepsilon^{n}(0)$ form a cofinal subset in the family of analytic charts of $X^\la$ centred at $x$.
	Hence we have by \Cref{Lemma 1 - Easier description of the stalk of locally analytic functions} (ii)
	\begin{equation*}
		C^\la_x (X^\la, K) \cong \varinjlim_{\varepsilon \in \abs{\CO_L \setminus \{0\}}} C^\rig (U_\varepsilon(L), K) .
	\end{equation*}
	But there are compatible isomorphisms of $L$-algebras
	\begin{equation}\label{Eq 1 - Isomorphism between rigid analytic function and rigid analytic function on manifold}
		\CO (U_\varepsilon) \lra C^\rig (U_\varepsilon(L),L) \,,\quad f \lto f\res{U_\varepsilon(L)} \,, \quad\text{ for all $\varepsilon \in \abs{\CO_L \setminus \{0\}}$.}
	\end{equation}
	Indeed, an inverse is given as follows:
	For $f \in C^\rig(U_\varepsilon(L),L)$, let $g \in \CA_\varepsilon (L^n, L)$ be a convergent power series such that $f(z) = g(\varphi(z))$, for all $z\in U_\varepsilon(L)$.
	As $\CA_\varepsilon (L^n, L) = \CO(\BB^n_\varepsilon)$, we map $f$ to $\varphi^\flat (g)$ where $\varphi^\flat \colon \CO(\BB^n_\varepsilon ) \xrightarrow{\cong} \CO(U_\varepsilon)$ is the isomorphism of affinoid $L$-algebras corresponding to $\varphi\res{U_\varepsilon}$.
	Passing to the tensor product of \eqref{Eq 1 - Isomorphism between rigid analytic function and rigid analytic function on manifold} with $K$ then yields (cf.\ \cite[\S 2.3]{Emerton17LocAnVect}) 
	\begin{equation*}
		\CO (U_\varepsilon) \botimes{L} K \cong C^\rig (U_\varepsilon(L),L) \botimes{L} K \cong C^\rig (U_\varepsilon(L),K)
	\end{equation*}
	Taking the direct limit over these isomorphisms gives \eqref{Eq 1 - Isomorphism of rigid analytic stalk and locally analytic stalk}.

	Furthermore we note that the isomorphisms \eqref{Eq 1 - Isomorphism between rigid analytic function and rigid analytic function on manifold} preserve the maximal ideals of functions vanishing at $x$ so that \eqref{Eq 1 - Isomorphism of rigid analytic stalk and locally analytic stalk} is an isomorphism of local $K$-algebras.	
\end{proof}

\begin{corollary}\label{Cor 1 - Hyperalgebra agrees with algebraic distribution algebra}
	Let $\bG$ be a smooth algebraic group over $L$.
	Then the isomorphism \eqref{Eq 1 - Isomorphism of rigid analytic stalk and locally analytic stalk} for $\bG^\rig$ induces a canonical isomorphism of $K$-Hopf algebras
	\begin{equation*}
		\hy(\bG^\la,K) \overset{\cong}{\lra} {\rm Dist}(\bG) \botimes{L} K 
	\end{equation*}
	where ${\rm Dist}(\bG)$ denotes the distribution algebra of $\bG$, cf.\ \cite[I.\ \S 7.7]{Jantzen03RepAlgGrp}.
	In particular, if ${\rm char}(L) =0$, then $\hy(\bG^\la,K) \cong U({\rm Lie}\,\bG) \botimes{L} K$.
\end{corollary}
\begin{proof}
	As noted in Remark \ref{Rmks 1 - Associating a locally analytic manifold with a scheme} (ii), we may apply \Cref{Prop 1 - Isomorphism between rigid analytic and locally analytic stalk} to obtain an isomorphism of local $K$-algebras
	\begin{equation*}
		\alpha \colon \CO_{\bG^\rig, e} \botimes{L} K \overset{\cong}{\lra} C^\la_e (\bG^\la, K) .
	\end{equation*}
	For every $n\in \BN_0$, we thus have a homomorphism
	\begin{equation}\label{Eq 1 - Homomorphism between hyperalgebra and dual of rigid analytic stalk}
		\hy(\bG^\la,K)_n \lra \left\{ \ell \in (\CO_{\bG^\rig, e} \botimes{L}K )^\ast \middle{|} \ell( \Fm_e^{n+1} \botimes{L} K) = 0 \right\} \,,\quad \mu \lto \mu \circ \alpha .
	\end{equation}
	Let $\FM_e = \Ker(\ev_e)\subset C^\la_e(\bG^\la,K)$ denote the maximal ideal.
	By \Cref{Prop 1 - Properties of the algebra of germs of locally analytic functions} (iii) every $\mu \in C^\la_e(\bG^\la,K)^\ast$ with $\mu (\FM_e^{n+1}) =0$ factors over a finite-dimensional quotient of $C^\la_e(\bG^\la,K)$. 
	Hence every such $\mu$ already is continuous, and it follows that \eqref{Eq 1 - Homomorphism between hyperalgebra and dual of rigid analytic stalk} an isomorphism of $K$-vector spaces.

	Moreover, let $\widehat{\CO}_{\bG,e}$ and $\widehat{\CO}_{\bG^\rig,e}$ denote the respective $\Fm_e$-adic completions.
	As stated in \cite[p.\ 113]{Bosch14LectFormRigidGeom} these completions are canonically isomorphic\footnote{In \cite{Bosch14LectFormRigidGeom} Bosch refers to \cite[Satz 2.1]{Koepf74EigentlFamAlgVar}, but this paper was not available to me.}.
	The homomorphisms
	\begin{equation*}
		\big(\CO_{\bG^\rig, e} \botimes{L}K \big)^\ast \longleftarrow \big(\widehat{\CO}_{\bG^\rig, e} \botimes{L}K \big)^\ast \cong \big(\widehat{\CO}_{\bG, e} \botimes{L}K \big)^\ast \lra \big(\CO_{\bG, e} \botimes{L}K \big)^\ast
	\end{equation*}
	thus restrict to yield an isomorphism
	\begin{align*}
		\left\{ \ell \in (\CO_{\bG^\rig, e} \botimes{L}K )^\ast \middle{|} \ell( \Fm_e^{n+1} \botimes{L} K) = 0 \right\} 
		&\cong \left\{ \ell \in (\CO_{\bG, e} \botimes{L}K )^\ast \middle{|} \ell( \Fm_e^{n+1} \botimes{L} K) = 0 \right\} \\
		&\cong {\rm Dist}_n(\bG) \botimes{L} K .
	\end{align*}
	That the resulting $\beta \colon  \hy(\bG^\la,K) \xrightarrow{\cong} {\rm Dist}(\bG) \botimes{L} K$ is an isomorphism of $K$-Hopf algebras follows from the commutativity of the following diagram, for $\mu, \nu \in \hy(\bG^\la,K)$:	
	\begin{equation*}
		\begin{tikzcd}
			\CO_{\bG, e} \botimes{L} K \ar[r, "\Delta"]\ar[d] &\big(\CO_{\bG, e} \botimes{L} K \big) \botimes{K} \big(\CO_{\bG, e} \botimes{L} K \big)  \ar[rd, "\beta(\mu) \otimes \beta(\nu)", start anchor = 353, end anchor = 130] \ar[d] & \\
			\CO_{\bG^\rig, e} \botimes{L} K \ar[r, "\Delta"]\ar[d, "\alpha"]&\big( \CO_{\bG^\rig, e} \botimes{L} K \big) \cotimes{K} \big(\CO_{\bG^\rig, e} \botimes{L} K \big) \ar[r]\ar[d, "\alpha \otimes \alpha"']&K \\
			C^\la_e(\bG^\la,K) \ar[r, "\Delta"]  &C^\la_e(\bG^\la,K) \cotimes{K} C^\la_e(\bG^\la,K) \ar[ru,  "\mu \otimes \nu"', start anchor = 7, end anchor = 230]  &.
		\end{tikzcd}
	\end{equation*}
\end{proof}

\vspace{3ex}

\newcommand{\sd}{{r}}
\newcommand{\rigbun}{{}}
\newcommand{\algbun}{{\mathrm{alg}}}
\newcommand{\boundedregion}{{\Delta}}

\section{$H^0(\CX,\CE)'_b$ and Local Cohomology Groups as Locally Analytic Representations}

Let $K$ be a complete non-archimedean field with non-trivial absolute value $\abs{\blank}$.

\subsection{Topologies on the Coherent and Local Cohomology of Rigid Analytic Spaces}\label{Sect 2 - Topologies on the Coherent Cohomology of Rigid Analytic Spaces}

In this section we want to consider a more general situation than the one we later need for the Drinfeld upper half space.
Let $X$ be a rigid analytic $K$-space and $\CE$ a coherent $\CO_X$-module.
Following \cite[1.6]{vanderPut92SerreDualRigidAnSp} we want to recall how the sections and the coherent cohomology groups of $\CE$ on admissible open subsets can canonically be endowed with locally convex topologies.

Let $U = \Sp\, A \subset X$ be an affinoid subdomain, for some affinoid $K$-algebra $A$.
By Kiehl's theorem \cite[9.4.3 Thm.\ 3]{BoschGuentzerRemmert84NonArchAna}, $\CE(U)$ is a finite $A$-module, and hence carries the structure of a complete normed $A$-module, unique up to equivalence of norms \cite[3.7.3 Prop.\ 3]{BoschGuentzerRemmert84NonArchAna}.
By fixing such a norm $\CE(U)$ becomes a $K$-Banach space.
For another affinoid subdomain $U'\subset U \subset X$, the induced restriction homomorphism $\CE(U)\ra \CE(U')$ is continuous.
%The norm of $\CE(U')$ is defined as the residue norm of some $B\otimes_A A^n$, where $U'=\Sp\, B$. But $A^n\ra B^n=B \otimes_A A^n$ is continuous.

Next, we want to look at an admissible open subset $U \subset X$.
For an admissible covering $U = \bigcup_{i\in I} U_i$ by affinoid open subdomains $U_i\subset X$, the intersections $U_i \cap U_j$, for $i,j\in I$, are admissible open again, and we can find admissible coverings $U_i \cap U_j = \bigcup_{k\in I_{ij}} V_{ijk}$ by affinoid open subdomains $V_{ijk}\subset X$.
Because the $\CE(U_i \cap U_j) \ra \prod_{k\in I_{ij}} \CE(V_{ijk})$ are monomorphisms of $K$-vector spaces, we have an exact sequence
\begin{equation}\label{Eq 2 - Exact sequence for sheaf}
	0 \lra \CE(U) \lra \prod_{i\in I} \CE(U_i) \lra \prod_{i,j\in I} \prod_{k\in I_{ij}} \CE(V_{ijk}) 
\end{equation}
of $K$-vector spaces.
We would like to use this exact sequence to endow $\CE(U)$ with a locally convex topology.
However in order to do this independently of the admissible coverings, we restrict ourselves to the situation that the involved coverings are at most countable (e.g.\ when $U$ is of countable type, see \Cref{Def 1 - Definition of rigid analytic spaces of countable type}).
In this case the products in \eqref{Eq 2 - Exact sequence for sheaf} are at most countable, and therefore are $K$-Fr\'echet spaces.
It follows that $\CE(U)$ with the subspace topology is a $K$-Fr\'echet space as well.
To see that this topology is independent of the (at most countable) admissible coverings, it suffices to consider the situation of a refinement of such coverings.
It is a classical result that the induced homomorphism of complexes is a quasi-isomorphism algebraically \cite[6.2 Thm.\ 5]{Bosch14LectFormRigidGeom}.
Then one argues analogously to the case of complexes of complex Fr\'echet spaces in \cite[VII.\ Lemma 1.32]{BanicaStanasila76AlgMethGlobThComplSp} to deduce that the induced isomorphisms between the cohomology groups are topological isomorphisms even.
%Using the non-Archimedean version of the Open-Mapping Theorem \cite[Prop.\ 8.6]{NFA}.
Note that the restriction homomorphisms $\CE(U) \ra \CE(U')$, for admissible open $U'\subset U$ which allow an at most countable admissible covering, are continuous.

Now consider an admissible open subset $U\subset X$ with an at most countable admissible covering $U = \bigcup_{i \in I} U_i$ by admissible open subsets that each allow an at most countable admissible covering by affinoid subdomains.
Moreover, assume that the covering of $U$ is $\CE$-acyclic, i.e.\ all higher cohomology groups of $\CE$ on the intersections of the $U_i$ vanish.
For instance, when $X$ is separated, this is fulfilled if the $U_i$ are affinoid or quasi-Stein \cite[1.6]{vanderPut92SerreDualRigidAnSp}.
%The intersection of two admissible opens $U=\bigcup_{i\in I} U_i$, $V = \bigcup_{i\in I} V_i$ which are quasi-Stein indeed is quasi-Stein again: Being quasi-Stein is equivalent to the $U_i\hookrightarrow U_{i+1}$ being Runge immersions \cite[4.2 Prop.\ 7]{Bosch14LectFormRigidGeom}, and the restriction of a Runge immersion is a Runge immersion again \cite[4.2 Rmk.\ 6]{Bosch14LectFormRigidGeom}. It follows that $U_i \cap V_i \ra U_{i+1}\cap V_i \ra U_{i+1}\cap V_{i+1}$ is a Runge immersion as the composition of such.
Then the associated \v{C}ech complex
\begin{align}\label{Eq 2 - Cech complex}
	\prod_{i\in I} \CE(U_i) \lra \prod_{i_0,i_1 \in I} \CE(U_{i_0}\cap U_{i_1}) \lra \ldots
\end{align}
computes $H^k(U,\CE)$ on the level of $K$-vector spaces \cite[6.2 Thm.\ 5]{Bosch14LectFormRigidGeom}, and its differentials are continuous.
Note that \eqref{Eq 2 - Cech complex} consists of $K$-Fr\'echet spaces because all of the above products are countable.
We endow $H^k(U,\CE)$ with the locally convex topology that is induced from being a subquotient of a term of \eqref{Eq 2 - Cech complex}.

By the open mapping theorem \cite[Prop.\ 8.6]{Schneider02NonArchFunctAna} this topology is Hausdorff if and only if the differentials of \eqref{Eq 2 - Cech complex} are strict.
In general however, this is not the case.
Like before one shows that the topology on $H^k(U,\CE)$ does not depend on the choice of the at most countable covering.

Furthermore, we want to consider the local cohomology groups of $\CE$ with respect to the set theoretical complement $Z \defeq X \setminus U$.
We now suppose that $X$ possesses an $\CE$-acyclic, at most countable, admissible covering $X = \bigcup_{j \in J} V_j$ by admissible open subsets which each have an at most countable, admissible covering by affinoid open subdomains.
We also suppose that the intersections $U_i \cap V_j$ admit an at most countable, admissible covering by affinoid open subdomains.
For example, this latter assumption is fulfilled if $X$ is quasi-separated.
In this setting, we may assume that the covering $U=\bigcup_{i \in I} U_i$ is a refinement of $U=\bigcup_{j \in J} V_j \cap U$.
Hence, we have a continuous homomorphism of the \v{C}ech complexes
\begin{equation*}
	\begin{tikzcd}
		\prod\limits_{j\in J} \CE(V_j) \ar[r]\ar[d]& \prod\limits_{j_0,j_1 \in J} \CE(V_{j_0}\cap V_{j_1}) \ar[r]\ar[d] & \ldots \\
		\prod\limits_{i\in I} \CE(U_i) \ar[r]& \prod\limits_{i_0,i_1 \in I} \CE(U_{i_0}\cap U_{i_1}) \ar[r] & \ldots
	\end{tikzcd}
\end{equation*}
which in turn induces continuous homomorphisms $H^k(X,\CE) \ra H^k(U,\CE)$, $k\geq 0$, by the usual diagram chase.
%in the category of locally convex topological vector spaces.
%Do the diagram chase using $Im(M^{j-1}\ra Z^j(M^\bullet) = \Ker(\Coker(...))$.

As $K$-vector spaces the local cohomology groups $H^k_Z(X,\CE)$ are defined via the right derived functors $H^k_Z(X,\blank)$ of the functor $\Ker\big(\Gamma(X,\blank)\ra \Gamma(U,\blank) \big)$ \cite[Exp.\ 2, Def.\ 2.1]{Grothendieck68SGA2}.
They fit into the long exact cohomology sequence of $K$-vector spaces \cite[Exp.\ 2, Cor.\ 2.9]{Grothendieck68SGA2}:
\begin{align}\label{Eq 2 - Long exact sequence of cohomology}
	\ldots \lra H^{k-1}_Z (X,\CE) \lra H^{k-1}(X,\CE) \lra H^{k-1} (U,\CE) \overset{\partial^{k-1}}{\lra} H^{k}_Z (X,\CE) \lra \ldots . 
\end{align}
We endow $H^k_Z(X,\CE)$ with the locally convex final topology with respect to $\partial^{k-1}$, i.e.\ the finest locally convex topology such that $\partial^{k-1}$ is continuous.

\begin{remark}\label{Rmk 2 - Strictness of edge homomorphism in long exact sequence of local cohomology}
	It follows from \cite[Lemma 5.1 (i)]{Schneider02NonArchFunctAna} that with this choice of topology on $H^k_Z(X,\CE)$ the homomorphism $H^k_Z(X,\CE) \ra H^k(X,\CE)$ is continuous as well.
	Moreover, then $\partial^{k-1}$ even is a strict homomorphism by \Cref{Lemma A1 - Final topology implies strict homomorphism}.
\end{remark}

\subsection{Coherent (Local) Cohomology of Equivariant Vector Bundles}\label{Sect 2 - Coherent Cohomology of Equivariant Vector Bundles}

Now consider a rigid analytic group variety $G$ over $K$ with multiplication morphism $m\colon G \times_K G \ra G$, and a rigid analytic $K$-variety $X$ with an action $\sigma \colon G \times_K X \ra X$ by this rigid analytic group variety.
Moreover, let $\CE$ be a $G$-equivariant coherent $\CO_X$-module, i.e.\ a coherent $\CO_X$-module with an isomorphism
\begin{equation*}
	\Phi \colon \sigma^\ast \CE \overset{\cong}{\lra} \pr_2^\ast \CE
\end{equation*}
of $\CO_{G\times_K X}$-modules which satisfies the cocycle condition on $G \times_K G \times_K X$:
\begin{equation*}
	\pr_{23}^{\ast} \Phi \circ (\mathrm{id}_G \times \sigma)^{\ast} \Phi = (m \times \mathrm{id}_{X})^{\ast} \Phi .
\end{equation*}
For any $g\in G(K)$, $g \colon \Sp\,K \ra G$, by a slight abuse of notation we also let $g$ denote the induced automorphism $X \xrightarrow{g \times \id_X} G \times_K X \xrightarrow{\sigma}  X$.
We thus obtain an isomorphism
\begin{equation*}
	\Phi_g \colon g^\ast \CE = (g\times \mathrm{id}_X)^\ast \sigma^\ast \CE \xrightarrow{(g\times \mathrm{id}_X)^\ast \Phi} (g\times \mathrm{id}_X)^{\ast} \pr_2^\ast \CE = \CE 
\end{equation*}
of $\CO_{X}$-modules, for each $g\in G(K)$.

Fix $g\in G(K)$, and let $U \subset X$ be an admissible open subset.
In the following, we keep the assumptions from the previous section, i.e.\ that $U$ admits an at most countable admissible, $\CE$-acyclic covering by open subsets which in turn admit at most countable admissible coverings by affinoid open subdomains.
We get an induced isomorphism of $K$-vector spaces on the cohomology groups
\begin{align}\label{Eq 2 - Group action isomorphism on cohomology}
	\varphi_g \colon H^k(U,\CE) \lra H^k(U,g_\ast g^\ast \CE) \cong H^k(g^{-1}(U), g^\ast \CE) \xrightarrow{\Phi_g (g^{-1}(U))} H^k(g^{-1}(U),\CE) .
\end{align}
In particular, we obtain an automorphism of $H^k(U,\CE)$ if $g$ stabilizes $U$.

For an affinoid subdomain $V\subset U$, we denote the continuous homomorphism of affinoid $K$-algebras corresponding to $g \colon g^{-1}(V) \ra V$  by $g^\flat_V \colon \CO(V) \ra \CO(g^{-1}(V))$ .
%A morphism $\varphi\colon Y \ra X$ of rigid varieties gives by definition for the admissible open subset $U \subset X$ an morphism of $K$-algebras $\varphi^{\ast}_U \colon \CO_X(U)\ra \CO_Y(\varphi^{-1}(U))$ [BGR p.337]. In our situation ($\varphi = g$), $(U,\CO_X\res{U})= \mathrm{Sp}\,\CO_X(U)$ and $(\varphi^{-1}(U),\CO_Y\res{\varphi^{-1}(U)}) = \mathrm{Sp}\,\CO_Y(\varphi^{-1}(U)) $ affinoid. Hence there exists unique morphism of rigid analytic varities $(\psi,\psi^{\ast})\colon (\varphi^{-1}(U),\CO_Y\res{\varphi^{-1}(U)}) \ra (U,\CO_X\res{U})$ such that $\psi_U^{\ast} = \varphi^{\ast}_U$ [BGR 9.3.1 Lemma 2].
Then
\begin{align*}
	\varphi_{g,V} \colon \CE(V) \lra g_\ast g^\ast \CE(V) = g^\ast \CE\big(g^{-1}(V)  \big) \xrightarrow{\Phi_g(g^{-1}(V))} \CE \big(g^{-1}(V) \big)
\end{align*}
is $g^\flat_V$-semilinear, i.e.\ $\varphi_{g,V}(a e)= g^\flat_{V} (a) \varphi_{g,V}(e)$, for all $a \in \CO(V), e\in \CE(V)$.
Using that $\CE(V)$ and $\CE(g^{-1}(V))$ are finite $\CO(V)$- respectively $\CO(g^{-1}(V))$-modules, this implies that $\varphi_{g,V}$ is continuous (cf.\ \cite[3.7.3 Prop.\ 2]{BoschGuentzerRemmert84NonArchAna} where this is shown for linear homomorphisms).
We apply this to all $V = U_{i_0}\cap\ldots\cap U_{i_r}$, for $i_0,\ldots,i_r\in I$, to obtain a continuous homomorphism of the \v{C}ech complexes
\begin{equation}\label{Eq 2 - Group action homomorphism of Cech complexes}
	\begin{tikzcd}
		\prod\limits_{i\in I} \CE(U_i) \arrow[r] \arrow[d] & \prod\limits_{i_0,i_1\in I} \CE(U_{i_0} \cap U_{i_1}) \arrow[d]\ar[r] & \ldots\\
		\prod\limits_{i\in I} \CE \big(g^{-1}(U_i) \big) \arrow[r]  & \prod\limits_{i_0,i_1\in I} \CE \big(g^{-1}(U_{i_0} \cap U_{i_1}) \big) \ar[r] & \ldots .
	\end{tikzcd} 
\end{equation}
In fact \eqref{Eq 2 - Group action homomorphism of Cech complexes} induces the isomorphisms \eqref{Eq 2 - Group action isomorphism on cohomology} which therefore are continuous.

Note that more generally, for an admissible open subset $V \subset g^{-1}(U)$ satisfying the previous assumptions, $g$ induces continuous homomorphisms
\begin{equation*}
	H^k(U,\CE) \overset{\varphi_g}{\lra} H^k(g^{-1}(U),\CE) \lra H^k( V,\CE) .
\end{equation*}

Turning to the local cohomology groups, we now consider $Z\defeq X \setminus U$, and a subset $W \subset X$ such that $g^{-1} (Z) \subset W$ and $X\setminus W$ is an admissible open subset satisfying the above assumptions.
The isomorphism $\Phi_g\colon g^\ast \CE \xrightarrow{\cong} \CE$ then induces homomorphisms
\begin{equation}\label{Eq 2 - Group action homomorphism for local cohomology}
	\varphi_g\colon H_{Z}^k (X, \CE) \lra H_{W}^k (X,\CE)
\end{equation}
of $K$-vector spaces.
These fit into an isomorphism of the long exact sequences of local cohomology
\begin{equation*}
	\begin{tikzcd}
		\ldots \ar[r] & H^k_{Z}(X,\CE) \ar[r]\ar[d, "\varphi_g"] & H^k(X,\CE) \ar[r] \ar[d, "\varphi_g"] & H^k(U,\CE) \ar[r, "\partial^k"] \ar[d, "\varphi_g"] & H^{k+1}_{Z} (X,\CE) \ar[r]\ar[d, "\varphi_g"] & \ldots \\
		\ldots \ar[r] & H^k_{ W}(X,\CE) \ar[r] & H^k(X,\CE) \ar[r]  & H^k(X \setminus W,\CE) \ar[r, "\partial^k"]  & H^{k+1}_{W} (X,\CE) \ar[r] & \ldots .
	\end{tikzcd}
\end{equation*}
Under the suitable assumptions on coverings of $X$, we conclude by \cite[Lemma 5.1 (i)]{Schneider02NonArchFunctAna} that the homomorphism $\varphi_g\colon H_{Z}^k (X, \CE) \ra H_{W}^k (X,\CE)$ is continuous, too.
Note that this $\varphi_g$ is a topological isomorphism if $W=g^{-1}(Z)$.

\subsection{Coherent Cohomology of the Drinfeld Upper Half Space}\label{Sect 2 - Coherent Cohomology of the Drinfeld Upper Half Space}

Let $K$ be non-ar\-chi\-me\-dean local field with ring of integers $\CO_K$ and residue characteristic $p>0$.
We denote the completion of the algebraic closure of $K$ by $C$, its ring of integers by $\CO_C$, and write $\abs{\blank}$ for the absolute value on $C$.
Moreover, we fix a uniformizer $\unif$ of $K$ so that $\Fm_K = (\unif)$.

For fixed $d\in \BN$, we now consider the action of the linear algebraic group $\bG \defeq \GL_{d+1,K}$ on the projective space $\BP_K^d$.
We write $G = \GL_{d+1}(K)$ for the $K$-rational points of $\bG$, and set $G_0 \defeq \GL_{d+1}(\CO_K)$ which is an open, maximal compact subgroup of $G$.

We will use the convention that 
$\BP^d_K = \Proj\, \mathrm{Sym} (K^{d+1})^{\ast}$ is the projective space of lines in $K^{d+1}$ where $\mathrm{Sym}(K^{d+1})^\ast = K[X_0,\ldots,X_d]$, 
with $X_0,\ldots,X_d$ being the dual basis of the standard basis of $K^{d+1}$.

Then the natural left action $\sigma \colon \GL_{d+1,K} \times_K \BP_K^d \ra \BP_K^d $ is given by
\begin{align*}
	\GL_{d+1}(C) \times \BP^d_K(C) \lra \BP^d_K(C) \,,\quad (g,z) \lto gz \defeq [z_0{\,:\,}\ldots{\,:\,}z_d]\cdot g^{-1} ,
\end{align*}
on $C$-valued points.
Note that the compatible left $G$-action on $K[X_0,\ldots,X_d]$ is given by $g\cdot X_j = \sum_{i=0}^d g_{ij} X_i$, for $g=(g_{ij}) \in G$, so that $\Fm_{gz} = g(\Fm_{z})$, where $\Fm_z \subset K[X_0,\ldots,X_d]$ denotes the maximal ideal corresponding to $z \in \BP_K^d(C)$.
We will continue to let $\GL_{d+1,K}$, $\BP_K^d$, and $\sigma$ denote the respective rigid analytifications and the rigid analytic group action when this causes no confusion.

We now recall the definition of the Drinfeld upper half space and its rigid-analytic structure following Schneider and Stuhler \cite[\S 1]{SchneiderStuhler91CohompAdicSymSp}.
Unless stated otherwise, for $z \in \BP^d_K(C)$, we always assume a representative $[z_0{\,:\,}\ldots{\,:\,}z_d]$ of $z$ to be \textit{unimodular}, i.e.\ to satisfy $\abs{z_i} \leq 1$, for all $i=0,\ldots,d$, and $\abs{z_i} = 1$, for some $i$.
Then, for any hyperplane $H \subset \BP^d_K(C)$, we let $\ell_H \in \big(\CO_C^{d+1}\big)^{\ast}$ be some unimodular linear form (i.e.\ $\ell_H (z) = \sum_{i=0}^d \lambda_i z_i$, for some unimodular $\lambda = (\lambda_0,\ldots,\lambda_d) \in \CO_C^{d+1}$) such that
\[H = \left\{z \in \BP^d_K(C) \middle{|} \ell_H(z) = 0 \right\} .\]
This linear form $\ell_H$ is determined up to a unit in $\CO_C$.

Let $\CH$ denote the set of all $K$-rational hyperplanes (i.e.\ there exists some  $\lambda \in \CO_K^{d+1}$ defining $H$) in $\BP^d_K(C)$.
As a set, the \textit{Drinfeld upper half-space} of dimension $d$ over $K$ is defined as 
\[\CX = \BP^d_K(C) \Big\backslash \bigcup_{H \in \CH} H .\]
To describe its structure as a rigid analytic variety, let
\[\CX_n \defeq \left\{ z \in \BP_K^d(C) \middle{|} \forall H \in \CH: \abs{\ell_H(z)} \geq \abs{\unif}^n \right\} ,\]
for $n\in \BN$.
The $\CX_n \subset \BP_K^d$ are open affinoid subvarieties.
Moreover, via the admissible covering $\CX = \bigcup_{n \in \BN} \CX_n$ the Drinfeld upper half space is an admissible open $K$-analytic subvariety of $\BP_K^d$, and in fact a Stein space \cite[\S 1, Prop.\ 4]{SchneiderStuhler91CohompAdicSymSp}.
Recall that this implies that, for any coherent $\CO_\CX$-module $\CE$, the higher cohomology groups vanish \cite[Satz 2.4]{Kiehl67ThmAundBNichtArchFunktTh}: $H^j(\CX,\CE) = 0$, for $j>0$.

The open affinoid subsets $\CX_n \subset \BP_K^d$ are stabilized under the action of $G_0$.
Indeed, let $g \in G_0$ and $z=[z_0{\,:\,}\ldots{\,:\,}z_d] \in \CX_n$.
Note that the representative $[z_0{\,:\,}\ldots{\,:\,}z_d]\cdot g^{-1}$ of $gz$ already is unimodular.
%Clearly all entries are of absolute value $\leq 1$. But $g^{-1} g z= z$, so if all entries are of absolute value $<1$, this contradicts $[z_0{\,:\,}\ldots{\,:\,}z_d]$ being unimodular.
For any $H \in \CH$, with $\ell_H$ corresponding to unimodular $\lambda \in \CO_C^{d+1}$, we have to check that $\abs{\ell_H(gz)} \geq \abs{\unif}^n$.
But we have
\[\abs{\ell_H(gz)} = \abs{[z_0{\,:\,}\ldots{\,:\,}z_d] \cdot g^{-1}\cdot \lambda^t } = \abs{\ell_{g^{-1}(H)}(z)} \geq \abs{\unif}^n ,\]
and $g^{-1}(H)$ is the hyperplane corresponding to $g^{-1}\cdot \lambda^t$.

Now consider the admissible open rigid analytic subgroups
\begin{equation}\label{Eq 2 - nth-level rigid analytic subgroup around G_0}
	H_{n+1} \defeq \left\{g\in \GL_{d+1}(\CO_C) \middle{|} \exists h \in G_0, h' \in \mathrm{M}_{d+1}(\CO_C): g = h + \unif^{n+1} h' \right\} \subset \GL_{d+1}(C), 
\end{equation}
for $n\in \BN$.
%An admissible covering by affinoids given by $H_{n+1} = \bigcup_{h' \in G_0} \left\{h\in \GL_{d+1}(C) \middle{|} \abs{\det(h)}=1, \forall i,j=1,\ldots,d: \abs{h_{ij}-h'_{ij}} \leq \abs{\unif}^{n+1}\right\}$.
%An admissible covering by affinoids given by $H_{n+1} = \bigcup_{h \in G_0} \left\{g\in \GL_{d+1}(\CO_C) \middle{|}  \forall i,j=1,\ldots,d: \abs{g_{ij}-h_{ij}} \leq \abs{\unif}^{n+1}\right\}$.
Then the action $\sigma$ of $\GL_{d+1,K}$ on $\BP_K^d$ restricts to an action of $H_{n+1}$ on $\CX_{n}$:
Let $z \in \CX_n$ and $g \in H_{n+1}$ with $g = h + \unif^{n+1} h'$, for $h \in G_0$, $h' \in \mathrm{M}_{d+1}(\CO_C)$.
We compute, for $H \in \CH$,
\begin{align*}
	\abs{\ell_H (g^{-1} z)} = \abs{\ell_{h(H)}(z) + \unif^{n+1} \ell_H([z_0{\,:\,}\ldots{\,:\,}z_d] \cdot h') }
	= \abs{\ell_{h(H)}(z)} \geq \abs{\unif}^{n}
\end{align*}
since $\abs{\ell_{h(H)}(z)} >  \abs{\unif^{n+1} \ell_H([z_0{\,:\,}\ldots{\,:\,}z_d] \cdot h')}$.

Consider now a $\bG$-equivariant vector bundle $\CE$ on $\BP_K^d$.
We want to show how the strong dual $H^0 (\CX,\CE)'_b$ of the global sections of $\CE$ on the Drinfeld upper half space $\CX$ is a locally analytic $G$-representation.
The methods are analogous to the ones used by Schneider and Teitelbaum \cite{SchneiderTeitelbaum02pAdicBoundVal} for the canonical line bundle $\CE = \Omega^d_{\BP_K^d}$.
But for the convenience of the reader we include the proofs.

\begin{proposition}[{cf.\ \cite[Cor.\ 3.9]{SchneiderTeitelbaum02pAdicBoundVal}}]\label{Prop 2 - Dual of the sections on DHS is locally analytic representation}
	The representation
	\begin{align}\label{Eq 2 - Representation on dual of global sections}
		G \times H^0(\CX,\CE)'_b \lra H^0(\CX,\CE)'_b \,,\quad (g,\ell) \lto \ell(g^{-1}. \blank) ,
	\end{align}
	is locally analytic.
	Moreover, the canonical map
	\begin{align}\label{Eq 2 - Dual of projective limit of sections}
		\varinjlim_{n\in \BN} H^0(\CX_n,\CE)' \lra  \bigg (\varprojlim_{n\in\BN} H^0(\CX_n,\CE) \bigg)_b' = H^0(\CX,\CE)_b'
	\end{align}
	is a topological isomorphism and $H^0(\CX,\CE)'_b$ is of compact type this way, i.e.\ it is the inductive limit of 
	\begin{equation}\label{Eq 2 - Dual of sections is of compact type}
		\begin{tikzcd}[column sep = small]
			H^0(\CX_1,\CE)' \ar[r, hook] &  \ldots \ar[r, hook] & H^0(\CX_n,\CE)' \ar[r, hook] & H^0(\CX_{n+1},\CE)' \ar[r, hook] & \ldots
		\end{tikzcd}
	\end{equation}
	where the transition maps induced from $\CX_n \subset \CX_{n+1}$ are compact and injective.
\end{proposition}

For this proposition we proceed in several steps.

\begin{lemma}\label{Lemma 2 - Continuity of action on sections of DHS}
	The natural actions of $G$ on $H^0(\CX,\CE)$ and of $H_{n+1}(K)$ on $H^0(\CX_n,\CE)$, for all $n\in \BN$, are given by continuous endomorphisms.
\end{lemma}
\begin{proof}
	This follows from the considerations in the previous section \Cref{Sect 2 - Coherent Cohomology of Equivariant Vector Bundles}.
\end{proof}

\begin{lemma}\label{Lemma 2 - Description of sections of DHS as projective limit}
	The space of global sections $H^0(\CX,\CE)$ is the projective limit of the $K$-Banach spaces
	\[H^0 (\CX,\CE) \cong \varprojlim_{n\in\BN} H^0(\CX_n,\CE) \]
	with respect to the restriction maps $\CE(\CX_{n+1})\ra \CE(\CX_{n})$.
	These homomorphisms are compact and have dense image.
	Moreover, the above isomorphism is $G_0$-equivariant.
\end{lemma}
\begin{proof}
	We apply the discussion of \Cref{Sect 2 - Topologies on the Coherent Cohomology of Rigid Analytic Spaces} to the admissible covering $\CX = \bigcup_{n\in \BN} \CX_{n}$.
	Noting that $\CX_{n} \subset \CX_{n+1}$, the topological isomorphism
	\begin{align*}
		H^0 (\CX,\CE) = \Ker \bigg( \prod_{n\in \BN} \CE(\CX_n) \ra \prod_{l,m \in \BN} \CE(\CX_{l}\cap \CX_{m}) \bigg) \cong \varprojlim_{n\in\BN} H^0(\CX_n,\CE)
	\end{align*}
	follows from \eqref{Eq 2 - Cech complex}.
	This isomorphism is $G_0$-equivariant by construction.
	
	We now argue analogously to \cite[\S 1 Prop.\ 4]{SchneiderTeitelbaum02pAdicBoundVal}. 
	For each $n\in \BN$, $\CX_n$ is a Weierstra{\ss} domain inside $\CX_{n+1}$ \cite[\S 1, Proof of Prop.\ 4]{SchneiderStuhler91CohompAdicSymSp}.
	This implies that the image of $\CO(\CX_{n+1})$ is dense inside $\CO(\CX_{n})$ \cite[7.3.4 Prop.\ 2]{BoschGuentzerRemmert84NonArchAna}.

	Furthermore, by \cite[\S 1, Proof of Prop.\ 4]{SchneiderStuhler91CohompAdicSymSp} the homomorphism $\psi \colon \CO(\CX_{n+1}) \ra \CO(\CX_{n})$ is inner in the sense of \cite[Def.\ 2.5.1]{Berkovich90SpecThAnGeom}, i.e.\ there exists a strict epimorphism
	\begin{equation*}
		\tau \colon K \langle T_1,\ldots, T_m \rangle \longtwoheadrightarrow \CO(\CX_{n+1})
	\end{equation*}
	of affinoid $K$-algebras, for some $m\in \BN$, such that 
	\begin{equation*}
		\sup_{y \in \CX_n} \abs{ \psi(\tau(T_i))(y)} <1,
	\end{equation*}
	for all $i=1,\ldots,m$.
	By \cite[\S 1 Lemma 5]{SchneiderTeitelbaum02pAdicBoundVal} it follows that $\psi$ is compact as a homomorphism of locally convex $K$-vector spaces.
	
	For a general vector bundle $\CE$, we know by Kiehl's theorem \cite[9.4.3 Thm.\ 3]{BoschGuentzerRemmert84NonArchAna} that, for some $k\in \BN_0$, there is a commutative diagram
	\begin{equation*}
		\begin{tikzcd}
			\CO(\CX_{n+1})^{\oplus k} \ar[r, "\psi^{\oplus k}"]\ar[d,two heads] & \CO(\CX_{n})^{\oplus k} \ar[d, two heads] \\
			\CE(\CX_{n+1}) \ar[r] & \CE(\CX_{n})
		\end{tikzcd}
	\end{equation*}
	where the vertical maps are strict epimorphisms.
	Then \Cref{Lemma A1 - Generalities on compact maps} (ii) and (iii) imply that $\CE(\CX_{n+1})\ra \CE(\CX_{n})$ is compact. 
	The above commutative diagram also shows that the image of this restriction homomorphism is dense.
\end{proof}

\begin{lemma}\label{Lemma 2 - Sections on X_n are locally analytic representation}
	For $n\in \BN$, the representation
	\[G_0 \times H^0(\CX_n,\CE) \lra H^0(\CX_n,\CE) \,, \quad (g,v) \lto g.v , \]
	on the $K$-Banach space $H^0(\CX_n,\CE)$ is locally analytic.
\end{lemma}
\begin{proof}
	We have already seen in \Cref{Lemma 2 - Continuity of action on sections of DHS} that each $g\in G_0 \subset H_{n+1}(K)$ acts by a continuous automorphism on $H^0(\CX_n,\CE)$.
	Hence, it suffices to show that, for every $v\in H^0(\CX_n,\CE)$, the orbit maps $G_0 \ra H^0(\CX_n,\CE)$, $g \, \to\, g.v$, are locally analytic.

	To this end we fix $v \in \CE(\CX_n)$ and $g \in G_0$, and proceed just as in \cite[Prop.\ 2.1']{SchneiderTeitelbaum02pAdicBoundVal}.
	We use the admissible open rigid analytic subgroup \eqref{Eq 2 - nth-level rigid analytic subgroup around G_0} and the rigid analytic chart
	\[\iota_g \colon D_{n+1} \defeq 1+ \unif^{n+1}\mathrm{M}_{d+1}(\CO_C) \lra H_{n+1} \,,\quad h \lto gh .\]
	Note that $D_{n+1}$ is isomorphic to a polydisc $\Sp \, K\langle T_1,\ldots,T_{(d+1)^2} \rangle$ as a rigid analytic variety.
	As $H_{n+1}$ fixes $\CX_n$, we can restrict the group action $\sigma$ and get the following commutative diagram:
	\[\begin{tikzcd}
		D_{n+1} \times_K \CX_n \arrow[r, "\iota_g \times \mathrm{id}"] \arrow[rd, "\pr_2"'] &[+5pt] H_{n+1} \times_K \CX_n \arrow[r, "\sigma"] \arrow[d, "\pr_2"] & \CX_n \\
		&\CX_n&
	\end{tikzcd}.\]
	Let $F_v \in \CE(\CX_n) \langle T_1,\ldots,T_{(d+1)^2} \rangle$ denote the power series to which $v$ is mapped under 
	\begin{equation*}
		\begin{tikzcd}[,/tikz/column 3/.append style={anchor=base west}]
			\CE(\CX_n) \ar[r] & (\iota_g \times \id)^\ast \sigma^\ast \CE(D_{n+1} \times_K \CX_n) \ar[d, "(\iota_g \times \mathrm{id})^{\ast} \Phi (D_{n+1} \times_K \CX_n)"] &[-30pt] \\
			&	(\iota_g \times \id)^\ast \pr_2^\ast \CE(D_{n+1} \times_K \CX_n) &[-30pt] \cong \pr_2^\ast \CE(D_{n+1} \times_K \CX_n) \\[-20pt]
			&&[-30pt] \cong \big( \CO(D_{n+1}) \cotimes{K} \CO(\CX_n) \big) \otimes_{\CO(\CX_n)} \CE(\CX_n) \\[-20pt]
			&&[-30pt] \cong \CE(\CX_n) \langle T_1,\ldots,T_{(d+1)^2} \rangle.
		\end{tikzcd}
	\end{equation*}
	Now consider, for a $K$-valued point $h \in D_{n+1}(K)$,
	\[\begin{tikzcd}
		D_{n+1} \times_K \CX_n \arrow[r, "\iota_g \times \id"] & H_{n+1} \times_K \CX_n \arrow[r, "\sigma"] & \CX_n \\
		\CX_n \arrow[u, "h \times \id"]  \arrow[urr,  "gh"', end anchor = 210] &&
	\end{tikzcd}.\]
	In terms of $K$-affinoid algebras the morphism $h \times \id\colon \CX_n \ra D_{n+1} \times_K \CX_n$ is given by the evaluation of power series
	\begin{align*}
		\mathrm{ev}_h \colon \CO(D_{n+1}) \cotimes{K} \CO(\CX_n) \cong \CO(\CX_n) \langle T_1,\ldots,T_{(d+1)^2} \rangle &\lra \CO(\CX_n) , \\
		F &\lto F(h) .
	\end{align*}
	Hence we arrive at the commutative diagram 
	\[\begin{tikzcd}
		\CE(\CX_n) \arrow[r] \arrow[rd, end anchor = 170] & (\iota_g\times\id)^\ast \sigma^\ast \CE(D_{n+1}\times_K \CX_n) \arrow[d, "(h\times\id)^\ast"] \arrow[r] & \CE(\CX_n)\langle T_1,\ldots,T_{(d+1)^2} \rangle \arrow[d, "\mathrm{ev}_h"]\\
		& (gh)^\ast \CE(\CX_n) \arrow[r, "\Phi_{gh}(\CX_n)"] &\CE(\CX_n) 
	\end{tikzcd}\]
	which shows that $gh.v = F_v(h)$.
	We conclude that the orbit map $G_0 \ra \CE(\CX_n)$ is analytic on the open neighbourhood $\iota_g (D_{n+1})(K)$ of $g$.
\end{proof}

\begin{corollary}\label{Cor 2 - Dual of sections on X_n are locally analytic representation}
	The contragredient representation
	\[G_0 \times H^0(\CX_n,\CE)' \lra H^0(\CX_n,\CE)' \,,\quad (g,\ell) \lto \ell(g^{-1}.\blank) ,\]
	on the $K$-Banach space $H^0(\CX_n,\CE)'$ is locally analytic.
\end{corollary}
\begin{proof}
	We argue analogous to the proof of \cite[Prop 3.8]{SchneiderTeitelbaum02pAdicBoundVal}.
	By \Cref{Prop 1 - Equivalent characterization for locally analytic representations on Banach spaces} the representation of $G_0$ on $H^0(\CX_n,\CE)$ from \Cref{Lemma 2 - Sections on X_n are locally analytic representation} is given by a locally analytic homomorphism of locally $K$-analytic Lie groups $\rho \colon G_0 \ra \GL(H^0(\CX_n,\CE))$.
	Then the contragredient representation
	\[\rho^\ast \colon G_0 \lra \GL\big(H^0(\CX_n,\CE)' \big)\,,\quad g \lto  \rho(g^{-1})^t , \]
	is given by a locally analytic homomorphism of Lie groups as well \cite[III.\ \S 3.11 Cor.\ 2]{Bourbaki89LieGrpLieAlg1to3}, and is locally analytic therefore.
\end{proof}

\begin{proof}[{Proof of \Cref{Prop 2 - Dual of the sections on DHS is locally analytic representation}}]
	For the statements about $H^0(\CX,\CE^\rigbun)'_b$, we argue analogously to the proof of \cite[Prop.\ 1.4]{SchneiderTeitelbaum02pAdicBoundVal}.
	By \Cref{Lemma 2 - Description of sections of DHS as projective limit}, the homomorphisms $H^0(\CX_{n+1},\CE) \ra H^0(\CX_n,\CE)$ have dense image.
	Therefore, the image of the projections $H^0(\CX,\CE) \ra H^0(\CX_n,\CE)$ is dense, too \cite[Ch.\ II.\ \S 3.5 Thm.\ 1]{Bourbaki66GenTop1}.
	We apply \cite[Prop.\ 16.5]{Schneider02NonArchFunctAna} to conclude that \eqref{Eq 2 - Dual of projective limit of sections} is a topological isomorphism.

	As $H^0(\CX_{n+1},\CE) \ra H^0(\CX_n,\CE)$ has dense image, the transpose $H^0(\CX_n,\CE)' \ra H^0(\CX_{n+1},\CE)'$ is injective.
	Moreover, \cite[Lemma 16.4]{Schneider02NonArchFunctAna} implies that they are compact.

	Concerning the $G$-action \eqref{Eq 2 - Representation on dual of global sections}, we have seen in \Cref{Lemma 2 - Continuity of action on sections of DHS} that $G$ acts by continuous endomorphisms on $H^0(\CX, \CE)$.
	Therefore the contragredient $G$-action on $H^0(\CX, \CE)'_b$ is by continuous endomorphisms as well.

	In view of \Cref{Prop 1 - Locally analytic representations and open subgroups} it suffices to show that the orbit maps of \eqref{Eq 2 - Representation on dual of global sections} restricted to $G_0$ are locally analytic.
	For any $\ell \in H^0(\CX,\CE)'_b$, there exists some $n\in \BN$ such that $\ell \in H^0(\CX_n,\CE)'$, and the inclusion $H^0(\CX_n, \CE)' \hookrightarrow H^0(\CX,\CE)'_b$ is $G_0$-equivariant by \Cref{Lemma 2 - Description of sections of DHS as projective limit}.

	But we have seen in \Cref{Cor 2 - Dual of sections on X_n are locally analytic representation} that the orbit map $G_0 \ra H^0(\CX_n,\CE)'$ of $\ell$ is locally analytic.
	As $H^0(\CX,\CE)'_b$ is of compact type with respect to \eqref{Eq 2 - Dual of sections is of compact type}, $H^0(\CX_n,\CE)' \hookrightarrow H^0(\CX,\CE)_b'$ constitutes a BH-subspace.
	Then by definition the orbit map $G_0 \ra H^0(\CX,\CE)'_b$ of $\ell$ is locally analytic as well.
\end{proof}

\subsection{Strictness of certain \v{C}ech Complexes for the Complement of Schubert Varieties}\label{Sect 2 - Strictness of certain Cech complexes}

Our next aim is to see how the strong duals of the local cohomology groups for $\CE$ with respect to Schubert varieties $\BP_K^j \subset (\BP_K^d)^\rig$ become locally analytic representations.
Moreover, we will show that these strong dual spaces are of compact type similarly to \eqref{Eq 2 - Dual of sections is of compact type} by giving an exhaustion by the local cohomology groups with respect to ``tubes'' around the Schubert varieties.

However, the first step in this section is to prove that certain \v{C}ech complexes which compute the cohomology of the complement of these Schubert varieties consist of strict continuous homomorphisms.
The strictness property enables us to pass to the local cohomology with respect to the ``tubes'' around the Schubert varieties in a well-behaved way.\\

First we define certain rigid analytic subvarieties of $\BP_K^d$ which are neighbourhoods of the Schubert varieties
\begin{equation*}
	\BP^\sd_{K} = \left\{ [z_0{\,:\,}\ldots{\,:\,} z_\sd{\,:\,} 0 {\,:\,} \ldots {\,:\,} 0] \in \BP_K^d \right\} \defeq V_+ (X_{\sd+1},\ldots,X_d) \subset \BP_K^d ,
\end{equation*}
for fixed $\sd \in \{0,\ldots,d-1\}$.
For $0<\varepsilon <1$, $\varepsilon \in \abs{\widebar{K}}$, the ``open'' $\varepsilon$-neighbourhood around $\BP_K^\sd$ is defined as
\begin{equation}\label{Eq 2 - Definition of open tube}
	\BP^\sd_K(\varepsilon) \defeq \left\{[z_0{\,:\,}\ldots{\,:\,}z_d]\in \BP^d_K \middle{|} \forall i=\sd+1,\ldots,d: \abs{z_i} \leq \varepsilon \right\},
\end{equation}
and the ``closed'' one as
\[ \BP^\sd_K(\varepsilon)^- \defeq \left\{[z_0{\,:\,}\ldots{\,:\,}z_d]\in \BP^d_K \middle{|} \forall i=\sd+1,\ldots,d: \abs{z_i} < \varepsilon \right\} .\]

We will describe some admissible coverings of the complements of these $\varepsilon$-neighbourhoods.
Let $\lambda\in K\unts$ and $m\in \BN$ such that $\varepsilon = \sqrt[m]{\abs{\lambda}}$.
We have the admissible covering 
\begin{equation}\label{Eq 2 - Covering of complement of the closed tubes}
	\BP^d_K \setminus \BP^\sd_K(\varepsilon)^- = \bigcup_{i=\sd+1}^d U_{i,\varepsilon}
\end{equation}
by the Weierstra{\ss} domains
\begin{align*}
	U_{i,\varepsilon} = D_+(X_i)_{\varepsilon} &\defeq \left\{ [z_0{\,:\,}\ldots{\,:\,}z_d] \in \BP^d_K \middle{|} \forall j=0,\ldots,d: \varepsilon \, \abs{z_j} \leq  \abs{z_i}\right\} \\
	&\phantom{\vcentcolon}= \left\{ [z_0{\,:\,}\ldots{\,:\,}z_d] \in \BP^d_K \middle{|} \forall j=0,\ldots,d: \left\lvert \lambda \frac{z_j^m}{z_i^m} \right\rvert \leq 1 \right\},
\end{align*}
cf.\ \cite[6.1.5 Thm.\ 4]{BoschGuentzerRemmert84NonArchAna}.
We will denote this covering \eqref{Eq 2 - Covering of complement of the closed tubes} by $\CU_{\varepsilon}$ in the following.
Note that $U_{i,\varepsilon}$ is contained in the standard open affine subset $D_+(X_i)$ of the projective space.

Moreover, for $0\leq\varepsilon <1$, $\varepsilon \in \abs{\widebar{K}}$, let
\begin{equation*}
	U_{i,\varepsilon}^- = D_+(X_i)_{\varepsilon}^- \defeq  \left\{ [z_0{\,:\,}\ldots{\,:\,}z_d] \in \BP^d_K \middle{|} \forall j=0,\ldots,d: \varepsilon \, \abs{z_j} < \abs{z_i} \right\} .
\end{equation*}
Then $U_{i,\varepsilon}^-$ becomes an open admissible subdomain of $\BP^d_K$ which is quasi-Stein via the admissible covering 
\begin{equation*}
	U_{i,\varepsilon}^- = \bigcup_{ \abs{\CO_{\widebar{K}} } \ni \varepsilon_m\searrow\varepsilon  } U_{i,\varepsilon_m} ,
\end{equation*}
for any strictly decreasing sequence of $(\varepsilon_m)_{m\in \BN} \subset \abs{\widebar{K}\unts}$, with $\varepsilon_m> \varepsilon$, $\varepsilon_m\ra \varepsilon$.
%We could also consider the net of these $\varepsilon'$ to get an admissible covering but to give sections on there a locally convex topology, we need a countable covering.
The condition on the associated homomorphisms of affinoid algebras to have dense image is fulfilled for every inclusion of Weierstra{\ss} domains \cite[7.3.4 Prop.\ 2]{BoschGuentzerRemmert84NonArchAna}.
%By \cite[7.2.3 Prop.\ 6]{BoschGuentzerRemmert84NonArchAna} the $D_+(X_k)_{\varepsilon_m}$ are Weierstra{\ss} domains inside each other.
In the extreme case of $\varepsilon =0$, we have $U_{i,0}^- = U_i^\rig$, cf.\ \cite[9.3.4 Example 2]{BoschGuentzerRemmert84NonArchAna}.
Here $U_i^\rig$ denotes the rigid analytification of the open affine subscheme $U_i \defeq D_+(X_i) \subset \BP_K^d$.

We let $\CU_{\varepsilon}^-$ denote the admissible covering
\begin{equation}\label{Eq 2 - Covering of complement of the open tubes}
	\BP^d_K \setminus \BP^\sd_K(\varepsilon) = \bigcup_{i=\sd+1}^d U_{i,\varepsilon}^-.
	% = \bigcup_{ \abs{\CO_{\widebar{K}} } \ni \varepsilon_m\searrow\varepsilon } \BP^d_K \setminus \BP^\sd_K(\varepsilon_m)^-
\end{equation}
In the extreme case $\varepsilon = 0$, we also write $(\BP^\sd_K)^\rig \subset \BP_K^d$ when we want to signify that the complement of $\BP_K^\sd \subset (\BP_K^d)^\rig$ is an admissible open subset
\begin{equation*}
	\BP^d_K \setminus \big( \BP^\sd_K\big)^\rig = \bigcup_{i=\sd+1}^d U_{i}^\rig = \bigcup_{i=\sd+1}^d \bigcup_{ \abs{\CO_{\widebar{K}} } \ni \varepsilon_m\searrow 0  } U_{i,\varepsilon_m} .
\end{equation*}
Moreover, we let $\CU$ denote the standard covering by open affine subschemes
\begin{equation}\label{Eq 2 - Covering of complement of Schubert varieties}
	\BP_K^d \setminus \BP_K^\sd = \bigcup_{i=\sd+1}^d U_i .
\end{equation}
Finally, we write, for any subset $I\subset \{0,\ldots,d\}$,
\begin{equation}\label{Eq 2 - Intersection of open subsets of the covering of the complement of the closed tubes}
	U_{I,\varepsilon} \defeq \bigcap_{i\in I} U_{i,\varepsilon}  ,
\end{equation}
and similarly $U_{I,\varepsilon}^-$, $U_I$, and $U_I^\rig$.
By convention we set $U_{\emptyset}=U_{\emptyset,\varepsilon} = U_{\emptyset,\varepsilon}^- \defeq \BP_K^d$.

\begin{lemma}\label{Lemma 2 - Density of algebraic sections in analytic sections}
	Let $0<\varepsilon <1$ with $\varepsilon \in \abs{\widebar{K}}$.
	Let $I\subset \{0,\ldots,d\}$ and denote by $\CE^\algbun(U_I)$ the sections of $\CE$ on the algebraic variety $U_I$.
	Then the homomorphism $\CE^\algbun( U_I) \ra \CE^\rigbun( U_{I,\varepsilon})$ which is induced from $U_{I,\varepsilon}$ being an admissible open subdomain of $U_I^\rig$ is injective and its image is dense.
	In particular, $\CE^\rigbun(U_{I,\varepsilon})$ is the completion of $\CE^\algbun(U_I)$ when the latter is considered with the topology coming from the Banach topology of $\CE^\rigbun (U_{I,\varepsilon})$.
\end{lemma}
\begin{proof}
	We may assume that $I = \{i_0,\ldots,i_m\} \neq \emptyset$.
	Then the isomorphism 
	\begin{align*}
		U_{i_0}=D_+(X_{i_0}) &\overset{\cong}{\lra} \Spec \,K \big[T_{(j,i_0)}\mid j=0,\ldots,d, j\neq i_0 \big] , \\
		\frac{X_j}{X_{i_0}} &\longmapsfrom T_{(j,i_0)} ,
	\end{align*}
	of schemes induces an isomorphism	$U_I \cong \Spec (A/\Fa)$ where
	\begin{align*}
		A &= K \big[T_{(j,i_k)} \,\big\vert\, j=0,\ldots,d, k=0,\ldots,m \text{, with } j \neq i_k \big] ,\\
		\Fa &= \big(T_{(i_j,i_k)} T_{(i_k,i_j)} -1 \,\big\vert\, j,k=0,\ldots,m \text{, with } j\neq k \big).
	\end{align*}
	Recall that $U_I^\rig$ is defined via the admissible covering \cite[9.3.4 Example 2]{BoschGuentzerRemmert84NonArchAna}
	\begin{equation*}
		U_I^\rig  = \bigcup_{n\geq 0} U_{I,\abs{\unif}^n} , 
	\end{equation*}
	and we have isomorphisms $U_{I,\abs{\unif}^n} \cong \Sp ( A_n/\Fa A_n )$, for 
	\begin{align*}
		A_n = K\big\langle \unif^n T_{(j,i_k)} \,\big\vert\, j=0,\ldots,d, k=0,\ldots,m \text{, with } j\neq i_k \big\rangle .
	\end{align*}
	We see that $A/\Fa \hookrightarrow A_n/\Fa A_n$ has dense image as $A$ is dense in $A_n$.
	Furthermore, the Weierstra{\ss} subdomain $U_{I,\varepsilon} \subset U_I^\rig$ is contained in some $U_{I,\abs{\unif}^n}$ so that $A/\Fa \hookrightarrow \CO(U_{I,\varepsilon})$ is dense as well \cite[7.3.4 Prop.\ 2]{BoschGuentzerRemmert84NonArchAna}.
	This settles the case of the structure sheaf $\CE = \CO$.
	
	For a general $\bG$-equivariant vector bundle $\CE$ on $\BP_K^d$ by \Cref{Lemma 2 - T-equivariant trivialization}, we find a trivialization $\CE^\algbun\res{U_{i_0}} \cong (\CO^\algbun)^{\oplus n}\res{U_{i_0}}$, for $n \defeq {\rm rk}(\CE)$.
	Therefore the claim follows from the compatible isomorphisms $\CE^\algbun(U_I) \cong \CO^\alg(U_I)^{\oplus n}$ and $\CE^\rigbun(U_{I,\varepsilon}) \cong \CO^\rigbun (U_{I,\varepsilon})^{\oplus n}$.	
\end{proof}

\begin{remark}\label{Rmk 2 - Concrete description of the norms of the sections on the affinoid subdomains}
	One can give a concrete description of the Gauss norm $\abs{\blank}_\varepsilon$ on the affinoid algebra $\CO (U_{I,\varepsilon})$ (and in turn on $\CO^\algbun (U_{I})$), cf.\ \cite[Proof of Lemma 1.3.1]{Orlik08EquivVBDrinfeldUpHalfSp}:
	\begin{equation}\label{Eq 2 - Norm on open affinoid}
		\bigg\lvert \sum_{ \ul{k}\in \Lambda_I} a_{\ul{k}} \, X_0^{k_0} \cdots X_d^{k_d}  \bigg\rvert_\varepsilon 
		= \sup_{ \ul{k} \in \Lambda_I}   \abs{ a_{\ul{k}}}  \bigg( \frac{1}{\varepsilon} \bigg)^{\abs{\max(0,\ul{k})}} .
	\end{equation}
	Here we use the notation
	\begin{align*}
		\Lambda_I &\defeq \bigg\{ \ul{k} \in \BZ^{d+1} \bigg\vert \sum_{i=0}^d k_i = 0 \text{ , and } \forall i\in \{0,\ldots,d\}\setminus I: k_i \geq 0 \bigg\} ,
	\end{align*}
	and
	\begin{align*}
		\max(0,\ul{k}) &\defeq \big( \max(0,k_0),\ldots , \max(0,k_d) \big) .
	\end{align*}
\end{remark}

We will need the following result about the \v{C}ech complex associated to the covering \eqref{Eq 2 - Covering of complement of Schubert varieties}\footnote{The statement of \Cref{Thm 2 - Differentials of algebraic Cech complex are strict} is found in the proof of \cite[Lemma 1.3.1]{Orlik08EquivVBDrinfeldUpHalfSp}. However the justification given there contains some flaws.}.

\begin{theorem}\label{Thm 2 - Differentials of algebraic Cech complex are strict}
	Let $\CE$ be a $\bG$-equivariant vector bundle on $\BP_K^d$, and $0<\varepsilon <1$ with $\varepsilon \in \abs{\widebar{K}}$.
	For the \v{C}ech complex $C^\bullet (\CU,\CE)$ associated with the covering \eqref{Eq 2 - Covering of complement of Schubert varieties} which computes the coherent cohomology $H^\bullet (\BP_K^d\setminus \BP_K^\sd ,\CE)$, the differentials
	\begin{equation*}
		d^q \colon C^q (\CU,\CE)  \lra C^{q+1} (\CU,\CE)
	\end{equation*}
	are strict continuous homomorphisms when each $\CE(U_I) \,(= \CE^\algbun (U_I) ) \subset \CE^\rigbun(U_{I,\varepsilon})$ is endowed with the topology coming from the Banach topology of $\CE^\rigbun (U_{I,\varepsilon})$, for $I \subset \{\sd +1,\ldots,d\}$.
\end{theorem}

Let $\bT \subset \bG$ denote the split maximal torus of diagonal matrices.
For the group of its characters, we have the usual identification $X(\bT) \cong \BZ^{d+1}$ by choosing the characters
\begin{equation*}
	\epsilon_i \colon \bT \lra \BG_m \,,\quad  {\rm diag}(t_0,\ldots,t_d) \lto t_i \quad\text{, for $i=0,\ldots,d$,}
\end{equation*}
as a $\BZ$-basis.
Recall that for any $\bT$-module $V$ in the sense of \cite[I.\ \S 2.7]{Jantzen03RepAlgGrp}, we have a decomposition into weight spaces \cite[I.\ \S 2.11]{Jantzen03RepAlgGrp}
\begin{equation*}
	V \cong \bigoplus_{\ul{\lambda}\in X(\bT)} V_\ul{\lambda} \quad\text{, for $V_\ul{\lambda} \defeq \big\{v\in V \,\big\vert\, \text{$\forall$ $K$-algebras $R$, $\forall t\in \bT(R)$} :t.(v\otimes 1) = v \otimes \ul{\lambda}(t)  \big\} $.}
\end{equation*}
In particular this is a decomposition into simultaneous eigenspaces with respect to the induced action of $T\defeq \bT(K)$ on $V$.
We say that $v\in V$ is \textit{$\bT$-homogeneous} if $v\in V_\ul{\lambda}$, for some $\ul{\lambda}\in X(\bT)$.

Note that the open affine subsets $U_i \subset \BP_K^d$ are stabilized under the action of $\bT$, for all $i=0,\ldots,d$.
Therefore, the $K$-vector space of sections $\CE(U_i)$ obtains the structure of a $\bT$-module which decomposes into weight spaces
\begin{equation}\label{Eq 2 - Weight space decomposition of sections}
	\CE(U_i) = \bigoplus_{\ul{\lambda}\in X(\bT)} \CE(U_i)_\ul{\lambda} .
\end{equation}

For $x\in \BP_K^d$, we let $\CE(x) \defeq \CE_{x} / \Fm_{x} \CE_{x}$ denote the fibre of $\CE$ at $x$ where $\Fm_{x}$ is the maximal ideal of the local ring $\CO_{\BP_K^d,x}$.
The $K$-vector space $\CE(x)$ is canonically isomorphic to the fibre at $x$ of the geometric vector bundle associated with $\CE$.

Note that the $\bT$-action on $\BP_K^d$ has the fixed points $x_i \defeq [0{\,:\,}\ldots{\,:\,}0{\,:\,}1{\,:\,}0{\,:\,}\ldots{\,:\,}0] \in U_i$, for $i=0,\ldots,d$.
We therefore obtain the structure of a $\bT$-module on $\CE_{x_i}$.
When considering the structure sheaf $\CO$ of $\BP_K^d$ as a $\bG$-equivariant sheaf in the usual way, the maximal ideal $\Fm_{x_i} \subset \CO_{x_i}$ is a $\bT$-submodule for the same reason. 
It follows from 
\begin{equation}\label{Eq 2 - Action on sections and multiplication}
	t. (s e) = (t.s) (t.e) \quad\text{, for all $t\in \bT(R)$, $s\in \CO(U_i) \botimes{K} R$, and $ e \in \CE(U_i) \botimes{K} R$,}
\end{equation}
that
\begin{equation*}
	t. (s e) = (t.s) (t.e) \quad\text{, for all $t\in \bT(R)$, $s\in \CO_{\BP_K^d,x_i} \botimes{K} R$, and $ e \in \CE_{x_i} \botimes{K} R$,}
\end{equation*}
for all $K$-algebras $R$.
Hence $\Fm_{x_i} \CE_{x_i} \subset \CE_{x_i}$ is a $\bT$-submodule as well, and we obtain the structure of a $\bT$-module on the quotient $\CE(x_i)$.

\begin{lemma}[{cf.\ \cite[Lemma 4.6]{Kempf78GrothendieckCousinComplexIndRep}}]\label{Lemma 2 - T-equivariant trivialization}
	Let $\CE$ be a $\bG$-equivariant vector bundle on $\BP_K^d$ and $i \in \{0,\ldots,d\}$.	
	For all open subsets $U\subset U_i$, there are $T$-equivariant isomorphisms of $\CO(U)$-modules
	\begin{equation*}
		\CE(U) \cong \CO(U) \botimes{K} \CE(x_i) ,
	\end{equation*}
	which are compatible with respect to the restriction homomorphisms.
\end{lemma}
\begin{proof}
	By \cite[I.5.16]{Jantzen03RepAlgGrp} and \cite[II.1.10]{Jantzen03RepAlgGrp}, the $\bG$-equivariant vector bundle $\CE$ admits a trivialisation on the open affine subset $U_i$.
	In particular, we can find sections in $\CE(U_i)$ which globally generate $\CE\res{U_i}$.
	We may take a $K$-basis $(e_j)_{j \in J}$, of $\CE(U_i)$  to do so.
	For $U\subset U_i$ open, we then define the homomorphism
	\begin{equation*}
		\varphi_U \colon \CE(U) \lra \CO(U) \botimes{K} \CE(x_i) \,,\quad \sum_{j\in J} s_j e_j\res{U} \lto \sum_{j\in J} s_j \otimes e_j(x_i) \quad\text{, for $s_j\in \CO(U)$,}
	\end{equation*}
	of $\CO(U)$-modules.
	Note that $\varphi_U$ is independent of the choice of the $K$-basis $(e_j)_{j \in J}$.
	An inverse to $\varphi_U$ is given by 
	\begin{equation*}
		\CO(U) \botimes{K} \CE(x_i) \lra \CE(U) \,,\quad s \otimes \sum_{j\in J} a_j e_j(x_i) \lto \sum_{j\in J} a_j s e_j\res{U} \quad\text{, for $s\in \CO(U)$, $a_j\in K$.}
	\end{equation*}
	To show that $\varphi_U$ is $T$-equivariant, let $t \in T$, $k\in J$, and $ t.e_k = \sum_{j \in J} a_j e_j$ in $\CE(U_i)$, for some $a_j\in K$.
	It follows from \eqref{Eq 2 - Action on sections and multiplication} that, for all $s \in \CO(U)$,
	\begin{align*}
		\varphi_U \big( t. (s e_k\res{U})\big)
		&= \varphi_U \bigg( \sum_{j\in J} (t.s) a_j e_j\res{U} \bigg) 
		=  t.s \otimes \sum_{j\in J} a_j e_j(x_i) \\
		&= t. \big( s \otimes e_k(x_i) \big) 
		= t. \varphi_U\big(s e_k\res{U}\big) .
	\end{align*}
	The compatibility of the $\varphi_U$ with the restriction homomorphisms is immediate.
\end{proof}

It will be helpful to fix certain norms on $\CE(U_I)$, for non-empty $I \subset \{0,\ldots,d\}$, which realise its locally convex topology.
Recall that
\begin{equation*}
	\CO(U_{I,\varepsilon}) = \Bigg\{ \sum_{\ul{k}\in \Lambda_I} a_\ul{k} \, X^\ul{k} \Bigg\vert \abs{a_\ul{k}} \bigg(\frac{1}{\varepsilon}\bigg)^{\abs{ \max(0,\ul{k})}} \ra 0 \text{ as $\abs{\ul{k}} \ra \infty$} \Bigg\} ,
\end{equation*}
with norm given by \eqref{Eq 2 - Norm on open affinoid}.
We now choose some $i\in I$, and let $v_1,\ldots,v_n \in \CE(x_i)$ be a $K$-basis of weight vectors of weights $\ul{\lambda}_1,\ldots,\ul{\lambda}_n \in X(\bT)$.
We endow $\CE(x_i)$ with the norm prescribed by this basis, i.e.\ 
\begin{equation*}
	\bigg\lvert \sum_{l=1}^n a_l \, v_l \bigg\rvert \defeq \max_{l=1}^n \abs{a_l} \quad\text{, for $a_1,\ldots,a_n \in K$.}
\end{equation*}
We continue to denote the elements of $\CE(U_i)$ corresponding to $1 \otimes v_l$ under the isomorphism of \Cref{Lemma 2 - T-equivariant trivialization} by $v_l$.
Consequently, we have an isomorphism of $\CO(U_i)$-modules
\begin{equation*}
	\CO(U_i)^{\oplus n} \overset{\cong}{\lra} \CE(U_i)\,,\quad (f_1,\ldots,f_n) \lto \sum_{l=1}^n f_l \, v_l .
\end{equation*}
This induces an isomorphism of $\CO(U_{I,\varepsilon})$-modules $\CE(U_{I,\varepsilon})\cong \CO(U_{I,\varepsilon})^{\oplus n}$ which endows the former with the norm
\begin{equation}\label{Eq 2 - Fixed norm on sections}
	\bigg\lvert \sum_{l=1}^n f_l \, v_l \bigg\lvert_\varepsilon \defeq \sup_{\substack{l=1,\ldots,n\\ \ul{k}\in \Lambda_I}} \abs{a_{l,\ul{k}}} \bigg( \frac{1}{\varepsilon}\bigg)^{\abs{ \max(0,\ul{k})}} \quad\text{, for $a_{l,\ul{k}}\in K$ with $f_l = \sum_{\ul{k}\in \Lambda_I} a_{l,\ul{k}} \, X^\ul{k} \in \CO(U_{I,\varepsilon})$.}
\end{equation}
We fix this norm on $\CE(U_{I,\varepsilon})$ and on its subspace $\CE(U_I) \, (=\CE^\algbun(U_I))$, and omit the $\varepsilon$ from the index.
Note that for a different choice of $i\in I$ or the basis of weight vectors of $\CE(x_i)$, the above construction yields an equivalent norm on $\CE(U_{I,\varepsilon})$ by \cite[3.7.3 Prop.\ 3]{BoschGuentzerRemmert84NonArchAna}.

As the main step towards proving \Cref{Thm 2 - Differentials of algebraic Cech complex are strict}, we want to show that, for ``large enough'' weights $\ul{\lambda}\in X(\bT)$, the weight spaces $\CE(U_I)_\ul{\lambda}$ change in a uniform way when one varies the ``extreme'' entries of $\ul{\lambda}$.
To this end we define, for $m\in \BN_0$,
\begin{equation*}
	\boundedregion_m \defeq \left\{ \ul{\lambda} \in X(\bT) \cong \BZ^{d+1} \middle{|} \forall j=0,\ldots,d: \abs{\lambda_j} \leq m \right\} .
\end{equation*}
Because the $\bT$-modules $\CE(x_i)$ are finite-dimensional, we find some $M \in \BN$ such that the weights of all $\CE(x_0),\ldots,\CE(x_d)$ are concentrated in $\boundedregion_M$, i.e.\ for all $i\in \{0,\ldots,d\}$ and all $\ul{\lambda} \in X(\bT)$ with $\CE(x_i)_\ul{\lambda} \neq \{0\}$, we have $\ul{\lambda} \in \boundedregion_M$.
We moreover set $N \defeq (2d+1)M +d$.

\begin{proposition}\label{Prop 2 - Existence of weight changing isomorphisms}
	For every $\ul{\mu} \in X(\bT)$, there exist $\ul{\nu} \in \boundedregion_N$ and $C\in \varepsilon^{\BN_0}$ such that, for all non-empty $I \subset \{0,\ldots,d\}$, there is an isomorphism 
	\begin{equation*}
		\varphi^I_{\ul{\mu},\ul{\nu}} \colon \CE(U_{I})_\ul{\mu} \overset{\cong}{\lra} \CE(U_{I})_\ul{\nu}
	\end{equation*}
	of $K$-vector spaces satisfying $\abs{\varphi^I_{\ul{\mu},\ul{\nu}} (v)} = C \,\abs{v}$, for all $v\in \CE(U_{I,\varepsilon})_\ul{\mu}$, and such that these isomorphisms are compatible with the restriction homomorphisms, i.e.\ for all $I \subset J \subset \{0,\ldots,d\}$, the following diagram commutes:
	\begin{equation*}
		\begin{tikzcd}
			\CE(U_{I})_\ul{\mu} \ar[r]\ar[d, "\varphi^I_{\ul{\mu},\ul{\nu}}"'] & \CE(U_{J})_\ul{\mu}	\ar[d, "\varphi^J_{\ul{\mu},\ul{\nu}}"] \\
			\CE(U_{I})_\ul{\nu} \ar[r]& \CE(U_{J})_\ul{\nu}	 .
		\end{tikzcd}
	\end{equation*}
\end{proposition}
\begin{proof}
	We proceed by induction on $\lVert \ul{\mu} \rVert \defeq \sum_{j=0}^d \abs{\mu_j}$.
	Note that for $\ul{\mu} \in \boundedregion_N$, we may take $\ul{\nu}\defeq \ul{\mu}$, $C\defeq 1$, and the identity homomorphisms to obtain the assertion of the proposition.
	This also deals with the base case $\lVert \ul{\mu}\rVert = 0$.

	Hence we now suppose that $\ul{\mu} \notin \boundedregion_N$.
	We let $u,v \in \{0,\ldots,d\}$ such that $\mu_u$ is maximal among the entries $\mu_0,\ldots,\mu_d$ and the entry $\mu_v$ is minimal.
	We set $\ul{\mu} - \alpha_{u,v} \defeq  \ul{\mu} - \epsilon_u + \epsilon_v $.

	For fixed $I\subset \{0,\ldots,d\}$ with $i\in I$, let $v_1,\ldots,v_n \in \CE(x_i)$ be a $K$-basis of weight vectors of weights $\ul{\lambda}_1,\ldots,\ul{\lambda}_n$.
	We then have the following $K$-basis for the weight space
	\begin{equation*}
		\CE(U_I)_\ul{\mu} = \bigoplus_{l\in L_{\ul{\mu},I}} K \cdot X^{\ul{\mu} - \ul{\lambda}_l} \, v_l
	\end{equation*}
	where $L_{\ul{\mu},I} \defeq \big\{ l\in \{1,\ldots,n\} \,\big\vert\, X^{\ul{\mu}-\ul{\lambda}_l} \in \CO(U_I) \big\}$.

	\begin{lemma}\label{Lemma 2 - Comparison of occuring weights}
		For $l\in \{1,\ldots,n\}$, we have 
		\begin{equation*}
			X^{\ul{\mu}-\ul{\lambda}_l} \in \CO(U_I) \quad\text{ if and only if }\quad X^{(\ul{\mu}-\alpha_{u,v})-\ul{\lambda}_l} \in \CO(U_I) .
		\end{equation*}
		If this is the case, for some $l \in \{1,\ldots,n\}$, we moreover have $\mu_u > M+1$ and $\mu_v < - (M+1)$ so that in particular $u \neq v$.
	\end{lemma}
	\begin{proof}
		We first note that $X^{\ul{\mu}-\ul{\lambda}_l} \in \CO(U_I)$ if and only if $\sum_{j=0}^d \mu_j -\lambda_{l,j} = 0$ and, for all $j \in \{0,\ldots,d\}\setminus I$, we have $\mu_j - \lambda_{l,j} \geq 0$.
		For $X^{(\ul{\mu}-\alpha_{u,v})-\ul{\lambda}_l}$ the analogous statement holds.
		To show the claimed equivalence it thus suffices to focus on the exponents of $X_u$ and $X_v$.

		Because $\ul{\mu}\notin \boundedregion_N$, there exists $j\in \{0,\ldots,d\}$ such that $\abs{\mu_j} > N$.
		We distinguish the two cases that $\mu_j >0$ and that $\mu_j <0$.
		If $\mu_j>0$, then $\mu_u > N\geq M+1$ by the maximality of $\mu_u$.
		Using $\abs{\lambda_{l,u}} \leq M$ we find that
		\begin{equation*}
			\mu_u -1 - \lambda_{l,u} > N - 1 - M \geq 0
		\end{equation*}
		and also $\mu_u -\lambda_{l,u} \geq 0$, so that the exponent of $X_u$ is not an obstacle in this case.
		Concerning the exponent of $X_v$, since $(\ul{\mu}-\alpha_{u,v})_v - \lambda_{l,v} \geq \mu_v - \lambda_{l,v}$, we find that if $X^{\ul{\mu}-\ul{\lambda}_l} \in \CO(U_I)$, then $X^{(\ul{\mu}-\alpha_{u,v})-\ul{\lambda}_l} \in \CO(U_I)$.
		Conversely suppose that $X^{(\ul{\mu}-\alpha_{u,v})-\ul{\lambda}_l} \in \CO(U_I)$ so that in particular $\sum_{j=0}^d \big( (\ul{\mu}-\alpha_{u,v})_j - \lambda_{l,j} \big) =0$ and $(\ul{\mu}-\alpha_{u,v})_v - \lambda_{l,v} \geq 0$.
		We compute
		\begin{align*}
			\bigg\lvert \sum_{j=0}^d \mu_j \bigg\rvert
			&= \bigg\lvert \sum_{j=0}^d (\mu_j - \lambda_{l,j}) + \sum_{j=0}^d \lambda_{l,j} \bigg\rvert 
			= \bigg\lvert \underset{=0}{\underbrace{\sum_{j=0}^d \big( (\ul{\mu}-\alpha_{u,v})_j - \lambda_{l,j} \big)}} + \sum_{j=0}^d \lambda_{l,j} \bigg\rvert \\
			&= \bigg\lvert \sum_{j=0}^d \lambda_{l,j} \bigg\rvert
			\leq \sum_{j=0}^d \abs{\lambda_{l,j}} 
			\leq (d+1) M .
		\end{align*}
		Then 
		\begin{equation*}
			(d+1) M \geq \bigg\lvert \sum_{j=0}^d \mu_j \bigg\rvert = \bigg\lvert \mu_u + \sum_{\substack{j=0 \\ j\neq u}}^d \mu_j \bigg\rvert
		\end{equation*}
		together with $\mu_u > N = (2d+1)M +d$ implies that $ \sum_{\substack{j=0 \\ j\neq u}}^d \mu_j  < -d (M+1)$.
		Hence there exists $j'\in \{0,\ldots,d\}$, $j'\neq u$, such that $\mu_{j'} < -(M+1)$.
		By the minimality of $\mu_v$, it follows that $\mu_v < -(M+1)$.
		But using $\abs{\lambda_{l,v}} \leq M$ we find that
		\begin{align*}
			(\ul{\mu}-\alpha_{u,v})_v - \lambda_{l,v} 
			&= \mu_v +1 - \lambda_{l,v}
			< -(M+1) +1 +M
			=0 .
		\end{align*}
		Therefore $X^{(\ul{\mu}-\alpha_{u,v})-\ul{\lambda}_l} \in \CO(U_I)$ can only occur if $v\in I$.
		We conclude that $X^{\ul{\mu}-\ul{\lambda}_l} \in \CO(U_I)$ as well which finishes the proof in this case.

		Now we consider the case that $\mu_j < 0$.
		Here we find that $\mu_v < - N \leq -(M+1)$ by the minimality of $\mu_v$.
		Like before, we have
		\begin{equation*}
			(d+1) M \geq \bigg\lvert \sum_{j=0}^d \mu_j \bigg\rvert = \bigg\lvert \mu_v + \sum_{\substack{j=0 \\ j\neq v}}^d \mu_j \bigg\rvert.
		\end{equation*}
		From this and $\mu_v < -N$ we conclude that there exists $j' \in \{0,\ldots,d\}$, $j' \neq v$, such that $\mu_{j'}>M+1$.
		As $\mu_u$ is maximal, it follows that $\mu_u >M+1$ as well.
		Therefore
		\begin{align*}
			(\ul{\mu}-\alpha_{u,v})_u - \lambda_{l,u} 
			&= \mu_u -1 - \lambda_{l,u}
			> (M+1) - 1 - M
			= 0 ,
		\end{align*}
		and moreover $\mu_u - \lambda_{l,u} \geq 0$.
		This shows that the exponent of $X_u$ is not an obstacle in this case.
		For the exponent of $X_v$, we compute
		\begin{align*}
			(\ul{\mu}-\alpha_{u,v})_v - \lambda_{l,v} 
			&= \mu_v +1 - \lambda_{l,v}
			< -N + 1 +M
			< 0 ,
		\end{align*}
		and also $\mu_v - \lambda_{l,v} <0$.
		Therefore either of $X^{\ul{\mu}-\ul{\lambda}_l} \in \CO(U_I)$ or $X^{(\ul{\mu}-\alpha_{u,v})-\ul{\lambda}_l} \in \CO(U_I)$ implies that $v\in I$.
		This shows the assertion in the case $\mu_j<0$.
	\end{proof}

	Using the above lemma, we see that
	\begin{equation*}
		\CE(U_I)_{\ul{\mu} -\alpha_{u,v}} = \bigoplus_{l\in L_{\ul{\mu},I}} K \cdot X^{(\ul{\mu} -\alpha_{u,v}) - \ul{\lambda}_l} \, v_l,
	\end{equation*}
	and we define the following isomorphism of $K$-vector spaces
	\begin{equation*}
		\varphi^I_{\ul{\mu},\ul{\mu}-\alpha_{u,v}} \colon \CE(U_I)_{\ul{\mu}} \lra \CE(U_I)_{\ul{\mu}-\alpha_{u,v}} \,,\quad X^{\ul{\mu} - \ul{\lambda}_l} \, v_l \lto X^{(\ul{\mu} -\alpha_{u,v}) - \ul{\lambda}_l} \, v_l .
	\end{equation*}
	For $l\in L_{\ul{\mu},I}$, it follows from $\mu_u > M+1$, $\mu_v < - (M+1)$, and $\ul{\lambda}_l\in \boundedregion_M$ that
	\begin{align*}
		(\ul{\mu}-\alpha_{u,v})_u - \lambda_{l,u} &\geq 0 , \hspace{-2cm} & \mu_u - \lambda_{l,u}  &\geq 0 , \\
		(\ul{\mu}-\alpha_{u,v})_v - \lambda_{l,v} &\leq 0 , \hspace{-2cm}& \mu_v - \lambda_{l,v}  &\leq 0 .
	\end{align*}
	Therefore
	\begin{alignat*}{3}
		\max\big(0,(\ul{\mu}-\alpha_{u,v})_u - \lambda_{l,u} \big) &= \mu_u -1 - \lambda_{l,u} &&= \max\big(0, \mu_u - \lambda_{l,u}\big) -1, \\
		\max\big(0,(\ul{\mu}-\alpha_{u,v})_v - \lambda_{l,v} \big) &= 0 &&= \max\big(0, \mu_v - \lambda_{l,v}\big) ,
	\end{alignat*}
	and we conclude that
	\begin{align*}
		\big\lvert X^{(\ul{\mu} -\alpha_{u,v}) - \ul{\lambda}_l} \, v_l \big\rvert
		&= \bigg(\frac{1}{\varepsilon}\bigg)^{\sum_{j=0}^d \max\big(0,(\ul{\mu}-\alpha_{u,v})_j - \lambda_{l,j}\big) } \\
		&= \bigg(\frac{1}{\varepsilon}\bigg)^{\big(\sum_{j=0}^d \max(0,\mu_j - \lambda_{l,j}) \big) -1 } 
		= \varepsilon \, \big\lvert X^{\ul{\mu}  - \ul{\lambda}_l} \, v_l \big\rvert .
	\end{align*}
	Hence we have $\abs{\varphi^I_{\ul{\mu},\ul{\mu}-\alpha_{u,v}}(v)} =\varepsilon \,\abs{v}$, for all $v\in \CE(U_{I})_\ul{\mu}$.

	Next we verify that the above family of isomorphisms $\varphi^I \defeq \varphi^I_{\ul{\mu},\ul{\mu}-\alpha_{u,v}}$ is compatible with the restriction maps.
	For this consider $I\subset J \subset \{0,\ldots,d\}$ with $i\in I$ and $j\in J$.
	Let
	\begin{align*}
		\CE(U_I) = \bigoplus_{l=1}^n \CO (U_I) \,v_l \quad\text{and}\quad
		\CE(U_J) = \bigoplus_{k=1}^n \CO (U_J) \,w_k,
	\end{align*}
	for a $K$-basis consisting of weight vectors $v_1,\ldots,v_n$ of $\CE(x_i)$ of weights $\ul{\lambda}_1,\ldots,\ul{\lambda}_n$, and a $K$-basis consisting of weight vectors $w_1,\ldots,w_n$ of $\CE(x_j)$ of weights $\ul{\kappa}_1,\ldots,\ul{\kappa}_n$ which yield the isomorphisms $\varphi^I $ and $\varphi^J$ respectively.
	Let ${\rm res}\colon \CE(U_I) \hookrightarrow \CE(U_J)$ denote the restriction map which is $\CO(U_I)$-linear and injective.
	Furthermore, there are $a_{l,k} \in K$ such that
	\begin{equation*}
		{\rm res} (v_l) = \sum_{k=1}^n a_{l,k} \, X^{\ul{\lambda}_l - \ul{\kappa}_k} \, w_k \quad\text{ , for all $l=1,\ldots,n$,}
	\end{equation*}
	because ${\rm res} (v_l) \in \CE(U_J)$ is of weight $\ul{\lambda}_l$.
	For $l \in \{1,\ldots,n\}$ such that $X^{\ul{\mu}- \ul{\lambda}_l} \in \CO(U_I)$ we now compute
	\begin{align*}
		{\rm res} \Big( \varphi^I \big( X^{\ul{\mu}-\ul{\lambda}_l}\, v_l \big) \Big)
		&= {\rm res} \Big( X^{(\ul{\mu}-\alpha_{u,v})-\ul{\lambda}_l} \,v_l \Big) 
		= X^{(\ul{\mu}-\alpha_{u,v})-\ul{\lambda}_l} \,{\rm res} (v_l) \\
		&= X^{(\ul{\mu}-\alpha_{u,v})-\ul{\lambda}_l} \sum_{k=1}^n a_{l,k} \,X^{\ul{\lambda}_l - \ul{\kappa}_k} \, w_k 
		= \sum_{k=1}^n a_{l,k} \, X^{(\ul{\mu}-\alpha_{u,v}) - \ul{\kappa}_k} \, w_k \\
		&= \varphi^J \bigg( \sum_{k=1}^n a_{l,k} \, X^{\ul{\mu} - \ul{\kappa}_k} \, w_k \bigg)
		= \varphi^J \bigg( X^{\ul{\mu} - \ul{\lambda}_l} \sum_{k=1}^n a_{l,k}\,  X^{\ul{\lambda}_l - \ul{\kappa}_k} \, w_k \bigg) \\
		&= \varphi^J \big( X^{\ul{\mu} - \ul{\lambda}_l} \, {\rm res} ( v_l ) \big) 
		= \varphi^J \Big( {\rm res} \big( X^{\ul{\mu} - \ul{\lambda}_l} \, v_l \big) \Big) 
	\end{align*}
	which shows the compatibility for $I \subset J$.

	Finally, we want to apply the induction hypothesis.
	In the case that, for all $I \subset \{0,\ldots,d\}$ with $I\neq \emptyset$, the set $L_{\ul{\mu},I}$ is empty, we find that $\CE(U_I)_\ul{\mu} = \{0\}$, for all such $I$.
	We may for example take $\ul{\nu}\defeq (M+1,\ldots,M+1) \in \boundedregion_N$, and see that, for all such $I$ and all $l=1,\ldots,n$, we have $\ul{\nu}-\ul{\mu}_l \notin \Lambda_I$.
	It follows that $\CE(U_I)_\ul{\nu} = \{0\}$, and we obtain the sought isomorphisms trivially.

	On the other hand, if there is some $I\subset \{0,\ldots,d\}$ with $I \neq \emptyset$ for which $L_{\ul{\mu},I}$ is non-empty, the second assertion of \Cref{Lemma 2 - Comparison of occuring weights} implies that 
	\begin{align*}
		\lVert \ul{\mu}-\alpha_{u,v} \rVert 
		&= \abs{\mu_u -1} + \abs{\mu_v +1} + \sum_{\substack{j=0 \\ j\neq u,v}}^d \abs{\mu_j}
		= (\mu_u - 1) - (\mu_v +1) + \sum_{\substack{j=0 \\ j\neq u,v}}^d \abs{\mu_j}
		= \lVert \ul{\mu} \rVert - 2 .
	\end{align*}
	Applying the induction hypothesis to $\ul{\mu}- \alpha_{u,v}$, we obtain $\ul{\nu}\in \boundedregion_N$ and the family of isomorphisms $\varphi^I_{\ul{\mu}- \alpha_{u,v}, \ul{\nu}}$, for $I\subset \{0,\ldots,d\}$, as specified.
	Then the statement follows for $\ul{\mu}$ by taking the compositions
	\begin{equation*}
		\varphi^I_{\ul{\mu},\ul{\nu}} \defeq \varphi^I_{\ul{\mu}- \alpha_{u,v}, \ul{\nu}} \circ \varphi^I_{\ul{\mu},\ul{\mu}-\alpha_{u,v}} .
	\end{equation*}
\end{proof}

\begin{proof}[{Proof of \Cref{Thm 2 - Differentials of algebraic Cech complex are strict}}]
	Recall that in \eqref{Eq 2 - Fixed norm on sections}, for every non-empty $I \subset \{0,\ldots,d\}$, we fixed a norm on $\CE(U_I)$.
	This endows the space of $q$-th \v{C}ech cochains, for $q\geq 0$, 
	\begin{equation*}
		C^q (\CU,\CE) = \bigoplus_{(i_0,\ldots,i_q) \in \{\sd+1,\ldots,d\}^{q+1}} \CE\big(U_{\{i_0,\ldots,i_q\}}\big)
	\end{equation*}
	with a norm as well.
	We also have a decomposition of $C^q(\CU,\CE)$ into weight spaces
	\begin{equation*}
		C^q(\CU,\CE)_\ul{\lambda} = \bigoplus_{(i_0,\ldots,i_q) \in \{\sd+1,\ldots,d\}^{q+1}} \CE\big(U_{\{i_0,\ldots,i_q\}}\big)_\ul{\lambda} \quad\text{, for $\ul{\lambda} \in X(\bT)$,}
	\end{equation*}
	under the induced $\bT$-action.
	Note that if we have a weight decomposition of $f\in C^q(\CU,\CE)$, it follows from \eqref{Eq 2 - Fixed norm on sections} that
	\begin{equation*}
		\abs{f} = \sup_{\ul{\lambda}\in X(\bT)} \abs{f_{\lambda}} \quad\text{, for $f = \sum_{\ul{\lambda}\in X(\bT)} f_\ul{\lambda}$, with $f_\ul{\lambda} \in C^q(\CU,\CE)_\ul{\lambda}$.}
	\end{equation*}
	Since the restriction maps ${\rm res}_{I,J}\colon \CE(U_I) \hookrightarrow \CE(U_J)$, for $I\subset J \subset \{\sd+1 ,\ldots,d\}$, are $\bT$-equivariant, so is the differential
	\begin{align*}
		d^q \colon C^q(\CU,\CE) &\lra C^{q+1} (\CU,\CE) ,\\
		\big( f_{(i_0,\ldots,i_q)} \big)_{(i_0,\ldots,i_q)} &\lto \bigg( \sum_{k=0}^{q+1} (-1)^k {\rm res}_{I,J}\big(f_{(j_0,\ldots, \widehat{j_k},\ldots,j_{q+1})} \big) \bigg)_{(j_0,\ldots,j_{q+1})} ,
	\end{align*}
	of the \v{C}ech complex.
	Therefore we may restrict $d^q$ to the individual weight spaces
	\begin{equation*}
		d^q_{\ul{\lambda}} \colon C^q(\CU,\CE)_\ul{\lambda} \lra C^{q+1} (\CU,\CE)_\ul{\lambda} \quad\text{, for $\ul{\lambda} \in X(\bT)$.}
	\end{equation*}
	Now let $N\in \BN$ be defined as before \Cref{Prop 2 - Existence of weight changing isomorphisms}, and consider
	\begin{equation*}
		d^q_N \colon \bigoplus_{\ul{\lambda}\in \boundedregion_N} C^q (\CU,\CE)_\ul{\lambda} \lra \bigoplus_{\ul{\lambda}\in \boundedregion_N} C^{q+1} (\CU,\CE)_\ul{\lambda} \,,\quad \big(f_\ul{\lambda}\big)_{\ul{\lambda}\in \boundedregion_N} \lto \big(d^q_\ul{\lambda}(f_\ul{\lambda})\big)_{\ul{\lambda}\in \boundedregion_N}.
	\end{equation*}
	Since this homomorphism takes values in a finite-dimensional $K$-vector space, $d^q_N$ is strict by \cite[I.2.3. Cor.]{Bourbaki87TopVectSp1to5}.
	Hence there exists $R>0$ such that
	\begin{equation}\label{Eq 2 - Condition for restricted differential to be strict}
		B_R(0) \cap \Im(d^q_N) \subset d_N^q \big( B_1(0)\big)
	\end{equation}
	where $B_{R}(0)$ and $B_1(0)$ denote the ``closed'' balls of radius $R$ and $1$ respectively.

	To see that $d^q$ itself is strict, it suffices to show that
	\begin{equation}\label{Eq 2 - Condition for differential to be strict}
		B_R(0) \cap \Im(d^q) \subset d^q \big( B_1(0) \big) 
	\end{equation}
	by scaling.
	For this, let $g \in \Im(d^q) \cap B_R(0)$, and let $g= \sum_{\ul{\lambda}\in X(\bT)} g_\ul{\lambda}$ be a weight decomposition with $g_\ul{\lambda} \in C^{q+1}(\CU,\CE)_\ul{\lambda}$.
	As $\abs{g} = \sup_{\ul{\lambda}\in X(\bT)} \abs{ g_\ul{\lambda}}$, we have $\abs{g_\ul{\lambda}} \leq R$, for all $\ul{\lambda}\in X(\bT)$.
	We moreover have $g_\ul{\lambda}\in \Im(d^q_\ul{\lambda})$ since $d^q$ is $\bT$-equivariant.
	For $\ul{\lambda}\in \boundedregion_N$, we can therefore conclude by \eqref{Eq 2 - Condition for restricted differential to be strict} that there exists $f_\ul{\lambda} \in C^q(\CU,\CE)_\ul{\lambda}$ such that $\abs{f_\ul{\lambda}} \leq 1$ and $d^q_\ul{\lambda} (f_\ul{\lambda}) = g_\ul{\lambda}$.

	For $\ul{\lambda} \notin \boundedregion_N$ in turn, we apply \Cref{Prop 2 - Existence of weight changing isomorphisms} to find $\ul{\nu} \in \boundedregion_N$, $C \in \varepsilon^{\BN_0}$, and a compatible family of isomorphisms
	\begin{equation*}
		\varphi^I_{\ul{\lambda},\ul{\nu}} \colon \CE(U_I)_\ul{\lambda} \overset{\cong}{\lra} \CE(U_I)_\ul{\nu} \quad\text{, for all non-empty $I\subset \{0,\ldots,d\}$,}
	\end{equation*}
	such that $\lvert \varphi^I_{\ul{\lambda},\ul{\nu}}(v) \rvert = C \, \abs{v}$, for all $v\in \CE(U_I)_\ul{\lambda}$.
	We take the direct sum of these isomorphisms to obtain isomorphisms
	\begin{equation*}
		\varphi^q_{\ul{\lambda},\ul{\nu}} \colon C^q(\CU,\CE)_\ul{\lambda} \overset{\cong}{\lra} C^q(\CU,\CE)_\ul{\nu} \quad\text{, for all $q\geq 0$,}
	\end{equation*}
	such that $\lvert \varphi^q_{\ul{\lambda},\ul{\nu}}(v) \rvert = C \,\abs{v}$, for all $v\in C^q(\CU,\CE)_\ul{\lambda}$.
	Moreover, the compatibility of the $\varphi^I_{\ul{\lambda},\ul{\nu}}$ with the canonical restriction maps implies that the diagram 
	\begin{equation*}
		\begin{tikzcd}
			C^q(\CU,\CE)_\ul{\lambda} \ar[r, "d^q_\ul{\lambda}"] \ar[d, "\varphi^q_{\ul{\lambda},\ul{\nu}}"'] & C^{q+1} (\CU,\CE)_\ul{\lambda} \ar[d, "\varphi^{q+1}_{\ul{\lambda},\ul{\nu}}"] \\
			C^q(\CU,\CE)_\ul{\nu} \ar[r, "d^q_{\ul{\nu}}"] & C^{q+1}(\CU,\CE)_\ul{\nu} 
		\end{tikzcd}
	\end{equation*}
	commutes.
	Since $\varepsilon \in \abs{\widebar{K}}$, we may assume that $\varepsilon = \abs{a}$, for some $a\in K$, after an extension of scalars by a finite extension of $K$.
	Then by scaling \eqref{Eq 2 - Condition for restricted differential to be strict} by an appropriate power of $a$ and restricting to the vectors which are rational with respect to the original $K$, we see that
	\begin{equation*}
		B_{CR} (0) \cap \Im(d^q_N) \subset d^q_N\big( B_C (0) \big) .
	\end{equation*}
	Because we have $\lvert \varphi^{q+1}_{\ul{\lambda},\ul{\nu}} (g_\ul{\lambda}) \rvert = C \, \abs{g_\ul{\lambda}} \leq CR$, we find $h\in C^q(\CU,\CE)_\ul{\nu}$ such that $\abs{h}\leq C$ and $d^q_\ul{\nu} (h) = \varphi^{q+1}_{\ul{\lambda},\ul{\nu}} (g_\ul{\lambda})$.
	Then $f_\ul{\lambda} \defeq \big(\varphi^q_{\ul{\lambda},\ul{\nu}}\big)^{-1} (h)$ satisfies $d^q_\ul{\lambda} (f_\ul{\lambda}) = g_\ul{\lambda}$ and $\abs{f_\ul{\lambda}}  = C^{-1} \abs{h} \leq 1$.

	In total we obtain, for every $\ul{\lambda}\in X(\bT)$, an element $f_\ul{\lambda} \in C^q(\CU,\CE)_\ul{\lambda}$ such that $d^q_\ul{\lambda} (f_\ul{\lambda}) = g_\ul{\lambda}$ and $\abs{f_\ul{\lambda}} \leq 1$.
	Taking the sum 
	\begin{equation*}
		f \defeq \sum_{\ul{\lambda}\in X(\bT)} f_\ul{\lambda}
	\end{equation*}
	we find that $d^q(f) = g$ and that $\abs{f} = \sup_{\ul{\lambda}\in X(\bT)} \abs{f_\ul{\lambda}} \leq 1$.
	This shows \eqref{Eq 2 - Condition for differential to be strict} and therefore the strictness of $d^q$.
\end{proof}

\subsection{Local Cohomology with respect to Tubes around Schubert Varieties}

We begin by fixing some notation.
For $\sd \in \{0,\ldots,d-1\}$ and $0<\varepsilon < 1$ with $\varepsilon \in \abs{\widebar{K}}$, we set
\begin{equation*}
	\widetilde{H}^i_{\BP_K^\sd (\varepsilon)} (\BP_K^d, \CE) 
	\defeq \Ker \Big( H^i_{\BP_K^\sd (\varepsilon)} (\BP_K^d, \CE) \lra H^i (\BP_K^d, \CE) \Big)
\end{equation*}
and endow it with the subspace topology.
We define $\widetilde{H}^i_{(\BP_K^\sd )^\rig} (\BP_K^d, \CE) $, $\widetilde{H}^i_{\BP_K^\sd (\varepsilon)^-} (\BP_K^d, \CE) $, and $\widetilde{H}^i_{\BP_K^\sd } (\BP_K^d, \CE) $ analogously.

\begin{lemma}\label{Lemma 2 - Compactness of transition maps}
	Let $0<\varepsilon<\varepsilon'< 1$ with $\varepsilon,\varepsilon' \in \abs{\widebar{K}}$, and $I \subset \{0,\ldots,d\}$.
	Then the transition map
	\begin{equation*}
		\CE^\rigbun (U_{I,\varepsilon}) \lra \CE^\rigbun(U_{I,\varepsilon'})
	\end{equation*}
	is a compact homomorphism between $K$-Banach spaces with dense image.
\end{lemma}
\begin{proof}
	If $I = \emptyset$, we have $U_{I,\varepsilon} = U_{I,\varepsilon'} = \BP_K^d$ and $\CE^\rigbun(\BP_K^d)$ is a finite dimensional $K$-vector space.
	Therefore, the homomorphism in question is compact by \cite[Lemma 16.4]{Schneider02NonArchFunctAna}.
	
	Now suppose that $I\neq \emptyset$.
	Our method here is similar to the proof of \Cref{Lemma 2 - Description of sections of DHS as projective limit}.
	For $I = \{i_0,\ldots,i_m\}$, we consider the affinoid $K$-algebra
	\begin{align*}
		A_\varepsilon &\defeq K \big\langle \varepsilon T_{(j,i_k)} \,\big\vert\, j=0,\ldots,d, k=0,\ldots,m \text{, with } j\neq i_k \big\rangle \\
		&= \Bigg\{ \sum_{ \ul{\mu} \in \BN^{d(m+1)}_0 } a_{\ul{\mu}} \, \ul{T}^{\ul{\mu}} \in K \llrrbracket{\ul{T}} \,\Bigg\vert\, \abs{a_{\ul{\mu}}} \left( \frac{1}{\varepsilon} \right)^{\abs{\ul{\mu}}} \ra 0 \text{ as $\abs{\ul{\mu}} \ra \infty$} \Bigg\} ,
	\end{align*}
	cf.\ \cite[6.1.5 Thm. 4]{BoschGuentzerRemmert84NonArchAna}.
	Like in the proof of \Cref{Lemma 2 - Density of algebraic sections in analytic sections}, we have $U_{I,\varepsilon} = \Sp ( A_\varepsilon/\Fa A_\varepsilon)$, for the ideal
	\[\Fa = \big( T_{(i_j,i_k)} T_{(i_k,i_j)} -1 \,\big\vert\, j,k=0,\ldots,m \text{, with } j\neq k \big) .\]
	With similar definitions for $U_{I,\varepsilon'}$, we obtain a commutative diagram of affinoid $K$-algebras
	\begin{equation*}
		\begin{tikzcd}
			A_\varepsilon \ar[r,hook, "\psi"] \ar[d,two heads, "\tau"'] & A_{\varepsilon '} \ar[d,two heads,"\tau'"] \\
			A_\varepsilon/\Fa A_\varepsilon \ar[r, "\widebar{\psi}"] & A_{\varepsilon'}/\Fa A_{\varepsilon'}
		\end{tikzcd}
	\end{equation*}
	where the vertical maps are strict epimorphisms, and $\widebar{\psi}$ is the transition homomorphism $\CO^\rigbun(U_{I,\varepsilon}) \ra \CO^\rigbun (U_{I,\varepsilon'})$.
	We claim that $\widebar{\psi}$ is an inner homomorphism of affinoid $K$-algebras in the sense of \cite[Def.\ 2.5.1]{Berkovich90SpecThAnGeom}.
	Indeed we have, for $T =  T_{(j,i_k)} \in A_\varepsilon$,
	\begin{align*}
		\sup_{x\in U_{I,\varepsilon'}} \abs{\widebar{\psi}(\tau(T))(x)} = \sup_{x\in U_{I,\varepsilon'}} \abs{\tau'(\psi(T))(x)} \leq  \abs{\psi(T)}_{\sup} = \abs{\psi(T)}_{A_{\varepsilon'}} = \frac{1}{\varepsilon'} < \frac{1}{\varepsilon}.
	\end{align*}
	It follows from \cite[\S 1 Lemma 5]{SchneiderTeitelbaum02pAdicBoundVal} that $\widebar{\psi}$ is a compact homomorphism of locally convex $K$-vector spaces.

	For a general $\bG$-equivariant vector bundle $\CE$, we again may apply Kiehl's theorem \cite[9.4.3 Thm.\ 3]{BoschGuentzerRemmert84NonArchAna} to find strict epimorphisms $\CO^\rigbun(U_{I,\varepsilon})^{\oplus k} \twoheadrightarrow \CE^\rigbun(U_{I,\varepsilon})$ and ${\CO^\rigbun(U_{I,\varepsilon'})^{\oplus k} \twoheadrightarrow \CE^\rigbun(U_{I,\varepsilon'})}$ which are compatible, for some $k\in \BN_0$.
	Then \Cref{Lemma A1 - Generalities on compact maps} (ii) and (iii) applied to 
	\begin{equation*}
		\begin{tikzcd}
			\CO^\rigbun(U_{I,\varepsilon})^{\oplus k} \ar[r, "\widebar{\psi}^{\oplus k}"] \ar[d, two heads] & \CO^\rigbun(U_{I,\varepsilon'})^{\oplus k} \ar[d, two heads] \\
			\CE^\rigbun(U_{I,\varepsilon}) \ar[r] & \CE^\rigbun(U_{I,\varepsilon'})
		\end{tikzcd}
	\end{equation*}
	show that $ \CE^\rigbun (U_{I,\varepsilon}) \ra \CE^\rigbun(U_{I,\varepsilon'})$ is compact.
	
	Finally it follows from \Cref{Lemma 2 - Density of algebraic sections in analytic sections} that the image of the transition map is dense.
\end{proof}

\begin{proposition}\label{Prop 2 - Projective limit description of local cohomology wrt open tubes}
	Let $0< \varepsilon < 1$ with $\varepsilon \in \abs{\widebar{K}}$.
	Let $(\varepsilon_m)_{m\in \BN}$ be some strictly decreasing sequence with $\varepsilon_m\in \abs{\widebar{K}}$, $\varepsilon < \varepsilon_m < 1$, and $\varepsilon_m \ra \varepsilon$.
	%	(Take for example $\big(\varepsilon  \abs{\unif}^{-\frac{1}{n+N}}\big)_{n\in \BN} $ for $N\geq 1$ large enough.
	Then the $K$-Fr\'echet space $H^i_{\BP_K^\sd (\varepsilon)} (\BP_K^d, \CE^\rigbun) $ is the projective limit of the $K$-Banach spaces
	\begin{equation}\label{Eq 2 - Projective limit of local cohomology groups}
		H^i_{\BP_K^\sd (\varepsilon)} (\BP_K^d, \CE^\rigbun)  \cong \varprojlim_{\varepsilon_m \searrow \varepsilon} H^i_{\BP_K^\sd (\varepsilon_m)^-} (\BP_K^d, \CE^\rigbun) .
	\end{equation}
	The transition homomorphisms
	\begin{equation*}
		H^i_{\BP_K^\sd (\varepsilon_m)^-} (\BP_K^d, \CE^\rigbun) \lra H^i_{\BP_K^\sd (\varepsilon_n)^-} (\BP_K^d, \CE^\rigbun) \quad\text{, for $\varepsilon_m < \varepsilon_n$,}
	\end{equation*}
	which are induced from the inclusion $\BP^\sd_K(\varepsilon_m)^- \subset \BP^\sd_K(\varepsilon_n)^-$ are compact and have dense image so that $H^i_{\BP_K^\sd (\varepsilon)} (\BP_K^d, \CE^\rigbun)$ is nuclear.
	Moreover, the local cohomology group $H^i_{\BP_K^\sd} (\BP_K^d, \CE)$ of the algebraic variety $\BP_K^d$ constitutes a dense subspace of $H^i_{\BP_K^\sd (\varepsilon)} (\BP_K^d, \CE^\rigbun)$.
	
	The analogous statements for $\widetilde{H}^i_{\BP_K^\sd (\varepsilon)} (\BP_K^d, \CE^\rigbun)$ hold as well.
\end{proposition}
\begin{proof}
	We first focus on the $K$-Banach spaces associated with the ``closed'' neighbourhoods of $\BP_K^\sd$.	
	For now we fix $0 < \varepsilon <1$ with $\varepsilon \in \abs{\widebar{K}}$ and consider the \v{C}ech complex $C^\bullet (\CU_{\varepsilon},\CE)$ associated with the covering \eqref{Eq 2 - Covering of complement of the closed tubes} which computes the cohomology groups $H^i (\BP^d_K \setminus \BP_K^\sd (\varepsilon)^-,\CE)$.
	For this \v{C}ech complex, we let 
	\begin{align*}
		Z^i (\CU_{\varepsilon}, \CE^\rigbun) &\defeq \Ker \big( C^i (\CU_{\varepsilon}, \CE^\rigbun) \overset{d^i_{\varepsilon}}{\lra} C^{i+1} (\CU_{\varepsilon},\CE^\rigbun) \big), \\
		B^i (\CU_{\varepsilon}, \CE^\rigbun) &\defeq \Im \big( C^{i-1} (\CU_{\varepsilon},\CE^\rigbun) \overset{d^{i-1}_{\varepsilon}}{\lra} C^i (\CU_{\varepsilon},\CE^\rigbun) \big) 
	\end{align*}
	denote the space of $i$-th \v{C}ech cocycles respectively the space of $i$-th \v{C}ech coboundaries.
	We use the analogous notation for further \v{C}ech complexes that will occur.

	Recall from \Cref{Lemma 2 - Density of algebraic sections in analytic sections} that the cochains $C^i(\CU_{\varepsilon},\CE)$ have the structure of a $K$-Banach space, and the cochains $C^i (\CU,\CE^\algbun)$ of algebraic sections constitute a dense subspace therein.
	As a first step we want to see that the induced locally convex topology on $H^i (\BP^d_K \setminus \BP_K^\sd (\varepsilon)^-,\CE)$, $\widetilde{H}^i_{\BP_K^\sd (\varepsilon)^-} (\BP^d_K ,\CE)$, and $H^i_{\BP_K^\sd (\varepsilon)^-} (\BP^d_K ,\CE)$ gives those spaces the structure of $K$-Banach spaces as well, and that their ``algebraic'' counterparts $H^i (\BP^d_K \setminus \BP_K^\sd,\CE^\algbun)$, $\widetilde{H}^i_{\BP_K^\sd} (\BP^d_K ,\CE^\algbun)$, and $H^i_{\BP_K^\sd} (\BP^d_K ,\CE^\algbun)$ embed as dense subspaces respectively.

	To this end we consider $C^i (\CU,\CE^\algbun)$ with the subspace topology coming from $C^i(\CU_{\varepsilon},\CE)$ so that $C^i(\CU_{\varepsilon},\CE)$ is the completion of $C^i (\CU,\CE^\algbun)$.
	We will endow all ``algebraic'' terms with the topologies induced from $C^i (\CU,\CE^\algbun)$, and all ``analytic'' terms with the ones induced from $C^i(\CU_{\varepsilon},\CE)$.

	From \Cref{Thm 2 - Differentials of algebraic Cech complex are strict} we know that the differential $d^i \colon C^i (\CU,\CE^\algbun) \ra C^{i+1} (\CU,\CE^\algbun)$ is strict, and the differential $d^i_{\varepsilon} \colon C^i(\CU_{\varepsilon},\CE) \ra C^{i+1}(\CU_{\varepsilon},\CE)$ is the completion of $d^i$.
	Therefore \cite[1.1.9 Prop.\ 5]{BoschGuentzerRemmert84NonArchAna} implies that $Z^i(\CU_{\varepsilon},\CE)$ is the completion of $Z^i (\CU, \CE^\algbun)$, and $B^{i+1} (\CU_{\varepsilon}, \CE)$ is the completion of $B^{i+1}(\CU, \CE^\algbun)$.
	Moreover note that $d^i_{\varepsilon}$ is strict by \cite[1.1.9 Prop.\ 4]{BoschGuentzerRemmert84NonArchAna}, i.e.\ $B^{i+1} (\CU_{\varepsilon}, \CE)$ is a closed subspace of $C^{i+1}(\CU_{\varepsilon},\CE)$ so that $H^i (\BP^d_K \setminus \BP_K^\sd (\varepsilon)^-,\CE)$ is a $K$-Banach space.

	We also have the short strictly exact sequence
	\begin{equation*}
		0 \lra B^i(\CU,\CE^\algbun) \lra Z^i (\CU, \CE^\algbun) \lra H^i (\BP_K^d \setminus \BP_K^\sd , \CE^\algbun) \lra 0 .
	\end{equation*}
	The completion of this sequence yields the short strictly exact sequence \cite[1.1.9 Cor.\ 6]{BoschGuentzerRemmert84NonArchAna}
	\begin{equation*}
		0 \lra B^{i} (\CU_{\varepsilon}, \CE) \lra Z^i(\CU_{\varepsilon},\CE) \lra H^i (\BP_K^d \setminus \BP_K^\sd , \CE^\algbun)^{\cpltn} \lra 0 .
	\end{equation*}
	Therefore the map $H^i (\BP_K^d \setminus \BP_K^\sd , \CE^\algbun)^\cpltn \ra H^i (\BP^d_K \setminus \BP_K^\sd (\varepsilon)^-,\CE)$ obtained by the universal property of the completion is a topological isomorphism, and $ H^i (\BP^d_K \setminus \BP_K^\sd (\varepsilon)^-,\CE)$ is the completion of $H^i (\BP_K^d \setminus \BP_K^\sd , \CE^\algbun)$.

	Now recall the long exact sequence of local cohomology
	\begin{equation}\label{Eq 2 - Long exact sequence of local cohomology for algebraic cohomology}
		\begin{aligned}
			\ldots \lra  H^{i-1} (\BP^d_K,\CE^\algbun) &\xrightarrow{\alpha^{i-1}}  H^{i-1}(\BP_K^d\setminus \BP^\sd_K,\CE^\algbun)\\
			&\xrightarrow{\partial^{i-1}}  H^i_{\BP_K^\sd} (\BP^d_K, \CE^\algbun) \xrightarrow{\beta^i} H^i(\BP^d_K,\CE^\algbun) \lra \ldots .
		\end{aligned}
	\end{equation}
	We endow the local cohomology group $H^i_{\BP_K^\sd} (\BP^d_K, \CE^\algbun)$ with the locally convex final topology with respect to $\partial^{i-1}$.
	As seen in \Cref{Rmk 2 - Strictness of edge homomorphism in long exact sequence of local cohomology} this makes $\partial^{i-1}$ strict and $\beta^i$ continuous.
	Furthermore, the homomorphisms $\alpha^i$ and $\beta^i$ are strict as well since $H^i (\BP_K^d, \CE^\algbun)$ is a finite-dimensional locally convex $K$-vector space \cite[Cor.\ 3.4.25]{PerezGarciaSchikhof10LocConvSpNonArch}.
	The analogous situation arises for $\BP^\sd_K (\varepsilon)^- \subset \BP_K^d$.

	From this long strictly exact sequence, we obtain the following commutative diagram with strictly exact rows
	\begin{equation}\label{Eq 2 - Diagram of short strictly exact sequences 1}
		\begin{tikzcd}
			0 \ar[r]	&\Ker(\alpha^{i-1}) \ar[r]\ar[d]	&H^{i-1} (\BP_K^d, \CE^\algbun) \ar[r, "\alpha^{i-1}" ]\ar[d, equal]	&\Im(\alpha^{i-1}) \ar[r]\ar[d]	&0	\\
			0 \ar[r]	&\Ker(\alpha^{i-1}_{\varepsilon}) \ar[r]	&H^{i-1} (\BP_K^d, \CE^\rigbun) \ar[r, "\alpha^{i-1}_{\varepsilon}"] \ar[r]	&\Im(\alpha^{i-1}_{\varepsilon}) \ar[r]	& 0.
		\end{tikzcd}
	\end{equation}
	Because $\alpha^{i-1}_{\varepsilon}$ is a strict epimorphism, it follows that $\Im(\alpha^{i-1}) \ra \Im(\alpha^{i-1}_{\varepsilon})$ is a strict epimorphism \cite[Prop.\ 1.1.8]{Schneiders98QuasiAbCat}.

	From \eqref{Eq 2 - Long exact sequence of local cohomology for algebraic cohomology} we moreover obtain the commutative diagram
	\begin{equation}\label{Eq 2 - Diagram of short strictly exact sequences 2}
		\begin{tikzcd}
			0 \ar[r]	&\Im(\alpha^{i-1}) \ar[r]\ar[d]	&H^{i-1} (\BP_K^d \setminus \BP_K^\sd ,\CE^\algbun) \ar[r, "\partial^{i-1}"]\ar[d]	&\widetilde{H}^{i}_{\BP_K^\sd} (\BP_K^d, \CE^\algbun) \ar[r]\ar[d]	&0 \\
			0 \ar[r]	&\Im(\alpha^{i-1}_{\varepsilon}) \ar[r]	&H^{i-1} (\BP_K^d \setminus \BP_K^\sd (\varepsilon)^-,\CE) \ar[r, "\partial^{i-1}_{\varepsilon}"]	&\widetilde{H}^{i}_{\BP_K^\sd(\varepsilon)^-} (\BP_K^d, \CE) \ar[r]	&0
		\end{tikzcd}
	\end{equation}
	with strictly exact rows.
	Here we have used the exactness of \eqref{Eq 2 - Long exact sequence of local cohomology for algebraic cohomology} for the identification $\Ker(\partial^{i-1}) = \Im(\alpha^{i-1})$ and $\Im(\partial^{i-1}) = \Ker(\beta^i)\eqdef \widetilde{H}^{i}_{\BP_K^\sd} (\BP_K^d, \CE^\algbun)$, and likewise for the ``analytic'' terms.
	The composition $\Im(\alpha^{i-1}) \ra H^{i-1} (\BP_K^d \setminus \BP_K^\sd ,\CE^\algbun) \ra H^{i-1} (\BP_K^d \setminus \BP_K^\sd (\varepsilon)^-,\CE)$ is a strict monomorphism \cite[Prop.\ 1.1.7]{Schneiders98QuasiAbCat}.
	Therefore $\Im(\alpha^{i-1}) \ra \Im(\alpha^{i-1}_{\varepsilon})$ is a strict monomorphism as well \cite[Prop.\ 1.1.8]{Schneiders98QuasiAbCat}, and we conclude that $\Im(\alpha^{i-1}) \cong \Im(\alpha_{\varepsilon}^{i-1}) $.
	By taking the completion of the first row of \eqref{Eq 2 - Diagram of short strictly exact sequences 2} and reasoning similarly to before, we find that $\widetilde{H}^{i}_{\BP_K^\sd(\varepsilon)^-} (\BP_K^d, \CE)$ is the completion of $\widetilde{H}^{i}_{\BP_K^\sd} (\BP_K^d, \CE^\algbun)$.

	From \eqref{Eq 2 - Diagram of short strictly exact sequences 1} we also can conclude that $\Ker(\alpha^{i-1}) = \Ker(\alpha^{i-1}_{\varepsilon})$ by applying the appropriate version of the snake lemma \ref{Lemma A1 - Snake lemma}.
	Since $\Im(\beta^i) = \Ker(\alpha^i)$ and $\Im(\beta^i_{\varepsilon}) = \Ker(\alpha^i_{\varepsilon})$ by the exactness of \eqref{Eq 2 - Long exact sequence of local cohomology for algebraic cohomology}, we arrive at the short strictly exact sequence
	\begin{equation*}
		0 \lra \widetilde{H}^{i}_{\BP_K^\sd} (\BP_K^d, \CE^\algbun) \lra H^{i}_{\BP_K^\sd} (\BP_K^d, \CE^\algbun) \lra \Ker(\alpha^i) \lra 0 .
	\end{equation*}
	The completion of this sequence yields the following commutative diagram with strictly exact rows \cite[1.1.9 Prop.\ 4, Cor.\ 6]{BoschGuentzerRemmert84NonArchAna}
	\begin{equation}\label{Eq 2 - Diagram of short strictly exact sequences 3}
		\begin{tikzcd}
			0 \ar[r]	&\widetilde{H}^{i}_{\BP_K^\sd} (\BP_K^d, \CE^\algbun)^\cpltn \ar[r]\ar[d,"\cong"]	&H^{i}_{\BP_K^\sd} (\BP_K^d, \CE^\algbun)^\cpltn \ar[r]\ar[d]	&\Ker(\alpha^{i}) \ar[r]\ar[d,"\cong"]	&0 \\
			0 \ar[r]	&\widetilde{H}^{i}_{\BP_K^\sd(\varepsilon)^-} (\BP_K^d, \CE) \ar[r]	&H^{i}_{\BP_K^\sd(\varepsilon)^-} (\BP_K^d, \CE) \ar[r]	&\Ker(\alpha^i_{\varepsilon}) \ar[r]	&0.
		\end{tikzcd}
	\end{equation}
	Via the snake lemma \ref{Lemma A1 - Snake lemma} it follows that the vertical homomorphism in the middle is a topological isomorphism, too.
	Therefore $H^{i}_{\BP_K^\sd(\varepsilon)^-} (\BP_K^d, \CE)$ is the completion of $H^{i}_{\BP_K^\sd} (\BP_K^d, \CE^\algbun)$.

	Fixing $0<\varepsilon <1$ with $\varepsilon \in \abs{\widebar{K}}$ and a strictly decreasing sequence $(\varepsilon_m)_{m\in \BN}$ as specified, we now consider $\varepsilon < \varepsilon_m <\varepsilon_n <1$ with $\varepsilon_m,\varepsilon_n \in \abs{\widebar{K}}$.
	Using the commutativity of 
	\begin{equation*}
		\begin{tikzcd}[column sep=tiny]
			&\widetilde{H}^{i}_{\BP_K^\sd} (\BP_K^d, \CE^\algbun) \ar[ld]\ar[rd] & \\
			\widetilde{H}^{i}_{\BP_K^\sd(\varepsilon_m)^-} (\BP_K^d, \CE) \ar[rr]&&\widetilde{H}^{i}_{\BP_K^\sd(\varepsilon_n)^-} (\BP_K^d, \CE)
		\end{tikzcd}
	\end{equation*}
	it follows from the statements about the density of the ``algebraic'' subspaces which we have just shown that the transition homomorphism $\widetilde{H}^{i}_{\BP_K^\sd(\varepsilon_m)^-} (\BP_K^d, \CE) \ra \widetilde{H}^{i}_{\BP_K^\sd(\varepsilon_n)^-} (\BP_K^d, \CE)$ has dense image.
	Analogously one argues for $H^{i}_{\BP_K^\sd(\varepsilon_m)^-} (\BP_K^d, \CE) \ra H^{i}_{\BP_K^\sd(\varepsilon_n)^-} (\BP_K^d, \CE)$.

	To show that the transition homomorphisms are compact, we start at the commutative diagram
	\begin{equation*}
		\begin{tikzcd}
			0 \ar[r] &Z^i (\CU_{\varepsilon_m}, \CE^\rigbun) \ar[r] \ar[d]	&C^i (\CU_{\varepsilon_m},\CE^\rigbun) \ar[r, "d^i_{\varepsilon_m}"] \ar[d]	&B^{i+1} (\CU_{\varepsilon_m}, \CE^\rigbun) \ar[r] \ar[d]	&0 \\
			0 \ar[r] &Z^i (\CU_{\varepsilon_n}, \CE^\rigbun) \ar[r] 		&C^i (\CU_{\varepsilon_n},\CE^\rigbun) \ar[r, "d^i_{\varepsilon_n}"] 		&B^{i+1} (\CU_{\varepsilon_n}, \CE^\rigbun) \ar[r]	&0
		\end{tikzcd}
	\end{equation*}
	with strictly exact rows.
	We have seen in \Cref{Lemma 2 - Compactness of transition maps} that the vertical homomorphism in the middle is compact.
	Hence \Cref{Lemma A1 - Generalities on compact maps} (i) implies that $Z^i (\CU_{\varepsilon_m}, \CE^\rigbun) \ra Z^i (\CU_{\varepsilon_n}, \CE^\rigbun)$ is compact.
	Using \Cref{Lemma A1 - Generalities on compact maps} (ii) one argues in the analogous way to conclude that $H^{i-1} (\BP_K^d \setminus \BP_K^\sd (\varepsilon_m)^-,\CE) \ra H^{i-1} (\BP_K^d \setminus \BP_K^\sd (\varepsilon_n)^-,\CE)$ and $\widetilde{H}^{i}_{\BP_K^\sd(\varepsilon_m)^-} (\BP_K^d, \CE) \ra \widetilde{H}^{i}_{\BP_K^\sd(\varepsilon_n)^-} (\BP_K^d, \CE)$ are compact.
	For
	\begin{equation*}
		\begin{tikzcd}
			0 \ar[r]	&\widetilde{H}^{i}_{\BP_K^\sd(\varepsilon_m)^-} (\BP_K^d, \CE) \ar[r]\ar[d]	&H^{i}_{\BP_K^\sd(\varepsilon_m)^-} (\BP_K^d, \CE) \ar[r]\ar[d]	&\Ker(\alpha^i_{\varepsilon_m}) \ar[r]\ar[d]	&0 \\
			0 \ar[r]	&\widetilde{H}^{i}_{\BP_K^\sd(\varepsilon_n)^-} (\BP_K^d, \CE) \ar[r]	&H^{i}_{\BP_K^\sd(\varepsilon_n)^-} (\BP_K^d, \CE) \ar[r]	&\Ker(\alpha^i_{\varepsilon_n}) \ar[r]	&0
		\end{tikzcd}
	\end{equation*}
	the short strictly exact sequences of locally convex $K$-vector spaces in both rows split compatibly as $\Ker(\alpha_{\varepsilon_m}^i) \cong \Ker(\alpha^i_{\varepsilon_n})$ is finite-dimensional \cite[Cor.\ 3.4.27]{PerezGarciaSchikhof10LocConvSpNonArch}.
	Therefore the transition homomorphism $H^{i}_{\BP_K^\sd(\varepsilon_m)^-} (\BP_K^d, \CE) \ra H^{i}_{\BP_K^\sd(\varepsilon_n)^-} (\BP_K^d, \CE)$ is compact by \Cref{Lemma A1 - Generalities on compact maps} (iii).

	\begin{remark}
		Having seen that the transition maps $H^{i}_{\BP_K^\sd(\varepsilon_m)^-} (\BP_K^d, \CE) \ra H^{i}_{\BP_K^\sd(\varepsilon_n)^-} (\BP_K^d, \CE)$ have dense image, we can apply \cite[Prop.\ 1.3.3]{Orlik08EquivVBDrinfeldUpHalfSp} to obtain the topological isomorphism \eqref{Eq 2 - Projective limit of local cohomology groups} at this point. 
		However, we will additionally present the alternative way mentioned in \cite[Rmk.\ 1.3.6]{Orlik08EquivVBDrinfeldUpHalfSp} to do so.
		To this end, one needs the following lemma.
	\end{remark}

	\begin{lemma}[{cf.\ \cite[Lemma 1.3.7]{Orlik08EquivVBDrinfeldUpHalfSp}}]\label{Lemma 2 - Exactness of projective limit when topological ML property}
		Consider a projective system of strictly exact sequences
		\begin{equation*}
			0 \lra V'_n \lra V_n \lra V''_n \lra 0 \quad\text{, for $n\in \BN$,}
		\end{equation*}
		of $K$-Fr\'echet spaces.
		If the transition maps $V'_{n+1}\ra V'_n$, for all $n\in \BN$, have dense image, then the sequence
		\begin{equation}\label{Eq 2 - Projective limit of strictly exact sequences}
			0 \lra \varprojlim_{n\in \BN} V'_n \lra \varprojlim_{n\in \BN} V_n \lra \varprojlim_{n\in \BN} V''_n \lra 0
		\end{equation}
		is strictly exact, too.
	\end{lemma}
	\begin{proof}
		We may view $V'_n$ as the kernel of $V_n \ra V''_n$.
		Since taking the (projective) limit commutes with kernels, we see that $\varprojlim_{n\in \BN} V'_n$ is the kernel of $\varprojlim_{n\in \BN} V_n \ra \varprojlim_{n\in \BN} V''_n$.

		Moreover, as the transition homomorphisms $V'_{n+1} \ra V'_n$ have dense image, the topological Mittag-Leffler condition is fulfilled for this inverse system.
		It follows that \eqref{Eq 2 - Projective limit of strictly exact sequences} is a short exact sequence of vector spaces \cite[13.2.4 (i)]{Grothendieck61EGA3.1}.
		Finally, the open mapping theorem \cite[Prop.\ 8.6]{Schneider02NonArchFunctAna} implies that $\varprojlim_{n\in \BN} V_n \ra \varprojlim_{n\in \BN} V''_n$ is strict, too.
	\end{proof}

	The differential $d^i_\varepsilon \colon C^i (\CU_\varepsilon^-,\CE^\rigbun) \ra C^{i+1} (\CU_{\varepsilon}^-,\CE^\rigbun)$ is the projective limit of the differentials $d^i_{\varepsilon_m} \colon C^i(\CU_{\varepsilon_m},\CE^\rigbun) \ra C^{i+1}(\CU_{\varepsilon_m},\CE^\rigbun) $.
	Therefore we have $C^i(\CU_{\varepsilon}^-,\CE^\rigbun) \cong \varprojlim_{\varepsilon_m \searrow \varepsilon} C^i(\CU_{\varepsilon_{m}},\CE^\rigbun)$ and $Z^i(\CU_{\varepsilon}^-,\CE^\rigbun) \cong \varprojlim_{\varepsilon_m \searrow \varepsilon} Z^i(\CU_{\varepsilon_{m}},\CE^\rigbun)$ because the (projective) limit commutes with taking kernels.
	But by \Cref{Lemma 2 - Exactness of projective limit when topological ML property}, there is the short strictly exact sequence
	\begin{equation*}
		0 \lra \varprojlim_{\varepsilon_m \searrow \varepsilon} Z^i(\CU_{\varepsilon_{m}}, \CE^\rigbun) \lra \varprojlim_{\varepsilon_m \searrow \varepsilon} C^i(\CU_{\varepsilon_m}, \CE^\rigbun) \lra \varprojlim_{\varepsilon_m \searrow \varepsilon} B^{i+1}( \CU_{\varepsilon_m},\CE^\rigbun) \lra 0 
	\end{equation*}
	so that we can conclude $B^{i+1} (\CU_{\varepsilon}^- , \CE^\rigbun) \cong \varprojlim_{\varepsilon_m \searrow \varepsilon} B^{i+1} (\CU_{\varepsilon_m} , \CE^\rigbun) $. 
	In a similar way we obtain $H^{i}( \BP_K^d \setminus \BP_K^\sd (\varepsilon) , \CE^\rigbun) \cong \varprojlim_{\varepsilon_m \searrow \varepsilon} H^{i}( \BP_K^d \setminus \BP_K^\sd (\varepsilon_m)^- , \CE^\rigbun)$.

	Likewise arguing for the homomorphisms $\beta^{i}_{\varepsilon} = \varprojlim_{\varepsilon_m \searrow \varepsilon} \beta^{i}_{\varepsilon_m}$ and $\alpha^i_\varepsilon = \varprojlim_{\varepsilon_m \searrow \varepsilon} \alpha^i_{\varepsilon_m}$, we find that $\widetilde{H}^{i}_{\BP_K^\sd(\varepsilon)} (\BP_K^d, \CE) \cong \varprojlim_{\varepsilon_m \searrow \varepsilon} \widetilde{H}^{i}_{\BP_K^\sd(\varepsilon_m)^-} (\BP_K^d, \CE)$ and $\Ker(\alpha_\varepsilon^i) \cong \varprojlim_{\varepsilon_m \searrow \varepsilon} \Ker(\alpha_{\varepsilon_m}^i) $.
	We now take the projective limit over the projective system
	\begin{equation*}
		0 \lra \widetilde{H}^{i}_{\BP_K^\sd(\varepsilon_m)^-} (\BP_K^d, \CE) \lra H^{i}_{\BP_K^\sd(\varepsilon_m)^-} (\BP_K^d, \CE) \lra \Ker(\alpha^i_{\varepsilon_m}) \lra 0
	\end{equation*}
	of short strictly exact sequences to arrive at the following commutative diagram with strictly exact rows
	\begin{equation*}
		\begin{tikzcd}
			0 \ar[r]	&[-3pt]\varprojlim\limits_{\varepsilon_m \searrow \varepsilon} \widetilde{H}^{i}_{\BP_K^\sd(\varepsilon_m)^-} (\BP_K^d, \CE) \ar[r]\ar[d,"\cong"]	&[-3pt]\varprojlim\limits_{\varepsilon_m \searrow \varepsilon} H^{i}_{\BP_K^\sd(\varepsilon_m)^-} (\BP_K^d, \CE) \ar[r]\ar[d]	&[-3pt]\varprojlim\limits_{\varepsilon_m \searrow \varepsilon} \Ker(\alpha^i_{\varepsilon_m}) \ar[r]\ar[d, "\cong"]	&[-3pt]0 \\
			0 \ar[r]	&\widetilde{H}^{i}_{\BP_K^\sd(\varepsilon)} (\BP_K^d, \CE) \ar[r]	&H^{i}_{\BP_K^\sd(\varepsilon)} (\BP_K^d, \CE) \ar[r] 	&\Ker(\alpha^i_\varepsilon) \ar[r]	&0.
		\end{tikzcd}
	\end{equation*}
	Since the outer vertical maps are topological isomorphisms, it follows from the snake lemma \ref{Lemma A1 - Snake lemma} that $H^{i}_{\BP_K^\sd(\varepsilon)} (\BP_K^d, \CE) \cong \varprojlim_{\varepsilon_m \searrow \varepsilon} H^{i}_{\BP_K^\sd(\varepsilon_m)^-} (\BP_K^d, \CE)$.
\end{proof}

The same reasoning shows the following statement in the extreme case $\varepsilon =0$.

\begin{corollary}\label{Cor 2 - Projective limit description by Banach spaces of rigid local cohomology of Schubert varieties}
	For any strictly decreasing sequence $(\varepsilon_m)_{m\in \BN} \subset  \abs{\widebar{K}}$ with $0 <\varepsilon_m < 1$ and $\varepsilon_m \ra 0$, the $K$-Fr\'echet space $H^i_{(\BP_K^\sd)^\rig} (\BP_K^d, \CE^\rigbun) $ is the projective limit of the $K$-Banach spaces
	\begin{equation*}
		H^i_{(\BP_K^\sd)^\rig} (\BP_K^d, \CE^\rigbun)   \cong \varprojlim_{\varepsilon_m \searrow 0} H^i_{\BP_K^\sd (\varepsilon_m)^-} (\BP_K^d, \CE^\rigbun) .
	\end{equation*}
	with compact transition maps which have dense image.
	Moreover, the local cohomology group $H^i_{\BP_K^\sd} (\BP_K^d, \CE) $ of the algebraic variety $\BP_K^d$ constitutes a dense subspace of $H^i_{(\BP_K^\sd)^\rig} (\BP_K^d, \CE^\rigbun) $.
	
	For $\widetilde{H}^i_{(\BP_K^\sd)^\rig} (\BP_K^d, \CE^\rigbun)$ the analogous assertions hold.
\end{corollary}

There also is a result for the projective limit of the local cohomology groups $H^i_{\BP_K^\sd (\varepsilon)} (\BP_K^d, \CE^\rigbun)$ with respect to the ``open'' $\varepsilon$-neighbourhoods.

\begin{corollary}\label{Cor 2 - Projective limit description of rigid local cohomology of Schubert varieties}
	For any strictly decreasing sequence $(\varepsilon_m)_{m\in \BN} \subset  \abs{\widebar{K}}$ with $0 <\varepsilon_m < 1$ and $\varepsilon_m \ra 0$, there is a topological isomorphism
	\begin{equation}\label{Eq 2 - Projective limit for rigid local cohomology groups}
		H^i_{(\BP_K^\sd)^\rig} (\BP_K^d, \CE^\rigbun) \cong \varprojlim_{\varepsilon_m \searrow 0} H^i_{\BP_K^\sd (\varepsilon_m)} (\BP_K^d, \CE^\rigbun)
	\end{equation}
	of $K$-Fr\'echet spaces.
	Moreover, the transition homomorphisms
	\begin{equation*}
		H^i_{\BP_K^\sd (\varepsilon_m)} (\BP_K^d, \CE^\rigbun) \lra H^i_{\BP_K^\sd (\varepsilon_n)} (\BP_K^d, \CE^\rigbun) \quad\text{, for $\varepsilon_m < \varepsilon_n$,}
	\end{equation*}
	are compact and have dense image.
	
	Again, the analogous statements are true for $\widetilde{H}^i_{(\BP_K^\sd)^\rig} (\BP_K^d, \CE^\rigbun)$.
\end{corollary}
\begin{proof}
	For the topological isomorphism \eqref{Eq 2 - Projective limit for rigid local cohomology groups} one argues analogously to the last part of the proof of the preceeding proposition.
	To show that the topological Mittag-Leffler condition for \Cref{Lemma 2 - Exactness of projective limit when topological ML property} is fulfilled one uses the statement of \Cref{Prop 2 - Projective limit description of local cohomology wrt open tubes} about the density of $H^i_{\BP_K^\sd} (\BP_K^d, \CE) \subset H^i_{\BP_K^\sd (\varepsilon_m)} (\BP_K^d, \CE^\rigbun)$.

	To see that the transition homomorphisms are compact, note that, for $\varepsilon_m < \varepsilon_n$, the transition map factors as
	\begin{equation}\label{Eq 2 - Factorization of transition homomorphism}
		H^i_{\BP_K^\sd (\varepsilon_m)} (\BP_K^d, \CE^\rigbun) \lra H^i_{\BP_K^\sd (\varepsilon_n)^-} (\BP_K^d, \CE^\rigbun) \lra  H^i_{\BP_K^\sd (\varepsilon_n)} (\BP_K^d, \CE^\rigbun) .
	\end{equation}
	The first homomorphism is a continuous linear map from a nuclear locally convex $K$-vector space to a $K$-Banach space and therefore is compact by \cite[Prop.\ 19.5]{Schneider02NonArchFunctAna}.
	It follows from \cite[Rmk.\ 16.7 (i)]{Schneider02NonArchFunctAna} that the composition \eqref{Eq 2 - Factorization of transition homomorphism} is compact as well.

	For $\widetilde{H}^i_{(\BP_K^\sd)^\rig} (\BP_K^d, \CE^\rigbun)$ one argues analogously.
\end{proof}

\begin{proposition}\label{Prop 2 - Vanshing behaviour of local cohomology with respect to Schubert varieties}
	We have the following description of the local cohomology groups
	\begin{equation*}
		H^i_{\BP_K^\sd} (\BP_K^d, \CE) = \begin{cases}
			0												& \text{ , for $i<d-\sd$ or $i>d$,} \\
			H^{d-\sd}_{\BP_K^\sd} (\BP_K^d, \CE)	& \text{ , for $i= d-\sd$,} \\
			H^i (\BP_K^d, \CE)						& \text{ , for $i> d-\sd$.}
		\end{cases}
	\end{equation*}
	For $0 < \varepsilon < 1 $ with $\varepsilon \in \abs{\widebar{K}}$, we have
	\begin{equation*}
		H^i_{\BP_K^\sd(\varepsilon)} (\BP_K^d, \CE^\rigbun) = \begin{cases}
			0															& \text{ , for $i<d-\sd$ or $i>d$,} \\
			H^{d-\sd}_{\BP_K^\sd(\varepsilon)} (\BP_K^d, \CE^\rigbun)	& \text{ , for $i= d-\sd$,} \\
			H^i (\BP_K^d, \CE^\rigbun)									& \text{ , for $i> d-\sd$,}
		\end{cases}
	\end{equation*}
	and similarly for $H^i_{\BP_K^\sd(\varepsilon)^-} (\BP_K^d, \CE^\rigbun)$ and $H^i_{(\BP_K^\sd)^\rig} (\BP_K^d, \CE^\rigbun)$.
\end{proposition}
\begin{proof}
	The statement for $H^i_{\BP_K^\sd} (\BP_K^d, \CE)$ is shown in \cite[pp.\ 595--597]{Orlik08EquivVBDrinfeldUpHalfSp} (and the reasoning there is independent of the field $K$).
	For $H^i_{\BP_K^\sd(\varepsilon)} (\BP_K^d, \CE^\rigbun)$, $H^i_{\BP_K^\sd(\varepsilon)^-} (\BP_K^d, \CE^\rigbun)$, and $H^i_{(\BP_K^\sd)^\rig} (\BP_K^d, \CE^\rigbun)$ the assertion then follows from the density of the ``algebraic'' local cohomology groups.	
\end{proof}

\subsection{Local Cohomology Groups with respect to Schubert Varieties as Locally Analytic Representations}

We let $\bB \subset \bG = \GL_{d+1,K}$ denote the Borel subgroup of lower triangular matrices, and $\bT \subset \bG$ the maximal torus of diagonal matrices.
For $i=0,\ldots,d$, let $\epsilon_i \colon \bT \ra \BG_m$ be the character defined via $\epsilon_i ({\rm diag}(t_0,\ldots,t_d)) = t_i$, and set $\alpha_{i,j} \defeq \epsilon_i - \epsilon_j$, for $i\neq j$, and $\alpha_i \defeq \alpha_{i+1,i}$, for $i=0,\ldots,d-1$.
Then the roots of $\bG$ with respect to $\bT$ are
\[\Phi = \{\alpha_{i,j} \mid 0\leq i \neq j \leq d \} \]
and its simple roots with respect to $\bT \subset \bB$ are
\[\Delta = \{ \alpha_0,\ldots, \alpha_{d-1} \} .\]

Moreover, for $I \subset \Delta$, we let $\bP_I \subset \bG$ denote the (lower) standard parabolic subgroup associated with $I$, i.e.\ the subgroup generated by $\bB$ and the root subgroups $\bU_{-\alpha}$, for $\alpha \in I$.
For example, we have $\bP_{\emptyset} = \bB$ and $\bP_{\Delta} = \bG$.
We write $P_I \defeq \bP_I(K)$ which is a locally $K$-analytic subgroup of $G$.
We set $P_{I,0} \defeq \bP_{I,\BZ}(\CO_K)$ which is a compact open subgroup of $P_I$.
Here $\bP_{I,\BZ}$ denotes the respective standard parabolic subgroup of $\GL_{d+1,\BZ}$.
Additionally, consider the canonical reduction homomorphism
\begin{equation*}
	p_n \colon G_0 \lra \GL_{d+1}(\CO_K/\unif^n),  
\end{equation*}
for $n\in \BN$.
We define
\begin{equation*}
	P^n_{I} \defeq p_n^{-1} \big( \bP_{I,\BZ}(\CO_K/\unif^n) \big) 
\end{equation*}
which is an open compact subgroup of $G_0$ containing $P_{I,0}$.

We now fix $\sd \in \{0,\ldots,d-1\}$.
Then the maximal parabolic subgroup $\bP_{\Delta \setminus \{\alpha_\sd\}}$ stabilizes the subvariety $\BP_K^\sd$ of $\BP_K^d$.
Consequently $P_{\Delta \setminus \{\alpha_\sd\}}$ stabilizes $(\BP_K^\sd)^\rig$ under the group action of $G$ on $(\BP_K^d)^\rig$.

We claim that, for any $n\in \BN$ and $\varepsilon \in \abs{\widebar{K}}$ with $\abs{\unif}^n \leq \varepsilon <1$, the subgroup $P^n_{\Delta\setminus \{\alpha_{\sd}\}} \subset G_0$ stabilizes $\BP^\sd_K(\varepsilon)$.
Indeed, for $[z_0{\,:\,}\ldots{\,:\,}z_d] \in \BP^\sd_K(\varepsilon)$ and $g=(g_{ij}) \in P^n_{\Delta\setminus \{\alpha_{\sd}\}}$, let us write $g^{-1}z \eqdef [w_0{\,:\,}\ldots{\,:\,}w_d]$.
We thus have, for $j=\sd+1,\ldots,d$,
\begin{align*}
	\abs{w_j} = \bigg\lvert \sum_{i=0}^d z_i g_{ij} \bigg\rvert \leq \max \Big(\max_{i=0}^\sd \abs{z_i g_{ij}} , \max_{i=\sd+1}^d \abs{z_i g_{ij}} \Big) 
	\leq \max \Big( \max_{i=0}^\sd  1\cdot \abs{\unif^n}, \max_{i=\sd+1}^d \varepsilon \cdot 1 \Big) \leq \varepsilon .
\end{align*}
Analogously one computes that $\BP^\sd_K(\varepsilon)^-$ is stabilized by $P^n_{\Delta\setminus \{\alpha_{\sd}\}}$, for any $\varepsilon \in \abs{\widebar{K}}$ and $n\in \BN$ with $\abs{\unif}^n < \varepsilon <1$.

\begin{proposition}[{cf.\ \cite[Cor.\ 1.3.9]{Orlik08EquivVBDrinfeldUpHalfSp}}]\label{Prop 2 - Dual of local cohomology wrt open tube is locally analytic representation}
	\begin{altenumerate}
		\item
		Let $\varepsilon \in \abs{\widebar{K}}$ and $n\in \BN$ with $\abs{\unif}^n < \varepsilon < 1$.
		Then the representation
		\begin{equation*}
			P^n_{\Delta\setminus \{\alpha_{\sd}\}} \times \widetilde{H}^i_{\BP^\sd_K(\varepsilon)} (\BP^d_K,\CE^\rigbun)'_b  \lra \widetilde{H}^i_{\BP^\sd_K(\varepsilon)} (\BP^d_K,\CE^\rigbun)'_b  \,,\quad (g,\ell) \lto \ell(g^{-1}\blank),
		\end{equation*}
		is locally analytic.
		Moreover, for any strictly decreasing sequence $(\varepsilon_m)_{m\in \BN} \subset \abs{\widebar{K}}$ with $\varepsilon_m \ra \varepsilon$ and $\varepsilon < \varepsilon_m < 1$, the canonical map
		\begin{equation*}
			\varinjlim_{\varepsilon_m\searrow \varepsilon} \widetilde{H}^i_{\BP^\sd_K(\varepsilon_m)^-} (\BP^d_K,\CE^\rigbun)' \lra \Big( \varprojlim_{\varepsilon_m\searrow \varepsilon} 	\widetilde{H}^i_{\BP^\sd_K(\varepsilon_m)^-} 	(\BP^d_K,\CE^\rigbun) \Big)'_b = \widetilde{H}^i_{\BP^\sd_K(\varepsilon)} (\BP^d_K,\CE^\rigbun)'_b 
		\end{equation*}
		is a topological isomorphism, and $\widetilde{H}^i_{\BP^\sd_K(\varepsilon)} (\BP^d_K,\CE^\rigbun)'_b$ is of compact type this way, i.e.\ it is the inductive limit of the $K$-Banach spaces
		\begin{equation*}
			\begin{tikzcd}[column sep=small]
				\widetilde{H}^i_{\BP^\sd_K(\varepsilon_1)^-}(\BP^d_K,\CE^\rigbun)'  \ar[r, hook] & \ldots \ar[r,hook] & \widetilde{H}^i_{\BP^\sd_K(\varepsilon_m)^-}(\BP^d_K,\CE^\rigbun)' \ar[r, hook] & 	\widetilde{H}^i_{\BP^\sd_K(\varepsilon_{m+1})^-} (\BP^d_K,\CE^\rigbun)' \ar[r,hook] & \ldots
			\end{tikzcd}
		\end{equation*}	
		with compact, injective transition homomorphisms.
		\item
		For the extreme case $\varepsilon = 0$, the representation 
		\begin{equation*}
			P_{\Delta\setminus \{\alpha_{\sd}\}} \times \widetilde{H}^i_{(\BP^\sd_K)^\rig} (\BP^d_K,\CE^\rigbun)'_b  \lra \widetilde{H}^i_{(\BP^\sd_K)^\rig} (\BP^d_K,\CE^\rigbun)'_b  \,,\quad (g,\ell) \lto \ell(g^{-1}\blank),
		\end{equation*}
		is locally analytic, and the underlying locally convex $K$-vector space $\widetilde{H}^i_{(\BP^\sd_K)^\rig} (\BP^d_K,\CE^\rigbun)'_b$ is of compact type analogously to (i).
	\end{altenumerate}
\end{proposition}

Like in \Cref{Sect 2 - Coherent Cohomology of the Drinfeld Upper Half Space}, we proceed step by step.

\begin{lemma}\label{Lemma 2 - Continuity of action on local cohomology wrt tubes}
	\begin{altenumerate}
		\item
		Let $0<\varepsilon <1$ with $\varepsilon \in \abs{\widebar{K}}$.
		For $n\in \BN$ with $\abs{\unif}^n \leq \varepsilon$ (respectively, $\abs{\unif}^n < \varepsilon$), the group $P^n_{\Delta\setminus \{\alpha_{\sd}\}}$ acts on $H^i_{\BP^\sd_K (\varepsilon)} (\BP^d_K,\CE^\rigbun)$ (respectively, on $H^i_{\BP^\sd_K (\varepsilon)^-} (\BP^d_K,\CE^\rigbun)$) by continuous endomorphisms, and the topological isomorphism \eqref{Eq 2 - Projective limit of local cohomology groups} is $P^n_{\Delta\setminus \{\alpha_{\sd}\}}$-equivariant.
		Moreover, the analogous assertions are true for $\widetilde{H}^i_{\BP^\sd_K (\varepsilon)} (\BP^d_K,\CE^\rigbun)$ and $\widetilde{H}^i_{\BP^\sd_K (\varepsilon)^-} (\BP^d_K,\CE^\rigbun)$.
		\item
		In the extreme case $\varepsilon = 0$, $P_{\Delta\setminus \{\alpha_{\sd}\}}$ acts on $H^i_{(\BP^\sd_K)^\rig} (\BP^d_K,\CE^\rigbun)$ and $\widetilde{H}^i_{(\BP^\sd_K)^\rig} (\BP^d_K,\CE^\rigbun)$ by continuous endomorphisms, and the topological isomorphism \eqref{Eq 2 - Projective limit of local cohomology groups} is $P_{\Delta\setminus \{\alpha_{\sd}\},0}$-equivariant.
	\end{altenumerate}
\end{lemma}
\begin{proof}
	For the local cohomology groups this follows from the discussion in \Cref{Sect 2 - Coherent Cohomology of Equivariant Vector Bundles}.
	The fact that the long exact sequence of local cohomology is equivariant for the respective group actions, implies the assertion for the kernels $\widetilde{H}^i_{\BP^\sd_K (\varepsilon)} (\BP^d_K,\CE^\rigbun)$, $\widetilde{H}^i_{\BP^\sd_K (\varepsilon)^-} (\BP^d_K,\CE^\rigbun)$, and $\widetilde{H}^i_{(\BP^\sd_K)^\rig} (\BP^d_K,\CE^\rigbun)$.
\end{proof}

\begin{lemma}\label{Lemma 2 - Local cohomology wrt closed tube is locally analytic representation}
	For $n\in \BN$ and $\varepsilon \in \abs{\widebar{K}}$ with $\abs{\unif}^n < \varepsilon < 1$, the representation 
	\begin{equation*}
		P^n_{\Delta\setminus \{\alpha_{\sd}\}} \times \widetilde{H}^i_{\BP^\sd_K(\varepsilon)^-} (\BP^d_K,\CE^\rigbun) \lra  \widetilde{H}^i_{\BP^\sd_K(\varepsilon)^-} (\BP^d_K,\CE^\rigbun) \,,\quad (g,v) \lto g.v,
	\end{equation*}
	on the $K$-Banach space $\widetilde{H}^i_{\BP^\sd_K(\varepsilon)^-} (\BP^d_K,\CE^\rigbun)$ is locally analytic.
\end{lemma}
\begin{proof}
	Like in \Cref{Lemma 2 - Sections on X_n are locally analytic representation} the only assertion left to show is that the orbit maps 
	\[P^n_{\Delta\setminus \{\alpha_{\sd}\}} \lra \widetilde{H}^i_{\BP^\sd_K(\varepsilon)^-} (\BP^d_K,\CE^\rigbun)\,,\quad g\lto g.v , \]
	are locally analytic, for every $v\in \widetilde{H}^i_{\BP^\sd_K(\varepsilon)^-} (\BP^d_K,\CE^\rigbun)$.
	To do so, we first show that the orbit maps for the $P^n_{\Delta\setminus \{\alpha_{\sd}\}}$-action on $H^i(\BP^d_K\setminus \BP^\sd_K(\varepsilon)^-,\CE^\rigbun)$ are locally analytic.
	
	Fix $g\in P^n_{\Delta\setminus \{\alpha_{\sd}\}}$ and consider the affinoid subdomain $g D_n \subset \GL_{d+1}(C)$, where we set $D_n \defeq 1+\unif^n M_{d+1}(\CO_C)$, with the rigid analytic chart
	\begin{equation*}
		\iota_g \colon D_{n}  \lra gD_{n} \,,\quad h \lto gh.
	\end{equation*}
	Moreover, we claim that, for every fixed $v\in \widetilde{H}^i_{\BP^\sd_K(\varepsilon)^-} (\BP^d_K,\CE^\rigbun)$, the orbit map 
	\begin{equation*}
		P^n_{\Delta\setminus \{\alpha_{\sd}\}} \lra H^i(\BP^d_K\setminus \BP^\sd_K(\varepsilon)^-,\CE^\rigbun)\,,\quad h \lto h.v,
	\end{equation*}
	restricted to $gD_n(K)$ is given by a convergent power series.

	To this end, we consider the admissible covering $\CU_{\varepsilon}$ of $\BP_K^d \setminus \BP^\sd_K (\varepsilon)^-$ whose \v{C}ech complex $C^i(\CU_\varepsilon,\CE^\rigbun)$ computes $H^i(\BP^d_K\setminus \BP^\sd_K(\varepsilon)^-,\CE^\rigbun)$.
	We fix a non-empty subset $I\subset \{\sd+1,\ldots,d\}$ and write $U \defeq U_{I,\varepsilon}$ with the notation from \eqref{Eq 2 - Intersection of open subsets of the covering of the complement of the closed tubes}.
	We then have $h(U) = g(U)$, for all $h\in gD_n$. 
	Indeed, let $z=[z_0{\,:\,}\ldots{\,:\,}z_d] \in U$ and $(1+h') \in D_{n}$ so that 
	\[ [z_0{\,:\,}\ldots{\,:\,}z_d] \cdot (1+h') \eqdef [w_0{\,:\,}\ldots{\,:\,}w_d] \quad\text{, with $w_i = z_i + \sum_{j=0}^d z_j h'_{ji}$.} \]
	For $i\in I$, we have 
	\begin{align*}
		\Big\lvert \sum_{j=0}^d z_j h'_{ji} \Big\rvert \leq \max_{j=0}^d \abs{z_j h'_{ji}} \leq \max_{j=0}^d \abs{z_j} \, \abs{\unif}^n < \max_{j=0}^d \,\varepsilon \, \abs{z_j} \leq \abs{z_i}
	\end{align*}
	where the last inequality holds because $z \in U \subset U_{i,\varepsilon}$.
	This implies $\abs{w_i} =\abs{z_i}$.
	Now we compute, for $i\in I$, $ j\in \{0,\ldots,d \}$:
	\begin{align*}
		\varepsilon \,\abs{w_j} &\leq \max \Big( \varepsilon \,\abs{z_j} , \max_{k=0}^{d} \,\varepsilon \,\abs{z_k h'_{kj}} \Big) 
		\leq \max \Big( \abs{z_i} , \max_{k=0}^d \,\varepsilon \,\abs{\unif}^n\, \abs{z_k} \Big) 
		\leq \abs{z_i} = \abs{w_i}
	\end{align*}
	which implies that $(1+h')^{-1}.z \in U$.

	Since the above non-empty subset $I\subset \{\sd +1,\ldots, d\}$ was arbitrary, we obtain a map
	\begin{equation*}
		gD_n(K) \times C^i(g(\CU_\varepsilon),\CE^\rigbun) \lra C^i(\CU_\varepsilon,\CE^\rigbun)
	\end{equation*}
	that affords the $P^n_{\Delta\setminus \{\alpha_{\sd}\}}$-action on $H^i(\BP^d_K\setminus \BP^\sd_K(\varepsilon)^-,\CE^\rigbun)$ restricted to $gD_n(K)$, cf.\ \eqref{Eq 2 - Group action homomorphism of Cech complexes}.
	Here $g(\CU_{\varepsilon})$ denotes the translated covering $\BP_K^d \setminus \BP^\sd_K (\varepsilon)^- = \bigcup_{i=\sd+1}^d g(U_{i,\varepsilon})$.
	Consequently it suffices to show that, for every $v\in \CE(g(U))$ with $U \defeq U_{I,\varepsilon}$, for non-empty $I\subset \{\sd+1,\ldots,d\}$, the map $gD_n(K) \ra \CE(U)$, $h \mto h.v$, is given by a convergent power series.

	For this we proceed similarly to the proof or \Cref{Lemma 2 - Sections on X_n are locally analytic representation}.
	Using that $h(U)=g(U)$, for all $h\in gD_{n}$, the group action $\sigma\colon \GL_{d+1,K} \times_K \BP_K^d \ra \BP_K^d$ induces the following commutative diagram:
	\begin{equation*}
		\begin{tikzcd}
			D_{n} \times_K U \arrow[r, "\iota_g \times \mathrm{id}"] \arrow[rd, "\pr_2"'] & g D_n \times_K U \arrow[r, "\sigma"] \arrow[d, "\pr_2"] & g(U) \\
			&U&.
		\end{tikzcd}
	\end{equation*}
	Let $F_v \in \CE(U)\langle T_1,\ldots, T_{(d+1)^2} \rangle$ be the power series to which $v \in \CE(g(U))$ is mapped under
	\begin{equation*}
		\begin{tikzcd}[,/tikz/column 3/.append style={anchor=base west}]
			\CE(g(U)) \ar[r] & (\iota_g \times \id)^\ast \sigma^\ast \CE(D_n \times_K U) \ar[d, "(\iota_g \times \mathrm{id})^{\ast} \Phi (D_n \times_K U)"] &[-30pt]\\
			& (\iota_g \times \id)^\ast \pr_2^\ast \CE(D_n \times_K U) &[-30pt] \cong \pr_2^\ast \CE(D_n \times_K U) \\[-20pt]
			&&[-30pt]	\cong \big( \CO(D_n) \cotimes{K} \CO(U) \big) \otimes_{\CO(U)} \CE(U) \\[-20pt]
			&&[-30pt] \cong \CE(U) \langle T_1,\ldots,T_{(d+1)^2} \rangle  .
		\end{tikzcd}
	\end{equation*}
	Now consider, for $h \in D_n(K)$,
	\begin{equation*}
		\begin{tikzcd}
			D_n \times_K U \arrow[r, "\iota_g \times \id"] & g D_n \times_K U \arrow[r, "\sigma"] & g(U) \\
			U \arrow[u, "h \times \id"]  \arrow[urr,  "gh"', end anchor = 210] && .
		\end{tikzcd}
	\end{equation*}
	In terms of $K$-affinoid algebras the morphism $h \times \id\colon U \ra D_n \times_K U$ is given by the evaluation homomorphism of power series
	\begin{align*}
		\mathrm{ev}_h \colon \CO(D_n) \cotimes{K} \CO(U) \cong \CO(U) \langle T_1,\ldots,T_{(d+1)^2} \rangle &\lra \CO(U) , \\
		F &\lto F(h) .
	\end{align*}
	Hence we arrive at the commutative diagram 
	\begin{equation*}
		\begin{tikzcd}
			\CE(g(U)) \arrow[r] \arrow[rd, end anchor = 170] & (\iota_g\times\id)^\ast \sigma^\ast \CE(D_n\times_K U) \arrow[d, "(h\times\id)^\ast"] \arrow[r] & \CE(U)\langle T_1,\ldots,T_{(d+1)^2} \rangle \arrow[d, "\mathrm{ev}_h"]\\
			& (gh)^\ast \CE(U) \arrow[r, "\Phi_{gh}(U)"] &\CE(U) 
		\end{tikzcd}
	\end{equation*}
	which shows that $gh.v = F_v(h)$.

	Having seen that the $P^n_{\Delta\setminus\{\alpha_r\}}$-representation $H^i(\BP^d_K\setminus \BP^\sd_K(\varepsilon)^-,\CE^\rigbun)$ is locally analytic, we now consider the long exact sequence of local cohomology 
	\begin{equation*}
		\ldots \lra H^{i-1}(\BP^d_K\setminus \BP^\sd_K(\varepsilon)^-,\CE^\rigbun) \xrightarrow{\partial^{i-1}_\varepsilon} H^i_{\BP_K^\sd (\varepsilon)^-} (\BP_K^d, \CE^\rigbun) \lra H^i (\BP_K^d, \CE^\rigbun) \lra \ldots .
	\end{equation*}
	Since this sequence is $P^n_{\Delta\setminus\{\alpha_r\}}$-equivariant, \Cref{Prop 1 - Subrepresentations and quotientrepresentations of locally analytic representations} (ii) implies that the kernel $\widetilde{H}^i_{\BP_K^\sd (\varepsilon)^-} (\BP_K^d, \CE^\rigbun)$ is a locally analytic representation, too.
\end{proof}

\begin{proof}[{Proof of \Cref{Prop 2 - Dual of local cohomology wrt open tube is locally analytic representation}}]
	We argue similarly to the proof of \Cref{Prop 2 - Dual of the sections on DHS is locally analytic representation}.
	We have seen in \Cref{Prop 2 - Projective limit description of local cohomology wrt open tubes} that the transition maps $\widetilde{H}^i_{\BP^\sd_K(\varepsilon_{m+1})^-}(\BP^d_K,\CE^\rigbun) \ra	\widetilde{H}^i_{\BP^\sd_K(\varepsilon_{m})^-} (\BP^d_K,\CE^\rigbun)$ are compact and have dense image.
	
	Like in the proof of \Cref{Cor 2 - Dual of sections on X_n are locally analytic representation} one deduces from \Cref{Lemma 2 - Local cohomology wrt closed tube is locally analytic representation} that the contragredient representation on $\widetilde{H}^i_{\BP^\sd_K(\varepsilon_{m})^-}(\BP^d_K,\CE^\rigbun)'$ is locally analytic, too.
	The proposition then follows analogously. 
\end{proof}

\begin{remark}
	The cohomology groups $H^i (\BP_K^d, \CE)$ are finite-dimensional algebraic $\bG$-re\-pre\-sen\-tations.
	It follows from \Cref{Cor 1 - Local analytification functor} that induced homomorphism $G \ra \GL(H^i (\BP_K^d,\CE) ) $ on $K$-valued points is a homomorphism of locally $K$-analytic Lie groups.
	Therefore the $H^i (\BP_K^d, \CE)$ are locally analytic $G$-representations by \Cref{Prop 1 - Equivalent characterization for locally analytic representations on Banach spaces}.
\end{remark}

\vspace{3ex}

\renewcommand{\sd}{{r}}
\renewcommand{\rigbun}{{}}
\renewcommand{\algbun}{{\mathrm{alg}}}

\newcommand{\sdd}{{r}}
\newcommand{\adbun}{{}}

\section{The $\GL_{d+1}(K)$-Representation $H^0(\CX,\CE)$}

Let $K$ be a non-archimedean local field, and $\CE$ a $\bG$-equivariant vector bundle on $\BP_K^d$.
Here we write $\bG = \GL_{d+1,K}$, and $G = \GL_{d+1}(K)$ for its associated locally $K$-analytic Lie group.

\subsection{Orlik's Fundamental Complex and the Associated Spectral Sequence}\label{Sect 3 - The Fundamental Complex}

In this section we want to recapitulate Orlik's method \cite{Orlik08EquivVBDrinfeldUpHalfSp} of using the geometric structure of the divisor at infinity $\CY \defeq \BP^d_K \setminus \CX$ to obtain a filtration by locally analytic $G$-subrepresentations of $H^0(\CX,\CE)'_b$, and to express the respective subquotients as extensions of certain locally analytic $G$-representations. 
Since the reasoning introduced there for a $p$-adic field $K$ carries over to the case of a general non-archimedean local field verbatim, we only present an overview.
At times we give some additional details but at others we refer to \cite{Orlik08EquivVBDrinfeldUpHalfSp} for the full proofs.

The space of global sections $H^0(\CX,\CE)$ that we are interested in relates to the complement $\CY$ via the long exact sequence of local cohomology
\begin{equation}\label{Eq 3 - Long exact sequence of local cohomology for complement of DHS}
	0	\lra H^0(\BP_K^d,\CE) \lra H^0 (\CX,\CE) \lra H^1_\CY (\BP^d_K,\CE)  
	\lra H^1 (\BP_K^d,\CE) \lra 0 .
\end{equation}
Here the higher cohomology groups $H^i(\CX,\CE)$, for $i>0$, vanish as $\CX$ is quasi-Stein.
Because the $H^i(\BP_K^d,\CE)$, for $i=0,1$, are finite-dimensional algebraic $G$-representations, the main difficulty lies in understanding $H^1_\CY (\BP^d_K, \CE)$.

In this regard the strategy of \cite{Orlik08EquivVBDrinfeldUpHalfSp} unfolds as follows.
Let $(\BP_K^d)^\ad$ and $\CX^\ad$ denote the adic spaces attached to $\BP_K^d$ and $\CX$ respectively.
Then one considers the complement
\begin{equation*}
	\CY^\ad \defeq (\BP_K^d)^\ad \setminus \CX^\ad
\end{equation*}
which is a closed pseudo-adic subspace of $(\BP_K^d)^\ad$ by \cite[Lemma 3.2]{Orlik05CohomPerDomRedGrpLocF}, cf.\ \cite[Ch.\ 1.10]{Huber96EtCohomAdSp}.
Since the Zariski topoi of $\CX$ and $\CX^\ad $, and the ones of $\BP_K^d = (\BP_K^d)^\rig$ and $(\BP_K^d)^\ad$ are equivalent (see \cite[Prop.\ 4.5 (i)]{Huber94GenFormalSchRigAnVar}),
it follows that $H^i_{\CY^\ad}\big((\BP_K^d)^\ad, \CE^\adbun \big) = H^i_{\CY}(\BP_K^d,\CE^\rigbun)$, for all $i\geq 0$.

Recall that, for a subset $I$ of the set of simple roots $\Delta = \{\alpha_0,\ldots,\alpha_{d-1} \}$ of $\bG \defeq \GL_{d+1,K}$, we denote the associated (lower) standard parabolic subgroup by $\bP_I$.
We write $P_I \defeq \bP_I(K)$, $P_{I,0} \defeq \bP_{I,\BZ} (\CO_K)$, and $P^n_I \defeq p^{-1}_n (\bP_{I,\BZ}(\CO_K/\unif^n))$ where
\begin{equation*}
	p_n \colon G_0 \defeq \GL_{d+1}(\CO_K) \lra \GL_{d+1} (\CO_K/\unif^n)
\end{equation*}
is the canonical reduction homomorphism.
Here $\bP_{I,\BZ}$ denotes the respective standard parabolic subgroup of $\GL_{d+1,\BZ}$.

For a subset $I \subsetneq \Delta$ with $\Delta\setminus I = \{\alpha_{i_1}, \ldots, \alpha_{i_s} \}$, $i_1{\,<\,}\ldots {\,<\,} i_s$, we define the closed subvariety
\begin{equation*}
	Y_I \defeq \BP_K \bigg( \bigoplus_{j=0}^{i_1} K \cdot e_j \bigg) = \BP_K^{i_1} = V_+ (X_{i_1 +1},\ldots, X_{d}) \subset \BP_K^d
\end{equation*}
so that
\begin{equation*}
	\CY^\ad = \bigcup_{I \subsetneq \Delta} \bigcup_{g \in G/P_I} g Y_I^\ad .
\end{equation*}
Moreover, for a compact open subset $W \subset G/P_I$, we consider
\begin{equation*}
	Z_I^W \defeq \bigcup_{g\in W} g Y_I^\ad
\end{equation*}
which is a closed pseudo-adic subspace of $(\BP_K^d)^\ad$, see \cite[Lemma 3.2]{Orlik05CohomPerDomRedGrpLocF}. 
In particular, we have $\CY^\ad = Z^{G/P_{\Delta\setminus\{\alpha_{d-1}\}}}_{\Delta\setminus\{\alpha_{d-1}\}}$.

Next one defines on $\CY^\ad$ certain \'etale sheaves of locally constant sections supported on $Z_I^{G/P_I}$.
To this end, let
\begin{alignat*}{3}
	\Phi_{I,g} &\colon g Y_I^\ad &&\lra \CY^\ad \\
	\Psi_{I,W} &\colon Z_I^W &&\lra \CY^\ad
\end{alignat*}
be the embeddings of closed pseudo-adic spaces and consider the \'etale sheaves
\begin{alignat*}{3}
	\BZ_{g,I} &\defeq (\Phi_{I,g})_\ast (\Phi_{I,g})^\ast \BZ_{\CY^\ad} \\
	\BZ_{Z^W_I} &\defeq (\Psi_{I,W})_\ast (\Psi_{I,W})^\ast \BZ_{\CY^\ad}
\end{alignat*}
where $\BZ_{\CY^\ad}$ denotes the constant \'etale sheaf on $\CY^\ad$ with stalks equal to $\BZ$.

Furthermore, we let $\CC_{G/P_I}$ denote the category of disjoint coverings of $G/P_I$ by compact open subsets with morphisms given by refinement.
For a covering $c\in \CC_{G/P_I}$ of the form $G/P_I = \bigcup_{j\in A} W_j$, let $\BZ_c$ denote the image of the sheaf homomorphism
\begin{equation*}
	\bigoplus\limits_{j\in A} \BZ_{Z^{W_j}_I} \lhook\joinrel\longrightarrow \prod\limits_{g \in G/P_I} \BZ_{g,I}
\end{equation*}
which is induced by the homomorphisms $\BZ_{Z_I^{W_j}} \ra \BZ_{g,I}$, for $g \in W_j$, cf.\ \cite[p.\ 621]{Orlik08EquivVBDrinfeldUpHalfSp}.
Taking the inductive limit over all coverings of $\CC_{G/P_I}$ one arrives at the aforementioned sheaf
\begin{equation*}
	\varinjlim_{c \in \CC_{G/P_I}} \BZ_c
\end{equation*}
of locally constant sections supported on $Z_{I}^{G/P_I}$ with values in $\BZ$.

With the appropriate sheaf homomorphisms given by restriction, these sheaves fit together to yield a complex of sheaves on $\CY^\ad$
\begin{equation}\label{Eq 3 - Fundamental complex}
	\begin{aligned}
		0 \lra \BZ_{\CY^\ad} \lra \bigoplus_{\substack{I\subset \Delta \\ \abs{\Delta\setminus I}=1 }} \varinjlim_{c\in \CC_{G/P_I}} \BZ_c &\lra \ldots \lra \bigoplus_{\substack{I\subset \Delta \\ \abs{\Delta\setminus I}=i}} \varinjlim_{c\in \CC_{G/P_I}} \BZ_c \lra \ldots \\[0.5cm]
		\ldots &\lra \bigoplus_{\substack{I\subset \Delta \\ \abs{\Delta\setminus I}=d-1}} \varinjlim_{c\in \CC_{G/P_I}} \BZ_c \lra \varinjlim_{c\in \CC_{G/P_\emptyset}} \BZ_c \lra 0 .	
	\end{aligned}	
\end{equation}

\begin{theorem}[{\cite[Thm.\ 2.1.1]{Orlik08EquivVBDrinfeldUpHalfSp}}]
	The complex \eqref{Eq 3 - Fundamental complex} is acyclic.
\end{theorem}

Now, let $0\ra \CE \ra \CI^0 \ra \CI^1 \ra \ldots$ be an injective resolution of the $\CO_{\BP_K^d}$-module $\CE$.
Let $\iota \colon \CY^\ad \hookrightarrow (\BP_K^d)^\ad$ denote the closed embedding.
We want to consider the double complex obtained by applying $\Hom(\iota_\ast(\blank),\CI^q)$ to the acyclic resolution \eqref{Eq 3 - Fundamental complex} of $\BZ_{\CY^\ad}$, i.e.\
\begin{equation}\label{Eq 3 - Double complex for spectral sequence}
	E_0^{p,q} \defeq \begin{cases} \displaystyle
		\Hom \Bigg( \iota_\ast \bigg( \bigoplus\limits_{\substack{I\subset \Delta \\ \abs{\Delta\setminus I}=-p+1}} \varinjlim\limits_{c\in \CC_{G/P_I}} \BZ_c \bigg), \CI^q \Bigg) & \text{, if $-(d-1)\leq p\leq 0$, $q \geq 0$,} \\
		0 & \text{, else.}
	\end{cases}
\end{equation}
This double complex is concentrated in the upper left quadrant.

There is a natural action of $G$ on $E_0^{\bullet,\bullet}$ as follows:
For fixed $g\in G$ and $I \subsetneq \Delta$, we have a homomorphism by functoriality of taking the inverse image under the automorphism $g$
\begin{equation}\label{Eq 3 - Action of G on Hom-sets 1}
	\Hom \bigg( \iota_\ast \Big(\varinjlim\limits_{c\in \CC_{G/P_I}} \BZ_c\Big) , \CI^q \bigg) \lra \Hom \bigg( g^{-1} \iota_\ast \Big( \varinjlim\limits_{c\in \CC_{G/P_I}} \BZ_c \Big) , g^{-1} \CI^q \bigg) .
\end{equation}
Moreover, for the restriction $\widebar{g} \colon \CY^\ad \ra \CY^\ad$ of $g$ we have $\widebar{g}^{-1} \BZ_{c} = \BZ_{g^{-1}c}$ where $g^{-1} c$ denotes the translated covering $G/P_I = \bigcup_{j\in A} g^{-1}W_j$, for $c = \{W_j\mid j\in A\}$.
This yields
\begin{equation*}
	g^{-1} \iota_\ast \Big( \varinjlim\limits_{c\in \CC_{G/P_I}} \BZ_{c} \Big) \cong \iota_\ast \widebar{g}^{-1}\Big( \varinjlim\limits_{c\in \CC_{G/P_I}} \BZ_{c} \Big) \cong \iota_\ast \Big( \varinjlim\limits_{c\in \CC_{G/P_I}} \BZ_{g^{-1}c} \Big).
\end{equation*}
%For the base change $\iota_\ast \widebar{g}^{-1} \cong g^{-1} \iota_\ast$ one uses \cite[Thm.\ 4.4.1]{Huber96EtCohomAdSp} since $\iota$ is proper by \cite[Lemma 1.10.17 (i)]{Huber96EtCohomAdSp}.
Together with the homomorphism $g^{-1} \CI^q \ra g^\ast \CI^q \ra \CI^q$ induced from $\Phi_g\colon g^\ast \CE \ra \CE$ in the second component we obtain a homomorphism
\begin{equation}\label{Eq 3 - Action of G on Hom-sets 2}
	\Hom \bigg( g^{-1} \iota_\ast \Big( \varinjlim\limits_{c\in \CC_{G/P_I}} \BZ_c \Big) , g^{-1} \CI^q \bigg) 
	\lra \Hom \bigg( \iota_\ast \Big( \varinjlim\limits_{c\in \CC_{G/P_I}} \BZ_{g^{-1}c} \Big) , \CI^q \bigg).
\end{equation}
Then $g$ acts via the endomorphism that arises as the composition of \eqref{Eq 3 - Action of G on Hom-sets 1} and \eqref{Eq 3 - Action of G on Hom-sets 2}.

Associated with the double complex \eqref{Eq 3 - Double complex for spectral sequence} we have two spectral sequences $^{\rm h}\!E^{p,q}_r$ and $^{\rm v}\!E^{p,q}_r$.
Since, for every $n\in \BZ$, there are only finitely many pairs $(p,q)\in \BZ^2$ with $p+q=n$ and $E^{p,q}_0 \neq 0$, both spectral sequences converge to the total cohomology of the double complex \cite[\href{https://stacks.math.columbia.edu/tag/0132}{Tag 0132}]{StacksProject}.
As all rows of $^{\rm h}\!E^{p,q}_0$ are exact apart from the entry at $p=0$, we compute
\begin{equation*}
	^{\rm h}\! E_1^{p,q} = \begin{cases}
		\Hom ( \iota_\ast \BZ_{\CY^\ad} , \CI^q ) & \text{, if $p=0$, $q\geq 0$,} \\
		0 & \text{, else.}
	\end{cases}
\end{equation*}
Therefore $^{\rm h}\! E_2^{p,q}$ reads as follows
\begin{equation*}
	^{\rm h}\! E_2^{p,q} = \begin{cases}
		\Ext^q( \iota_\ast \BZ_{\CY^\ad}, \CE ) & \text{, if $p=0$, $q\geq 0$,} \\
		0& \text{, else,}
	\end{cases}
\end{equation*}
and the spectral sequence $^{\rm h}\! E_r^{p,q}$ collapses at the second page.
%The pushforward of an injective $\CI$ from the Zariski topos to the \'etale topos remains injective since the pullback functor between these topoi is exact. Moreover this pushforward is left exact as it is a right adjoint.
Moreover, we have the $G$-equivariant isomorphism
\begin{equation*}
	^{\rm h}\! E_2^{0,q} = \Ext^q (\iota_\ast \BZ_{\CY^\ad}, \CE ) = H^q_{\CY^\ad} \big((\BP^d_K)^\ad, \CE \big) = H^q_\CY (\BP^d_K, \CE) ,
\end{equation*}
for all $q \geq 0$, by \cite[Prop.\ 2.3 bis.]{Grothendieck68SGA2}.

We now turn to the spectral sequence $^{\rm v}\! E_r^{p,q}$ and compute that
\begin{equation*}
	^{\rm v}\! E^{p,q}_1 = \begin{cases} \displaystyle
		\Ext^q \Bigg( \iota_\ast \bigg( \bigoplus\limits_{\substack{I\subset \Delta \\ \abs{\Delta\setminus I}=-p+1}} \varinjlim\limits_{c\in \CC_{G/P_I}} \BZ_c \bigg) , \CE \Bigg) & \text{, if $-(d-1)\leq p\leq 0$, $q \geq 0$,} \\
		0 & \text{, else.}
	\end{cases}
\end{equation*}
Furthermore, for all $I\subsetneq \Delta$, it is shown in \cite[Prop.\ 2.2.1]{Orlik08EquivVBDrinfeldUpHalfSp} that there is an isomorphism
\begin{equation}\label{Eq 3 - Isomorphism between Ext groups and local cohomology}
	\Ext^q \bigg( \iota_\ast \Big( \varinjlim_{c\in \CC_{G/P_I}} \BZ_c \Big) , \CE \bigg) = \varprojlim_{n\in \BN} \bigoplus_{g \in G_0/P_I^n} H^q_{g Y_I(\varepsilon_n)} (\BP_K^d, \CE) ,
\end{equation}
for all $q\geq 0$, with the definition of the ``open'' $\varepsilon_n$-neighbourhood of $Y_I$ from \eqref{Eq 2 - Definition of open tube}.
Here and in the following, we abbreviate $\varepsilon_n \defeq \abs{\unif}^n$, for $n\in \BN$.

\begin{remark}\label{Rmk 3 - G-action on terms of spectral sequence}
	We want to explain how the $G$-action on the $\Ext$-groups on the left hand side of \eqref{Eq 3 - Isomorphism between Ext groups and local cohomology} transfers to the right hand side.
	While doing so, we will introduce some useful notation.
	
	First note that $G/P_I \cong G_0/P_{I,0}$ by the Iwasawa decomposition (see \cite[\S 3.5]{Cartier79ReppAdicGrpsSurvey}).
	The proof of the isomorphism \eqref{Eq 3 - Isomorphism between Ext groups and local cohomology} uses that the family of coverings
	\begin{equation*}
		G/P_I  = \bigcup_{g \in G_0/P^n_I} gP_I^n/P_I \quad\text{, with $ gP_I^n/P_I \defeq  \big\{ gpP_I \in G/P_I \mid p \in P^n_I \big\}$, }
	\end{equation*}
	for $n\in \BN$, is cofinal in $\CC_{G/P_I}$.
	This shows that 
	\begin{equation*}
		\varinjlim_{c\in \CC_{G/P_I}} \BZ_c  = \varinjlim_{n\in \BN} \bigoplus_{g \in G_0/P_I^n} \BZ_{Z_I^{gP^n_I}} .
	\end{equation*}
	One proceeds by applying \cite[Prop.\ 2.3 bis.]{Grothendieck68SGA2} and arguing that certain higher derived inverse limits vanish.

	Then the isomorphism \eqref{Eq 3 - Isomorphism between Ext groups and local cohomology} is $G$-equivariant when the right hand side is equipped with the following $G$-action by continuous endomorphisms:
	For fixed $g\in G$, to give an endomorphism by which $g$ acts it suffices to define compatible homomorphisms
	\begin{equation}\label{Eq 3 - Homomorphism of group action for inverse limit of cohomology groups}
		\varprojlim_{n\in \BN} \bigoplus_{g \in G_0/P_I^n} H^q_{g Y_I(\varepsilon_n)} (\BP_K^d, \CE) \lra \bigoplus_{i=1}^{s_m} H^q_{g_i Y_I(\varepsilon_m)} (\BP_K^d, \CE) ,
	\end{equation}
	for all $m \in \BN$.
	Here the $g_1,\ldots,g_{s_m}$ are some coset representatives of $G_0/P_I^m$ so that we have $G/P_I = \bigcup_{i=1}^{s_m} g_i P_{I}^m/P_I$.

	We choose $n=n(g,m) \geq m$ to be large enough such that the covering $G/P_I = \bigcup_{j=1}^{s_n} h_j P^n_I/P_I$, for coset representatives $h_1,\ldots,h_{s_n}$ of $G_0/P^n_I$, is a refinement of the translated covering $G/P_I = \bigcup_{i=1}^{s_m} g g_i P^m_I/P_I$.
	In this way we obtain a surjection
	\begin{equation*}
		\sigma_g \colon \{1,\ldots, s_{n}\} \lra \{ 1, \ldots, s_m \}
	\end{equation*}
	where $\sigma_g(j)$ is defined via $h_j P^{n}_I/P_I \subset g g_{\sigma_g(j)} P^m_I /P_I$.

	One computes that
	\begin{equation*}
		h_j Y_I (\varepsilon_{n}) = h_j P^{n}_I . Y_I 
		\subset g g_{\sigma_g(j)} P^m_I . ( P_I . Y_I )
		= g g_{\sigma_g(j)} Y_I(\varepsilon_m)
	\end{equation*}
	Then the homomorphisms
	\begin{equation*}
		\varphi_g \colon H^q_{h_j Y_I(\varepsilon_{n})} (\BP_K^d, \CE) \lra H^q_{g_{\sigma_g(j)} Y_I(\varepsilon_m)} (\BP_K^d, \CE)
	\end{equation*}
	from \eqref{Eq 2 - Group action homomorphism for local cohomology} give
	\begin{align*}
		\bigoplus_{j=1}^{s_{n}} H^q_{h_j Y_I(\varepsilon_{n})} (\BP_K^d, \CE) &\lra \bigoplus_{i=1}^{s_m} H^q_{g_i Y_I(\varepsilon_m)} (\BP_K^d, \CE) , \\
		(v_1,\ldots,v_{s_n}) &\lto \Big(\sum_{j\in \sigma_g^{-1}(\{1\})} \varphi_g(v_{j}) ,\ldots, \sum_{j\in \sigma_g^{-1}(\{s_m\})} \varphi_g(v_{j}) \Big) .
	\end{align*}
	Combining this with the projection 
	\begin{equation*}
		\varprojlim_{n\in \BN} \bigoplus_{g \in G_0/P_I^n} H^q_{g Y_I(\varepsilon_n)} (\BP_K^d, \CE) \lra \bigoplus_{j=1}^{s_{n}} H^q_{h_j Y_I(\varepsilon_{n(g,m)})} (\BP_K^d, \CE)
	\end{equation*}
	yields the sought homomorphism \eqref{Eq 3 - Homomorphism of group action for inverse limit of cohomology groups}.
	One checks that these homomorphisms are compatible and do not depend on the choice of $n$.
	\qed
\end{remark}

Note that when $g\in G_0$, the $n$ for the action of $g$ can always be chosen to be $n=m$.
In this case, $g$ even acts on each constituent of projective limit in \eqref{Eq 3 - Isomorphism between Ext groups and local cohomology} individually.

Using the isomorphisms 
\begin{equation*}
	\varphi_{g_i} \colon H^q_{g_i Y_I(\varepsilon_m)} (\BP_K^d, \CE) \lra H^q_{ Y_I(\varepsilon_m)} (\BP_K^d, \CE), 
\end{equation*}
for coset representatives $g_1,\ldots, g_{s_m}$ of $G_0/P_I^m$, we obtain a $G_0$-equivariant isomorphism
\begin{equation*}
	\bigoplus_{i=1}^{s_m} H^q_{g_i Y_I(\varepsilon_m)} (\BP_K^d, \CE) \overset{\cong}{\lra} \Ind_{P_{I}^m}^{G_0} \Big( H^q_{Y_I(\varepsilon_m)} (\BP_K^d, \CE) \Big) \,,\quad
	(v_1,\ldots,v_{s_m}) \lto \sum_{i=1}^{s_m} g_i \bullet \varphi_{g_i} (v_i) .
\end{equation*}
Consequently the spectral sequence reads as follows
\begin{equation*}
	^{\rm v}\! E_1^{p,q} = \bigoplus\limits_{\substack{I\subset \Delta \\ \abs{\Delta\setminus I}=-p+1}} \varprojlim_{n\in \BN} \,\Ind^{G_0}_{P_I^n} \Big( H^q_{Y_I(\varepsilon_n)} (\BP_K^d, \CE) \Big)
	\Rightarrow {}^{\rm v}\!E_\infty^{p+q} = {}^{\rm h}\!E_\infty^{p+q} = H^{p+q}_\CY (\BP_K^d, \CE) .
\end{equation*}

One continues analysing this spectral sequence by considering complexes $K^\bullet_{q,n}$ defined as
\begin{equation*}
	K^p_{q,n} \defeq \bigoplus\limits_{\substack{I\subset \Delta \\ \abs{\Delta\setminus I}=-p+1}} \Ind^{G_0}_{P_I^n} \Big( H^q_{Y_I(\varepsilon_n)} (\BP_K^d, \CE) \Big) ,
\end{equation*}
for $-(d-1)\leq p\leq 0$, $q \geq 0$, so that $^{\rm v}\! E_1^{\bullet,q} = \varprojlim_{n\in \BN} K^\bullet_{q,n}$.
Applying \Cref{Prop 2 - Vanshing behaviour of local cohomology with respect to Schubert varieties} to the local cohomology groups with respect to $Y_I(\varepsilon_n) = \BP_K^{i_1}$, for $I \subsetneq \Delta$ with $\Delta\setminus I = \{\alpha_{i_1},\ldots,\alpha_{i_s} \}$, $i_1{\,<\,}\ldots{\,<\,}i_s$, and $q \geq 0$, we find that
\begin{equation*}
	H^q_{ Y_I(\varepsilon_n)} (\BP_K^d, \CE) = \begin{cases}
		0 & \text{, if $\{\alpha_0,\ldots,\alpha_{d-q-1}\} \nsubset I$,} \\
		H^q_{ \BP_K^{d-q}(\varepsilon_n)} (\BP_K^d, \CE) & \text{, if $\{\alpha_0,\ldots,\alpha_{d-q-1}\}\subset I$ and $\alpha_{d-q}\notin I$,} \\
		H^q (\BP_K^d, \CE) & \text{, if $\{\alpha_0,\ldots,\alpha_{d-q} \} \subset I$.} 
	\end{cases}
\end{equation*}
In particular the complex $K^\bullet_{q,n}$ is concentrated in the degrees $p=-q+1,\ldots, 0$.
%Since for $p<-q+1$ and some $I \subset \Delta$ with $\abs{\Delta \setminus I} = -p +1$, $\Delta \setminus I \eqdef \{\alpha_{i_1},\ldots,\alpha_{i_s} \}$, $i_1<\ldots < i_s$, we have $s = -p+1 >q$. Hence there must exist $i \in \{0,\ldots,d-q-1\}$ such that $\alpha_i \in \Delta \setminus I$. The q-th local cohomology groups therefore vanish for all these $I$.

The next step is to write $K^\bullet_{q,n}$ as an extension of complexes $K'^{\bullet}_{q,n}$ and $K''^{\bullet}_{q,n}$ via the $G$-equivariant, short strictly exact sequence
\begin{equation*}
	\begin{tikzcd}[column sep= 0pt]
		0 \ar[d] &[-5pt] & 0 \ar[d] \\
		K'^p_{q,n} \ar[d] &[-5pt]\defeq\joinrel=\joinrel= & \hspace{-0.6cm}\displaystyle\bigoplus\limits_{\substack{I\subset \Delta \\ \abs{\Delta\setminus I}=-p+1 \\ \alpha_0,\ldots,\alpha_{d-q-1} \in I \\ \alpha_{d-q} \notin I}} \hspace{-0.6cm} \Ind^{G_0}_{P_I^n} \Big( \widetilde{H}^q_{\BP_K^{d-q}(\varepsilon_n)} (\BP_K^d, \CE) \Big) \ar[d, start anchor ={[yshift=4ex]}] \\[-10pt]
		K_{q,n}^{p} \ar[d]&[-5pt] =\joinrel=\joinrel=  & \hspace{-0.6cm}\displaystyle\bigoplus\limits_{\substack{I\subset \Delta \\ \abs{\Delta\setminus I}=-p+1 \\ \alpha_0,\ldots,\alpha_{d-q-1} \in I }} \hspace{-0.6cm} \Ind^{G_0}_{P_I^n} \Big( H^q_{Y_I(\varepsilon_n)} (\BP_K^d, \CE) \Big) \ar[d, start anchor ={[yshift=3ex]}] \\[-10pt]
		K''^p_{q,n} \ar[d] &[-5pt] \defeq\joinrel=\joinrel= & \hspace{-0.6cm}\displaystyle\bigoplus\limits_{\substack{I\subset \Delta \\ \abs{\Delta\setminus I}=-p+1 \\ \alpha_0,\ldots,\alpha_{d-q-1} \in I }} \hspace{-0.6cm} \Ind^{G_0}_{P_I^n} \big( H^q (\BP_K^d, \CE) \big) \ar[d, start anchor ={[yshift=3ex]}] \\[-10pt]
		0 &[-5pt] & 0
	\end{tikzcd}
\end{equation*}
%\begin{equation*}
%\begin{tikzcd}[column sep= 0.5 cm]
%	0 \ar[r] & K'^p_{q,n} \ar[r] \arrow[equal, start anchor ={[yshift=-0.5ex]}]{d}[pos = 0.05, above]{..} & K_{q,n}^{p} \ar[r] \ar[d, equal] & K''^p_{q,n} \ar[r] \arrow[equal, start anchor ={[yshift=-0.5ex]}]{d}[pos = 0.05, above]{..} & 0  \\
%	0 \ar[r] & \hspace{-0.6cm}\displaystyle\bigoplus\limits_{\substack{I\subset \Delta \\ \abs{\Delta\setminus I}=-p+1 \\ \alpha_0,\ldots,\alpha_{d-q-1} \in I \\ \alpha_{d-q} \notin I}} \hspace{-0.6cm} \Ind^{G_0}_{P_I^n} \Big( \widetilde{H}^q_{\BP_K^{d-q}(\varepsilon_n)} (\BP_K^d, \CE) \Big) \ar[r]
%		& \hspace{-0.6cm}\displaystyle\bigoplus\limits_{\substack{I\subset \Delta \\ \abs{\Delta\setminus I}=-p+1 \\ \alpha_0,\ldots,\alpha_{d-q-1} \in I }} \hspace{-0.6cm} \Ind^{G_0}_{P_I^n} \Big( H^q_{Y_I(\varepsilon_n)} (\BP_K^d, \CE) \Big) \ar[r]
%		& \hspace{-0.6cm}\displaystyle\bigoplus\limits_{\substack{I\subset \Delta \\ \abs{\Delta\setminus I}=-p+1 \\ \alpha_0,\ldots,\alpha_{d-q-1} \in I }} \hspace{-0.6cm} \Ind^{G_0}_{P_I^n} \big( H^q (\BP_K^d, \CE) \big) \ar[r] & 0 
%\end{tikzcd}
%\end{equation*}
of $K$-Fr\'echet spaces.
This sequence is induced by the short strictly exact sequences
\begin{equation*}
	0 \lra 0 \lra H^q  (\BP_K^d, \CE) \lra H^q  (\BP_K^d, \CE) \lra 0 ,
\end{equation*}
for $I\subsetneq \Delta$ with $\alpha_0,\ldots, \alpha_{d-q} \in I$, and by
\begin{equation*}
	0 \lra \widetilde{H}^q_{\BP_K^{d-q}(\varepsilon_n)} (\BP_K^d, \CE) \lra H^q_{\BP_K^{d-q}(\varepsilon_n)} (\BP_K^d, \CE) \lra H^q  (\BP_K^d, \CE) \lra 0,
\end{equation*}
for $I\subsetneq \Delta$ with $\alpha_0,\ldots,\alpha_{d-q-1} \in I$, $\alpha_{d-q} \notin I$.
Here the second homomorphism in the latter sequence is surjective as $H^q (\BP_K^d\setminus \BP_K^{d-q}(\varepsilon_n),\CE) =0$.

By \Cref{Prop 2 - Projective limit description of local cohomology wrt open tubes} the projective systems $(K'^{p}_{q,n})_{n\in \BN}$ satisfy the topological Mittag-Leffler condition of \Cref{Lemma 2 - Exactness of projective limit when topological ML property}.
Therefore we obtain the $G_0$-equivariant, short strictly exact sequence 
\begin{equation}\label{Eq 3 - Extension of complexes in projective limit}
	0 \lra \varprojlim_{n\in \BN} K'^{\bullet}_{q,n} \lra {}^{\rm v}\! E_1^{\bullet,q} \lra \varprojlim_{n\in \BN} K''^{\bullet}_{q,n} \lra 0
\end{equation}
of complexes of $K$-Fr\'echet spaces.

One finds that the complexes $K'^{\bullet}_{q,n}$ and $K''^{\bullet}_{q,n}$, for all $q\geq 0$, $n \in \BN$, are acyclic aside from the very left and right position \cite[Lemma 2.2.5]{Orlik08EquivVBDrinfeldUpHalfSp}.
After checking the (topological) Mittag-Leffler conditions one then concludes that the complexes $\varprojlim_{n\in \BN} K'^{\bullet}_{q,n}$ and $\varprojlim_{n\in \BN} K''^{\bullet}_{q,n}$ are acyclic apart from the very left and right position as well \cite[Rmk.\ 2.2.6]{Orlik08EquivVBDrinfeldUpHalfSp}.
To compute $^{\rm v}\! E_2^{p,q}$ we can consider the long exact sequence of the cohomology of the complexes \eqref{Eq 3 - Extension of complexes in projective limit}.
It follows that the only non-vanishing terms of $^{\rm v}\! E_2^{p,q} = H^p \big(\varprojlim_{n\in\BN} K^{\bullet}_{p,n}\big)$ are the ones for $p=-q+1$, $q=1,\ldots,d$, and the ones for $p=0$, $q \geq 2$.
For these terms we obtain a short strictly exact sequence
\begin{equation}\label{Eq 3 - Extension for terms of E_2-page}
	\begin{aligned}
		0 \lra \Ker\Big( \varprojlim_{n\in \BN}  &{\hspace{1.5pt}} K'^{ -q+1}_{q,n} \ra \varprojlim_{n\in \BN} K'^{-q+2}_{q,n} \Big) \lra {}^{\rm v}\! E_2^{-q+1,q} \\
		&\lra \Ker\Big( \varprojlim_{n\in \BN} K''^{ -q+1}_{q,n} \ra \varprojlim_{n\in \BN} K''^{-q+2}_{q,n} \Big) \lra 0 ,
	\end{aligned}
\end{equation}
for $q=1,\ldots,d$, and one shows that
\begin{equation*}
	{}^{\rm v}\! E_2^{0,q} = H^q(\BP_K^d, \CE) ,
\end{equation*}
for $q \geq 2 $.
Moreover, from this one concludes that the spectral sequence $^{\rm v}\! E_r^{p,q} $ degenerates at the $E_2$-page \cite[p.\ 633]{Orlik08EquivVBDrinfeldUpHalfSp}.

We want to investigate the strong dual of \eqref{Eq 3 - Extension for terms of E_2-page}.
We first look at $\varprojlim_{n\in \BN} K''^{p}_{q,n}$.
Since each $K''^p_{q,n}$ is finite-dimensional, the transition homomorphisms of $\big(K''^p_{q,n}\big)_{n\in \BN}$ are compact \cite[Lemma 16.4]{Schneider02NonArchFunctAna} so that $\varprojlim_{n\in \BN} K''^{p}_{q,n}$ is a nuclear $K$-Fr\'echet space \cite[Prop.\ 19.9]{Schneider02NonArchFunctAna}.
Then $\Ker\big( \varprojlim_{n\in \BN} K''^{ -q+1}_{q,n} \ra \varprojlim_{n\in \BN} K''^{-q+2}_{q,n} \big)$ is a nuclear $K$-Fr\'echet space, too \cite[Prop.\ 19.4 (i)]{Schneider02NonArchFunctAna}.
It follows that 
\begin{align*}
	\Ker \Big( \varprojlim_{n\in \BN} K''^{-q+1}_{q,n} \ra \varprojlim_{n\in \BN} K''^{-q+2}_{q,n} \Big)'_b 
	&\cong \Big(\varprojlim_{n\in \BN} \Ker \big( K''^{-q+1}_{q,n} \ra \varprojlim_{n\in \BN} K''^{-q+2}_{q,n} \big) \Big)'_b \\
	&\cong \varinjlim_{n\in \BN} \Ker \big( K''^{-q+1}_{q,n} \ra  K''^{-q+2}_{q,n} \big)'_b
\end{align*}
by \cite[Prop.\ 16.5]{Schneider02NonArchFunctAna}.
%and the fact that the inverse limit (a limit) commutes with taking the kernel (a limit as well).
As $K''^\bullet_{q,n}$ is exact at $K''^{-q+2}_{q,n}$, the image of $K''^{-q+1}_{q,n} \ra  K''^{-q+2}_{q,n}$ is closed and this homomorphism is strict \cite[IV.\ \S 4.2 Thm.\ 1]{Bourbaki87TopVectSp1to5}.
Hence \cite[IV.\ \S 4.1 Prop.\ 2]{Bourbaki87TopVectSp1to5} implies that there is topological isomorphism
\begin{equation*}
	\Ker \big( K''^{-q+1}_{q,n} \ra  K''^{-q+2}_{q,n} \big)'_b
	\cong \Coker\Big( \big(K''^{ -q+2}_{q,n}\big)'_b \ra \big(K''^{-q+1}_{q,n}\big)'_b \Big) .
\end{equation*}

To simplify the notation we write $\bQ_{d-q} \defeq \bP_{I_{d-q}}$ for the standard parabolic subgroup corresponding to the subset $I_{d-q} = \{ \alpha_0, \ldots, \alpha_{d-q-1} \} \subset \Delta$.
%This parabolic subgroup corresponds to the partition $(d-q+1,1,\ldots,1)$.
In Remark \ref{Rmk 1 - Dual of finite induction} (i) we have seen that taking the strong dual and finite induction commute with each other.
Therefore, we obtain a $G$-equivariant, topological isomorphism
\begin{align*}
	\Coker &\Big( \big(K''^{-q+2}_{q,n}\big)'_b \ra \big(K''^{-q+1}_{q,n}\big)'_b \Big) \\
	&\cong \Ind^{G_0}_{Q^n_{d-q}} \big( H^q (\BP_K^d, \CE)' \big) \bigg/ \sum_{i=d-q}^{d-1} \Ind^{G_0}_{P^n_{I_{d-q} \cup \{\alpha_i\}}} \big( H^q (\BP_K^d, \CE)' \big) \\
	&\cong \Big( H^q (\BP_K^d, \CE)' \otimes_K \Ind^{G_0}_{Q^n_{d-q}} (K)  \Big)  \bigg/ \sum_{i=d-q}^{d-1} H^q (\BP_K^d, \CE)' \otimes_K \Ind^{G_0}_{P^n_{I_{d-q} \cup \{\alpha_i\}}} (K) \\
	&\cong H^q (\BP_K^d, \CE)' \otimes_K  v^{G_0}_{Q^n_{d-q}} 
\end{align*}
where
\begin{equation}\label{Eq 3 - Definition of generalized Steinberg representation for finite level}
	v^{G_0}_{Q^n_{d-q}} \defeq \Ind^{G_0}_{Q^n_{d-q}} (K) \bigg/ \sum_{i=d-q}^{d-1} \Ind^{G_0}_{P^n_{I_{d-q} \cup \{\alpha_i\}}} (K) .
\end{equation}
We have used the push-pull formula from Remark \ref{Rmk 1 - Dual of finite induction} (ii), since $H^q (\BP_K^d, \CE)'$ already is a $G$-representation, and that taking the projective tensor product with $H^q (\BP_K^d, \CE)'$ is exact \cite[Lemma 2.1 (ii)]{BreuilHerzig18TowardsFinSlopePartGLn}.

Furthermore, \cite[Prop.\ 1.1.32 (i)]{Emerton17LocAnVect} yields
\begin{align*}
	\varinjlim_{n\in \BN} \Big( H^q (\BP_K^d, \CE)' \otimes_K  v^{G_0}_{Q^n_{d-q}} \Big)
	\cong H^q (\BP_K^d, \CE)' \cotimes{K} \varinjlim_{n\in \BN} \Big( v^{G_0}_{Q^n_{d-q}}  \Big) .
\end{align*}
The second factor of this tensor product can be expressed as a \textit{smooth generalized Steinberg representation} (cf.\ \cite{Casselman74pAdicVanishingThmGarland})
\begin{equation*}
	v^G_{P_I} \defeq \Ind^{\sm,G}_{P_I} (K) \bigg/ \sum_{I\subsetneq J \subset \Delta} \Ind^{\sm,G}_{P_J} (K) ,
\end{equation*}
for $I \subset \Delta$.
Here $\Ind^{\sm,G}_{P_I}(K)$ denotes the smooth induction of the trivial $P_I$-representation, i.e.\ the space locally constant functions invariant under $P_I$ endowed with left regular $G$-action:
\begin{equation*}
	\Ind^{\sm,G}_{P_I}(K) \defeq \big\{ f\in C^\sm (G,K) \,\big\vert\, \forall g\in G , p \in P_I : f(gp) = f(g) \big\} .
\end{equation*}

\begin{lemma}\label{Lemma 3 - Generalized Steinberg representation as inductive limit}
	For $I\subset \Delta$, there is a $G$-equivariant, topological isomorphism
	\begin{equation*}
		\varinjlim_{n \in \BN} \bigg( \Ind^{G_0}_{P^n_{I}} (K) \bigg/ \sum_{I \subsetneq J \subset \Delta} \Ind^{G_0}_{P_{J}^n} (K) \bigg) \cong v^{G}_{P_{I}}
	\end{equation*}
	where $v^G_{P_I}$ denotes the generalized smooth Steinberg representation endowed with the finest locally convex topology.	
\end{lemma}
\begin{proof}
	For all $J\subset \Delta$ and $n\in \BN$, we have a well-defined, $G_0$-equivariant inclusion
	\begin{equation}\label{Eq 3 - Inclusion of smooth inductions}
		\Ind^{G_0}_{P_J^n} (K) \lra \Ind_{P_{J}}^{\sm, G} (K) \,,\quad \sum_{i=1}^{s_n} g_i \bullet \lambda_i \lto \big[ g \mto \lambda_i \text{ , if $gP_J\in g_i P_J^n/P_J$} \big] .
	\end{equation} 
	Here $g_1,\ldots,g_{s_n}$ are coset representatives of $G_0/P^n_J$.
	Since $\Ind^{G_0}_{P_J^n} (K)$ is finite dimensional, this is a continuous homomorphism when the right hand side carries the finest locally convex topology.

	Moreover, given a locally constant function $f\colon G \ra K$ in $\Ind^{\sm, G}_{P_{J}} (K) $, there exists $n\in \BN$ such that for the covering $G/P_J = \bigcup_{i=1}^{s_n} g_i P^n_J/P_J $ the function $f$ is constant on each open subset $g_i P^n_J \cdot P_J$.
	Thus $f$ is contained in the image of \eqref{Eq 3 - Inclusion of smooth inductions}.
	With the choice of the finest locally convex topology on $\Ind^{\sm, G}_{P_{J}} (K)$ it follows that $\Ind^{\sm,G}_{P_{J}} (K) \cong \varinjlim_{n\in \BN} \Ind^{G_0}_{P^n_J} (K)$ is a topological isomorphism.

	With $v^{G_0}_{P^n_I}$ defined in the obvious way, we take the inductive limit over the diagrams
	\begin{equation*}
		\begin{tikzcd}
			\displaystyle\bigoplus_{I \subsetneq J \subset \Delta} \Ind^{G_0}_{P_{J}^n} (K) \ar[r] \ar[d] & \Ind^{G_0}_{P^n_I} (K) \ar[r] \ar[d] & v^{G_0}_{P^n_I}  \ar[r] \ar[d] & 0 \\
			\displaystyle\bigoplus_{I \subsetneq J \subset \Delta} \Ind^{\sm, G}_{P_{J}} (K) \ar[r]  & \Ind^{\sm, G}_{P_{I}} (K) \ar[r]  & v^{G}_{P_{I}} \ar[r] & 0 
		\end{tikzcd}
	\end{equation*}
	which have strictly exact rows.
	The snake lemma \ref{Lemma A1 - Snake lemma} then shows that the induced $G_0$-equivariant homomorphism $\varinjlim_{n\in \BN} v^{G_0}_{P^n_I} \ra v^{G}_{P_{I}}$ is a topological isomorphism.
	Moreover, this isomorphism is $G$-equivariant when $\varinjlim_{n\in \BN} v^{G_0}_{P^n_I}$ carries the $G$-action induced from \Cref{Rmk 3 - G-action on terms of spectral sequence}.
\end{proof}

We now turn towards $ \varprojlim_{n\in \BN} K'^{p}_{q,n}$.
Here the transition homomorphisms of $\big( K'^{p}_{q,n} \big)_{n\in \BN}$ are compact by \Cref{Cor 2 - Projective limit description of rigid local cohomology of Schubert varieties} so that $\varprojlim_{n\in \BN} K'^{p}_{q,n}$ is a nuclear $K$-Fr\'echet space, too. 
Similarly to before one finds that
\begin{equation*}
	\Ker \Big( \varprojlim_{n\in \BN} K'^{ -q+1}_{q,n} \ra \varprojlim_{n\in \BN} K'^{-q+2}_{q,n} \Big)'_b 
	\cong \varinjlim_{n\in \BN} \Coker\Big( \big(K'^{-q+1}_{q,n}\big)'_b \ra \big(K'^{-q+2}_{q,n}\big)'_b \Big)
\end{equation*}
with injective, compact transition homomorphisms in the inductive limit of the right hand side.

We write $\bP_{d-q} \defeq \bP_{J_{d-q}}$ for the standard parabolic subgroup corresponding to the subset $J_{d-q} \defeq \Delta \setminus \{\alpha_{d-q}\} \subset \Delta$.
We recall from \Cref{Prop 2 - Dual of local cohomology wrt open tube is locally analytic representation} (ii) that $W_n \defeq \widetilde{H}^{q}_{\BP_K^{d-q}(\varepsilon_n) } (\BP_K^d, \CE)'_b$ is a locally analytic $P^n_{d-q}$-representation.
Using the exactness of the ``finite'' induction $\Ind^{G_0}_{P^n_{d-q}}(\blank)$, the push-pull formula from Remark \ref{Rmk 1 - Dual of finite induction} (ii), and exactness of tensoring with $W_n$ \cite[Lemma 2.1 (ii)]{BreuilHerzig18TowardsFinSlopePartGLn}, we compute that
\begin{align*}
	\Coker&\Big( \big(K'^{-q+2}_{q,n}\big)'_b \ra \big(K'^{-q+1}_{q,n}\big)'_b \Big) \\
	&\cong \Ind^{G_0}_{Q^n_{d-q}} (W_n) \Bigg/ \sum_{i=d-q+1}^{d-1} \Ind^{G_0}_{P^n_{I_{d-q} \cup \{ \alpha_i \} }} (W_n) \\
	&\cong \Ind^{G_0}_{P^n_{d-q}} \Bigg( \Ind^{P^n_{d-q}}_{Q^n_{d-q}} (W_n) \Bigg/ \sum_{i=d-q+1}^{d-1} \Ind^{P^n_{d-q}}_{P^n_{I_{d-q} \cup \{\alpha_i \}}} (W_n) \Bigg) \\
	&\cong \Ind^{G_0}_{P^n_{d-q}} \Bigg( \Big( W_n \otimes_K \Ind^{P^n_{d-q}}_{Q^n_{d-q}} (K) \Big) \Bigg/ \sum_{i=d-q+1}^{d-1} W_n \otimes_K \Ind^{P^n_{d-q}}_{P^n_{I_{d-q} \cup \{\alpha_i \}}} (K) \Bigg) \\
	&\cong \Ind^{G_0}_{P^n_{d-q}} \Big( W_n \otimes_K v^{P^n_{d-q}}_{Q^n_{d-q}} \Big)
\end{align*}
with the finite-dimensional representations
\begin{equation*}
	v^{P^n_{d-q}}_{Q^n_{d-q}} \defeq  \Ind^{P^n_{d-q}}_{Q^n_{d-q}} (K) \Bigg/ \sum_{i= d-q+1}^{d-1} \Ind^{P^n_{d-q}}_{P^n_{I_{d-q} \cup \{\alpha_i \}}} (K) .
\end{equation*}

\begin{remark}\label{Rmk 3 - G-action on inductive limit of inductions}
	We want to describe the induced $G$-action on 
	\begin{equation*}
		\Ker \Big( \varprojlim_{n\in \BN} K'^{ -q+1}_{q,n} \ra \varprojlim_{n\in \BN} K'^{-q+2}_{q,n} \Big)'_b 
		\cong \varinjlim_{n\in\BN}\, \Ind^{G_0}_{P^n_{d-q}} \Big( \widetilde{H}^{q}_{\BP_K^{d-q}(\varepsilon_n)} (\BP_K^d, \CE)'_b \otimes_K v^{P^n_{d-q}}_{Q^n_{d-q}} \Big) .
	\end{equation*}
	To ease the notation, we write $\bP= \bP_{d-q}$, and $\bQ = \bQ_{d-q}$ here.
	Fix $g\in G$ and consider an element $v$ of the right hand side term. 
	Let $m \in \BN$ such that
	\begin{equation*}
		v = \sum_{i=1}^{s_m} g_i \bullet v_i \in \Ind^{G_0}_{P^m} \Big( \widetilde{H}^{q}_{\BP_K^{d-q}(\varepsilon_m)} (\BP_K^d, \CE)'_b \otimes_K v^{P^m}_{Q^m} \Big) ,
	\end{equation*}
	where $g_1,\ldots,g_{s_m}$ are coset representatives of $G_0/P^m$.

	Similarly to \Cref{Rmk 3 - G-action on terms of spectral sequence}, let $n\geq m$ such that $G/Q = \bigcup_{j=1}^{s_n} h_j Q^n/ Q$ is a refinement of the translated covering $G/Q = \bigcup_{i=1}^{s_m} g g_i Q^m/Q$, for coset representatives $h_1,\ldots, h_{s_n}$ of $G_0/Q^n$.
	We can consider the induced coverings of $G/P$ under the surjection $G/Q \ra G/P$.
	After enlarging $n$ we may assume that the induced covering $G/P = \bigcup_{i=1}^{s_n} h_j P^n/P$ is a refinement of the induced covering $G/P = \bigcup_{i=1}^{s_m} g g_i P^m/P$.

	For $j=1,\ldots,s_n$, we have $h_j Q^n/Q \subset g g_{\sigma_g(j)} Q^m/Q$ with the notation of \Cref{Rmk 3 - G-action on terms of spectral sequence}.
	We claim that this implies $h_j P^n/Q \subset g g_{\sigma_g(j)} P^m/Q$.
	Indeed, using
	\begin{align*}
		h_j P^n/Q = \bigcup_{h_{j'} P^n = h_j P^n} h_{j'} Q^n/Q
		\subset \bigcup_{h_{j'} P^n = h_j P^n} g g_{\sigma_g(j')} Q^m/Q
		\subset \bigcup_{h_{j'} P^n = h_j P^n} g g_{\sigma_g(j')} P^m/Q
	\end{align*}	
	it suffices to show that $g_{\sigma_g(j)} P^m = g_{\sigma_g(j')} P^m$ if $h_j P^n = h_{j'} P^n$.
	For this, we compute that $h_j Q^n/P \subset h_j P^n/P \cap g g_{\sigma_g(j)} P^m/P$.
	By the assumption on $n$ with respect to the covering of $G/P$ this shows that $h_j P^n/P \subset g g_{\sigma_g(j)} P^m/P$ \footnote{Recall that the function $\sigma_g$ is defined with regard to the coverings of $G/Q$.}.
	The, if $h_j P^n = h_{j'} P^n$, the sets $ g g_{\sigma_g(j)} P^m/P $ and $g g_{\sigma_g(j')} P^m/P $ have non-empty intersection which implies $g_{\sigma_g(j)} P^m = g_{\sigma_g(j')} P^m$.

	We now set $p_{g,j} \defeq h^{-1}_j g g_{\sigma_g(j)}$ so that we have $ P^n/Q \subset p_{g,j} P^m/Q$.
	Then $g$ induces continuous homomorphisms
	\begin{equation*}
		\varphi_{p_{g,j}}^t \colon \widetilde{H}^{q}_{\BP_K^{d-q}(\varepsilon_m)} (\BP_K^d, \CE)'_b  \lra  \widetilde{H}^{q}_{\BP_K^{d-q}(\varepsilon_n)} (\BP_K^d, \CE)'_b .
	\end{equation*}
	Furthermore, we have $p_{g,j}^{-1} P^n \subset P^m \cdot Q$.
	Written in terms of locally constant functions $p_{g,j}$ gives a continuous homomorphism
	\begin{align*}
		\Ind^{P^m}_{Q^m} (K) &\lra \Ind^{P^n}_{Q^n} (K) , \\
		f &\lto f(p_{g,j}^{-1} \blank ) = \big[ p \mto f(p') \text{ , if $p_{g,j}^{-1} p = p' q$, for $p'\in P^m$, $q\in Q$}\big] .
	\end{align*}
	This in turn yields a continuous homomorphism $\psi_{p_{g,j}} \colon v^{P^m}_{Q^m} \ra v^{P^n}_{Q^n}$.
	We let $\tau_{p_{g,j}}$ denote the tensor product
	\begin{equation*}
		\tau_{p_{g,j}} \defeq \varphi^t_{p_{g,j}} \otimes \psi_{p_{g,j}} \colon \widetilde{H}^{q}_{\BP_K^{d-q}(\varepsilon_m)} (\BP_K^d, \CE)'_b \botimes{K} v^{P^m}_{Q^m} \lra  \widetilde{H}^{q}_{\BP_K^{d-q}(\varepsilon_n)} (\BP_K^d, \CE)'_b \botimes{K} v^{P^n}_{Q^n} .
	\end{equation*}
	In total the homomorphism by which $g$ acts is given by
	\begin{equation*}
		g .\bigg( \sum_{i=1}^{s_m} g_i \bullet v_i \bigg) = \sum_{j=1}^{s_n} h_j \bullet \tau_{p_{g,j}} (v_{\sigma_g(j)})
		\in \Ind^{G_0}_{P^n} \Big( \widetilde{H}^{q}_{\BP_K^{d-q}(\varepsilon_n)} (\BP_K^d, \CE)'_b \otimes_K v^{P^n}_{Q^n} \Big) .
	\end{equation*}
	\qed
\end{remark}

Since the spectral sequence ${}^{\rm v}\! E_r^{p,q}$ degenerates at the $E_2$-page, we obtain a filtration of $H^1_\CY (\BP_K^d, \CE)$ by $G$-invariant subspaces whose successive subquotients are isomorphic to ${}^{\rm v}\! E_2^{-q+1,q}$.
More precisely and taking into account the long exact sequence \eqref{Eq 3 - Long exact sequence of local cohomology for complement of DHS} of local cohomology with respect to $\CY \subset \BP_K^d$, we arrive at the following theorem.

\begin{theorem}[{cf.\ \cite[Thm.\ 2.2.8]{Orlik08EquivVBDrinfeldUpHalfSp}}]\label{Thm 3 - Orliks theorem on global sections of an equivariant vector bundle on the DHS}
	Let $\CE$ be a $\bG$-equivariant vector bundle on $\BP_K^d$.
	Then there exist a filtration by closed $D(G)$-submodules
	\begin{equation*}
		H^0 (\CX,\CE) = V^d \supset V^{d-1} \supset \ldots \supset V^{1} \supset V^{0} = H^0(\BP_K^d, \CE) ,
	\end{equation*}
	and, for $q = 1,\ldots, d$, short strictly exact sequences of locally analytic $G$-representations
	\begin{equation}\label{Eq 3 - Extension from spectral sequence}
		\begin{aligned}
			0 \lra &{\hspace{2pt}} H^q (\BP_K^d, \CE)' \botimes{K} v^G_{Q_{d-q}} \lra \big(V^q/V^{q-1} \big)'_b \\
			&\lra \varinjlim_{n\in \BN} \Ind^{G_0}_{P^n_{d-q}} \Big( \widetilde{H}^q_{\BP_K^{d-q} (\varepsilon_n)} (\BP_K^d, \CE^\rigbun)'_b \botimes{K} v^{P^n_{d-q}}_{Q^n_{d-q}} \Big) \lra 0 .
		\end{aligned}
	\end{equation}
\end{theorem}

When $K$ is a $p$-adic field, Orlik analyses the right hand side term of \eqref{Eq 3 - Extension from spectral sequence} further.
In \cite[p.\ 634]{Orlik08EquivVBDrinfeldUpHalfSp} he shows that there is an isomorphism of locally analytic $G$-representations
\begin{equation}\label{Eq 3 - Isomorphism of locally analytic representations from Orlik}
	\varinjlim_{n\in \BN} \Ind^{G_0}_{P^n_{d-q}} \Big( \widetilde{H}^q_{\BP_K^{d-q} (\varepsilon_n)} (\BP_K^d, \CE^\rigbun)'_b \botimes{K} v^{P^n_{d-q}}_{Q^n_{d-q}} \Big)
	\cong \Ind^{\la, G}_{P_{d-q}} \Big( N'_{d-q} \otimes v^{P_{d-q}}_{Q_{d-q}} \Big)^{\Fd_{d-q}} .
\end{equation}
We explain the notation used here.
Let $U(\Fg)$ and $U(\Fp_{d-q})$ denote the universal enveloping algebras of the Lie algebras of $\bG$ and $\bP_{d-q}$ respectively.
One shows that $\widetilde{H}^q_{\BP_K^{d-q}} (\BP_K^d, \CE)$ is a quotient of a generalized Verma module for $U(\Fg)$.
More precisely, there exists a finite-dimensional $P_{d-q}$-subrepresentation $N_{d-q} \subset \widetilde{H}^q_{\BP_K^{d-q}} (\BP_K^d, \CE)$ which generates it as a $U(\Fg)$-module \cite[Lemma 1.2.1]{Orlik08EquivVBDrinfeldUpHalfSp}, i.e.\ there exists an epimorphism of $U(\Fg)$-modules
\begin{equation*}
	U(\Fg) \otimes_{U(\Fp_{d-q})} N_{d-q} \relbar\joinrel\twoheadrightarrow \widetilde{H}^q_{\BP_K^{d-q}} (\BP_K^d, \CE) .
\end{equation*}
Let $\Fd_{d-q}$ denote the kernel of this epimorphism.
Then $\Ind^{\la, G}_{P_{d-q}} \big( N'_{d-q} \otimes v^{P_{d-q}}_{Q_{d-q}} \big)^{\Fd_{d-q}}$ indicates the subspace of those functions in the locally analytic induction that are annihilated by $\Fd_{d-q}$, cf.\ \cite[p.\ 607]{Orlik08EquivVBDrinfeldUpHalfSp}.
In particular, the right hand side of \eqref{Eq 3 - Isomorphism of locally analytic representations from Orlik} does not depend on the choice of $N_{d-q}$.

However, when $K$ is of positive characteristic $\widetilde{H}^q_{\BP_K^{d-q}} (\BP_K^d, \CE)$ in general is no longer finitely generated, even if we replace $U(\Fg)$ by the algebraic distribution algebra ${\rm Dist}(\bG) \cong \hy(G)$ of $\bG$, see \cite[Ch.\ 2.3]{Kuschkowitz16EquivVBRigidCohomDrinfeldUpHalfSpFF}.
We tackle this problem by adapting a different description of $\Ind^{\la, G}_{P_{d-q}} \big( N'_{d-q} \otimes v^{P_{d-q}}_{Q_{d-q}} \big)^{\Fd_{d-q}}$ via the functors $\CF_P^G$ defined by Orlik and Strauch in \cite{OrlikStrauch15JordanHoelderSerLocAnRep}.
In \Cref{Sect - The Functors of Orlik--Strauch} we then compare Orlik's and our description for the case of a $p$-adic field.

\subsection{The Subquotients of $H^0(\CX,\CE)'_b$ as Locally Analytic $\GL_{d+1}(\CO_K)$-Re\-pre\-sen\-ta\-tions}

We now want to analyse the $G$-representations that occur as a quotient of $\big( {}^{\rm v}\! E_2^{-q+1,q}\big)'_b$ in \eqref{Eq 3 - Extension from spectral sequence}.
We change the notation in as much as we fix $\sdd \defeq d-q \in \{0,\ldots,d-1\}$, and write $\bP \defeq \bP_{d-q}$ as well as $\bQ \defeq \bQ_{d-q}$.
For $n\in \BN$, we have seen in the proof of \Cref{Prop 2 - Dual of local cohomology wrt open tube is locally analytic representation} that 
\begin{equation}\label{Eq 3 - Definition of V_n}
	V_n \defeq \widetilde{H}^{d-\sdd}_{\BP_K^\sdd(\varepsilon_n)^-} (\BP_K^d, \CE^\rigbun)' \otimes_K v^{P^n}_{Q^n}(K)
\end{equation}
is a locally analytic $P^n$-representation whose underlying locally convex vector space is a $K$-Banach space.

\begin{lemma}\label{Lemma 3 - Limit of the local cohomology as a locally analytic representation}
	The transition homomorphisms $V_{n}\ra V_{n+1}$ are injective and compact, i.e.\
	\[ V \defeq \varinjlim_{n\in \BN} V_n \]
	is of compact type this way.
	In particular via the $P^n$-actions on the $V_n$, $V$ becomes a locally analytic $(\hy(G),P_0)$-module in the sense of \Cref{Def 1 - Compatible hyperalgebra modules}.
\end{lemma}
\begin{proof}
	We have seen in \Cref{Prop 2 - Projective limit description of local cohomology wrt open tubes} that the transition maps
	\[\widetilde{H}^{d-\sdd}_{\BP_K^\sdd(\varepsilon_{n+1})^-} (\BP_K^d, \CE^\rigbun) \lra \widetilde{H}^{d-\sdd}_{\BP_K^\sdd(\varepsilon_{n})^-} (\BP_K^d, \CE^\rigbun) \]
	are compact and have dense image.
	Therefore their transposes are injective and compact by \cite[Lemma 16.4]{Schneider02NonArchFunctAna}.
	Moreover the homomorphisms $v^{P^n}_{Q^n} \ra v^{P^{n+1}}_{Q^{n+1}}$ which are induced by the restriction map $\Ind^{P^n}_{Q^n} (K) \ra \Ind^{P^{n+1}}_{Q^{n+1}} (K),\, f \mto f\res{P^{n+1}}$, are injective as well.
	%	Let $f$ be in the kernel of $\Ind^{\mathrm{sm},P^n}_{Q^n}(K) \ra v^{P^{n+1}}_{Q^{n+1}}$, i.e.\ for $k=\sdd+1,\ldots,d-1$, there are $f_k \in \Ind^{sm, P^{n+1}}_{P^{n+1}_k}$ such that $\sum_{k=\sdd+1}^{d-1} f_k \circ \mathrm{pr}_k^{n+1} = f \circ \mathrm{red}_{Q}^{n}$ where $\mathrm{pr}_k^{n+1} \colon P^{n+1}/Q^{n+1} \ra P^{n+1}/P^{n+1}_k$ and $\mathrm{red}_Q^n \colon P^{n+1}/Q^{n+1} \ra P^n/Q^n$. For each $g\in P^n/P_k^n$ choose compatible lifts to $P^{n+1}/P_k^{n+1}$ and $P^{n+1}/Q^{n+1}$ and define this way $\tilde{f}_k \in \Ind^{\mathrm{sm},P^{n}}_{P^n_k}$ such that $f_k = \tilde{f}_k \circ \mathrm{red}^n_{P_k}$ on these lifts. We claim $ f= \sum_{k=\sdd+1}^{d-1} \tilde{f}_k \circ \mathrm{pr}^n_k$ which shows that $f = 0 $ in $v^{P^n}_{Q^n}$. To show this claim it suffices to check $f \circ \mathrm{red}^n_{Q} = \sum_{k=\sdd+1}^{d-1} \tilde{f}_k \circ \mathrm{pr}_k^n \circ \mathrm{red}_Q^n$ on the lifts. But there:
	%	\begin{align*}
		%		\sum_{k=\sdd+1}^{d-1} \tilde{f}_k \circ \mathrm{pr}_k^n \circ \mathrm{red}_Q^n &= \sum_{k=\sdd+1}^{d-1} \tilde{f}_k \circ \mathrm{red}_{P_k}^n \circ \mathrm{pr}_k^{n+1} \\
		%		&= \sum_{k=\sdd+1}^{d-1} f_k \circ \mathrm{red}_Q^n \\
		%		&= f \circ \mathrm{red}^n_Q .
		%	\end{align*}	
	Hence it follows from \cite[Cor.\ 1.1.27]{Emerton17LocAnVect} that the tensor product $V_n \ra V_{n+1}$ of these maps is injective.
	The homomorphism $v^{P^n}_{Q^n} \ra v^{P^{n+1}}_{Q^{n+1}}$ is compact as a homomorphism between finite-di\-men\-sion\-al $K$-vector spaces by \cite[Lemma 16.4]{Schneider02NonArchFunctAna}.
	Therefore \cite[Lemma 18.12]{Schneider02NonArchFunctAna} implies that $V_n\ra V_{n+1}$ is compact as well.

	Finally, we have $\hy(P^n) = \hy(G)$ because $P^n\subset G$ is an open subgroup.	
	Hence each $V_n$ is a locally analytic $(\hy(G),P^n)$-module via \Cref{Rmk 1 - Locally analytic representation induces compatible hyperalgebra module structure}, and therefore a locally analytic $(\hy(G),P_0)$-module in particular.
	It follows that $V$ is a locally analytic $(\hy(G),P_0)$-module as well.	
\end{proof}

\begin{remark}\label{Rmk 3 - P-action on inductive limit}
	In fact, $V$ is even a locally analytic $P$-representation.
	For fixed $p\in P$, we find $n\geq m$ large enough such that $h_j P^n/Q \subset p g_{\sigma_g(j)} P^m/Q$, for all $j=1,\ldots, s_n$, like in \Cref{Rmk 3 - G-action on inductive limit of inductions}.
	Considering the images under $G/Q \ra G/P$, it follows from $1 \in p^{-1} P^n/P $ that $P^n/P \subset p P^m/P$.
	This shows that $\BP_K^\sdd (\varepsilon_n)^- \subset p \BP_K^\sdd(\varepsilon_m)^-$ and we obtain a homomorphism
	\begin{equation*}
		\varphi_p^t \colon \widetilde{H}^{d-\sdd}_{\BP_K^\sdd(\varepsilon_m)^-} (\BP_K^d, \CE^\rigbun)' \lra \widetilde{H}^{d-\sdd}_{\BP_K^\sdd(\varepsilon_n)^-} (\BP_K^d, \CE^\rigbun)' .
	\end{equation*}
	Moreover, we have the continuous homomorphism $\psi_p\colon v^{P^m}_{Q^m} \ra v^{P^n}_{Q^n}$ induced by
	\begin{equation*}
		\Ind^{P^m}_{Q^m} (K) \lra \Ind^{P^n}_{Q^n} (K) \,,\quad f \lto f(p^{-1} \blank ) .
	\end{equation*}
	The tensor product of these yields the continuous homomorphism $\tau_p = (\varphi_p^t \otimes \psi_p) \colon V_m \ra V_n$.
	Like before the collection of these $\tau_p$, for all $m \in \BN$, gives the action of $p$ on $V$.
	This $P$-action extends the one of $P_0$ and $V$ is a locally analytic $(\hy(G),P)$-module this way.
	\qed
\end{remark}

\begin{lemma}\label{Lemma 3 - Dual of limit of local cohomology of Schubert varieties}
	There is a canonical topological isomorphism of locally analytic $(\hy(G),P)$-modules
	\begin{equation*}
		V \cong \widetilde{H}^{d-\sdd}_{(\BP_K^{\sdd})^\rig} (\BP_K^d, \CE)'_b \cotimes{K} v^{\GL_{d-\sdd}(K)}_{B_{d-\sdd}} .
	\end{equation*}
	Here $\GL_{d-\sdd,K}$ is viewed as a subgroup of the standard Levi factor $\bL \cong \GL_{\sdd+1,K}\times_K \GL_{d-\sdd,K}$ of $\bP=\bP_{\Delta\setminus\{\alpha_r\}}$, and $\bB_{d-\sdd}$ denotes the standard (lower) Borel subgroup of $\GL_{d-\sdd,K}$.
	On $v^{\GL_{d-\sdd}(K)}_{B_{d-\sdd}}$ the group $P$ acts via inflation, $\hy(G)$ acts trivially, and it carries the finest locally convex topology.
\end{lemma}
\begin{proof}
	First note that we have an isomorphism
	\begin{equation*}
		\Ind^{P^n}_{Q^n} (K) \overset{\cong}{\lra} \Ind^{\GL_{d-\sdd}(\CO_K)}_{B_{d-\sdd}^n} (K) \,,\quad f \lto f\res{\GL_{d-\sdd}(\CO_K)} ,
	\end{equation*}
	since $P^n/Q^n \cong \GL_{d-\sdd}(\CO_K)/B_{d-\sdd}^n$.
	Here we view $\GL_{d-\sdd}(\CO_K) \subset P_0$ as a subgroup.
	This yields isomorphisms $v^{P^n}_{Q^n} \cong v^{\GL_{d-\sdd}(\CO_K)}_{B_{d-\sdd}^n}$, for all $n\in \BN$.
	Taking the inductive limit over these isomorphisms and applying \Cref{Lemma 3 - Generalized Steinberg representation as inductive limit} to the case of $B_{d-\sdd} \subset \GL_{d-\sdd}(K)$, we obtain an isomorphism $\varinjlim_{n \in \BN} v^{P^n}_{Q^n} \cong v^{\GL_{d-\sdd}(K)}_{B_{d-\sdd}}$ of locally convex $K$-vector spaces.
	Moreover, one computes that this isomorphism is $P$-equivariant when $\varinjlim_{n \in \BN} v^{P^n}_{Q^n}$ carries the $P$-action from \Cref{Rmk 3 - P-action on inductive limit}, and $v^{\GL_{d-\sdd}(K)}_{B_{d-\sdd}}$ the one by inflation.

	We recall from \Cref{Prop 2 - Dual of local cohomology wrt open tube is locally analytic representation} (ii) that $\widetilde{H}^{d-\sdd}_{(\BP_K^\sdd )^\rig} (\BP_K^d , \CE)'_b \cong \varinjlim_{n \in \BN} \widetilde{H}^{d-\sdd}_{\BP_K^\sdd(\varepsilon_n)^-} (\BP_K^d , \CE)'$.
	By \cite[Prop.\ 1.1.32]{Emerton17LocAnVect} we then obtain a $P$-equivariant topological isomorphism
	\begin{equation}\label{Eq 3 - Comparison isomorphism for dual of limit of local cohomology of Schubert varieties}
		V = \varinjlim_{n \in \BN} \Big( \widetilde{H}^{d-\sdd}_{\BP_K^\sdd(\varepsilon_n)^-} (\BP_K^d , \CE)' \botimes{K} v^{P^n}_{Q^n} \Big) 
		\cong \widetilde{H}^{d-\sdd}_{(\BP_K^\sdd )^\rig} (\BP_K^d , \CE)'_b \cotimes{K} v^{\GL_{d-\sdd}(K)}_{B_{d-\sdd}} .
	\end{equation}
	Finally note that $\hy(G)$ acts trivially on $v^{P^n}_{Q^n}$, for each $n\in \BN$, as $Q^n\subset P^n$ is an open subgroup.
	On $\widetilde{H}^{d-\sdd}_{(\BP_K^\sdd )^\rig} (\BP_K^d , \CE)'_b$ the $\hy(G)$-action is induced by the actions on the $\widetilde{H}^{d-\sdd}_{\BP_K^\sdd(\varepsilon_n)^-} (\BP_K^d , \CE)'$.
	Therefore \eqref{Eq 3 - Comparison isomorphism for dual of limit of local cohomology of Schubert varieties} is $\hy(G)$-equivariant.
\end{proof}

We come back to the locally analytic $G$-representation $\varinjlim_{n \in \BN} \Ind_{P^n}^{G_0} (W_n)$ from \eqref{Eq 3 - Extension from spectral sequence}, for $W_n \defeq \widetilde{H}^{d-\sdd}_{\BP_K^\sdd(\varepsilon_n)} (\BP_K^d, \CE^\rigbun)'_b \otimes_K v^{P^n}_{Q^n}(K)$.
The $K$-Banach spaces $V_n$ allow us to exhibit this as a locally convex $K$-vector space of compact type.

\begin{lemma}\label{Lemma 3 - Limit of inductions is of compact type}
	There is a canonical topological isomorphism of locally analytic $G$-re\-pre\-sen\-ta\-tions
	\begin{equation}\label{Eq 3 - Limit of inductions isomorphic to limit over Banach spaces}
		\varinjlim_{n\in \BN}\, \Ind^{G_0}_{P^n}  (W_n) \cong \varinjlim_{n\in \BN}\, \Ind^{G_0}_{P^n}  (V_n).
	\end{equation}
	Here the $G$-action on $\varinjlim_{n\in \BN}\, \Ind^{G_0}_{P^n}  (V_n)$ is given in the way analogous to \Cref{Rmk 3 - G-action on inductive limit of inductions}.
	Moreover, the transition maps $\Ind^{G_0}_{P^n} (V_n) \ra \Ind^{G_0}_{P^{n+1}} (V_{n+1})$ of the right hand side are compact and injective so that the underlying locally convex $K$-vector space of the above $G$-representation is of compact type.
\end{lemma}
\begin{proof}
	In view of the inductive limit description of \Cref{Prop 2 - Dual of local cohomology wrt open tube is locally analytic representation}, for $n\in \BN$, the homomorphism
	\begin{equation*}
		\widetilde{H}^{d-\sdd}_{\BP^\sdd_K(\varepsilon_{n})} (\BP_K^d,\CE^\rigbun)'_b \lra \widetilde{H}^{d-\sdd}_{\BP^\sdd_K(\varepsilon_{n+1})} (\BP_K^d,\CE^\rigbun)'_b
	\end{equation*}
	factors over $\widetilde{H}^{d-\sdd}_{\BP^\sdd_K(\varepsilon_{n})^-} (\BP_K^d,\CE^\rigbun)'$.
	We therefore obtain a commutative diagram of locally analytic $P^{n+1}$-re\-pre\-sen\-tations
	\begin{equation*}
		\begin{tikzcd}
			\widetilde{H}^{d-\sdd}_{\BP^\sdd_K(\varepsilon_{n-1})^-} (\BP_K^d,\CE^\rigbun)' \ar[r]\ar[d] &  \widetilde{H}^{d-\sdd}_{\BP^\sdd_K(\varepsilon_{n})^-} (\BP_K^d,\CE^\rigbun)' \ar[d] \\
			\widetilde{H}^{d-\sdd}_{\BP^\sdd_K(\varepsilon_{n})} (\BP_K^d,\CE^\rigbun)'_b \ar[r]\ar[ur] &	\widetilde{H}^{d-\sdd}_{\BP^\sdd_K(\varepsilon_{n+1})} (\BP_K^d,\CE^\rigbun)'_b .
		\end{tikzcd}
	\end{equation*}
	Combined with the canonical homomorphisms $v^{P^n}_{Q^n}\ra v^{P^{n+1}}_{Q^{n+1}}$, this gives factorizations of the transition maps of both inductive limits in \eqref{Eq 3 - Limit of inductions isomorphic to limit over Banach spaces}
	\begin{equation*}
		\begin{tikzcd}
			\Ind^{G_0}_{P^{n-1}} \Big( \widetilde{H}^{d-\sdd}_{\BP^\sdd_K(\varepsilon_{n-1})^-} (\BP_K^d,\CE^\rigbun)' \otimes_K v^{P^{n-1}}_{Q^{n-1}} \Big) \ar[r]\ar[d] & \Ind^{G_0}_{P^{n}} \Big( \widetilde{H}^{d-\sdd}_{\BP^\sdd_K(\varepsilon_{n})^-} (\BP_K^d,\CE^\rigbun)' \otimes_K v^{P^{n}}_{Q^{n}} \Big) \ar[d] \\
			\Ind^{G_0}_{P^{n}} \Big( \widetilde{H}^{d-\sdd}_{\BP^\sdd_K(\varepsilon_{n})} (\BP_K^d,\CE^\rigbun)'_b \otimes_K v^{P^{n}}_{Q^{n}} \Big) \ar[r]\ar[ru] & \Ind^{G_0}_{P^{n+1}} \Big( \widetilde{H}^{d-\sdd}_{\BP^\sdd_K(\varepsilon_{n+1})} (\BP_K^d,\CE^\rigbun)'_b \otimes_K v^{P^{n+1}}_{Q^{n+1}} \Big)  .
		\end{tikzcd}
	\end{equation*}
	We conclude that both inductive limits are topologically isomorphic to each other.
	
	Now consider, for $n\in \BN$, coset representatives $h_1,\ldots,h_{s_{n+1}}$ of $G_0/P^{n+1}$ and $g_1,\ldots,g_{s_n}$ of $G_0/P^n$ such that $h_j$ gets mapped to $g_{i(j)}$ under $G_0/P^{n+1}\ra G_0/P^n$.
	Then the transition map $\Ind^{G_0}_{P^n} (V_n) \ra \Ind^{G_0}_{P^{n+1}} (V_{n+1})$ is given by
	\begin{equation*}
		\bigoplus_{i=1}^{s_n} g_i \bullet V_n \lra \bigoplus_{j=1}^{s_{n+1}} h_j \bullet V_{n+1} \,,\quad  \sum_{i=1}^{s_n} g_i \bullet v_i \lto \sum_{j=1}^{s_{n+1}} h_j \bullet v_{i(j)} .
	\end{equation*}
	Therefore it is injective and compact as the finite direct sum of compact homomorphisms, by \Cref{Lemma A1 - Generalities on compact maps} (iii).
\end{proof}

Our next goal is to interpret $\varinjlim_{n\in \BN} \Ind^{G_0}_{P^n} (V_n)$ as a subspace of $C^\la (G_0,V)$.
We know from \Cref{Prop 1 - Isomorphism between locally analytic and finite induction} that $\Ind^{G_0}_{P^n} (V_n) \cong \Ind^{\la,G_0}_{P^n} (V_n)$.
As each $V_n$ is a BH-subspace of $V$, we consequently obtain injective continuous homomorphisms, for all $n\in\BN$:
\begin{align*}
	\iota_n \colon	\Ind^{G_0}_{P^n} (V_n) &\lra C^\la (G_0,V_n) \lra C^\la (G_0,V) , \\
	\sum_{i=1}^{s_n} g_i \bullet v_i &\lto \big[ g \mto p^{-1}.v_i \text{ , for $g = g_i p$ with $p\in P^n$} \big] ,
\end{align*}
Given $f\in C^\la(G_0,V)$, $g\in G_0$, and $p\in P_0$, we write $f(g\blank p)$ for the locally analytic function
\begin{equation*}
	f(g\blank p) \colon G_0 \lra V \,,\quad h \lto f(g h p) .
\end{equation*}
We let $\mu (f)$ denote $\mu \in \hy(G)$ applied to a function $f\in C^\la (G_0,V)$ via the pairing \eqref{Eq 1 - Pairing for stalk of locally analytic functions}.
The $\hy(G)$-module action on $V$ is denoted by $\mu \ast v$, for $v\in V$.
Also recall that $\dot{\mu}$ signifies the involution from \Cref{Lemma 1 - Involution of hyperalgebra}.

\begin{lemma}\label{Lemma 3 - Embedding of direct limit of inductions}
	The homomorphism
	\begin{equation*}
		\iota\colon \varinjlim_{n\in \BN} \Ind^{G_0}_{P^n} (V_n) \lra C^\la (G_0,V)
	\end{equation*}
	induced by the $\iota_n$ is a closed embedding with
	\begin{equation*}
		\Im(\iota) = \left\{ f\in C^\la(G_0,V)\middle{|} \forall g \in G_0, p\in P_0, \mu \in \hy(G): \mu \big(f(g\blank p) \big) = p^{-1}. \dot{\mu} \ast f(g)\right\} .
	\end{equation*}
\end{lemma}
\begin{proof}
	It suffices to show the statement about $\Im(\iota)$.
	Indeed, then $\Im(\iota)$ is the intersection of the kernels of the continuous homomorphisms
	\begin{equation}\label{Eq 3 - Family of homomorphisms for limit of inductions}
		C^\la (G_0,V) \lra V \,,\quad f \lto \mu \big( f(g \blank p) \big) - p^{-1} . \dot{\mu} \ast f(g) ,
	\end{equation}
	for $g \in G_0, p\in P_0, \mu \in \hy(G)$.
	As $V$ is Hausdorff, these kernels are closed subspaces, and so is $\Im(\iota)$.
	Since $C^\la(G_0,V)$ is of compact type by \Cref{Prop 1 - Direct limit description of locally analytic functions for compact manifold} (iii), $\Im(\iota)$ is of compact type as well using \Cref{Lemma A1 - Closed subspaces and quotients of spaces of compact type}.
	Moreover, the induced homomorphism $\iota$ is a continuous bijection onto its image.
	Because $\varinjlim_{n\in \BN} \Ind^{G_0}_{P^n} (V_n)$ is the inductive limit of $K$-Banach spaces, we can apply a version of the open mapping theorem \cite[II.\ \S 4.6 Cor.]{Bourbaki87TopVectSp1to5} to conclude that $\iota$ is strict.

	It remains to show the statement about $\Im(\iota)$.
	Let $\rho_n \colon P^n \ra \GL(V_n)$ denote the representation \eqref{Eq 3 - Definition of V_n} and $\rho_{n,v} \colon P^n \ra V_n$, for $v\in V_n$, its locally analytic orbit maps.
	First, consider $f \in \Im(\iota)$.
	Then there exists $n\in \BN$ such that $f\in \Im(\iota_n)$, i.e.\ $f \in C^\la (G_0,V_n)$ and $f(gp)= p^{-1}.f(g)$, for all $g\in G_0$, $p\in P^n$.
	Let $D_n \defeq 1 + \unif^n M_{d+1}(\CO_K)$, and note that $P^n = D_n \cdot P_0$.
	It follows that
	\begin{equation}\label{Eq 3 - Identity on D_n}
		f(gxp)= (xp)^{-1}.f(g) = (\rho_{n,f(g)} \circ \inv)(xp) ,
	\end{equation}
	for all $g\in G_0$, $x\in D_n$, $p\in P_0$.
	We fix $g\in G_0$ and $p\in P_0$ for the moment and consider \eqref{Eq 3 - Identity on D_n} as an identity of locally analytic functions in $x$ on $D_n$.
	The homomorphism
	\begin{align*}
		V_n \lra C^{\rm la}(P^n,V_n) \,,\quad v \lto (\rho_{n,v} \circ \inv )( \blank p) ,
	\end{align*}
	is $P^n$-equivariant with respect to the left regular representation on $C^{\rm la}(P^n,V_n)$.
	Therefore, it is equivariant for the $\hy(G)$-action as well, and for $v = f(g)$, we obtain
	\begin{equation}\label{Eq 3 - Equivariance of hyperalgebra-action on map that sends to orbit map}
		(\rho_{n, \mu \ast f(g)}\circ \inv)(\blank p)= \mu \ast \big( (\rho_{n,f(g)}\circ \inv)( \blank p) \big),
	\end{equation}
	for all $\mu \in \hy(G)$.
	Finally, we apply $\mu \in \hy(G)$ to \eqref{Eq 3 - Identity on D_n} and compute for the functions restricted to $D_n$:
	\begin{equation}\label{Eq 3 - Computation for f and hyperalgebra}
		\begin{aligned}
			\mu \big(f(g\blank p) \big) &= \mu \big( (\rho_{n,f(g)}\circ \inv)(\blank p)\big) && \\
			&= \big( \dot{\mu} \ast \big( (\rho_{n,f(g)}\circ \inv)(\blank p)\big)\big) (1) &&\quad \text{, by \Cref{Prop 1 - Compatibility of convolution and pairing for left regular representation} } \\
			&= \big( (\rho_{n, \dot{\mu} \ast f(g)}\circ \inv)(\blank p) \big)(1) &&\quad \text{, by \eqref{Eq 3 - Equivariance of hyperalgebra-action on map that sends to orbit map}} \\
			&= p^{-1}.\dot{\mu} \ast f(g) . &&
		\end{aligned}
	\end{equation}

	On the other hand, let $f \in C^\la(G_0,V)$ such that $\mu \big( f(g \blank p) \big) = p^{-1} . \dot{\mu} \ast f(g)$, for all $g \in G_0$, $p\in P_0$, and $\mu \in \hy(G)$.
	As $G_0$ is compact, there exists some $n \in \BN$ such that $f$ factors over the BH-space $V_n$ of $V$.
	Moreover, we find $m\geq n$ such that $f$ is locally analytic with respect to the finite covering 
	\begin{equation}\label{Eq 3 - Covering for local analyticity}
		G_0 = \bigcup_{u\in G_0/P^m } \bigcup_{p\in P^m/D_m} u p D_m ,
	\end{equation}
	i.e.\ $f\res{u p D_m}$ is analytic, for all cosets $u p D_m$.
	By \Cref{Prop 1 - Equivalent characterization for locally analytic representations on Banach spaces}, $\rho_n$ is a locally $K$-analytic map of locally $K$-analytic manifolds.
	After increasing $m$, we therefore may assume that $\rho_n \circ \inv$ is analytic on each coset of \eqref{Eq 3 - Covering for local analyticity} as well.
	%$\inv$ is locally analytic as $P^n$ is a non-Archimedean Lie group.
	In particular, $(\rho_{n,v} \circ \inv)\res{up D_m}$ is analytic, for all cosets $up D_m$ and all $v \in V_n$.

	To show that $f \in \Im(\iota_m)$, we fix $g\in G_0$ and $p \in P^m$ and write $p=xp_0$, for $x \in D_m$, $p_0 \in P_0$.
	Then $f(g \blank p_0)\res{D_m}$ and $(\rho_{n,v}\circ \inv)(\blank p_0)\res{D_m}$ are analytic, for all $v\in V_n$.
	Indeed, let $u'p'D_m$ by the coset of \eqref{Eq 3 - Covering for local analyticity} containing $gp_0$.
	Because $f$ is analytic on $u'p'D_m$, $f(g \blank p_0)$ is analytic on $g^{-1} u'p'D_m p_0^{-1}$.
	But using that $D_m$ is normal in $G_0$ being the kernel of the reduction homomorphism, we see that this last set is equal to $D_m$.
	Similarly $(\rho_{n,v} \circ \inv)$ is analytic on the coset $g^{-1} u'p'D_m$ so that $(\rho_{n,v}\circ \inv)(\blank p_0)$ is analytic on $D_m$.
	If we apply $\mu \in \hy(G)$ to the analytic function $(\rho_{n,v}\circ \inv)(\blank p_0)\res{D_m}$, we compute analogously to \eqref{Eq 3 - Computation for f and hyperalgebra}
	\begin{equation*} 
		p_0^{-1}. \dot{\mu} \ast f(g) = \mu \big( (\rho_{n,f(g)}\circ \inv)(\blank p_0) \big) . 
	\end{equation*}
	Combining this with the assumption $\mu \big( f(g \blank p) \big) = p^{-1}.\dot{\mu} \ast f(g)$, we see that the functions on $D_m$ satisfy
	\begin{equation*}
		\begin{aligned}
			\mu \big( f(g\blank p_0) \big) = \mu \big( (\rho_{n,f(g)}\circ \inv)(\blank p_0) \big),
		\end{aligned}
	\end{equation*}
	for all $\mu \in \hy(G)$.
	By \Cref{Prop 1 - Pairing for germs of locally analytic functions}, this implies that the locally analytic germs at $1$ of $f(g\blank p_0)$ and $(\rho_{n,f(g)}\circ\inv)(\blank p_0)$ agree.
	As both functions are analytic on all of $D_m$, we conclude that they agree there, and we have, for all $g\in G_0$, $p=xp_0 \in D_m \cdot P_0 = P^m$:
	\begin{equation*}
		f(gp)= f(gxp_0) = (\rho_{n,f(g)} \circ \inv)(xp_0) = (xp_0)^{-1}.f(g) = p^{-1}.f(g) .
	\end{equation*}
\end{proof}

Because $V$ is of compact type, by \Cref{Prop 1 - Direct limit description of locally analytic functions for compact manifold} (iii) we have the topological isomorphism $C^{\rm la}(G_0,K) \cotimes{K} V \cong C^{\rm la}(G_0,V)$ induced by $f\otimes v \mto f(\blank) \, v$.
Under this identification, the homomorphisms \eqref{Eq 3 - Family of homomorphisms for limit of inductions} are given by 
\[C^\la (G_0,K) \cotimes{K} V \lra V \,,\quad f\otimes v \lto \mu \big( f(g\blank p) \big) \, v - f(g) \, p^{-1}.\dot{\mu} \ast v .\]
Therefore, $\iota$ fits into the sequence of continuous homomorphism
\begin{align}\label{Eq 3 - Exact sequence for limit of inductions}
	0 \lra \varinjlim_{n\in \BN} \Ind_{P^n}^{G_0} (V_n) \overset{\iota}{\lra} C^\la (G_0,K) \cotimes{K} V &\overset{\psi}{\lra} \prod_{g\in G_0, p\in P_0 , \mu \in \hy(G)} V\\
	f\otimes v &\lto \Big(\mu \big(f(g \blank p) \big)\, v - f(g) \, p^{-1}.\dot{\mu} \ast v \Big)_{g,p,\mu} \nonumber
\end{align}
which is algebraically exact, and $\iota$ is strict.

We want to consider the strong dual of this sequence \eqref{Eq 3 - Exact sequence for limit of inductions}.
Note that by \cite[Rmk.\ 16.1 (ii)]{Schneider02NonArchFunctAna}, the homomorphisms of this dualized sequence are continuous again.
By \cite[Prop.\ 9.11]{Schneider02NonArchFunctAna} there is a topological isomorphism
\begin{equation*}
	\bigoplus_{\substack{g\in G_0\\ p\in P_0 \\\mu \in \hy(G)}} V'_b \overset{\cong}{\lra} \Bigg( \prod_{\substack{g\in G_0\\ p\in P_0 \\\mu \in \hy(G)}} \!\!V\,\, \Bigg)'_b \,,\quad
	\sum \ell_{g,p,\mu} \lto \left[ (v_{g,p,\mu}) \mto \sum \ell_{g,p,\mu}(v_{g,p,\mu}) \right] .
\end{equation*}
Under this isomorphism, the transpose of $\psi$ is given by 
\begin{align*}
	\psi^t \colon \bigoplus_{g\in G_0, p\in P_0 , \mu \in \hy(G)} V'_b &\lra \big( C^{\rm la}(G_0,K) \cotimes{K} V \big)'_b \\
	\sum \ell_{g,p,\mu} &\lto \left[ f\otimes v \mto \sum \mu \big( f(g \blank p) \big) \, \ell_{g,p,\mu}(v) - f(g) \, \ell_{g,p,\mu} (p^{-1}.\dot{\mu}\ast v) \right] .
\end{align*}
By applying \Cref{Cor 1 - Fubini theorem} at various points, we have
\begin{align*}
	\mu \big( f(g\blank p) \big) &= \mu \big[  h_1 \mto f(gh_1 p)  \big]\\
	%		&= \mu \big[  h_1 \mto \delta_g \big[ h_2 \mto f(h_2 h_1 p) \big] \big]\\
	&= \mu \big[ h_1 \mto \delta_g  \big[ h_2\mto \delta_p \big[ h_3 \mto f(h_2 h_1 h_3) \big] \big] \big] \\
	&= \delta_g \big[ h_2 \mto \mu \big[ h_1 \mto \delta_p \big[ h_3 \mto f(h_2 h_1 h_3) \big] \big] \big] \\
	&= (\delta_g \ast \mu \ast \delta_p)(f) .
\end{align*}
Furthermore, by the definition of the contragredient action of $P_0$ and $\hy(G)$ on $V'_b$:
\[ \ell_{g,p,\mu}(p^{-1}.\dot{\mu}\ast v) = (p.\ell_{g,p,\mu})(\dot{\mu} \ast v) = (\mu\ast p.\ell_{g,p,\mu})(v) . \]

Moreover, both $V$ and $C^\la (G_0,K)$ are reflexive with their strong duals being reflexive Fr\'echet spaces.
Hence by \cite[Prop.\ 20.13]{Schneider02NonArchFunctAna} and \cite[Cor.\ 20.14]{Schneider02NonArchFunctAna}, we have a topological isomorphism
\begin{equation*}
	D(G_0) \cotimes{K} V'_b \overset{\cong}{\lra} \big( C^{\rm la}(G_0,K) \cotimes{K} V \big)'_b \,,\quad \delta\otimes\ell \lto \big[ f\otimes v \mto \delta(f) \, \ell(v)\big] .
\end{equation*}
All in all, we see that the strong dual of \eqref{Eq 3 - Exact sequence for limit of inductions} is the complex
\begin{equation}
	\begin{aligned}
		\bigoplus_{g\in G_0, p\in P_0 ,\mu \in \hy(G)} V'_b &\overset{\psi^t}{\lra} D(G_0) \cotimes{K} V'_b \overset{\iota^t}{\lra} \Big( \varinjlim_{n\in \BN} \Ind_{P^n}^{G_0} (V_n) \Big)'_b \lra 0 \\
		\sum \ell_{g,p,\mu} &\lto \sum \delta_g \ast \mu\ast \delta_p \otimes \ell_{g,p,\mu} - \delta_g \otimes \mu\ast p.\ell_{g,p,\mu} .
	\end{aligned}
\end{equation}

As $\iota$ is a closed embedding, the Hahn--Banach Theorem \cite[Cor.\ 9.4]{Schneider02NonArchFunctAna} implies that $\iota^t$ is surjective.
It follows from the open mapping theorem \cite[Prop.\ 8.6]{Schneider02NonArchFunctAna} that $\iota^t$ is strict.
Moreover, by \cite[IV. \S 4.1 Prop.\ 2]{Bourbaki87TopVectSp1to5} we have $\Ker(\iota^t) = \Im(\iota)^\perp$ where
\begin{equation*}
	\Im(\iota)^\perp \defeq \big\{ \ell \in D(G_0) \cotimes{K} V'_b \,\big\vert\, \forall v \in \Im(\iota): \ell(v) = 0 \big\} .
\end{equation*}
Since $\Im(\iota)^\perp = \Ker(\psi)^\perp$ by the algebraic exactness of \eqref{Eq 3 - Exact sequence for limit of inductions}, \Cref{Lemma A1 - Annihilator of kernel is weak closure of image of the transpose} implies that $\Ker(\iota^t) = \Ker(\psi)^\perp \subset \widebar{\Im(\psi^t)}$.
As $\Im(\psi^t) \subset \Ker(\iota^t)$ and $\Ker(\iota^t)$ is closed, we conclude that $\Ker(\iota^t) = \widebar{\Im(\psi^t)}$.

Under the equivalence of \Cref{Prop 1 - Equivalences for categories of locally analytic representations} (ii), the homomorphism $\iota^t$ becomes a homomorphism of $D(G_0)$-modules when $D(G_0) \cotimes{K} V'_b$ carries the $D(G_0)$-module structure via multiplication on the left in the first factor. 
We therefore obtain a topological isomorphism of $D(G_0)$-modules
\begin{equation*}
	\Big( \varinjlim_{n\in \BN} \Ind_{P^n}^{G_0} (V_n) \Big)'_b \cong  \big( D(G_0) \cotimes{K} V'_b \big) \big/ \Ker(\iota^t) \cong \big( D(G_0) \cotimes{K} V'_b \big) \big/ \widebar{\Im(\psi^t)} .
\end{equation*}
The submodule $\Im(\psi^t)$ in turn is generated by the elements
\begin{equation*}
	\delta_g \ast \mu \ast \delta_p \otimes \ell - \delta_g \otimes \mu\ast p.\ell \quad\text{, for $g\in G_0$, $\mu \in \hy(G)$, $p \in P_0$, $\ell \in V'_b$.}
\end{equation*}
Recall from Proposition \ref{Prop 1 - Description of product of hyperalgebra and distribution algebra of subgroup} that $D(\Fg,P_0)$ is generated by the elements of the form $\mu \ast \lambda$, for $\mu \in \hy(G)$, $\lambda \in D(P_0)$.
Together with the density of the Dirac distributions% in $D(G_0)$ respectively $D(P_0)$ (see \Cref{Prop 1 - Properties of the distribution algebra} (iv))
, it follows that $\widebar{\Im(\psi^t)}$ is equal to the closure of the $D(G_0)$-submodule generated by the vectors
\begin{equation*}
	\delta \ast \nu \otimes \ell - \delta \otimes \nu \ast \ell \quad \text{, for $\delta \in D(G_0)$, $\nu \in D(\Fg,P_0)$, $\ell \in V'_b$,}
\end{equation*}
where $V'_b$ is a separately continuous $D(\Fg,P_0)$-module via \Cref{Cor 1 - Equivalence between compatible hyperalgebra modules and modules over product of hyperalgebra and distribution algebra of subgroup}.
Using \Cref{Rmk 1 - Completion of projective tensor product over algebras} we conclude the following:

\begin{proposition}\label{Prop 3 - Isomorphism as distribution algebra modules}
	There is a canonical topological isomorphism of $D(G_0)$-modules
	\begin{equation*}
		\bigg( \varinjlim_{n\in \BN} \Ind_{P^n}^{G_0} (V_n) \bigg)'_b
		\cong D(G_0) \cotimes{D(\Fg,P_0)} V'_b .
	\end{equation*}
\end{proposition}

\subsection{The Subquotients of $H^0(\CX,\CE)'_b$ as Locally Analytic $\GL_{d+1}(K)$-Re\-pre\-sen\-ta\-tions}

On the level of locally analytic $G_0$-representations, the above \Cref{Prop 3 - Isomorphism as distribution algebra modules} already is a description of the term
\begin{equation*}
	\varinjlim_{n\in\BN}\, \Ind^{G_0}_{P^n} \left( \widetilde{H}^{q}_{\BP_K^{d-q}(\varepsilon_n)} (\BP_K^d, \CE)'_b \otimes_K v^{P^n}_{Q^n} \right)
\end{equation*}
occuring in \Cref{Thm 3 - Orliks theorem on global sections of an equivariant vector bundle on the DHS}.
However, we want to extend this to a description as locally analytic $G$-re\-pre\-sen\-tations or equivalently as $D(G)$-modules.

\begin{lemma}\label{Lemma 3 - Tensor products over distribution algebras}
	\begin{altenumerate}
		\item
		There is a topological isomorphism of $D(G_0)$-$D(P)$-bi-modules
		\begin{equation*}
			D(G_0) \botimes{D(P_0),\iota} D(P) \overset{\cong}{\lra} D(G) \,,\quad \mu \otimes \nu \lto \mu \ast \nu ,
		\end{equation*}
		(cf.\ \cite[Lemma 6.1 (i)]{SchneiderTeitelbaum05DualAdmLocAnRep} for the statement on the algebraic level).
		\item
		For a separately continuous $D(\Fg,P)$-module $M$, natural inclusion $D(G_0) \hookrightarrow D(G)$ induces a topological isomorphism
		\begin{equation}\label{Eq 3 - Isomorphism of tensor products with distribution algebras}
			D(G_0) \cotimes{D(\Fg,P_0),\iota} M \overset{\cong}{\lra} D(G) \cotimes{D(\Fg,P),\iota} M \,,\quad \delta \otimes \ell \mto \delta \otimes \ell ,
		\end{equation}
		of $D(G_0)$-modules.
	\end{altenumerate}
\end{lemma}
\begin{proof}
	For (i), essentially the proof from \cite[Lemma 6.1 (i)]{SchneiderTeitelbaum05DualAdmLocAnRep} for the algebraic statement applies:
	The Iwasawa decomposition $G = G_0 P$ with $G_0 \cap P = P_0$ (see \cite[\S 3.5]{Cartier79ReppAdicGrpsSurvey}) gives a disjoint covering $G = \bigcup_{p \in P_0 \backslash P} G_0 p$ by compact open subsets.
	In view of \Cref{Prop 1 - Properties of the distribution algebra} (iii) this yields a topological isomorphism
	\begin{align*}
		D(G) \cong \bigoplus_{p \in P_0 \backslash P} D(G_0) \ast \delta_p
	\end{align*}
	of $D(G_0)$-$D(P)$-bi-modules, and one of $D(P_0)$-$D(P)$-bi-modules
	\begin{align*}
		D(P) \cong \bigoplus_{p \in P_0 \backslash P} D(P_0) \ast \delta_p .
	\end{align*}
	Moreover, there is a topological isomorphism \cite[Lemma 1.2.13]{Kohlhaase05InvDistpAdicAnGrp}
	\begin{equation*}
		D(G_0) \indotimes \bigg( \bigoplus_{p \in P_0 \backslash P} D(P_0) \ast \delta_p \bigg)  \cong \bigoplus_{p \in P_0 \backslash P} D(G_0) \indotimes D(P_0) \ast \delta_p 
	\end{equation*}
	of $D(G_0)$-$D(P)$-bi-modules.
	Passing to the quotients we obtain the topological isomorphism
	\begin{align*}
		D(G_0) \botimes{D(P_0),\iota} D(P)
		&\cong D(G_0) \botimes{D(P_0),\iota} \bigg( \bigoplus_{p \in P_0 \backslash P} D(P_0) \ast \delta_p \bigg) \\
		&\cong \bigoplus_{p \in P_0 \backslash P} D(G_0) \botimes{D(P_0),\iota} D(P_0) \ast \delta_p \\
		&\cong \bigoplus_{p \in P_0 \backslash P} D(G_0) \ast \delta_p \cong D(G).
	\end{align*}
	of $D(G_0)$-$D(P)$-bimodules.

	For (ii), clearly the continuous homomorphism $D(G_0) \botimes{K,\iota} M  \ra D(G) \indotimes M$ induces the continuous homomorphism \eqref{Eq 3 - Isomorphism of tensor products with distribution algebras} by passing to the homomorphism between the quotients
	\begin{equation*}
		D(G_0) \botimes{D(\Fg, P_0),\iota} M \lra D(G) \botimes{D(\Fg, P),\iota} M
	\end{equation*}
	and completing.
	Moreover, using the statement of (i) and \Cref{Lemma 1 - Tensor identities for modules over locally convex algebras} (ii) we have a topological isomorphism of $D(G_0)$-modules
	\begin{equation*}
		D(G) \botimes{K,\iota} M  \cong \big( D(G_0) \botimes{D(P_0),\iota} D(P) \big) \indotimes M
		\cong  D(G_0) \botimes{D(P_0),\iota} \big( D(P)  \indotimes M \big) .
	\end{equation*}	
	Together with the well-defined continuous homomorphism of $D(G_0)$-modules
	\begin{align*}
		D(G_0) \otimes_{D(P_0),\iota} \big( D(P) \indotimes M \big) &\lra D(G_0) \botimes{D(\Fg,P_0),\iota} M,  \\
		\delta \otimes \lambda \otimes \ell   &\lto \delta \otimes  \lambda.\ell .
	\end{align*}	
	we obtain $D(G) \indotimes M \ra D(G_0) \botimes{D(\Fg,P_0),\iota} M$.
	This homomorphism factors over the quotient as
	\begin{equation*}
		D(G) \botimes{D(\Fg,P),\iota} M \lra D(G_0) \botimes{D(\Fg,P_0),\iota} M ,
	\end{equation*}
	and one verifies that the completion of the latter homomorphism is an inverse to \eqref{Eq 3 - Isomorphism of tensor products with distribution algebras}.	
	%Consider $\delta_0 \otimes \ell \in D(G_0) \cotimes{D(\Fg,P_0),\pi} V'_b$. Under \eqref{Eq 3 - Isomorphism of tensor product with distribution algebras} this gets mapped to $\delta_0 \otimes \ell \in D(G) \cotimes{D(\Fg,P),\pi} V'_b$. The "lift" of this in $D(G_0) \otimes_{D(P_0),\iota} D(P) \projotimes V'_b$ is $\delta_0  \otimes 1 \otimes \ell$ so this gets again mapped to $\delta_0 \otimes \ell \in D(G_0) \cotimes{D(\Fg,P_0),\pi} V'_b$. 
	%On the other hand, let $\delta \otimes \ell \in D(G) \cotimes{D(\Fg,P),\pi} V'_b$ with $\delta = \sum_{i=1}^n \delta_{0,i} \ast \lambda_i$, for $\delta_{0,i} \in D(G_0)$, $\lambda_i \in D(P)$, so that it is identified with $\sum_{i=1}^n \delta_{0,i} \otimes \lambda_i \otimes \ell \in D(G_0) \otimes_{D(P_0),\iota} D(P) \projotimes V'_b$. This gets mapped to $\sum_{i=1}^n \delta_{0,i} \otimes \lambda_i . \ell \in D(G_0) \cotimes{D(\Fg,P_0),\pi} V'_b$, and also to the same expression in $D(G) \cotimes{D(\Fg,P),\pi} V'_b$ under \eqref{Eq 3 - Isomorphism of tensor product with distribution algebras}. But in this quotient, it is equal to $\sum_{i=1}^n \delta_{0,i} \ast \lambda_i \otimes  \ell = \delta \otimes \ell$.
\end{proof}

\begin{theorem}\label{Thm 3 - Main theorem}
	Let $\CE$ be a $\bG$-equivariant vector bundle on $\BP_K^d$. 
	For the terms that occur on the right hand side of the description of the subquotients in \Cref{Thm 3 - Orliks theorem on global sections of an equivariant vector bundle on the DHS}, there are topological isomorphisms of $D(G)$-modules
	\begin{align*}
		&\bigg( \varinjlim_{n\in\BN}\, \Ind^{G_0}_{P^n_{d-q}} \Big( \widetilde{H}^{q}_{\BP_K^{d-q}(\varepsilon_n)} (\BP_K^d, \CE)'_b \otimes_K v^{P^n_{d-q}}_{Q^n_{d-q}} \Big) \bigg)'_b \\
		&\qquad\cong D(G) \cotimes{D(\Fg,P_{d-q}),\iota} \bigg( \widetilde{H}^{q}_{(\BP_K^{d-q})^\rig} (\BP_K^d, \CE) \cotimes{K} \Big( v^{\GL_{q}(K)}_{B_{q}} \Big)'_b \bigg) ,
	\end{align*}
	for $q=1,\ldots,d$.
	Here $P_{d-q}$ acts via inflation from the subgroup $\GL_q(K)$ of its standard Levi factor $L_{d-q}$ on $v^{\GL_{q}(K)}_{B_{q}(K)}$, and $\hy(G)$ acts trivially there\footnote{Since $v^{\GL_{q}(K)}_{B_{q}}$ carries the finest locally convex topology, $\Big( v^{\GL_{q}(K)}_{B_{q}} \Big)'$ equals the algebraic dual. The smooth dual of $v^{\GL_{q}(K)}_{B_{q}}$ is a $K$-subspace thereof, but this inclusion is not an equality in general.}.
\end{theorem}
\begin{proof}
	We keep the simplified notation with $\sdd = d-q \in \{0,\ldots,d-1\}$ fixed and $\bP \defeq \bP_{{d-q}}$, $\bQ \defeq \bQ_{{d-q}}$.
	Combining the topological isomorphism 
	\begin{equation*}
		\bigg(\varinjlim_{n\in\BN}\, \Ind^{G_0}_{P^n} (V_n)\bigg)'_b \cong D(G_0) \cotimes{D(\Fg,P_{0})} V'_b
	\end{equation*}
	from \Cref{Prop 3 - Isomorphism as distribution algebra modules} which was obtained via the topological embedding 
	\begin{equation*} 
		\iota \colon \varinjlim_{n\in\BN}\, \Ind^{G_0}_{P^n} (V_n) \lhook\joinrel\longrightarrow C^\la (G_0,K) \cotimes{K} V
	\end{equation*}
	with the statement of \Cref{Lemma 3 - Tensor products over distribution algebras} (ii), already gives a topological isomorphism of $D(G_0)$-modules
	\begin{equation*}
		\omega \colon \Big(\varinjlim_{n\in\BN}\, \Ind^{G_0}_{P^n} (V_n)\Big)'_b \overset{\cong}{\lra} D(G) \cotimes{D(\Fg,P),\iota} V'_b .
	\end{equation*}
	It remains to show that $\omega$ is $D(G)$-linear.

	To do so we first construct a $G$-equivariant homomorphism 
	\begin{equation*}
		\tilde{\iota} \colon \varinjlim_{n\in\BN}\, \Ind^{G_0}_{P^n} (V_n) \lra C^\la(G,V)
	\end{equation*}
	which is compatible with $\iota$ and the restriction map 
	\begin{equation*}
		(\blank)\res{G_0}\colon C^\la(G,V) \lra C^\la(G_0,V) \,,\quad f \lto f\res{G_0},
	\end{equation*}
	in the sense that $(\blank)\res{G_0} \circ \tilde{\iota}= \iota$.
	Let $f \in \Ind^{G_0}_{P^n} (V_n)$ correspond to the locally analytic function $f\colon G_0 \ra V_n$ so that $f(gp)= p^{-1}. f(g)$, for all $g\in G_0$, $p\in P^n$.
	We define an associated function $\tilde{f} \in C^\la(G,V)$ as follows.
	For $g\in G$ with Iwasawa decomposition $g= g_0 p$, with $g_0 \in G_0$, $p\in P$, we set
	\begin{equation*}
		\tilde{f}(g) = p^{-1}. f(g_0) \in V
	\end{equation*}
	where $p^{-1}$ acts on $f(g_0) \in V$ as explained in \Cref{Rmk 3 - P-action on inductive limit}.
	This gives a well-defined function $\tilde{f}\colon G \ra V$, because, for a different decomposition $g=g'_0 p'$ with $g'_0 =  g_0 p^{-1}_0$, $p'= p_0 p$, for some $p_0 \in P_0$, we have
	\begin{equation*}
		\tilde{f}(g'_0 p') = (p')^{-1} . f(g'_0) = p^{-1} .p_0^{-1}. f( g_0 p^{-1}_0) = p^{-1} . f(g_0) = \tilde{f}(g) 
	\end{equation*}
	as $p_0 \in P^n$.

	The function $\tilde{f}$ is locally analytic:
	For fixed $g=g_0 p$ with $g\in G_0$, $p\in P$, we consider the open neighbourhood $G_0 p$ of $g$.
	There $\tilde{f}\res{G_0  p} \colon h  p \mto \tilde{f}(hp) = p^{-1}.f(h)$ is locally analytic by \Cref{Prop 1 - Functorialities for the space of locally analytic functions} (i) since $f$ is locally analytic.
	In total we obtain the sought homomorphism
	\begin{equation*}
		\tilde{\iota} \colon \varinjlim_{n\in\BN}\, \Ind^{G_0}_{P^n} (V_n) \lra C^\la(G,V) \,,\quad f \lto \tilde{f} ,
	\end{equation*}
	with $(\blank)\res{G_0} \circ \tilde{\iota} = \iota$.

	Next we want to show that $\tilde{\iota}$ is $G$-equivariant with respect to the left-regular $G$-action\footnote{As $G$ is not compact, the left-regular $G$-representation is not locally analytic but it is continuous nevertheless.} on $C^\la(G,V)$.
	To this end, let $f \in \Ind^{G_0}_{P^m}(V_m)$ be given by $\sum_{i=1}^{s_m} g_i \bullet v_i$ so that $f(g_i p ) = p^{-1}. v_i$, for $p\in P^m$.
	As usual $g_1,\ldots,g_{s_m}$ denote coset representatives of $G_0/P^m$.

	We fix $g\in G$, and want to show that $\tilde{\iota}(g.f) = g. \tilde{\iota}(f)$.
	Let $n\geq m$ like in \Cref{Rmk 3 - G-action on inductive limit of inductions} so that $P^n/Q \subset p_{g,j} P^m/Q$ with $p_{g,j} \defeq h_j^{-1} g g_{\sigma_g(j)}$.
	Then we have seen that 
	\begin{equation*}
		g.f = g.\bigg( \sum_{i=1}^{s_m}  g_i \bullet v_i \bigg) = \sum_{j=1}^{s_n}  h_j \bullet \tau_{p_{g,j}} (v_{\sigma_g(j)}) \in  \Ind^{G_0}_{P^n} (V_n).
	\end{equation*}
	Now consider $h\in G$, and let $j\in \{1,\ldots, s_n\}$ such that $h \in h_j P^n/P$, i.e.\ $h = h_j p_{(n)} p$, for some $p_{(n)} \in P^n$, $p\in P$.
	We compute that
	\begin{align*}
		\big( \tilde{\iota}(g.f)\big) (h) &= \big( \tilde{\iota}(g.f)\big) (h_j p_{(n)} p)
		= p^{-1}. \big( (g.f)(h_j p_{(n)}) \big) \\
		&= p^{-1}. p_{(n)}^{-1}. \tau_{p_{g,j}} (v_{\sigma_g(j)})
		= (\tau_{p^{-1}} \circ \tau_{p_{(n)}^{-1}} \circ \tau_{p_{g,j}})(v_{\sigma_g(j)}) .
	\end{align*}
	On the other hand, we have
	\begin{equation*}
		g^{-1} h = (g_{\sigma_g(j)} p_{g,j}^{-1} h_j^{-1} ) ( h_j p_{(n)} p) = g_{\sigma_g(j)} p^{-1}_{g,j} p_{(n)} p .
	\end{equation*}
	Hence
	\begin{align*}
		\big(g.\tilde{\iota}(f)\big)(h) &= \tilde{\iota}(f)(g^{-1}h) 
		= \tilde{\iota}(f)(g_{\sigma_g(j)} p^{-1}_{g,j} p_{(n)} p )\\
		&= \big(p_{g,j}^{-1} p_{(n)} p \big)^{-1} . f( g_{\sigma_g(j)}) 
		= (\tau_{p^{-1}} \circ \tau_{p_{(n)}^{-1}} \circ \tau_{p_{g,j}} ) (v_{\sigma_g(j)})
	\end{align*}
	This shows that indeed $\tilde{\iota}$ is $G$-equivariant.

	We now use the injective continuous integration homomorphism from \Cref{Prop 1 - Integration map} (i) together with \cite[Cor.\ 18.8]{Schneider02NonArchFunctAna} to obtain
	\begin{equation*}
		C^\la (G,V) \lra \CL_b (D(G),V) \cong D(G)'_b \projcotimes V .
	\end{equation*}
	As $C^\la(G,V)$ is reflexive by \Cref{Cor 1 - Space of K-valued locally analytic functions is reflexive and barrelled} this yields an injective continuous homomorphism 
	\begin{equation*}
		\varinjlim_{n\in\BN}\, \Ind^{G_0}_{P^n} (V_n) \overset{\tilde{\iota}}{\lra} C^\la (G,V) \lra C^\la (G,K) \projcotimes V
	\end{equation*}
	that we continue to call $\tilde{\iota}$.

	Moreover, let $G = \bigcup_{i\in I} g_i G_0$ be a disjoint covering, for coset representatives $g_i$ of $G/G_0$, so that $C^\la(G,K) \cong \prod_{i\in I} C^\la(g_i G_0,K)$ and $D(G) \cong \bigoplus_{i\in I} D(g_i G_0,K)$.
	Then there are topological isomorphisms
	\begin{equation*}
		\begin{aligned}
			\big( C^\la(G,K) \projcotimes V \big)'_b &\cong \bigg( \prod_{i\in I} C^\la (g_i G_0, K) \cotimes{K} V \bigg)'_b &&\quad\text{, by \cite[Lemma 2.1 (iii)]{BreuilHerzig18TowardsFinSlopePartGLn}} \\
			&\cong \bigoplus_{i\in I} \big( C^\la (g_i G_0, K) \cotimes{K} V \big)'_b &&\quad\text{, by \cite[Prop.\ 9.11]{Schneider02NonArchFunctAna}} \\
			&\cong \bigoplus_{i\in I} D(g_i G_0, K) \cotimes{K} V'_b &&\quad\text{, by \cite[Prop.\ 1.1.32 (ii)]{Emerton17LocAnVect}} \\
			&\cong D(G) \indcotimes V'_b . &&\quad\text{, by \cite[Cor.\ 1.2.14]{Kohlhaase05InvDistpAdicAnGrp}}
		\end{aligned}
	\end{equation*}
	The resulting $G$-equivariant isomorphism fits into the commutative square 
	\begin{equation*}
		\begin{tikzcd}
			\big( C^\la(G_0,K) \cotimes{K} V \big)'_b \ar[r] \ar[d, "\cong"'] & \big( C^\la(G,K) \projcotimes V \big)'_b \ar[d, "\cong"] \\
			D(G_0) \cotimes{K} V'_b \ar[r] & D(G) \indcotimes V'_b
		\end{tikzcd}
	\end{equation*}
	where the horizontal homomorphism on the top is the transpose of $(\blank)\res{G_0}$, and the bottom one is induced by the embedding $D(G_0) \hookrightarrow D(G)$.
	We finally arrive at the commutative diagram
	\begin{equation*}
		\begin{tikzcd}[column sep = tiny]
			D(G_0) \cotimes{K} V'_b \ar[rrr, two heads] \ar[dddd] &[-25pt] &[+17pt] &[-25pt] D(G_0) \cotimes{D(\Fg,P_0)} V'_b \ar[dddd, "\cong"] \\
			& \big( C^\la(G_0,K) \cotimes{K} V \big)'_b \ar[lu, "\cong"'] \ar[dd] \ar[rd, two heads, "\iota^t"] & &\\[-25pt]
			& &\Big( \varinjlim_{n\in \BN} \Ind_{P^n}^{G_0} (V_n) \Big)'_b \ar[ruu, "\cong"] \ar[rdd, "\omega"'] & \\[-25pt]
			& \big( C^\la(G,K) \projcotimes V \big)'_b \ar[ld, "\cong"] \ar[ru, "\tilde{\iota}^t"'] & &\\
			D(G) \indcotimes V'_b \ar[rrr, two heads] & & &D(G) \cotimes{D(\Fg,P),\iota} V'_b .
		\end{tikzcd}
	\end{equation*}
	Since $\iota^t$ is surjective, so is $\tilde{\iota}^t$. 
	It follows that the induced topological isomorphism $\omega$ is $G$-equivariant because all the other homomorphisms in the lower ``square'' are.
	We conclude that $\omega$ is $D(G)$-linear by using the density of the Dirac distributions in $D(G)$.
\end{proof}

\subsection{The Functors $\CF_P^G$ of Orlik--Strauch}\label{Sect - The Functors of Orlik--Strauch}

In this section we want to relate our description from \Cref{Thm 3 - Main theorem} to the functors $\CF_P^G$ introduced by Orlik and Strauch in \cite{OrlikStrauch15JordanHoelderSerLocAnRep}.
To this end we suppose that the non-archimedean local field $K$ is of mixed characteristic, i.e.\ a finite extension of $\BQ_p$.
We begin by recapitulating the definition of the functors $\CF_P^G$, but for simplicity only under the assumption that the field of definition $L$ agrees with the field of coefficients $K$.
We normalize the absolute value of $K$ such that $\abs{p} = p^{-1}$.

Let $\bG$ be a connected split reductive group over $K$.
We fix a split maximal torus and a Borel subgroup $\bT \subset \bB \subset \bG$, as well as a standard parabolic subgroup $\bP \supset \bB$ with Levi decomposition $\bP = \bL_\bP \bU_\bP$ with $\bT \subset \bL_\bP$.
We assume that $\bG$ and the above subgroups already are defined over $\CO_K$.
Let $G=\bG(K)$, $P=\bP(K)$, etc.\ denote the associated locally $K$-analytic Lie groups and write $G_0 = \bG(\CO_K)$, $P_0 = \bP(\CO_K)$, etc.\ by abuse of notation.
Furthermore let $\Fg= {\rm Lie}(\bG)$, $\Fp = {\rm Lie}(\bP)$, etc.\ denote the corresponding Lie algebras.
We consider the following subcategories of modules for the universal enveloping algebra $U(\Fg)$.

\begin{definition}[{\cite[\S 2.5]{OrlikStrauch15JordanHoelderSerLocAnRep}}]
	\begin{altenumerate}
		\item
		Let $\CO^\Fp$ be the full subcategory of $U(\Fg)$-modules $M$ satisfying
		\begin{altenumeratelevel2}
			\item 
			$M$ is finitely generated as a $U(\Fg)$-module,
			\item
			viewed as an $\Fl_\bP$-module, $M$ is the direct sum of finite-dimensional simple modules,
			\item
			the action of $\Fu_\bP$ on $M$ is locally finite, i.e.\ for every $m\in M$, the $K$-vector subspace $U(\Fu_\bP) m \subset M$ is finite-dimensional.
		\end{altenumeratelevel2}
		\item
		Let ${\rm Irr}(\Fl_\bP)^{\rm fd}$ be the set of isomorphism classes of finite-dimensional irreducible $\Fl_\bP$-re\-pre\-sen\-tations.
		We define $\CO_\alg^\Fp$ to be the full subcategory of $\CO^\Fp$ of $U(\Fg)$-modules $M$ such that for a decomposition 
		\[ M = \bigoplus_{\Fa \in {\rm Irr}(\Fl_\bP)^{\rm fd}} M_\Fa \]
		into the $\Fa$-isotypic components as in (2), we have: If $M_\Fa \neq (0)$ then $\Fa$ is the Lie algebra representation induced by a finite-dimensional algebraic $\bL_\bP$-representation.
	\end{altenumerate}
\end{definition}

Note that for $\Fp = \Fb$, $\CO \defeq \CO^\Fb$ is the adaptation of the classical category $\CO$ introduced by Bernstein, Gelfand and Gelfand for semi-simple Lie algebras over the complex numbers.

The functor $\CF_P^G$ from \cite[\S 3,4]{OrlikStrauch15JordanHoelderSerLocAnRep} is now defined as follows:
Let $M \in \CO^\Fp_\alg$ and let $V$ be a smooth admissible representation of the Levi subgroup $L_\bP \subset P$ on a $K$-vector space.
We regard $V$ as a smooth $P$-representation via inflation and endow it with the finest locally convex topology so that it becomes a locally analytic $P$-representation of compact type, see \cite[\S 2]{SchneiderTeitelbaumPrasad01UgFinLocAnRep}.
By the conditions on $M \in \CO^\Fp$, there exists a finite-dimensional $U(\Fp)$-submodule $W\subset M$ which generates $M$ as a $U(\Fg)$-module, i.e.\ there is a short exact sequence of $U(\Fg)$-modules
\begin{equation*}
	0 \lra \Fd \lra U(\Fg) \botimes{U(\Fp)} W \lra M \lra 0  .
\end{equation*}
Then the $\Fp$-representation $W$ uniquely lifts to the structure of an algebraic $\bP$-representation on $W$ \cite[Lemma 3.2]{OrlikStrauch15JordanHoelderSerLocAnRep}.
There is a pairing (cf.\ \cite[(3.2.2)]{OrlikStrauch15JordanHoelderSerLocAnRep})
\begin{align*}
	\langle \blank , \!\blank \rangle_{C^\la(G,V)} \colon D(G) \botimes{D(P)} W  \times \Ind^G_P ( W'\botimes{K} V ) &\lra C^\la(G,V) , \\
	(\delta \otimes w , f) &\lto \big[ g \mto \big( \delta \ast_r (\ev_w \circ f) \big) (g) \big]	.
\end{align*}
Here we use the identification $W' \botimes{K} V \cong \CL_b (W,V)$ from \cite[Cor.\ 18.8]{Schneider02NonArchFunctAna} and denote the evaluation homomorphism by $\ev_w \colon \CL_b(W,V) \ra V$, $h \mto h(w)$.
Moreover, ``$\ast_r$'' signifies the $D(G)$-module action on $C^\la(G,V)$ induced from the right regular action of $G$ (see \Cref{Expl 1 - Examples of locally analytic representations} (ii)).
Via the injective map
\begin{equation*}
	U(\Fg) \botimes{U(\Fp)} W \longhookrightarrow D(G) \botimes{D(P)} W
\end{equation*}
one may consider the subspace of $\Ind^G_P ( W'\botimes{K} V)$ annihilated by $\Fd$ and define \cite[(4.4.1)]{OrlikStrauch15JordanHoelderSerLocAnRep}
\begin{equation*}
	\CF^G_P (M, V) \defeq \Ind^G_P (W'\botimes{K} V)^\Fd = \big\{ f\in \Ind^G_P ( W'\botimes{K} V) \,\big\vert\, \forall \Fz \in \Fd : \langle \Fz, f \rangle_{C^\la(G,V)} = 0 \big\} .
\end{equation*}
The resulting $\CF^G_P(M,V)$ is an admissible\footnote{In the sense of \cite[\S 6]{SchneiderTeitelbaum03AlgpAdicDistAdmRep}.} locally analytic $G$-representation which even is strongly admissible\footnote{Meaning that as a representation of any (equivalently, of one) compact open subgroup $H \subset G$, it is strongly admissible in the sense of \cite[\S 3]{SchneiderTeitelbaum02LocAnDistApplToGL2}, i.e.\ its strong dual is finitely generated as a $D(H)$-module.} if $V$ is of finite length \cite[Prop.\  4.8]{OrlikStrauch15JordanHoelderSerLocAnRep}.
This construction yields an exact bi-functor 
\begin{equation*}
	\CF^G_P \colon \CO^\Fp_\alg \times {\rm Rep}_K^{\sm, {\rm adm}}(L_\bP) \lra {\rm Rep}_K^{\la,{\rm adm}}(G) \,,\quad (M,V) \lto \CF_P^G(M,V) ,
\end{equation*}
which is contravariant in $M$ and covariant in $V$, see \cite[Prop.\ 4.7]{OrlikStrauch15JordanHoelderSerLocAnRep}.

In the case that $V = K$ is the trivial representation, there is another description of $\CF^G_P(M) = \CF^G_P(M,K)$, for $M \in \CO_\alg^\Fp$.
Since $M$ is the union of finite-dimensional $U(\Fp)$-submodules, via lifting each of those to an algebraic $\bP$-representation one obtains a $D(P)$-module structure on $M$, cf.\ \cite[\S 3.4]{OrlikStrauch15JordanHoelderSerLocAnRep}.
This yields a unique $D(\Fg,P)$-module structure on $M$ such that the two actions of $U(\Fp)$ agree and the Dirac distributions $\delta_p \in D(P)$ act like the group elements $p\in P$ on $M$ (\cite[Cor.\ 3.6]{OrlikStrauch15JordanHoelderSerLocAnRep}).
Then there are isomorphisms of $D(G)$- resp.\ $D(G_0)$-modules \cite[Prop.\ 3.7]{OrlikStrauch15JordanHoelderSerLocAnRep}
\begin{equation}\label{Eq 3 - Alternative description for Orlik-Strauch functors}
	\CF^G_P (M)' \cong D(G) \botimes{D(\Fg,P)} M \cong D(G_0) \botimes{D(\Fg,P_0)} M .
\end{equation}
By \cite[Thm.\ 5.1]{SchneiderTeitelbaum03AlgpAdicDistAdmRep} the locally analytic distribution algebra $D(G_0)$ is a Fr\'echet--Stein algebra since $G_0$ is compact.
It holds that $D(G) \botimes{D(\Fg,P)} M$ is a coadmissible $D(G)$-module \cite[Prop.\ 3.7]{OrlikStrauch15JordanHoelderSerLocAnRep}.
Recall that any coadmissible module over a Fr\'echet--Stein algebra can be endowed with a canonical Fr\'echet topology \cite[\S 3]{SchneiderTeitelbaum03AlgpAdicDistAdmRep}.
We note that with this topology \eqref{Eq 3 - Alternative description for Orlik-Strauch functors} is a topological isomorphism.

\begin{remark}
	It is expected that a description similar to \eqref{Eq 3 - Alternative description for Orlik-Strauch functors} holds for the case of non-trivial $V$ as well.
	In fact, in \cite{AgrawalStrauch22FromCatOLocAnRep} Agrawal and Strauch consider functors defined by 
	\begin{equation*}
		\check{\CF}_P^G (M,V) \defeq D(G) \botimes{D(\Fg,P)} \big({\rm Lift}(M,\log) \botimes{K} V' \big)
	\end{equation*}
	to generalize $\CF_P^G$ to the case of $M$ being an element of the extension closure $\CO_\alg^{\Fp,\infty}$ of $\CO_\alg^{\Fp}$.
	They show that the resulting $D(G)$-modules $\check{\CF}_P^G (M,V)$ are coadmissible, see \cite[Thm.\ 4.2.3]{AgrawalStrauch22FromCatOLocAnRep}.
\end{remark}

We now come back to the concrete situation of $\bG = \GL_{d+1,K}$.
As mentioned towards the end of \Cref{Sect 3 - The Fundamental Complex}, Orlik shows in \cite[Lemma 1.2.1]{Orlik08EquivVBDrinfeldUpHalfSp} that the $U(\Fg)$-module $\widetilde{H}^{q}_{\BP_K^{d-q}} (\BP_K^d, \CE)$ is contained in $\CO_\alg^{\Fp_{d-q}}$.
Moreover, let $\bB_q = \bB \cap \GL_{q,K}$ denote the induced Borel subgroup of $\GL_{q,K} \hookrightarrow \bL_{\bP_{d-q}}$.
Then the Steinberg representation $v^{\GL_{q}(K)}_{B_{q}}$ is an irreducible smooth representation of $\GL_{q}(K)$ (\cite[Thm.\ 2 (a)]{Casselman74pAdicVanishingThmGarland}), and hence of $P_{d-q}$.
Orlik then obtains a description of the locally analytic $G$-representation \eqref{Eq 3 - Isomorphism of locally analytic representations from Orlik} as being isomorphic to 
\begin{equation}\label{Eq 3 - Concrete description via Orlik-Strauch functors}
	\CF^{G}_{P_{d-q}} \Big( \widetilde{H}^{q}_{\BP_K^{d-q}} (\BP_K^d, \CE) , v^{\GL_{q}(K)}_{B_{q}} \Big) .
\end{equation}
The other term $H^q (\BP_K^d, \CE)' \botimes{K} v^G_{Q_{d-q}}$ of the extension \eqref{Eq 3 - Extension from spectral sequence} is a strongly admissible locally analytic $G$-representation as well, cf.\ the proof of \cite[Lemma 2.4]{OrlikStrauch15JordanHoelderSerLocAnRep}.
%$V \defeq v^G_{Q_{d-q}}$ is irreducible as a smooth $G$-representation and hence of finite length. By \cite[Prop.\ 2.2]{SchneiderTeitelbaumPrasad01UgFinLocAnRep} $V$ is strongly admissible as a smooth $G$-representation in the sense there. Hence $V'_b$ (with $V$ carrying the finest locally convex topology) is an analytic $D(G)$-module by \cite[Prop.\ 2.1]{SchneiderTeitelbaumPrasad01UgFinLocAnRep}. It follows from \cite[Lemma 1.2]{SchneiderTeitelbaumPrasad01UgFinLocAnRep} and the very definition of such a module that $V'_b$ is a finitely generated $D(G_0)$-module.
%We take a composition $E_0 \supset E_1 \supset\ldots\supset E_{r+1}= 0$ of the finite-dimensional algebraic $G$-representation $E_0 \defeq H^q (\BP_K^d, \CE)$. For the resulting extension $0 \ra E_{r} \botimes{K} V'_b \ra E_{r-1} \botimes{K} V'_b \ra (E_{r-1}/E_r) \botimes{K} V'_b \ra 0$, we know by \cite[Lemma 3.3]{SchneiderTeitelbaumPrasad01UgFinLocAnRep} that the outer terms are finitely generated $D(G_0)$-modules. It follows that the middle term is finitely generated as well, and we can conclude inductively that $E_0 \botimes{K} V'_b$ is a finitely generated $D(G_0)$-module. Hence its strong dual $H^q (\BP_K^d, \CE)' \botimes{K} v^G_{Q_{d-q}}$ is a strongly admissible $G$-representation.
It follows that the extension $(V^q/V^{q-1})'_b$ of the two terms is strongly admissible.
Since the homomorphisms between these (strongly) admissible representations in Orlik's description are necessarily strict (see \cite[Prop.\ 6.4 (ii)]{SchneiderTeitelbaum03AlgpAdicDistAdmRep}), we can conclude that there is a topological isomorphism between \eqref{Eq 3 - Concrete description via Orlik-Strauch functors} and the strong dual of
\begin{align}\label{Eq 3 - Our description}
	D(G) \cotimes{D(\Fg,P_{d-q}),\iota} \bigg( \widetilde{H}^{q}_{(\BP_K^{d-q})^\rig} (\BP_K^d, \CE) \cotimes{K} \Big( v^{\GL_{q}(K)}_{B_{q}} \Big)'_b \bigg) 
\end{align}
from our description in \Cref{Thm 3 - Main theorem} by the uniqueness of the quotient.

We recall from \Cref{Cor 2 - Projective limit description by Banach spaces of rigid local cohomology of Schubert varieties} that $\widetilde{H}^{q}_{(\BP_K^{d-q})^\rig} (\BP_K^d, \CE)$ is the completion of its subspace $\widetilde{H}^{q}_{\BP_K^{d-q}} (\BP_K^d, \CE)$.
Hence \eqref{Eq 3 - Our description} is nothing but the Hausdorff completion of
\begin{equation}\label{Eq 3 - Concrete tensored up module}
	D(G) \botimes{D(\Fg,P_{d-q}),\iota} \bigg( \widetilde{H}^{q}_{\BP_K^{d-q}} (\BP_K^d, \CE) \botimes{K,\pi} \Big( v^{\GL_{q}(K)}_{B_{q}} \Big)'_b \bigg) .
\end{equation}
Therefore it is natural to ask how this Hausdorff completion relates to the coadmissible abstract $D(G)$-module underlying \eqref{Eq 3 - Concrete tensored up module} when one considers the latter with its canonical Fr\'echet topology.
We answer this question in \Cref{Cor 3 - The two topologies on the local cohomology groups agree} by showing that they agree, i.e.\ that \eqref{Eq 3 - Concrete tensored up module} already is complete and its locally convex topology is the same as the canonical Fr\'echet topology.
\\

For the first ingredient used to this end, we return to the general setup of connected split reductive $\bG$ considered at the beginning of this section.
We fix on $U(\Fg) \subset D(G)$ the subspace topology, and likewise on all subalgebras of $U(\Fg)$. 
As $U(\Fg)$ already is contained in $D(G_0)$ and the latter is a $K$-Fr\'echet algebra, $U(\Fg)$ and its subalgebras become (jointly continuous) locally convex $K$-algebras this way.

We want to give a concrete description of this topology of $U(\Fg)$.
Let $\Fx_1,\ldots,\Fx_s$ be a $K$-basis of $\Fg$, and $\varepsilon >0$.
Via the associated PBW basis for $U(\Fg)$ consisting of $\Fx_1^{k_1} \cdots \Fx_s^{k_s}$, for $\ul{k} \in \BN_0^s$, we define the norm
\begin{equation}\label{Eq 3 - Norms on universal enveloping algebra}
	\bigg\lvert \sum_{\ul{k}\in \BN_0^{s}} a_{\ul{k}} \, \Fx_1^{k_1} \cdots \Fx_s^{k_s} \bigg\rvert_\varepsilon \defeq \sup_{\ul{k}\in \BN_0^{s}} \abs{a_{\ul{k}}} \bigg( \frac{1}{\varepsilon} \bigg)^{\abs{\ul{k}}}
\end{equation}
on $U(\Fg)$ where $\abs{\ul{k}} \defeq k_1 + \ldots +k_s$.

\begin{lemma}\label{Lemma 3 - Subspace topology of universal enveloping algebra is given by family of norms}
	The topology on $U(\Fg)$ prescribed by the family of norms $\abs{\blank}_\varepsilon$, for $0<\varepsilon<1$, equals the topology defined by regarding $U(\Fg)$ as a subspace of $D(G)$.
\end{lemma}
\begin{proof}
	By \cite[Lemma 2.4]{SchneiderTeitelbaum02LocAnDistApplToGL2}, the subspace topology of $U(\Fg) \subset D_e(G)\subset D(G)$ is defined via the family of norms
	\begin{equation*}
		\bigg\lVert \sum_{\ul{k}\in \BN_0^{s}} a_{\ul{k}} \, \Fx_1^{k_1} \cdots \Fx_s^{k_s}  \bigg\rVert_\varepsilon \defeq \sup_{\ul{k}\in \BN_0^{s}} \abs{a_{\ul{k}} \, \ul{k}!} \bigg( \frac{1}{\varepsilon} \bigg)^{\abs{\ul{k}}} ,
	\end{equation*}
	for $\varepsilon >0$, where $\ul{k}! \defeq (k_1 !) \cdots (k_s !)$. 
	Here $\lVert \blank \rVert_\varepsilon$ yields a topology finer than the one of $\lVert \blank \rVert_{\varepsilon'}$, for $\varepsilon \leq \varepsilon'$.
	We immediately find that $\lVert \blank \rVert_\varepsilon \leq \abs{\blank}_\varepsilon$ since $\abs{\ul{k}!} \leq 1$.

	On the other hand, we obtain by Legendre's formula for the $p$-adic valuation of $n!$ that $v_p(n!) \leq \frac{n}{p-1}$, for $n\in \BN$.
	Hence 
	\begin{equation*}
		\abs{\ul{k}!} = p^{- v_p(\ul{k}!)} \geq \Big( p^{\frac{1}{p-1}}\Big)^{-\abs{\ul{k}}} .
	\end{equation*}
	It follows that $\lVert \blank \rVert_\varepsilon \geq \lvert\blank \rvert_{p^\frac{1}{p-1} \varepsilon}$, and we conclude that the topology defined by the family $(\lvert\blank\rvert_{\varepsilon})_{0<\varepsilon<1}$ is equal to the topology defined by the family $(\lVert\blank\rVert_{\varepsilon})_{0<\varepsilon}$.
\end{proof}

\begin{remark}\label{Rmk 3 - Isomorphism between distribution algebra of stalk and Arens-Michael envelope}
	In \cite[Thm.\ 2.1]{Schmidt10StabFlatNonArchHyperenvAlg} Schmidt shows that there is a natural isomorphism of topological $K$-algebras between the completion of $U(\Fg)$ with respect to the family of norms $\abs{\blank}_\varepsilon$, for $0<\varepsilon<1$, and the Arens--Michael envelope $\widehat{U}(\Fg)$ of $U(\Fg)$.
	The latter is defined as the Hausdorff completion of $U(\Fg)$ with respect to all submultiplicative seminorms on $U(\Fg)$.
	Using that the completion of $U(\Fg) \subset D(G)$ with respect to the subspace topology is its closure $D_e(G)$ in $D(G)$ (see \Cref{Cor 1 - Pairing for hyperalgebra and germs of locally analytic functions}), it follows that there is a natural isomorphism of topological $K$-algebras $\widehat{U}(\Fg) \cong D_e(G)$ as well.
\end{remark}

Recalling that we fix the subspace topology on $U(\Fg)\subset D(G)$, we now consider a finitely generated $U(\Fg)$-module $M$ (e.g.\ $M \in \CO_\alg^\Fp$).
By assumption we then find some $n\in \BN$ and an epimorphism of $U(\Fg)$-modules
\begin{equation}\label{Eq 3 - Epimorphism onto finitely generated Lie algebra module}
	U(\Fg)^{\oplus n} \longtwoheadrightarrow M .
\end{equation}

\begin{lemma}\label{Lemma 3 - Topologies on Lie algebra modules}
	\begin{altenumerate}
		\item
		When $M$ is endowed with the quotient topology via \eqref{Eq 3 - Epimorphism onto finitely generated Lie algebra module}, it becomes a locally convex $U(\Fg)$-module.
		Its Hausdorff completion is a nuclear $K$-Fr\'echet space.
		
		\item
		Any homomorphism $f \colon M \ra M'$ between finitely generated $U(\Fg)$-modules is continuous and strict when $M$ and $M'$ carry the quotient topology induced by some epimorphism as above.
		In particular, this topology on $M$ does not depend on the choice of the epimorphism \eqref{Eq 3 - Epimorphism onto finitely generated Lie algebra module}.
	\end{altenumerate}
\end{lemma}
\begin{proof}
	For (i), since the quotient topology on $M$ is locally convex, we only have to show that the multiplication $U(\Fg) \times M \ra M$ is continuous.
	To do so, we consider the commutative diagram
	\begin{equation*}
		\begin{tikzcd}
			U(\Fg) \times U(\Fg)^{\oplus n} \ar[r] \ar[d, two heads] & U(\Fg)^{\oplus n} \ar[d, two heads] \\
			U(\Fg) \times M \ar[r] & M
		\end{tikzcd}
	\end{equation*}
	where the vertical maps are open by our choice of topology on $M$.
	But the multiplication map for the finite free $U(\Fg)$-module $U(\Fg)^{\oplus n}$ is continuous which implies that the multiplication map for $M$ is as well.

	The Hausdorff completion of $U(\Fg)$ is its closure $D_e(G)$ in $D(G)$ which is a nuclear $K$-Fr\'echet space.
	In particular, $D_e(G)$ is hereditarily complete, see the discussion after \cite[Def.\ 1.1.39]{Emerton17LocAnVect}.
	Hence the strict epimorphism \eqref{Eq 3 - Epimorphism onto finitely generated Lie algebra module} induces a strict epimorphism
	\[ D_e(G)^{\oplus n} \longtwoheadrightarrow \widehat{M} \]
	onto the Hausdorff completion of $M$ by \cite[Cor.\ 2.2]{BreuilHerzig18TowardsFinSlopePartGLn}.
	As a quotient of a nuclear $K$-Fr\'echet space $\widehat{M}$ then is one itself, see \cite[Prop.\ 8.3]{Schneider02NonArchFunctAna} and \cite[Prop.\ 19.4 (ii)]{Schneider02NonArchFunctAna}.

	For (ii), we argue similarly to \cite[3.7.3 Prop.\ 2]{BoschGuentzerRemmert84NonArchAna}.
	Consider an epimorphism $\varphi \colon U(\Fg)^{\oplus n} \twoheadrightarrow M$ of $U(\Fg)$-modules which endows $M$ with its topology.
	Then the homomorphism $\varphi' \defeq f \circ \varphi$ is continuous because the addition and multiplication maps of the locally convex $U(\Fg)$-module $M'$ are.
	As $\varphi$ is open by definition, it follows that $f$ is continuous.
	Furthermore, $M/\Ker(f)$ and $\Im(f)$ are isomorphic as abstract $U(\Fg)$-modules.
	The homomorphisms between these finitely generated $U(\Fg)$-modules that arrange this isomorphism are continuous.
	Therefore $M/\Ker(f)$ and $\Im(f)$ are topologically isomorphic, i.e.\ $f$ is strict.	
\end{proof}

\begin{lemma}\label{Lemma 3 - Description of tensor product of Lie algebra and distribution algebra of subgroup}
	The multiplication map $U(\Fg) \times D(P_0) \ra D(G_0)$, $(\mu,\delta)\mto \mu\ast\delta$ induces a topological isomorphism
	\begin{equation*}
		U(\Fg) \botimes{U(\Fp),\pi} D(P_0) \overset{\cong}{\lra} D(\Fg,P_0)
	\end{equation*} 
	of $U(\Fg)$-$D(P_0)$-bi-modules.
\end{lemma}
\begin{proof}
	Since the convolution product is jointly continuous here, it induces a continuous homomorphism $U(\Fg) \botimes{U(\Fp),\pi} D(P_0) \ra D(G_0)$ of $U(\Fg)$-$D(P_0)$-bi-modules.
	By \cite[Lemma 4.1]{SchmidtStrauch16DimLocAnRep} this homomorphism is injective with image being precisely $D(\Fg,P_0)$.

	On the other hand, let $\Fu^-$ denote the Lie algebra of the opposite unipotent radical $\bU^-$ of $\bP$.
	Then the direct sum decomposition $\Fg = \Fu^- \oplus \Fp$ yields an isomorphism $U(\Fg) \cong U(\Fu^-) \botimes{K} U(\Fp)$.
	As mentioned in \Cref{Rmk 3 - Isomorphism between distribution algebra of stalk and Arens-Michael envelope} the completion of $U(\Fg)$ with respect to the subspace topology $U(\Fg)\subset D(G)$ is topologically isomorphic to the Arens--Michael envelope $\widehat{U}(\Fg)$ of $U(\Fg)$ considered in \cite[\S 3.2]{Schmidt13VermModArensMichaelEnv}.
	Hence we obtain a commutative diagram
	\begin{equation*}
		\begin{tikzcd}
			\widehat{U}(\Fu^-) \projcotimes \widehat{U}(\Fp) \ar[r, "\cong"] & \widehat{U}(\Fg) \\
			U(\Fu^-) \projotimes U(\Fp) \ar[r] \ar[u, hookrightarrow] & U(\Fg) \ar[u, hookrightarrow]
		\end{tikzcd}
	\end{equation*}
	where the vertical maps are the canonical embeddings, and the upper map is a topological isomorphism by \cite[Lemma 3.2.4]{Schmidt13VermModArensMichaelEnv}.
	It follows from \cite[Lemma 2]{KopylovKuzminov00KerCokerSeqSemiAbCat} that the bijective continuous homomorphism $U(\Fu^-) \projotimes U(\Fp) \ra U(\Fg)$ is strict and therefore a topological isomorphism.

	Via this topological isomorphism, \Cref{Lemma 1 - Tensor identities for modules over locally convex algebras} (ii), and the exactness of the projective tensor product (see \cite[Lemma 2.1 (ii)]{BreuilHerzig18TowardsFinSlopePartGLn}) we obtain the topological embedding
	\begin{equation*}
		U(\Fg) \botimes{U(\Fp),\pi} D(P_0) \cong U(\Fu^-) \botimes{K,\pi} D(P_0) \longhookrightarrow D(U^-_0) \projotimes D(P_0) \longhookrightarrow D(U^-_0) \projcotimes D(P_0) .
	\end{equation*}
	The group multiplication induces an isomorphism $U^-_0 \times P_0 \cong G_0$ which yields a topological isomorphism
	\begin{equation*}
		D(U^-_0) \projcotimes D(P_0) \cong	D(U^-_0 \times P_0) \cong D(G_0) .
	\end{equation*}
	This isomorphism is given by the multiplication in $D(G_0)$	like in the definition of the convolution product in the proof of \Cref{Prop 1 - Convolution product}.
	Hence we obtain the commutative square
	\begin{equation*}
		\begin{tikzcd}
			D(U^-_0) \projcotimes D(P_0) \ar[r, "\cong"] & D(G_0) \\
			U(\Fg) \botimes{U(\Fp),\pi} D(P_0) \ar[r] \ar[u, hookrightarrow] & D(\Fg,P_0) \ar[u, hookrightarrow]
		\end{tikzcd}
	\end{equation*}
	where the vertical maps are strict monomorphisms.
	Again, \cite[Lemma 2]{KopylovKuzminov00KerCokerSeqSemiAbCat} implies that $U(\Fg) \botimes{U(\Fp),\pi} D(P_0) \ra D(\Fg,P_0)$ is strict as well, and thus a topological isomorphism as claimed.	
\end{proof}

\begin{proposition}\label{Lemma 3 - Comparision between Orlik-Strauch functors and topological tensor product}
	Let $M \in \CO^\Fp_\alg$ be endowed with the topology coming from an epimorphism \eqref{Eq 3 - Epimorphism onto finitely generated Lie algebra module}.
	Moreover, let $V$ be a strongly admissible smooth $P$-representation endowed with the finest locally convex topology and considered as a locally analytic representation of $P$.
	Then there is a topological isomorphism of separately continuous $D(G)$-modules
	\begin{equation}\label{Eq 3 - Comparison between topologies of tensoring up functor}
		D(G) \cotimes{D(\Fg,P),\iota} \big(M \projcotimes V'_b \big) \cong D(G) \botimes{D(\Fg,P)} \big(M \botimes{K} V' \big)
	\end{equation}
	in the sense that $D(G) \botimes{D(\Fg,P),\iota} \big( M \projotimes V'_b \big)$ defined according to \Cref{Def 1 - Completed projective tensor product over algebra} is complete, and its topology agrees with the canonical Fr\'echet topology induced from its underlying (abstract) $D(G)$-module being coadmissible.
\end{proposition}
\begin{proof}
	By the assumptions on $M \in \CO^\Fp_\alg$, we may find a finite-dimensional $U(\Fp)$-module $W$ and an epimorphism $\varphi \colon U(\Fg) \botimes{U(\Fp)} W \twoheadrightarrow M$ of $U(\Fg)$-modules.
	Since $U(\Fg) \botimes{U(\Fp),\pi} W \cong U(\Fg)^{\oplus \dim_K(W)}$ is a quotient of $U(\Fg) \botimes{K,\pi} W$, the composition
	\begin{equation*}
		U(\Fg) \botimes{K,\pi} W \longtwoheadrightarrow U(\Fg) \botimes{U(\Fp),\pi} W \overset{\varphi}{\longtwoheadrightarrow} M
	\end{equation*}
	is a strict epimorphism by \Cref{Lemma 3 - Topologies on Lie algebra modules} (i).
	Hence it follows from \cite[Lemma 2 (3)]{KopylovKuzminov00KerCokerSeqSemiAbCat} that $\varphi$ is a strict epimorphism, too.

	Using \Cref{Lemma 3 - Description of tensor product of Lie algebra and distribution algebra of subgroup} and \Cref{Lemma 1 - Tensor identities for modules over locally convex algebras} there is a  topological isomorphism
	\begin{align*}
		D(\Fg,P_0) \botimes{D(P_0),\pi} W 
		&\cong \big( U(\Fg) \botimes{U(\Fp),\pi} D(P_0) \big) \botimes{D(P_0),\pi} W  \\
		&\cong  U(\Fg) \botimes{U(\Fp),\pi} \big(D(P_0)  \botimes{D(P_0),\pi} W \big)	\\
		&\cong U(\Fg) \botimes{U(\Fp),\pi} W
	\end{align*}
	which maps $\mu \ast \delta \otimes w$ to $\mu \otimes \delta \ast w$, for $\mu \in U(\Fg)$, $\delta\in D(P_0)$, $w\in W$.
	Therefore $\varphi$ induces a strict epimorphism
	\begin{equation*}
		D(\Fg,P_0) \botimes{D(P_0),\pi} W  \longtwoheadrightarrow M
	\end{equation*}
	which one checks to be $D(\Fg,P_0)$-linear via the method of the proof of Proposition \ref{Prop 1 - Description of product of hyperalgebra and distribution algebra of subgroup}.

	As taking the projective tensor product is exact (see \cite[Lemma 2.1 (ii)]{BreuilHerzig18TowardsFinSlopePartGLn}), we obtain a $D(\Fg,P_0)$-linear strict epimorphism
	\begin{equation*}
		D(\Fg,P_0) \botimes{D(P_0),\pi} \big(  W \botimes{K,\pi} V'_b\big)  \cong \big( D(\Fg,P_0) \botimes{D(P_0),\pi} W \big) \botimes{K,\pi} V'_b \longtwoheadrightarrow M \botimes{K,\pi} V'_b 
	\end{equation*}
	using \Cref{Lemma 1 - Tensor identities for modules over locally convex algebras} (ii) once again.
	Here we extend the trivial $U(\Fp)$-action on $V'_b$ (recall that $V$ is a smooth representation) to the trivial $U(\Fg)$-action.
	We note that this extended action together with the given $P$-action on $V$ satisfies the condition (2) of \Cref{Def 1 - Compatible hyperalgebra modules}.
	%Use here that the conjugation action $f\mto f(g^{-1}\blank g)$ preserves the order $n$ of vanishing of $f$ at $e\in G$ ($f\in \Fm_e^n$). Hence $\Ad_n(g)(\mu) \in K \cdot 1$ if and only if $\mu \in K \cdot 1$.

	Moreover, the locally convex $D(P_0)$-module $W \botimes{K,\pi} V'_b$ is finitely generated as an abstract $D(P_0)$-module, cf.\ the proof of \cite[Prop.\ 4.1.5]{AgrawalStrauch22FromCatOLocAnRep} and \cite[Prop.\ 6.4.1]{AgrawalStrauch22FromCatOLocAnRep}.
	Such an epimorphism $D(P_0)^{\oplus n} \twoheadrightarrow W \botimes{K,\pi} V'_b$, for some $n\in \BN$, is necessarily strict by the open mapping theorem \cite[Prop.\ 8.6]{Schneider02NonArchFunctAna} since $W\botimes{K,\pi} V'_b$ is a $K$-Fr\'echet space.
	Therefore we obtain a commutative diagram 
	\begin{equation*}
		\begin{tikzcd}
			D(\Fg,P_0) \botimes{K,\pi} D(P_0)^{\oplus n} \ar[r, two heads] \ar[d, two heads] & D(\Fg,P_0) \botimes{K,\pi} \big(W \botimes{K,\pi} V'_b \big) \ar[d, two heads] \\
			D(\Fg,P_0)^{\oplus n} \ar[r] & D(\Fg,P_0) \botimes{D(P_0),\pi} \big(W\botimes{K,\pi} V'_b \big)
		\end{tikzcd}
	\end{equation*}
	where the top homomorphism is a strict epimorphism by \cite[Lemma 2.1 (ii)]{BreuilHerzig18TowardsFinSlopePartGLn}.
	It follows from \cite[Lemma 2]{KopylovKuzminov00KerCokerSeqSemiAbCat} that the bottom epimorphism is strict as well.

	In total we thus arrive at a strict epimorphism (cf.\ \cite[Prop.\ 4.1.5]{AgrawalStrauch22FromCatOLocAnRep} for the statement that the abstract module is finitely presented)
	\begin{equation*}
		\psi \colon D(\Fg,P_0)^{\oplus n} \longtwoheadrightarrow M \botimes{K,\pi} V'_b .
	\end{equation*}
	Using the exactness of the projective tensor product again, we obtain the commutative diagram
	\begin{equation*}
		\begin{tikzcd}
			D(G_0) \projotimes D(\Fg,P_0)^{\oplus n} \ar[d, two heads] \ar[r, two heads] & D(G_0) \projotimes \big(M \projotimes V'_b \big) \ar[d, two heads]  \\
			D(G_0)^{\oplus n} \ar[r, "\widebar{\psi}"] & D(G_0) \botimes{D(\Fg,P_0),\pi} \big(M \projotimes V'_b \big) 
		\end{tikzcd}
	\end{equation*}
	where all maps except a priori the bottom one are strict epimorphisms.
	Similarly to before one argues that the epimorphism $\widebar{\psi}$ is strict.
	On the other hand, the (abstract) $D(G_0)$-module $D(G_0) \botimes{D(\Fg,P_0)} \big(M \botimes{K} V'_b \big)$ is coadmissible by \cite[Thm.\ 4.2.3]{AgrawalStrauch22FromCatOLocAnRep}.
	Therefore $\widebar{\psi}$ also is strict when $D(G_0) \botimes{D(\Fg,P_0)} \big(M \botimes{K} V'_b \big)$ carries its canonical Fr\'echet topology (see \cite[\S 3]{SchneiderTeitelbaum03AlgpAdicDistAdmRep}).
	By the uniqueness of the quotient we conclude that this Fr\'echet topology agrees with the topology on $D(G_0) \botimes{D(\Fg,P_0),\pi} (M \projotimes V'_b) $ from \Cref{Def 1 - Completed projective tensor product over algebra}. 
	In particular, the latter already is complete so that taking the Hausdorff completion is redundant:
	\begin{equation*}
		D(G_0) \cotimes{D(\Fg,P_0)} \big(M \projcotimes V'_b \big) = D(G_0) \botimes{D(\Fg,P_0),\pi} \big(M \projotimes V'_b \big) .
	\end{equation*}
	Note here that since $\widehat{M}$ is a $K$-Fr\'echet space by \Cref{Lemma 3 - Topologies on Lie algebra modules} (i), $M \projcotimes V'_b \cong \widehat{M} \projcotimes V'_b$ is so as well by \cite[Lemma 19.10 (i)]{Schneider02NonArchFunctAna} and the discussion after \cite[Prop.\ 17.6]{Schneider02NonArchFunctAna}.
	Therefore the projective and inductive tensor product of $D(G_0)$ and $M \projcotimes V'_b$ over $D(\Fg,P_0)$ agree.
	Finally, the Iwasawa decomposition yields the isomorphism
	\begin{equation*}
		D(G) \botimes{D(\Fg,P)} \big(M \botimes{K} V' \big) \cong D(G_0) \botimes{D(\Fg,P_0)} \big(M \botimes{K} V' \big)
	\end{equation*}
	of $D(G_0)$-modules (cf.\ \cite[\S 4.2.2]{AgrawalStrauch22FromCatOLocAnRep}), and \Cref{Lemma 3 - Tensor products over distribution algebras} (ii) yields the topological isomorphism
	\begin{equation*}
		D(G) \cotimes{D(\Fg,P),\iota} \big(M \cotimes{K,\pi} V'_b \big) \cong D(G_0) \cotimes{D(\Fg,P_0)} \big(M \cotimes{K,\pi} V'_b \big)
	\end{equation*}
	of separately continuous $D(G_0)$-modules.
	Together they give the topological isomorphism \eqref{Eq 3 - Comparison between topologies of tensoring up functor} of separately continuous $D(G)$-modules.
\end{proof}

We return to the concrete setting of $\bG = \GL_{d+1,K}$ and the parabolic subgroup $\bP = \bP_\sdd$, for $\sdd \in \{ 0,\ldots, d-1\}$, from \Cref{Thm 3 - Main theorem}.
We want to apply \Cref{Lemma 3 - Comparision between Orlik-Strauch functors and topological tensor product} to show that \eqref{Eq 3 - Concrete tensored up module} is complete and its locally convex topology agrees with the canonical Fr\'echet topology.

Let $I\subset \{0,\ldots,d\}$ be a non-empty subset, and let $U_I \subset \BP_K^d$ denote the intersection of the corresponding principal open subsets as considered in \Cref{Sect 2 - Strictness of certain Cech complexes}.
Via the countable admissible covering
\begin{equation*}
	U_I^\rig  = \bigcup_{\substack{ 0<\varepsilon<1 \\ \varepsilon \in \abs{\widebar{K}}}} U_{I,\varepsilon} ,
\end{equation*}
and \Cref{Lemma 2 - Density of algebraic sections in analytic sections} we regard $\CE(U_I)$ as a subspace of $\CE(U_I^\rig) = \varprojlim_{\varepsilon \searrow 0} \CE(U_{I,\varepsilon})$.
Then the topology of $\CE(U_I)$ is induced by the family of norms $\abs{\blank}_\varepsilon$, for $0<\varepsilon<1$, defined in \eqref{Eq 2 - Fixed norm on sections}.

\begin{proposition}\label{Prop 3 - Sections on principal open subset are finite module over the universal enveloping algebra}
	For every non-empty subset $I \subset \{ 0,\ldots,d \}$, there exist $m\in \BN$ and a surjective homomorphism of $U(\Fg)$-modules
	\begin{equation*}
		\varphi \colon U(\Fg)^{\oplus m} \longtwoheadrightarrow \CE(U_I)
	\end{equation*}
	such that $\varphi$ is continuous and strict when $U(\Fg) \subset D(G)$ and $\CE(U_I) \subset \CE(U_I^\rig)$ carry their respective subspace topologies.
\end{proposition}

\begin{corollary}\label{Cor 3 - The two topologies on the local cohomology groups agree}
	The $U(\Fg)$-module $\widetilde{H}^{d-\sdd}_{\BP_K^\sdd} (\BP_K^d, \CE)$ is finitely generated, and its locally convex topology induced by an epimorphism from a finite free $U(\Fg)$-module via \eqref{Eq 3 - Epimorphism onto finitely generated Lie algebra module} agrees with the subspace topology $\widetilde{H}^{d-\sdd}_{\BP_K^\sdd} (\BP_K^d, \CE) \subset \widetilde{H}^{d-\sdd}_{(\BP_K^{\sdd})^\rig} (\BP_K^d, \CE)$ from \Cref{Cor 2 - Projective limit description by Banach spaces of rigid local cohomology of Schubert varieties}.
	Consequently, the separately continuous $D(G)$-module \eqref{Eq 3 - Concrete tensored up module} already is complete and its topology agrees with the canonical Fr\'echet topology.
\end{corollary}
\begin{proof}[{Proof of \Cref{Cor 3 - The two topologies on the local cohomology groups agree}}]
	Recall that we defined the \v{C}ech complex $C^\bullet(\CU, \CE)$ for the covering $\CU$ given by
	\begin{equation*}
		\BP_K^d \setminus \BP_K^\sdd = \bigcup_{i= \sdd+1}^d U_i .
	\end{equation*}
	We fix on $C^q (\CU,\CE)$, for $q\geq 0$, the locally convex topology given via the subspace topologies of $\CE(U_I) \subset \CE(U_I^\rig)$, for all non-empty $I \subset \{\sdd+1,\ldots,d\}$.
	Then we have seen in \Cref{Cor 2 - Projective limit description by Banach spaces of rigid local cohomology of Schubert varieties} (or rather in the proof of \Cref{Prop 2 - Projective limit description of local cohomology wrt open tubes}) that the thereby induced topology on $\widetilde{H}^{d-\sdd}_{\BP_K^\sdd} (\BP_K^d, \CE) $ agrees with the topology of the latter as a subspace of $\widetilde{H}^{d-\sdd}_{(\BP_K^{\sdd})^\rig} (\BP_K^d, \CE)$.

	On the other hand, the epimorphisms of $U(\Fg)$-modules onto $\CE(U_I)$ from \Cref{Prop 3 - Sections on principal open subset are finite module over the universal enveloping algebra} yield $m_q \in \BN$ and epimorphisms 
	\begin{equation}\label{Eq 3 - Epimorphism onto Cech cochains}
		\varphi_q \colon U(\Fg)^{\oplus m_q} \longtwoheadrightarrow C^q(\CU,\CE) \quad\text{, for all $q\geq 0$,}
	\end{equation}
	which are strict with regard to the subspace topologies of $U(\Fg) \subset D(G)$ and $\CE(U_I) \subset \CE(U_I^\rig)$.
	Therefore the topology on $C^q(\CU,\CE)$ we fixed above agrees with the one induced via \eqref{Eq 3 - Epimorphism onto Cech cochains} by the uniqueness of the quotient.

	Since $U(\Fg)$ is noetherian and the differentials of $C^\bullet(\CU,\CE)$ are $U(\Fg)$-linear, the subspace $Z^q(\CU,\CE) \subset C^q(\CU,\CE)$ of \v{C}ech cocycles is a finitely generated $U(\Fg)$-module, too.
	\Cref{Lemma 3 - Topologies on Lie algebra modules} (ii) implies that its topology induced by some epimorphism from a finite free $U(\Fg)$-module agrees with its topology as a subspace of $C^q(\CU,\CE)$.
	We conclude that $H^q(\BP_K^d\setminus \BP_K^\sdd, \CE)$ is finitely generated as a $U(\Fg)$-module, and its topology via an epimorphism from a finite free $U(\Fg)$-module agrees with its topology as a subquotient of $C^q(\CU,\CE)$.
	Arguing similarly for $\widetilde{H}^{d-\sdd}_{\BP_K^\sdd} (\BP_K^d, \CE) $, we find that its topology induced via some epimorphism from a finite free $U(\Fg)$-module agrees with the topology as a subspace of $\widetilde{H}^{d-\sdd}_{(\BP_K^{\sdd})^\rig} (\BP_K^d, \CE)$.

	\Cref{Lemma 3 - Comparision between Orlik-Strauch functors and topological tensor product} applied to $M = \widetilde{H}^{d-\sdd}_{\BP_K^\sdd} (\BP_K^d, \CE) $ and $V = v^{\GL_{d-\sdd}(K)}_{B_{d-\sdd}}$ then shows the last statement.
\end{proof}

\begin{proof}[{Proof of \Cref{Prop 3 - Sections on principal open subset are finite module over the universal enveloping algebra}}]

We begin by fixing a suitable $K$-basis of $\Fg$.
For $\alpha_{u,v}\in \Phi$ with $(u,v) \in \{0,\ldots,d\}^2$, $u\neq v$, let $L_{(u,v)}$ be the standard generator of the root space $\Fg_{\alpha_{u,v}}$.
Moreover, let $L_0,\ldots,L_d$ be the standard basis of $\Ft$.
Then the $L_{(u,v)}$ together with the $L_i$ constitute a $K$-basis of $\Fg$.
The action of $U(\Fg)$ on $\CO(U_I)$ is given by \cite[(1.5)]{Orlik08EquivVBDrinfeldUpHalfSp}
\begin{align*}
	L_{(u,v)}^k . X^\ul{\mu} &= \frac{\mu_v!}{(\mu_v - k)!} \, X^{\ul{\mu} +k \alpha_{u,v}} ,\\
	L^k_j . X^\ul{\mu} &= \mu_j^k \, X^\ul{\mu} ,
\end{align*}
for $k\in \BN_0$, where we use the convention that $\frac{\mu_v!}{(\mu_v-k)!} = \mu_v\, (\mu_v -1) \cdots (\mu_v - k+1)$.

Let the fixed subset $I \subset \{0,\ldots,d\}$ be $I = \{i_1,\ldots,i_r\}$ with pairwise distinct $i_j$.
Let $E$ be a finitely generated graded $K[X_0,\ldots, X_d]$-module such that $\CE$ is associated with $E$, e.g.\
\begin{equation*}
	E = \bigoplus_{k \in \BN_0} \Gamma \big( \BP_K^d , \CE(k) \big) ,
\end{equation*}
cf.\ \cite[\href{https://stacks.math.columbia.edu/tag/0BXD}{Tag 0BXD}]{StacksProject}.
As $\CE(U_I)$ is the homogeneous localization $(E_{X_{i_1}\cdots X_{i_r}})^0$ which is a finitely generated $\CO(U_I)$-module, we find homogeneous elements $v_1,\ldots,v_n\in E$ and $l_1,\ldots,l_n \in \BN_0$ with $\deg(v_k) = r l_k$ such that $\CE(U_I)$ is generated by 
\begin{equation*}
	\frac{v_k}{(X_{i_1}\cdots X_{i_r})^{l_k}} \quad\text{, for $k=1,\ldots,n$,}
\end{equation*}
as an $\CO(U_I)$-module.
For $N\defeq \max_{k=1}^n l_k$, we set $E' \defeq \bigoplus_{l=0}^N E_{rl}$ and may assume that $v_1,\ldots,v_n$ is a $K$-basis of $E'$.
Since $\bG$ acts on $E$ by homogeneous endomorphisms, see \cite[\S 2.3.1]{Kuschkowitz16EquivVBRigidCohomDrinfeldUpHalfSpFF}, so does $\Fg$ and $E'$ is a $U(\Fg)$-submodule of $E$.
Moreover, for $\varepsilon>0$ small enough, we may suppose that the estimate
\begin{equation}\label{Eq 3 - Smallness condition on epsilon}
	\big\lVert L_{(u,v)}. v_k \big\rVert_{E'} \leq \frac{1}{\varepsilon} \quad\text{, for all $(u,v)\in \{0,\ldots,d\}^2$ with $u\neq v$, and $k=1,\ldots,n$,}
\end{equation}
for the maximum norm on $E'$ with respect to the basis $v_1,\ldots,v_n$ holds.

We recall the notation
\begin{align*}
	\Lambda_I &\defeq \bigg\{ \ul{\mu} \in \BZ^{d+1} \bigg\vert \sum_{j=0}^d \mu_j =0 , \forall j\in \{0,\ldots,d\} \setminus I : \mu_j \geq 0 \bigg\} , \\
	\Delta_d &\defeq \left\{ \ul{\mu} \in \BZ^{d+1} \middle{|} \forall j=0,\ldots,d: \abs{ \mu_j} \leq d \right\} .
\end{align*}
As a candidate for the sought $U(\Fg)$-linear epimorphism $\varphi$ we consider
\begin{align*}
	\varphi \colon U(\Fg)^{(I)} \defeq \bigoplus_{\substack{\ul{\nu} \in \Lambda_I \cap \Delta_d \\ k=1,\ldots,n}} U(\Fg) \, e_{\ul{\nu},k} \lra \CE(U_I) \,,\quad
			\FX_{\ul{\nu},k} \, e_{\ul{\nu},k} \lto \FX_{\ul{\nu},k}. \bigg(X^\ul{\nu} \, \frac{v_k}{(X_{i_1}\cdots X_{i_r})^{l_k}} \bigg) .
\end{align*}
Let $\abs{\blank}_\varepsilon$ denote the norm on $U(\Fg)$ defined in \eqref{Eq 3 - Norms on universal enveloping algebra} in terms of the $K$-basis of $\Fg$ chosen in the beginning of the proof.
Then it follows from \Cref{Lemma 3 - Subspace topology of universal enveloping algebra is given by family of norms} that the topology on the locally convex direct sum $U(\Fg)^{(I)}$ is induced by the family of norms
\begin{equation*}
	\bigg\lvert \sum_{\substack{\ul{\nu} \in \Lambda_I \cap \Delta_d \\ k=1,\ldots,n}} \FX_{\ul{\nu},k} \, e_{\ul{\nu},k} \bigg\rvert_\varepsilon
		\defeq \max_{\substack{\ul{\nu} \in \Lambda_I \cap \Delta_d \\ k=1,\ldots,n}} \lvert \FX_{\ul{\nu},k} \rvert_\varepsilon 
		\quad \text{, for $0<\varepsilon < 1$.}
\end{equation*}

\begin{lemma}
	The map $\varphi$ is a continuous homomorphism of locally convex $K$-vector spaces.
\end{lemma}
\begin{proof}
	It suffices to show that, for all $v\in \CE(U_I)$, the map $U(\Fg) \ra \CE(U_I)$, $\FX \mto \FX.v$, is continuous.
	This in turn follows once we see that, for all $v\in \CE(U_I)$ and $0< \varepsilon < 1$, there exists $0< \delta < 1$ such that
	\begin{equation}\label{Eq 3 - Estimate for multiplication map}
		\lvert \FX.v \rvert_\varepsilon \leq \lvert \FX \rvert_\delta \lvert v \rvert_\varepsilon \quad\text{, for all $\FX \in U(\Fg)$.}
	\end{equation}
	We denote the $K$-basis elements of $\Fg$ fixed earlier by
	\begin{equation*}
		\big\{\Fx_1,\ldots,\Fx_s \big\} \defeq \big\{ L_{(u,v)} \,\big\vert\, \alpha_{u,v} \in \Phi \big\} \cup \big\{L_0,\ldots,L_d \big\} .
	\end{equation*}
	Moreover, let $w_1,\ldots,w_m$ be an $\CO(U_I)$-basis of $\CE(U_I)$ which defines the family of norms $\abs{\blank}_\varepsilon$ as in \eqref{Eq 2 - Fixed norm on sections}.
	For a given $0< \varepsilon < 1 $, we choose $\delta > 0 $ small enough such that $\delta \leq \varepsilon $ and
	\begin{equation*}
		\lvert \Fx_j w_l \rvert_\varepsilon \leq \frac{1}{\delta} \quad\text{, for all $j=1,\ldots,s$ and $l=1,\ldots,m$.}
	\end{equation*}

	We first want to show that
	\begin{equation}\label{Eq 3 - Norm estimate for Lie algebra basis element}
		\lvert \Fx . f \rvert_\varepsilon \leq \frac{1}{\varepsilon} \, \lvert f \rvert_\varepsilon \quad \text{, for all $\Fx\in \{\Fx_1,\ldots,\Fx_s\}$ and $f \in \CO(U_I)$.}
	\end{equation}
	To this end it suffices to consider monomials $f = X^\ul{\mu}$, for $\ul{\mu} \in \Lambda_I$.
	For $\Fx = L_j$, we have
	\begin{equation*}
		\lvert L_j . X^\ul{\mu} \rvert_\varepsilon = \lvert \mu_j \, X^\ul{\mu} \rvert_\varepsilon = \abs{\mu_j} \lvert X^\ul{\mu} \rvert_\varepsilon \leq \lvert X^\ul{\mu} \rvert_\varepsilon \leq  \frac{1}{\varepsilon}\, \lvert X^\ul{\mu} \rvert_\varepsilon .
	\end{equation*}
	Moreover, for $\Fx = L_{(u,v)}$, we compute
	\begin{align*}
		\lvert L_{(u,v)} . X^\ul{\mu} \rvert_\varepsilon
		= \lvert \mu_v \, X^{\ul{\mu} + \alpha_{u,v}} \rvert_\varepsilon
		= \abs{\mu_v} \bigg( \frac{1}{\varepsilon} \bigg)^{\abs{\max(0,\ul{\mu}+\alpha_{u,v})}}
		\leq \bigg( \frac{1}{\varepsilon} \bigg)^{\abs{\max(0,\ul{\mu}+\alpha_{u,v})}}.
	\end{align*}
	The sought inequality $\lvert L_{(u,v)} . X^\ul{\mu} \rvert_\varepsilon \leq  \big( \frac{1}{\varepsilon} \big)^{1+ \abs{\max(0,\ul{\mu})}} = \frac{1}{\varepsilon} \lvert X^\ul{\mu} \rvert_\varepsilon $ then follows from
	\begin{equation*}
		\max(0,\mu_u+1) + \max(0,\mu_v-1) \leq 1 + \max(0,\mu_u) + \max(0,\mu_v) .
	\end{equation*}

	Now, we turn towards $\CE(U_I) \cong \CO(U_I) \botimes{K} \CE(x_i)$.
	To prove \eqref{Eq 3 - Estimate for multiplication map} it suffices to show the following inequalities on the PBW basis of $U(\Fg)$:
	\begin{equation*}
		\bigg\lvert (\Fx_1^{k_1} \cdots \Fx_s^{k_s}) . \bigg(\sum_{l=1}^m f_l \,w_l \bigg) \bigg\rvert_\varepsilon
		\leq \bigg(\frac{1}{\delta} \bigg)^{\abs{\ul{k}}} \max_{l=1}^m \, \lvert f_l \rvert_\varepsilon
		= \lvert \Fx_1^{k_1} \cdots \Fx_s^{k_s} \rvert_\delta \, \bigg\lvert \sum_{l=1}^m f_l \,w_l \bigg\rvert_\varepsilon ,
	\end{equation*}
	for all $\ul{k} \in \BN_0^s$ and $ \sum_{l=1}^m f_l \,w_l\in \CE(U_I)$.
	We do so via induction on $\abs{\ul{k}}$.
	With the base case being clear, let $\ul{k} \in \BN^s_0$ with $\abs{\ul{k}}>0$, and $\sum_{l=1}^m f_l\, w_l \in \CE(U_I)$.
	Let $j\in \{1,\ldots,s\}$ be maximal such that $k_j > 0$ and $k_{j'} =0$, for all $j' > j$.
	By the Leibniz product rule we have 
	\begin{align*}
		\bigg\lvert \Fx_j . \bigg( \sum_{l=1}^m f_l \, w_l \bigg) \bigg\rvert_\varepsilon
		&= \bigg\lvert \sum_{l=1}^m (\Fx_j. f_l)\, w_l + f_l \,(\Fx_j .w_l) \bigg\rvert_\varepsilon 
		\leq \max_{l=1}^m \, \max \big( \, \lvert \Fx_j .f_l \rvert_\varepsilon \,,\, \lvert f_l \rvert_\varepsilon \, \lvert \Fx_j. w_l \rvert_\varepsilon \big) \\
		&\leq \max_{l=1}^m \, \max \bigg( \frac{1}{\varepsilon} \, \lvert f_l \rvert_\varepsilon \,,\, \lvert f_l \rvert_\varepsilon  \,\frac{1}{\delta} \bigg)
		\leq \frac{1}{\delta}  \max_{l=1}^m \, \lvert f_l \rvert_\varepsilon .
	\end{align*}
	Here we have additionally used \eqref{Eq 3 - Norm estimate for Lie algebra basis element}, $\delta \leq \varepsilon $, and that $\lvert \Fx_j. w_l \rvert_\varepsilon \leq \frac{1}{\delta}$, for all $l=1,\ldots,m$.
	Together with the induction hypothesis we then conclude that
	\begin{align*}
		\bigg\lvert (\Fx_1^{k_1}\cdots \Fx_s^{k_s}) . \bigg( \sum_{l=1}^m f_l \, w_l  \bigg) \bigg\rvert_\varepsilon
		&= \Bigg\lvert (\Fx_1^{k_1}\cdots \Fx_j^{k_j-1 }) . \Bigg( \Fx_j . \bigg( \sum_{l=1}^m f_l \, w_l \bigg) \Bigg) \Bigg\rvert_\varepsilon \\
		&\leq \lvert \Fx_1^{k_1}\cdots \Fx_j^{k_j-1 } \rvert_\delta \, \bigg\lvert \Fx_j . \bigg( \sum_{l=1}^m f_l \, w_l  \bigg) \bigg\rvert_\varepsilon
		%		&\leq \bigg(\frac{1}{\varepsilon}\bigg)^{\abs{\ul{k}}-1} \bigg\lVert \Fx_j . \sum_{l=1}^m f_l \, w_l  \bigg\rVert
		\leq \bigg(\frac{1}{\delta}\bigg)^{\abs{\ul{k}}} \max_{l=1}^m  \, \lvert f_l \rvert_\varepsilon .
	\end{align*}
\end{proof}

We also introduce, for $\ul{\mu} \in \Lambda_I$, the auxiliary constants
\begin{equation*}
	A(\ul{\mu}) \defeq \prod_{k=0}^N  \Bigg( \prod_{\substack{j=0 \\ j\in I}}^d \frac{1}{\abs{\min(0, \mu_j +1 -k)!}} \cdot \prod_{\substack{j=0 \\ j\notin I}}^d \frac{1}{\abs{\min(0, \mu_j +1)!}} \Bigg) .
\end{equation*}
Here and later we use the conventions
\begin{align*}
	(-n)! 				&\defeq (-n)(-n+1)\cdots(-1) = (-1)^n n! \quad\text{, for $n\in \BN_0$, and} \\
	\max(0,\ul{\mu})	&\defeq \big(\max(0,\mu_0),\ldots, \max(0,\mu_d) \big) .
\end{align*}

\begin{lemma}\label{Lemma 3 - Estimates for auxiliary function}
	\begin{altenumerate}
		\item
			For all $l = 0,\ldots,N$, $\ul{\mu} \in \Lambda_I$, and $u,v \in \{0,\ldots,d\}$ with $u \neq v$, $\mu_u >1$, $\mu_v <-1$, it holds that
			\begin{equation*}
				\frac{1}{\abs{\mu_v +1 - l}} \, A(\ul{\mu} - \alpha_{u,v}) \leq A(\ul{\mu}) .
			\end{equation*}
		\item
			There are constants $C>0$, $a\geq 1$ such that $A(\ul{\mu}) \leq C a^\abs{\max(0,\ul{\mu})}$, for all $\ul{\mu}\in \Lambda_I$.
	\end{altenumerate}
\end{lemma}
\begin{proof}
	Writing $\ul{\mu}' \defeq \ul{\mu}-\alpha_{u,v}$ in the setting of (i), it suffices to show that the rational number
	\begin{align*}
		\frac{1}{\mu_v +1 - l} \,\prod_{k=0}^N  \Bigg( \prod_{\substack{j=0 \\ j\in I}}^d \frac{\min(0, \mu_j +1 -k)!}{\min(0, \mu'_j +1 -k)!} \cdot \prod_{\substack{j=0 \\ j\notin I}}^d \frac{\min(0, \mu_j +1)!}{\min(0, \mu'_j +1)!} \Bigg)
	\end{align*}
	in fact is an integer.
	We may focus on the terms in the product for $j=u$ and $j=v$.
	But using $\mu_u >1$ and $\mu_v <-1$, we indeed find that
	\begin{align*}
		&\frac{1}{\mu_v +1 - l} \,\prod_{k=0}^N   \frac{\min(0, \mu_v +1 -k)!}{\min(0, \mu'_v +1 -k)!} \cdot \frac{\min(0, \mu_u +1)!}{\min(0, \mu'_u +1)!} \\
			&\qquad = \frac{1}{\mu_v +1 - l} \,\prod_{k=0}^N   \frac{(\mu_v +1 -k)!}{(\mu_v +2 -k)!} \cdot \frac{0!}{0!} 
%			&\qquad	= \frac{(\mu_v +1 - l)!}{(\mu_v +1 - l)!} \,\prod_{\substack{k=0 \\ k\neq l}}^N   \frac{(\mu_v +1 -k)!}{(\mu_v +2 -k)!} \cdot \frac{0!}{0!} \\
					= \prod_{\substack{k=0 \\ k\neq l}}^N (\mu_v +1 - k) \in \BZ .
	\end{align*}

	For (ii), we again use the estimate $v_p(n!) \leq \frac{n}{p-1}$ of the $p$-adic valuation with $n\in \BN_0$.
	This yields, for $j\in I$,
	\begin{align*}
		v_p \big( \min(0,\mu_j +1 - k)! \big)
			\leq \frac{- \min (0,\mu_j +1 -k)}{p-1}
			\leq \frac{  - \min(0,\mu_j) + k }{p-1}.
	\end{align*}
	It follows from similar estimates for $j \notin I$ and from $\sum_{j=0}^d \mu_j =0$ that
	\begin{align*}
		&v_p \Bigg( \prod_{k=0}^N  \Bigg( \prod_{\substack{j=0 \\ j\in I}}^d \min(0, \mu_j +1 -k)! \cdot \prod_{\substack{j=0 \\ j\notin I}}^d \min(0, \mu_j +1)! \Bigg) \Bigg) \\
			&\qquad\leq \frac{1}{p-1} \sum_{k=0}^N \bigg( r k - \sum_{j=0}^d \min(0,\mu_j) \bigg)
			= \frac{1}{p-1} \bigg( r \, \frac{N(N+1)}{2} + (N+1) \abs{\max(0,\ul{\mu})} \bigg) .
	\end{align*}
	Hence we conclude that
	\begin{align*}
		A(\ul{\mu}) \leq p^{\frac{1}{p-1} \big( r \, \frac{N(N+1)}{2} + (N+1) \abs{\max(0,\ul{\mu})} \big)}
%			= p^{\frac{s N (N+1)}{2(p-1)}} p^{\frac{N+1}{p-1} \abs{\max(0,\ul{\mu})}}
	\end{align*}
	which shows the claim with the constants $C = p^{\frac{r N (N+1)}{2(p-1)}}$ and $a = p^{\frac{N+1}{p-1}}$.
\end{proof}

\begin{lemma}\label{Lemma 3 - Existence of good preimages}
	Let $0<\varepsilon <1$ be small enough such that it satisfies the condition of \eqref{Eq 3 - Smallness condition on epsilon}.
	Then, for all $\ul{\mu} \in \Lambda_I$ and $k\in \{1,\ldots,n\}$, there exists $\FY_{\ul{\mu},k} \in U(\Fg)^{(I)}$ such that
	\begin{altenumeratelevel2}
		\item
			$ \displaystyle \varphi \big(\FY_{\ul{\mu},k} \big) = X^\ul{\mu} \, \frac{v_k}{(X_{i_1}\cdots X_{i_r})^{l_k}}$,
		\item
			$ \displaystyle \big\lvert \FY_{\ul{\mu},k} \big\rvert_\varepsilon \leq A(\ul{\mu}) \bigg( \frac{1}{\varepsilon}\bigg)^\abs{\max(0,\ul{\mu})}$.
	\end{altenumeratelevel2}
\end{lemma}
\begin{proof}
	We use induction on $\lVert \ul{\mu} \rVert \defeq \sum_{j=0}^d \abs{\mu_j}$.
	For $\ul{\mu} \in \Delta_d$, we can take $\FY_{\ul{\mu},k} \defeq 1 \, e_{\ul{\mu},k}$ and the claim is trivial.
	Therefore we may assume that there is $u \in \{0,\ldots,d\}$ such that $\abs{ \mu_u} > d$.
	If $\mu_u > 0$, it follows from 
	\[0 = \sum_{j=0}^d \mu_j = \mu_u + \sum_{\substack{j=0 \\ j\neq u}}^d \mu_j \]
	that $\sum_{j=0, j\neq u}^d \mu_j < -d$, and hence that there exists $v \in I$, $v \neq u$ such that $\mu_v < -1$.
	Arguing similarly for the case $\mu_u < 0$, we thus may assume that there are $u \in \{0,\ldots,d\}$, $v\in I$, $u \neq v$, such that $\mu_u >1$, $\mu_v < -1$.

	We now set $\ul{\mu}' \defeq \ul{\mu} - \alpha_{u,v} = \ul{\mu} -\epsilon_u + \epsilon_v$ so that $\lVert \ul{\mu}'\rVert = \lVert \ul{\mu} \rVert -2$.
	For $k \in \{1,\ldots,n\}$, we express $L_{(u,v)}.v_k$ in the $K$-basis $v_1,\ldots,v_n$ of $E'$:
	\[ L_{(u,v)}.v_k = \sum_{k'=1}^n a_{k'} v_{k'} \quad\text{, for $a_1,\ldots,a_n \in K$,} \]
	so that $\max_{k'=1}^n \abs{a_{k'}} = \lVert L_{(u,v)}.v_k \rVert_{E'} \leq \frac{1}{\varepsilon}$ using \eqref{Eq 3 - Smallness condition on epsilon}.
	Note that for the degrees of these homogeneous elements we have $l_{k'}=l_k$, for all $k'\in \{1,\ldots,n\}$ with $ a_{k'}\neq 0$, since $\Fg$ acts on $E'$ by homogeneous endomorphisms.
	Moreover, by the induction hypothesis for $\ul{\mu}'$, for all $k'=1,\ldots,n$, there exist $\FY_{\ul{\mu}',k'} \in U(\Fg)^{(I)}$ as specified in the statement of the lemma.
	We set
	\begin{equation*}
		\FY_{\ul{\mu},k} \defeq \frac{1}{\mu_v +1 - l_k} \bigg( L_{(u,v)} \FY_{\ul{\mu}',k} - \sum_{k'=1}^n a_{k'} \FY_{\ul{\mu}',k'} \bigg) .
	\end{equation*}
	Applying the Leibniz product rule we then find that
	\begin{align*}
		L_{(u,v)}.\bigg( X^{\ul{\mu}'} \frac{v_k}{(X_{i_1}\cdots X_{i_r})^{l_k}} \bigg)
			&= \bigg( L_{(u,v)}. \frac{X^{\ul{\mu}'}}{(X_{i_1}\cdots X_{i_r})^{l_k} } \bigg) v_k  + \frac{X^{\ul{\mu}'}}{(X_{i_1}\cdots X_{i_r})^{l_k} } \big(L_{(u,v)}. v_k \big) \\
			&= (\mu_v +1 - l_k) \frac{X^{\ul{\mu}}}{(X_{i_1}\cdots X_{i_r})^{l_k} } v_k + \sum_{k'=1}^n a_{k'} X^{\ul{\mu}'} \frac{v_{k'}}{(X_{i_1}\cdots X_{i_r})^{l_{k'}}} .
	\end{align*}
	From this it follows that
	\begin{equation*}
		\varphi \big( \FY_{\ul{\mu},k} \big) = \frac{X^{\ul{\mu}}}{(X_{i_1}\cdots X_{i_r})^{l_k} } v_k .
	\end{equation*}
%	\begin{align*}
%		\varphi \big( \FY_{\ul{\mu},k} \big) 
%			&= \frac{1}{\mu_v +1 - l_k} \bigg( L_{(u,v)}. \bigg( X^{\ul{\mu}'} \frac{v_k}{(X_{i_1}\cdots X_{i_r})^{l_k}} \bigg) - \sum_{k'=1}^n a_{k'} X^{\ul{\mu}'} \frac{v_{k'}}{(X_{i_1}\cdots X_{i_r})^{l_{k'}}} \bigg) \\
%			&= \frac{1}{\mu_v +1 - l_k} \bigg( \bigg( L_{(u,v)}. \frac{X^{\ul{\mu}'}}{(X_{i_1}\cdots X_{i_r})^{l_k} } \bigg) v_k  + \frac{X^{\ul{\mu}'}}{(X_{i_1}\cdots X_{i_r})^{l_k} } \big(L_{(u,v)}. v_k \big) - \sum_{k'=1}^n a_{k'} X^{\ul{\mu}'} \frac{v_{k'}}{(X_{i_1}\cdots X_{i_r})^{l_{k'}}} \bigg) \\
%			&=  \frac{X^{\ul{\mu}}}{(X_{i_1}\cdots X_{i_r})^{l_k} } v_k .
%	\end{align*}
	Moreover, the estimates from the induction hypothesis for the $\FY_{\ul{\mu}',k'}$ yield
	\begin{align*}
		\abs{\FY_{\ul{\mu},k}}_\varepsilon 
			&\leq \frac{1}{\abs{\mu_v +1 - l_k}} \max \Big( \abs{L_{(u,v)} \FY_{\ul{\mu}',k} }_\varepsilon \,,\, \max_{k'=1}^n \abs{a_{k'}} \abs{ \FY_{\ul{\mu}',k'}}_\varepsilon \Big) \\
			&\leq \frac{1}{\abs{\mu_v +1 - l_k}} \max \Bigg( A(\ul{\mu}') \bigg(\frac{1}{\varepsilon}\bigg)^{\abs{\max(0,\ul{\mu}')}+1} , \Big( \max_{k'=1}^n \abs{a_{k'}} \Big) A(\ul{\mu}') \bigg(\frac{1}{\varepsilon}\bigg)^{\abs{\max(0,\ul{\mu}')}} \Bigg) \\
			&\leq \frac{1}{\abs{\mu_v +1 - l_k}} A(\ul{\mu}') \bigg(\frac{1}{\varepsilon}\bigg)^{\abs{\max(0,\ul{\mu}')}+1} \\
			&\leq A(\ul{\mu}) \bigg(\frac{1}{\varepsilon}\bigg)^{\abs{\max(0,\ul{\mu})}} .
	\end{align*}
	In the last step we also have used \Cref{Lemma 3 - Estimates for auxiliary function} (i) and that $\abs{\max(0,\ul{\mu}')} = \abs{\max(0,\ul{\mu})} - 1$.	
\end{proof}

	After these preparations we now show that the epimorphism $\varphi$ is strict.
	It suffices to prove that $\varphi$ is an open map.
	Hence, let $U \subset U(\Fg)^{(I)}$ be an open neighbourhood of $0$ and we want to show that $\varphi(U)$ is an open neighbourhood of $0$.
	By \cite[Thm.\ 3.3.6]{PerezGarciaSchikhof10LocConvSpNonArch} there exist $0<\varepsilon_1 ,\ldots, \varepsilon_t < 1$, for some $t\in \BN$, and $R>0$ such that
	\begin{equation*}
		\big\{ \FX\in U(\Fg)^{(I)} \big\vert \max_{i=1}^t \, \lvert \FX \rvert_{\varepsilon_i} < R \big\} \subset U .
	\end{equation*}
	We may suppose that $\varepsilon \defeq \varepsilon_1 \leq \ldots \leq \varepsilon_t$ so that the former set is equal to 
	\begin{equation*}
		B^{(\abs{\blank}_\varepsilon),-}_R (0) \defeq \big\{ \FX \in U(\Fg)^{(I)} \,\big\vert\, \lvert \FX \rvert_{\varepsilon} < R \big\}
	\end{equation*}
	as $\lvert \blank \rvert_{\varepsilon'} \leq \lvert\blank\rvert_{\varepsilon}$, for $\varepsilon \leq \varepsilon'$.
	It follows that $\varphi\big( B^{(\abs{\blank}_\varepsilon),-}_R (0) \big) \subset \varphi ( U ) $.
	We moreover may assume that $\varepsilon>0$ is sufficiently small as in \eqref{Eq 3 - Smallness condition on epsilon}.

	Then let $\delta \defeq a^{-1} \varepsilon$ where $a \geq 1$ is the constant from \Cref{Lemma 3 - Estimates for auxiliary function} (ii).
	Let $\lVert\blank\rVert_\delta$ denote the quotient norm on $\CE(U_I)$ induced via 
	\begin{equation*}\label{Eq 3 - Epimorphism with homogeneous generators of sections}
		\bigoplus_{k=1}^n \CO(U_I) \, e_k \longtwoheadrightarrow \CE(U_I)\,,\quad X^\ul{\mu} \, e_k \lto X^\ul{\mu} \frac{v_k}{(X_{i_1}\cdots X_{i_r})^{l_k}} ,
	\end{equation*}
	with $\CO(U_I)\subset \CO(U_{I,\delta})$ carrying the Gauss norm $\abs{\blank}_\delta$, see \Cref{Rmk 2 - Concrete description of the norms of the sections on the affinoid subdomains}.
	Note that the norm $\lvert\blank\rvert_\delta$ on $\CE(U_I)$ (i.e.\ the one given by a fixed $\CO(U_I)$-basis as in \eqref{Eq 2 - Fixed norm on sections}) is equivalent to $\lVert\blank\rVert_\delta$ by \cite[3.7.3 Prop.\ 3]{BoschGuentzerRemmert84NonArchAna}.	
	To prove the proposition it therefore suffices to show that
	\begin{equation}\label{Eq 3 - Inclusion of balls}
		B_{C^{-1} R}^{(\lVert\blank\rVert_\delta), - } (0) \subset \varphi \big( B_{R}^{(\abs{\blank}_\varepsilon), - } (0) \big)
	\end{equation}
	where $C>0$ is the constant from \Cref{Lemma 3 - Estimates for auxiliary function} (ii).
	Indeed, then $B_{C^{-1} R}^{(\lVert\blank\rVert_\delta), - } (0)  \subset \varphi(U)$ is an open neighbourhood of $0$ with respect to $\lvert\blank\rvert_\delta$ on $\CE(U_I)$, and hence with respect to the subspace topology of $\CE(U_I)\subset \CE(U_I^\rig)$ as well.

	To show \eqref{Eq 3 - Inclusion of balls}, let $v\in \CE(U_I)$ with $\lVert v \rVert_\delta < \frac{R}{C}$.
	By the definition of the quotient norm $\lVert\blank\rVert_\delta$, there exist $a_{\ul{\mu},k} \in K$ such that
	\begin{align*}
		v = \sum_{\substack{\ul{\mu} \in \Lambda_I \\ k=1,\ldots,n}} a_{\ul{\mu},k} \, X^\ul{\mu} \frac{v_k}{(X_{i_1}\cdots X_{i_r})^{l_k}}
		\qquad \text{and} \qquad
		\sup_{\substack{\ul{\mu} \in \Lambda_I \\ k=1,\ldots,n}} \abs{a_{\ul{\mu},k}} \bigg( \frac{1}{\delta} \bigg)^{\abs{\max(0,\ul{\mu})}} < \frac{R}{C} .
	\end{align*}
	For all $\ul{\mu} \in \Lambda_I$ and $k=1,\ldots,n$, we find $\FY_{\ul{\mu},k} \in U(\Fg)^{(I)}$ as specified in \Cref{Lemma 3 - Existence of good preimages}.
	Then
	\begin{equation*}
		\FY \defeq \sum_{\substack{\ul{\mu} \in \Lambda_I\\ k=1,\ldots,n }} a_{\ul{\mu},k} \, \FY_{\ul{\mu},k} \in U(\Fg)^{(I)} 
	\end{equation*}
	satisfies $\varphi(\FY) = v$.
	Moreover, it follows from \Cref{Lemma 3 - Estimates for auxiliary function} (ii) that
	\begin{align*}
		\lvert \FY \rvert_\varepsilon
			&\leq \sup_{\substack{\ul{\mu} \in \Lambda_I\\ k=1,\ldots,n }} \abs{a_{\ul{\mu},k}} \abs{\FY_{\ul{\mu},k}}_\varepsilon \\
			&\leq \sup_{\substack{\ul{\mu} \in \Lambda_I\\ k=1,\ldots,n }} \abs{a_{\ul{\mu},k}} A(\ul{\mu}) \bigg(\frac{1}{\varepsilon}\bigg)^{\abs{\max(0,\ul{\mu})}}
			\leq C \sup_{\substack{\ul{\mu} \in \Lambda_I\\ k=1,\ldots,n }} \abs{a_{\ul{\mu},k}} \bigg(\frac{1}{\delta}\bigg)^{\abs{\max(0,\ul{\mu})}}
			< R
	\end{align*}
	showing that $v \in \varphi\big( B_{R}^{(\abs{\blank}_\varepsilon), - } (0) \big)$.
\end{proof}

Finally, we allow $K$ to be a finite extension of a non-archimedean local field $L$ of arbitrary characteristic, and return to the setting of a connected split reductive group $\bG$ over $L$ as considered in the beginning of this section.
Let $\bP$ a standard parabolic subgroup of $\bG$.
The preceding considerations as well as the generalization of the functors $\CF^G_P$ in the $p$-adic situation due to Agrawal and Strauch \cite{AgrawalStrauch22FromCatOLocAnRep} motivate the following definition.

\begin{definition}[{cf.\ \cite[Def.\ 4.2.1]{AgrawalStrauch22FromCatOLocAnRep}}]
	For a separately continuous $D(\Fg,P,K)$-module $M$ and an admissible smooth $L_\bP$-representations $V$ on a $K$-vector space considered with the finest locally convex topology, we define the functor
	\begin{equation*}
		\widetilde{\CF}^G_P(M,V) \defeq  D(G,K) \cotimes{D(\Fg,P,K), \iota} \big(M \projcotimes V'_b \big) 
	\end{equation*}
	which takes values in the category of separately continuous $D(G, K)$-modules.
\end{definition}
%If $M$ and $N$ are separately continuous modules over a separately continuous $K$-algebra $A$ on $K$-Fr\'echet spaces, then so is $M \projcotimes N$:
%For fixed $a\in A$, the multiplication map with $a$ is given by $M\projotimes N \ra (A\indotimes M) \projotimes (A \indotimes N) \ra M \projotimes N$, $m\otimes n\mto (a \otimes m) \otimes( a \otimes n) \mto am \otimes a n$  which is continuous. Then we take the completion of this map.
%For fixed $m\in M, n\in N$, the map $A \ra A\times A \ra (A \indotimes M) \times (A \indotimes N) \ra M\times N \ra M\projotimes N \ra M\projcotimes N$, $a \mto (a,a) \mto (a \otimes m, a \otimes n) \mto (am,an)\mto am\otimes an$ is continuous. As $M\projcotimes N$ is a $K$-Fr\'echet space, every of its elements can be approximated by a sequence of sums of pure tensors. For such sums we have seen that multiplication by them is continuous. For arbitrary elements it follows by the Banach--Steinhaus theorem.

Like before $V'_b$ is a nuclear $K$-Fr\'echet space, for such an admissible smooth $L_\bP$-re\-pre\-sen\-tation $V$.
If the Hausdorff completion $\widehat{M}$ of $M$ is a nuclear $K$-Fr\'echet space as well, it follows from the discussion after \cite[Prop.\ 17.6]{Schneider02NonArchFunctAna} and \cite[Prop.\ 19.11, Prop.\ 20.4]{Schneider02NonArchFunctAna} that $M \projcotimes V'_b = \widehat{M} \projcotimes V'_b $ is a nuclear $K$-Fr\'echet space, too.
Since \Cref{Lemma 3 - Tensor products over distribution algebras} (ii) yields a topological isomorphism
\begin{equation*}
	D(G) \cotimes{D(\Fg,P), \iota} \big(M \projcotimes V'_b \big) \cong D(G_0) \cotimes{D(\Fg,P_0)} \big(M \projcotimes V'_b \big) ,
\end{equation*}
it follows from \cite[Prop.\ 19.4 (ii)]{Schneider02NonArchFunctAna} that $\widetilde{\CF}^G_P(M,V)$ is a nuclear $K$-Fr\'echet space.
Hence the strong dual $\widetilde{\CF}^G_P(M,V)'_b$ is a locally analytic $G$-representation of compact type in this case, by \Cref{Prop 1 - Equivalences for categories of locally analytic representations}.

\begin{remark}
	Let $M$ be an abstract module for the (algebraic) distribution algebra ${\rm Dist}(\bG) \botimes{L} K = \hy(G)$.
	An obvious question is under which algebraic conditions on the module $M$ (such as $M \in \CO_\alg^\Fp$ in the $p$-adic case) this lifts to a separately continuous $D(\Fg,P)$-module structure on $M$.
	Here analogues of the BGG category $\CO$ in the setting of ${\rm char}(K)>0$ should play a role.
	For recent developments in regard to these analogues, see \cite{Andersen22BGGCatPrimeChar,  Andersen22CharFormCatOp, Fiebig21PerSubquotModCatO, Orlik21EqVBDrinfeldFF}.
\end{remark}

\vspace{3ex}

\appendix

\numberwithin{theorem}{section}

\section{Non-Archimedean Functional Analysis}

Here we collect some results of non-archimedean functional analysis. 
They are all (slight generalizations of) statements that can be found in the literature or adaptions from the archimedean setting.
Let $K$ be a spherically complete non-archimedean field with ring of integers $\CO_K= \{x \in K\mid \abs{x}\leq 1 \}$.
\\

We recall the definition of a compact continuous homomorphism between locally convex $K$-vector spaces \cite[Ch.\ 8.8 p.\ 334]{PerezGarciaSchikhof10LocConvSpNonArch}:

\begin{definition}
	\begin{altenumerate}
		\item
		A continuous homomorphism $f\colon V \ra W$ between Hausdorff locally convex $K$-vector spaces is \textit{compact} if there exists a complete compactoid subset $X \subset W$ (cf.\ \cite[Def.\ 3.8.1]{PerezGarciaSchikhof10LocConvSpNonArch}) such that $f^{-1}(X)\subset V$ is a neighbourhood of $0$ in $V$.
		\item
		A continuous homomorphism $f\colon V \ra W$ between Hausdorff locally convex $K$-vector spaces is \textit{semicompact} if there exists a compactoid Banach disk $B \subset W$ (cf.\ \cite[p. 414]{PerezGarciaSchikhof10LocConvSpNonArch}) such that $f^{-1}(B)\subset V$ is a neighbourhood of $0$ in $V$.
	\end{altenumerate}
	
\end{definition}

\begin{remarks}\label{Rmk A1 - Remarks about compact and semicompact homomorphisms}
	\begin{altenumerate}
		\item
		If $f$ is compact, then $f$ is semicompact.
		If $W$ is quasi-complete, then the converse holds as well \cite[Ch.\ 8.8 p.\ 334]{PerezGarciaSchikhof10LocConvSpNonArch}.
		\item
		This definition is equivalent to the one in \cite[\S 16]{Schneider02NonArchFunctAna}.
		There $f$ is defined to be compact if there exists an open lattice $L\subset V$ such that the closure $\widebar{f(L)}\subset W$ is bounded and c-compact.
	\end{altenumerate}
\end{remarks}
\begin{proof}[Proof of (ii)]
	For an open lattice $L\subset V$, by \cite[Prop.\ 12.7]{Schneider02NonArchFunctAna} $\widebar{f(L)}$ is bounded and c-compact if and only if it is compactoid and complete.
	Hence $X \defeq \widebar{f(L)}$ yields a complete compactoid subset such that $f^{-1}(X) \supset L$ is a neighbourhood of $0$ in $V$.

	On the other hand, given a complete compactoid $X\subset W$ such that $f^{-1}(X) $ is a neighbourhood of $0$ in $V$, i.e.\ $f^{-1}(X)$ contains an open lattice $L\subset V$, it follows that $\widebar{f(L)} \subset \widebar{X}$.
	But this implies that $\widebar{f(L)}$ itself is compactoid by \cite[Thm.\ 3.8.4]{PerezGarciaSchikhof10LocConvSpNonArch}, and complete by \cite[Rmk.\ 7.1 (iv), (v)]{Schneider02NonArchFunctAna}.
\end{proof}

\begin{lemma}\label{Lemma A1 - Generalities on compact maps}
	\begin{altenumerate}
		\item
		Consider the following commutative square of Hausdorff locally convex $K$-vector spaces:
		\begin{equation}\label{Eq A1 - Commutative diagram for compactness lemma}
			\begin{tikzcd}
				V \ar[r,"f"]\ar[d,"g"] & W \ar[d,"g'"] \\
				V'\ar[r,"f'"] & W' .
			\end{tikzcd}
		\end{equation}
		If $g'$ is compact and $f'$ a strict monomorphism, then $g$ is compact.
		\item
		In the commutative square \eqref{Eq A1 - Commutative diagram for compactness lemma}, if $g$ is compact and $f$ a strict epimorphism, then $g'$ is compact.
		\item
		Finite products of compact homomorphisms are compact.	
\end{altenumerate}
\end{lemma}
\begin{proof}
For (i), it follows from \cite[Rmk.\ 16.7 (i)]{Schneider02NonArchFunctAna} that $f' \circ g = g' \circ f$ is compact if $g'$ is. 
Since $\Im(f') \cong V'$, \cite[Rmk.\ 16.7 (ii)]{Schneider02NonArchFunctAna} implies that $g$ is compact.

In the situation of (ii), we again know from \cite[Rem.\ 16.7 (i)]{Schneider02NonArchFunctAna} that $g'\circ f = f'\circ g$ is compact.
Then by definition there exists an open lattice $L \subset V$ such that $\widebar{(g'\circ f)(L)}$ is bounded and c-compact.
But as $f$ is a strict epimorphism, $f(L)$ is an open lattice in $W$ which shows that $g'$ is compact.

For (iii), we consider compact homomorphisms $f\colon V \ra W$ and $f'\colon V'\ra W'$ with open lattices $L\subset V$, $L'\subset V'$ such that $\widebar{f(L)}$ is bounded and c-compact in $W$ and the same holds for $\widebar{f'(L')}$ in $W'$.
Then $L \times L'$ is an open lattice in $V \times V'$ and $\widebar{(f\times f') (L \times L')} = \widebar{f(L)} \times \widebar{f'(L')}$ is bounded in $W\times W'$.
Moreover, $\widebar{f(L)} \times \widebar{f'(L')}$ is c-compact by \cite[Prop.\ 12.2]{Schneider02NonArchFunctAna}.
\end{proof}

\begin{proposition}[{\cite[Prop.\ 1.2 (i)]{SchneiderTeitelbaum02LocAnDistApplToGL2}}]\label{Lemma A1 - Closed subspaces and quotients of spaces of compact type}
Let $V$ be a locally convex $K$-vector space of compact type.
If $U \subset V$ is a closed subspace then $U$ and $V/U$ are of compact type again.
\end{proposition}
\begin{proof}
Let $(V_n)_{n\in \BN}$ be an inductive sequence of $K$-Banach spaces with injective and compact transition maps such that $V \cong \varinjlim_{n \in \BN} V_n$.
By \cite[Thm.\ 3.1.16]{DeGrandeDeKimpeKakolPerezGarciaSchikofS97pAdicLocConvIndLim}, $U$ with its subspace topology is topologically isomorphic to the inductive limit of the sequence $(U_n)_{n\in \BN}$ where $U_n \defeq U\cap V_n$.
Then the transition maps of $(U_n)_{n\in \BN}$ are compact again \cite[Rmk.\ 16.7]{Schneider02NonArchFunctAna}.
Therefore $U$ is of compact type.

Furthermore, by \Cref{Lemma A1 - Generalities on compact maps} (ii), the induced maps $V_n/U_n \ra V_{n+1}/U_{n+1}$ are injective and compact.
Hence $\varinjlim_{n \in \BN} V_n/ U_n$ is of compact type.
Moreover, taking the inductive limit over the short strictly exact sequences $0\ra U_n \ra V_n \ra V_n/U_n \ra 0$ we arrive at the sequence of continuous homomorphisms
\begin{equation*}
	0 \lra U \lra V \lra \varinjlim_{n\in \BN} V_n/U_n \lra 0
\end{equation*}
which is exact as a sequence of $K$-vector spaces.
It follows from the open mapping theorem \cite[II.\ \S 4.6 Cor.]{Bourbaki87TopVectSp1to5} that the continuous surjection $V \ra \varinjlim_{n \in \BN} V_n/U_n$ is strict. 
Therefore $\varinjlim_{n \in \BN} V_n/U_n \ra V/U$ is a topological isomorphism.
\end{proof}

\begin{proposition}[{\cite[Thm.\ 1.3]{SchneiderTeitelbaum02LocAnDistApplToGL2}}]\label{Prop A1 - Duality of spaces of compact type and nuclear Frechet spaces}
The strong dual of a locally convex $K$-vector space of compact type is a nuclear Fr\'echet space, and the strong dual of a nuclear $K$-Fr\'echet space is of compact type.
\end{proposition}
\begin{proof}
If $V$ is a locally convex $K$-vector space of compact type, it follows from \cite[Thm.\ 3.1.7 (vii),(viii)]{DeGrandeDeKimpeKakolPerezGarciaSchikofS97pAdicLocConvIndLim} that $V'_b$ is a nuclear $K$-Fr\'echet space.

Conversely, let $V$ be a nuclear $K$-Fr\'echet space.
For a decreasing neighbourhood base of $0$ consisting of lattices $(L_n)_{n\in \BN}$, one obtains an inductive sequence $(V'_{L_n^{\circ}})_{n\in \BN}$ of certain $K$-Banach spaces $V'_{L_n^{\circ}}$ with $V'= \bigcup_{n\in \BN} V'_{L_n^{\circ}}$, see \cite[Def.\ 2.5.2]{DeGrandeDeKimpeKakolPerezGarciaSchikofS97pAdicLocConvIndLim}.
Then \cite[Prop.\ 3.1.13]{DeGrandeDeKimpeKakolPerezGarciaSchikofS97pAdicLocConvIndLim} says that this sequence is semicompact; even compact by Remark \ref{Rmk A1 - Remarks about compact and semicompact homomorphisms} (ii).
Note that if $K$ is spherically complete every locally convex $K$-vector space is polar \cite[Thm.\ 4.4.3 (i)]{PerezGarciaSchikhof10LocConvSpNonArch}.

Moreover, by \cite[Cor.\ 8.5.3]{PerezGarciaSchikhof10LocConvSpNonArch}, $V$ in particular is reflexive, and therefore $V'_b$ is barrelled by \cite[Thm.\ 7.4.11 (i)]{PerezGarciaSchikhof10LocConvSpNonArch}.
(If $K$ is spherically complete, barrelled and polarly barrelled are equivalent \cite[Thm.\ 7.1.9 (ii)]{PerezGarciaSchikhof10LocConvSpNonArch}.)
Now \cite[Cor.\ 2.5.9]{DeGrandeDeKimpeKakolPerezGarciaSchikofS97pAdicLocConvIndLim} implies that the inductive limit topology on $\bigcup_{n\in \BN} V'_{L_n^{\circ}}$ agrees with the strong topology of $V'$.
Hence $V'_b$ is of compact type.
\end{proof}

\begin{lemma}\label{Lemma A1 - Homomorphisms from projective limit into normed space}
Let $(V_n)_{n\in \BN}$ be a projective sequence of locally convex $K$-vector spaces and $W$ a normed $K$-vector space.
Then the canonical continuous homomorphism
\begin{equation}\label{Eq A1 - Homomorphisms from projective limit}
	\varinjlim_{n\in \BN} \CL_b(V_n,W) \lra \CL_b \bigg(\varprojlim_{n\in \BN} V_n, W \bigg)
\end{equation}
is surjective.
If all projections $\pr_n \colon \varprojlim_{n\in \BN} V_n \ra V_n$ have dense image, or if all $V_n$ are Hausdorff and the transition homomorphisms $V_{n+1} \ra V_n$ are compact, then \eqref{Eq A1 - Homomorphisms from projective limit} even is bijective.	
\end{lemma}
\begin{proof}	
First consider $f \in \CL \big( \varprojlim_{n\in \BN} V_n, W \big)$, and let $B_W \defeq \left\{ w\in W \middle{|} \lVert w \rVert_W \leq 1 \right\}$ denote the unit ball of $W$ so that $f^{-1}(B_W) \subset \varprojlim_{n\in \BN} V_n$ is open.
By the definition of the initial topology of $\varprojlim_{n\in \BN} V_n$, there exist integers $n_1,\ldots,n_r \in \BN$ and open lattices $L_{n_i}\subset V_{n_i}$, $i=1,\ldots,r$, such that
\begin{equation*}
	\pr_{n_1}^{-1}(L_{n_1}) \cap \ldots \cap \pr_{n_r}^{-1}(L_{n_r}) \subset f^{-1}(B_W).
\end{equation*}
Let $m\geq n_1,\ldots,n_r$, and note that $\Ker(\pr_m) \subset \Ker(\pr_{n_i})\subset \pr_{n_i}^{-1}(L_{n_i})$, for $i=1,\ldots,r$.
Hence we find that $\Ker(\pr_m) \subset f^{-1}(B_W)$.
As $\Ker(\pr_m)$ is a $K$-subvector space, it follows that $\Ker(\pr_m) \subset \Ker(f)$.
%Let $v\in \Ker(\pr_m)$. Let $(a_n)\subset K^\times$ be a sequence with $\abs{a_n} \ra 0$. Then
%\begin{equation*}
%	\lVert f(v) \rVert_W = \abs{a_n} \lVert f(a^{-1}v) \rVert_W \leq \abs{a_n} \cdot 1 \ra 0.
%\end{equation*}
%Therefore $f(v)=0$ as $W$ is Hausdorff.
Therefore $f$ factors over $V_m$ via $\pr_m$.
This shows that \eqref{Eq A1 - Homomorphisms from projective limit} is surjective.

If all projections $\pr_n$ have dense image, then the homomorphisms
\begin{equation*}
	\CL (V_n, W) \lra \CL\bigg(\varprojlim_{n\in \BN} V_n, W\bigg) \,,\quad f \lto f \circ \pr_n ,
\end{equation*}
are injective because $W$ is Hausdorff.
As \eqref{Eq A1 - Homomorphisms from projective limit} is the direct limit of these homomorphisms, its injectivity follows.

If all $V_n$ are Hausdorff and the transition maps are compact, then there exists a projective system $(U_n)_{n\in\BN}$ such that $\varprojlim_{n\in \BN} V_n$ and $\varprojlim_{n\in \BN} U_n$ are topologically isomorphic, and the canonical projections of the latter have dense image \cite[p.\ 93]{Schneider02NonArchFunctAna}.
Moreover, we have $\varinjlim_{n\in \BN} \CL_b(V_n,W) \cong \varinjlim_{n \in \BN} \CL_b(U_n,W)$ by functoriality, and can conclude using the previous case.
\end{proof}

\begin{proposition}[{\cite[Prop.\ 1.5]{SchneiderTeitelbaum02LocAnDistApplToGL2}}]\label{Prop A1 - Continuous bijection for tensor product of space of compact type and Banach space}
Let $V$ be a locally convex $K$-vector space of compact type, expressed as $V = \varinjlim_{n \in \BN} V_n$, for a sequence of $K$-Banach spaces $V_n$ with compact and injective transition maps.
Moreover, let $W$ be a $K$-Banach space.
Then the canonical continuous homomorphism
\begin{equation*}
	\varinjlim_{n\in \BN} \big( V_n \cotimes{K} W \big) \lra V \cotimes{K} W
\end{equation*}
is bijective.
\end{proposition}
\begin{proof}
To ease the notation, we will simply denote the strong dual of a locally convex $K$-vector space $U$ by $U'$.
Using \cite[Cor.\ 18.8]{Schneider02NonArchFunctAna} and the fact that $V$ is reflexive, we have a topological isomorphism
\begin{equation*}
	V \cotimes{K} W \overset{\cong}{\lra} V'' \cotimes{K} W \overset{\cong}{\lra} \CL_b (V', W) \,,\quad v\otimes w \lto \big[\ell \mto \ell(v)\, w\big] .
\end{equation*}
For each $V_n$, the duality map and \cite[Lemma 18.1]{Schneider02NonArchFunctAna} at least give a continuous homomorphism
\begin{equation*}
	V_n \cotimes{K} W \lra V_n'' \cotimes{K} W \lra \CL_b (V_n', W) \,,\quad v \otimes w \lto \big[\ell \mto \ell(v)\, w\big] ,
\end{equation*}
so that we arrive at a commutative diagram of continuous homomorphism
\begin{equation}\label{Eq A1 - Commutative diagram for tensor product of space of compact type and Banach space}
	\begin{tikzcd}
		\varinjlim_{n \in \BN} (V_n \cotimes{K} W) \ar[r] \ar[d] &V \cotimes{K} W  \ar[d, "\cong"] \\
		\varinjlim_{n \in \BN} \CL_b (V_n', W) \ar[r] &\CL_b(V', W).
	\end{tikzcd}		
\end{equation}
By \cite[Lemma 16.4 (ii)]{Schneider02NonArchFunctAna} the projective system $(V_n')_{n\in \BN}$ is compact, and therefore the lower map in \eqref{Eq A1 - Commutative diagram for tensor product of space of compact type and Banach space} is a bijection by \Cref{Lemma A1 -  Homomorphisms from projective limit into normed space}.

Hence the claim follows if we show that the left homomorphism is an isomorphism.
But by \cite[Lemma 16.4 (iii)]{Schneider02NonArchFunctAna}, the transition maps $V_n'' \ra V_{n+1}''$ factor over $V_{n+1}\subset V_{n+1}''$ which gives 
\begin{equation*}
	\varinjlim_{n \in \BN} (V_n \cotimes{K} W) \overset{\cong}{\lra} \varinjlim_{n\in\BN} \big( V_n'' \cotimes{K} W \big) .
\end{equation*}
Moreover, the image of $V_n'' \cotimes{K} W$ in $\CL_b(V_n', W)$ precisely is the subspace of compact homomorphisms $\CC(V_n', W)$ by \cite[Prop.\ 18.11]{Schneider02NonArchFunctAna}.
It follows from \cite[Rmk.\ 16.7 (i)]{Schneider02NonArchFunctAna} that the transition maps $\CL_b( V_n', W) \ra \CL_b(V_{n+1}',W)$ factor over $\CC(V_{n+1}',W) \subset \CL_b(V_{n+1}', W)$ because they are given by precomposition with the compact maps $V_{n+1}' \ra V_n'$.
Hence
\begin{equation*}
	\varinjlim_{n\in\BN} \big( V_n'' \cotimes{K} W \big)  \cong \varinjlim_{n \in \BN} \CC(V_{n}',W)  \overset{\cong}{\lra} \varinjlim_{n \in \BN} \CL_b(V'_n, W) 
\end{equation*}
is a topological isomorphism, too.
\end{proof}

\begin{lemma}[{cf.\ \cite[Lemma 1.3]{SchneiderTeitelbaum02pAdicBoundVal}}]\label{Lemma A1 - Separately continuous is jointly continuous under maps by locally compact space on barrelled vector space}
Let $X$ be a locally compact topological space.
%locally $L$-analytic manifolds over a locally compact field $L$ fulfill this
Let $V$ and $W$ be locally convex $K$-vector spaces, and assume that $V$ is barrelled.
If $\beta \colon X \times V \ra W$ is a separately continuous map and $\beta(x,\blank) \colon V \ra W$ is $K$-linear, for every $x\in X$, then $\beta$ is jointly continuous.
\end{lemma}
\begin{proof}
By the linearity of the $\beta(x,\blank)$ it suffices to show that $\beta$ is continuous at $(x,0)$, for every $x\in X$.
Fix $x_0\in X$, and let $U\subset X$ be a compact neighbourhood of $x_0$.
We claim that $H \defeq \left\{\beta(x,\blank) \middle{|} x \in U \right\} \subset \CL_s(V,W)$ is bounded.
Assuming this claim for the moment, it follows from the Banach-Steinhaus theorem \cite[Prop.\ 6.15]{Schneider02NonArchFunctAna} that $H$ is equicontinuous.
Hence, for any open lattice $M\subset W$, there exists an open lattice $L\subset V$ such that $\beta(U,L) \subset M$.
This shows that $\beta$ is continuous at $(x_0,0)$.
%For every $M$, $U\times L$ is an open neighbourhood of $(x_0,0)$ that gets mapped into $M$.

To show the claim, recall that the seminorms of $\CL_s(V,W)$ are $q_{v,p}$, for $v\in V$ and $p$ a continuous seminorm of $W$, and defined by $q_{v,p}(f) \defeq p(f(v))$, for $f\in \CL_s(V,W)$ \cite[Example 1 after 6.6]{Schneider02NonArchFunctAna}.
For such a seminorm, we have
\begin{equation*}
	\sup_{f \in H} q_{v,p}(f) = \sup_{x\in U} p(\beta(x,v))  < \infty
\end{equation*}
because the image of the compact subset $U$ under the continuous map $p\circ \beta(\blank,v)$ is bounded.
It follows that $H\subset \CL_s(V,W)$ is bounded.
\end{proof}

\begin{lemma}\label{Lemma A1 - Isomorphism for space of linear maps on direct sum with strong topology}
Let $(V_i)_{i\in I}$ and $W$ be locally convex $K$-vector spaces.
Then there exists a natural topological isomorphism
\begin{equation*}
	\CL_b \bigg(\bigoplus_{i\in I} V_i, W \bigg) \overset{\cong}{\lra} \prod_{i\in I}\CL_b (V_i,W) \,,\quad f \lto (f\res{V_i})_{i\in I} .
\end{equation*}
(For the case of $W=K$, see \cite[Prop.\ 9.10]{Schneider02NonArchFunctAna})
%Die analoge Aussage für $\CL_s(\blank, W)$ kann man zeigen, indem man feststellt, dass dieser Endofunktor auf der Kategorie der lokal konvexen $K$-Vektorräume zu sich selbst von rechts adjungiert ist (man benutzt die beiden Adjunktionen zum induktiven Tensorprodukt). Dann folgt die Aussage aus abstract nonsense.
\end{lemma}
\begin{proof}
Let $\iota_i \colon V_i \ra V \defeq \bigoplus_{i\in I} V_i$ denote the canonical embeddings, and $\pr_i \colon V \ra V_i$ the canonical projections.
%$\pr_i$ is defined by factoring over $V \ra \prod_{i\in I} V_i$.
By \cite[III.\ \S 3 Ex.\ 5 on p.\ III.41]{Bourbaki87TopVectSp1to5} there is a natural topological isomorphism
\begin{equation*}
	\CL_\CB (V, W) \overset{\cong}{\lra} \prod_{i\in I} \CL_b (V_i,W) \,,\quad f \lto (f\res{V_i})_{i \in I},
\end{equation*}
where $\CB$ is the family $\left\{\iota_i(B_i)\middle{|} i\in I, B_i \subset V_i \text{ bounded}\right\}$ of bounded sets of $V$.
We want to show that the $\CB$-topology coincides with the strong topology on $\CL(V,W)$.
In view of \cite[Lemma 6.5]{Schneider02NonArchFunctAna}, it suffices to show that for a given bounded subset $B \subset V$, there exist $\iota_{i_1} (B_{1}),\ldots, \iota_{i_m}(B_{m}) \in \CB$ such that $B$ is contained in the closure of the $\CO_K$-module generated by $\iota_{i_1}(B_{1}) \cup \ldots \cup \iota_{i_m}(B_{m})$.

To do so, we proceed similarly to the proof of \cite[Prop.\ 9.10]{Schneider02NonArchFunctAna}, and fix a bounded subset $B \subset V$.
By \cite[Thm.\ 3.6.4 (ii)]{PerezGarciaSchikhof10LocConvSpNonArch}, all $\pr_i (B)\subset V_i$ are bounded and there exists a finite subset $J\subset I$ such that $\pr_i(B) \subset \widebar{\{0\}}$, for all $i \in I\setminus J$.
We define the bounded subsets $B_j \defeq \pr_j(B)$, for $j\in J$.
For given $v\in B$, we write $v = \sum_{i\in I} \iota_i(v_i)$ where $v_i\in  V_i$.
We then have $\sum_{j\in J} \iota_j(v_j) \in \sum_{j\in J} \iota_j (B_j)$.

To deal with $\sum_{i\in I\setminus J} \iota_i (v_i)$, let $q$ be a continuous seminorm of $V$.
By \cite[Lemma 5.1 (ii)]{Schneider02NonArchFunctAna}, $q \circ \iota_{i}$ is a continuous seminorm of $V_{i}$.
Hence
\begin{equation*}
	q( \iota_{i}(v_i)) = (q \circ \iota_{i})(\pr_{i}(v))  = 0, 
\end{equation*}
for all $i \in I\setminus J$, as $\pr_i(B) \subset \widebar{\{0\}}$ for those $i$.
It follows that, for $i\in I\setminus J$,
\begin{equation*}
	\iota_{i} (v_i) \in \bigcap_{q} \Ker(q) = \widebar{\{0\}}
\end{equation*}
where we take the intersection over all continuous seminorms $q$ of $V$.
This shows that $v$ is contained in the closure of $\sum_{j\in J} \iota_j (B_j)$ in $V$.
\end{proof}

\begin{proposition}[{cf.\ \cite[Prop.\ 1]{Browder62AnaPartDiffEq1}}]\label{Prop A1 - Strong and weak analytic functions on a Banach space}
Let $K$ be a complete field with absolute value $\abs{\blank}$.
Let $U \subset K^n$ be open and let $f \colon U \lra E$ be a function into a $K$-Banach space $E$.
Let $V$ and $W$ be $K$-Banach spaces and let 
\begin{equation}\label{Eq A1 - Bilinear pairing for weak and strong analytic functions on a Banach space}
	E \times V \lra W \, ,\quad  (u,v) \lto \langle u, v \rangle,
\end{equation}
be a continuous $K$-bilinear map which induces an isometric embedding of $E$ into the $K$-Banach space $\CL_b(V,W)$, i.e.\ such that, for all $u\in E$,
\[ \abs{u} = \left\| \langle u, \_\, \rangle \right\|_{\CL_b(V,W)} \defeq \sup_{v \in V \setminus \{0\}} \frac{\abs{\langle u,v \rangle}}{\abs{v}} \, .\]
Then, for $z_0 \in U$, the following are equivalent:
\begin{altenumeratelevel2}
	\item
	The function $f$ is analytic in some open neighbourhood of $z_0$.
	\item
	For all $v \in V$, there exists an open neighbourhood of $z_0$ such that the function
	\[f_v \colon U \lra W \, ,\quad  z \lto \langle f(z) , v\rangle ,\]
	is analytic there.
\end{altenumeratelevel2}
\end{proposition}
\begin{proof}
First assume that there exists $r\in \BR_{>0}$ such that $f$ is given by the power series
\begin{equation*}
	f(z) = \sum_{\ul{i}\in \BN^n_0} a_{\ul{i}} \, (z-z_0)^\ul{i} \quad\text{, for $x\in B_r^{n}(z_0)$,}
\end{equation*}
for certain $a_{\ul{i}} \in E$.
For $v\in V$, it follows from the continuity of \eqref{Eq A1 - Bilinear pairing for weak and strong analytic functions on a Banach space} that
\begin{equation*}
	f_v(z) = \langle f(z),v\rangle  = \sum_{\ul{i}\in \BN^n_0} \langle a_{\ul{i}} , v\rangle \,  (z-z_0)^\ul{i} \quad\text{, for all $z \in B_r^{n}(z_0)$.}
\end{equation*}
This shows that (i) implies (ii).

Conversely, for all $v \in V$, let $f_v$ be analytic in a neighbourhood of $z_0$.
This means that there exists a radius $r^{(v)} \in \BR_{>0}$ such that $f_v$ is given by the convergent power series
\[f_v (z) = \sum_{\ul{i} \in \BN^n_0} a_\ul{i}^{(v)} (z-z_0)^\ul{i} \quad \text{, for all $z \in B^{n}_{r^{(v)}}(z_0)$,}\]
for certain $a_\ul{i}^{(v)} \in W$.
Hence there is a constant $C^{(v)} >0$ such that
\begin{align}\label{Eq A1 - pointwise bound}
	\abs{a_\ul{i}^{(v)}} \leq  C^{(v)} \left( \frac{1}{r^{(v)}} \right)^{\abs{\ul{i}}} \, ,
\end{align}
for all $\ul{i} \in \BN^n_0$.
As an intermediate step, we want to show:\\

\noindent{\bf Claim: }\textit{There are $r>0$ and $C>0$ such that, for all $\ul{j}\in \BN^n_0$, the maps}
\[ b_\ul{j} \colon  V \lra W \, , \quad  v \lto a_\ul{j}^{(v)} , \]
\textit{are $K$-linear and continuous with}
\begin{align}\label{Eq A1 - uniform bound}
	\norm{ b_\ul{j} }_{\CL_b(V,W)} \leq C \left(\frac{1}{r}\right)^{\abs{\ul{j}}} \, .
\end{align}

\begin{proof}[Proof of the Claim]
	The bilinearity of $\langle \blank,\blank \rangle$ implies that 
	\[ V \lra \mathrm{Map}(U,W) \, , \quad v \lto f_v , \]
	is $K$-linear. 
	The linearity of $b_\ul{j}$ then follows from the identity theorem for the coefficients of convergent power series \cite[3.2.1 resp.\ 4.2.1]{Bourbaki07VarDiffAnFasciDeResult}.
	
	To proof the continuity and the bound \eqref{Eq A1 - uniform bound}, we endow $\BN^n_0$ with the lexicographical order, i.e.\ $\ul{i}<\ul{j}$ if and only if $i_k < j_k$ for the smallest $k \in \{1,\ldots,n\}$ where $i_k \neq j_k$.
	We proceed by induction on $\ul{j} \in \BN^n_0$ with respect to this order.
	
	Fix $\ul{j} \in \BN^n_0$ and assume that the claim holds for all $\ul{i}<\ul{j}$ (a vacuous assumption if $\ul{j}=(0,\ldots,0)$).
	Hence there are $r_\ul{j} >0$ and $C_\ul{j} >0$ such that
	\[\abs{a_\ul{i}^{(v)}} \leq C_\ul{i} \left(\frac{1}{r_\ul{j}}\right)^{\abs{\ul{i}}} \abs{v} \quad \text{, for all $v \in V$ and all $\ul{i}<\ul{j}$.} \]
	Therefore the power series $\sum_{\ul{i}<\ul{j}} a_\ul{i}^{(v)} h_1^{i_1}\cdots h_n^{i_n} $	converges, for all $(h_1,\ldots,h_n) \in B^{n}_{r_\ul{j}}(0)$ and all $v\in V$.
	Moreover, for any $z\in U$, the linear map $V \ra W$, $ v \mto f_v(z)=\langle f(z), v \rangle$, is continuous.	
	Hence, for any $h \in K^n$ with $z_0+h \in B^{n}_{r_\ul{j}}(z_0) \cap U$, the operator
	\begin{equation*}
		T_h \colon V \lra W \,,\quad  v \lto \frac{1}{h_1^{j_1}\cdots h_n^{j_n}} \bigg( f_v(z_0 + h) - \sum_{\ul{i}< \ul{j}} a_\ul{i}^{(v)} h^\ul{i} \bigg) ,
	\end{equation*}
	is $K$-linear and continuous, too.
	For fixed $v\in V$, we may assume $h \in B_{r^{(v)}}^{n}(0)$, and we compute:
	\begin{align*}
		\lim_{h_n\rightarrow 0}\ldots \lim_{h_1\rightarrow 0} T_h(v) =& \lim_{h_n\rightarrow 0}\ldots \lim_{h_1\rightarrow 0} \,\,\, \frac{1}{h_1^{j_1}\cdots h_n^{j_n}} \bigg( \sum_{\ul{i} \in \BN^n_0} a_\ul{i}^{(v)} h^\ul{i} \, - \sum_{\ul{i}< \ul{j}} a_\ul{i}^{(v)} h^\ul{i} \bigg) \\
		=& \lim_{h_n\rightarrow 0}\ldots \lim_{h_1\rightarrow 0} \,\,\, \sum_{\ul{i}\geq \ul{j}} a_\ul{i}^{(v)} h_1^{i_1-j_1}\cdots h_n^{i_n - j_n} \\
		=& \lim_{h_n\rightarrow 0}\ldots \lim_{h_2\rightarrow 0} \,\, \sum_{\substack{\ul{i}\geq \ul{j} \\ i_1 = j_1} } a_\ul{i}^{(v)} h_2^{i_2 - j_2}\cdots h_n^{i_n - j_n} \\[-18pt]
		\vdots& \\
		=& \,\,\, a_\ul{j}^{(v)} = b_{\ul{j}}(v)  .
	\end{align*}
	By the Banach--Steinhaus theorem \cite[III.\ \S 4.2 Cor.\ 2]{Bourbaki87TopVectSp1to5} we conclude that $b_\ul{j}$ is continuous as the pointwise limit of the operators $T_h$.
	
	It remains to show the bound \eqref{Eq A1 - uniform bound} for all $\ul{i}\leq \ul{j}$.
	Define, for $k,l \in \BN$,
	\[B^{k,l}_\ul{j} = \big\{ v\in V \,\big\vert\, \forall \ul{i} \leq \ul{j} \colon \abs{a_\ul{i}^{(v)}} \leq k\, l^{\abs{\ul{i}}} \big\} .\]
	Then each $B_\ul{j}^{k,l}$ is closed in $V$ by the continuity of $b_\ul{i}$, for $\ul{i}\leq \ul{j}$.
	By the pointwise bounds \eqref{Eq A1 - pointwise bound}, we have $ V  =\bigcup_{k,l \in \BN} B_\ul{j}^{k,l}$.
	Hence it follows from Baire's theorem \cite[Ch.\ IX.\ \S 5.3 Thm.\ 1]{Bourbaki66GenTop2} that some $B_\ul{j}^{k,l}$ contains an open ball $B_\varepsilon (v_0)$ of radius $\varepsilon >0$ centred at some $v_0 \in B_\ul{j}^{k,l}$.
	Hence, for all $v \in V$ with $\abs{v}< \varepsilon$ and all $\ul{i}\leq \ul{j}$, we have
	\[\abs{a_\ul{i}^{(v)}} = \abs{a_\ul{i}^{(v+v_0)} - a_\ul{i}^{(v_0)}} \leq \abs{a_\ul{i}^{(v+v_0)}} + \abs{a_\ul{i}^{(v_0)}} \leq 2 k \, l^{\abs{\ul{i}}} \, . \]
	We fix some $\varpi \in K$ with $0 < \abs{\varpi} < 1$.
	For any fixed $v \in V \setminus \{0\}$, let $m \in \BZ$ such that $ \abs{\varpi}^m < \frac{\varepsilon}{\abs{v}} \leq \abs{\varpi}^{m-1}$.
	Then
	\begin{align*}
		\abs{a_\ul{i}^{(v)}} = \frac{\abs{a_\ul{i}^{(\varpi^m v)}}}{\abs{\varpi^m}} \leq \frac{2 k}{\varepsilon \,\abs{\varpi}} l^{\abs{\ul{i}}} \, \abs{v}
	\end{align*}
	which implies
	\[\left\| b_\ul{i} \right\|_{\CL_b(V,W)}= \sup_{v \in V \setminus \{0\}} \frac{\abs{a_\ul{i}^{(v)}}}{\abs{v}} \leq \frac{2 k}{\varepsilon \, \abs{\varpi}} l^{\abs{\ul{i}}} \,,\]
	for all $\ul{i}\leq \ul{j}$.
\end{proof}

It follows from the claim that the power series $\sum_{\ul{j} \in \BN^n_0} b_\ul{j} \, (z-z_0)^\ul{j}$ defines an analytic function with values in $\CL_b(V,W)$, on some open neighbourhood $U' \subset U$ of $z_0$.
On $U'$, this power series agrees with $f$ viewed as a map
\[f \colon U' \lra \CL_b(V,W) \, , \quad  z \lto \langle f(z), \blank \rangle  .\]
Indeed, for $z\in U'$ and $v \in V$:
\begin{align*}
	\bigg( \sum_{\ul{j} \in \BN^n_0} b_\ul{j}\,  (z-z_0)^\ul{j} \bigg)(v) = \sum_{\ul{j} \in \BN^n_0} b_\ul{j}(v) \, (z-z_0)^{\ul{j}} = \sum_{\ul{j} \in \BN^n_0} a_\ul{j}^{(v)} (z-z_0)^\ul{j} = f_v(z) = \langle f(z), v \rangle  .
\end{align*}
Hence $f(z) = \sum_{\ul{j} \in \BN^n_0} b_\ul{j} \, (z-z_0)^\ul{j}$ is analytic on $U'$ as a map into $\CL_b(V,W)$.

Finally, $E$ can be identified with a closed subspace of $\CL_b(V,W)$ by assumption.
By the same reasoning as in the proof of the claim, the coefficients $b_\ul{j}$ of $f$ can be computed as the limits of sequences in $E$.
Therefore $b_\ul{j} \in E$, for all $\ul{j}\in \BN^n_0$, and $f \colon U' \ra E$ is given by a convergent power series.
\end{proof}

\begin{lemma}\label{Lemma A1 - Annihilator of kernel is weak closure of image of the transpose}
Let $f \colon V \ra W$ be a continuous homomorphism between locally convex $K$-vector spaces, and assume that $V$ is Hausdorff.
Then in $V'$ we have the equality
\[ \widebar{\Im(f^t)}^s = \Ker(f)^\perp \defeq \big\{ \ell \in V' \,\big\vert\, \forall v \in \Ker(f) : \ell(v) = 0 \big\} \]
where $\widebar{\Im(f^t)}^s$ denotes the closure of the image of the transpose $f^t \colon W' \ra V'$ in $V'_s$.

Moreover, if in addition $V$ is semi-reflexive, then $\Ker(f)^\perp \subset \widebar{\Im(f^t)}$ in $V'_b$.
\end{lemma}
\begin{proof}
Taking the transpose twice
\begin{equation*}
	\CL (V,W) \lra \CL(W'_s, V'_s) \lra \CL \big( (V'_s)'_s, (W'_s)'_s \big) ,
\end{equation*}
we see that $f$ still defines a continuous homomorphism when $V$ and $W$ carry the respective weak topologies $V_s = (V'_s)'_s$ and $W_s = (W'_s)'_s$.
Then the statement $\Ker(f)^\perp = \widebar{\Im(f^t)}^s$ is part of \cite[II.\ \S 6.4 Cor.\ 2]{Bourbaki87TopVectSp1to5}.

%	A standard computation shows that $\Ker(g^t) = \Im(g)^\perp$, for any continuous homomorphism $g$ of locally convex $K$-vector spaces.
%	We apply this to the transpose $f^t \colon W'_s \ra V'_s$ which is continuous \cite[0.3.8]{Emerton17LocAnVect} to obtain $\Ker\big((f^{t})^t \big) = \Im(f^t)^\perp $ in $(V'_s)' $.
%	Under the isomorphism $V \cong (V'_s)'$ of $K$-vector spaces \cite[Prop.\ 9.7]{Schneider02NonArchFunctAna}, we have the identification $f = (f^t)^t$.
%	We now consider the pseudo-polar \cite[\S 13, Def.]{Schneider02NonArchFunctAna}
%	\begin{equation*}
	%		\Ker(f)^p \defeq \{ \ell \in V' \mid \forall v \in \Ker(f) : \abs{\ell(v)} < 1 \} .
	%	\end{equation*}
%	Since $\Ker(f)$ is a linear subspace of $V$, we see that $\Ker(f)^p = \Ker(f)^\perp$.
%	%Indeed, for $\ell \in \Ker(f)^p$ and $v\in \Ker(f)$, we have $\abs{\ell(v)} = \abs{\lambda \ell (\lambda^{-1}v)} < \abs{\lambda}$, for arbitrarily small $\lambda \in K\unts$, i.e.\ $\abs{\ell(v)} =0$.
%	With the same reasoning, we have $\Im(f^t)^\perp = \Im(f^t)^p$.
%	On the other hand, we can also consider the pseudo-bipolar 
%	\begin{equation*}
	%		 \Im(f^t)^{pp} \defeq \{\ell \in V'_s \mid \forall v \in \Im(f^t)^p: \abs{\ell(v)}<1 \} \subset  V'_s
	%	\end{equation*}
%	which satisfies $\Im(f^t)^{pp} = (\Im(f^t)^p)^p$ when one views $\Im(f^t)^p \subset (V'_s)'_s$ \cite[\S 13, Def.]{Schneider02NonArchFunctAna}.
%	Therefore, we arrive at $\Ker(f)^\perp = \Im(f^t)^{pp}$.
%	But the latter subspace is the closure of $\Im(f^t)$ in $V'_s$ by \cite[Prop.\ 13.4]{Schneider02NonArchFunctAna}.

It is a consequence of the Hahn--Banach theorem that the closed subspace $\widebar{\Im(f^t)} \subset V'_b$ is weakly closed as well \cite[Thm.\ 5.2.1]{PerezGarciaSchikhof10LocConvSpNonArch}.
If we assume that $V$ is semi-reflexive, i.e.\ the duality homomorphism $V \ra (V'_b)'$ is bijective, then the topology on the weak dual of $V$ agrees with the weak topology of the strong dual: $V'_s = (V'_b)_s$, see \cite[Thm.\ 7.4.9]{DeGrandeDeKimpeKakolPerezGarciaSchikofS97pAdicLocConvIndLim}.
Since $\widebar{\Im(f^t)}$ now is a closed subset of $V'_s$ which contains $\Im(f^t)$, it follows that $\widebar{\Im(f^t)}^s \subset \widebar{\Im(f^t)}$.
\end{proof}

\begin{lemma}\label{Lemma A1 - Final topology implies strict homomorphism}
Let $V$ be a locally convex $K$-vector space and $f\colon V \ra W$ a homomorphism of $K$-vector spaces.
If $W$ is given the locally convex final topology with respect to $f$, then $f$ is strict.
\end{lemma}
\begin{proof}
Note that $f$ as a homomorphism between locally convex $K$-vector spaces is continuous if and only if the induced algebraic isomorphism $\widebar{f}\colon V/\Ker(f) \ra \Im(f)$ is continuous with respect to the quotient respectively subspace topology.

We assume by the way of contradiction that $\widebar{f}$ is not a homeomorphism if $W$ carries the locally convex final topology with respect to $f$.
In this case we find an open lattice $L \subset V$ such that $\widebar{f}(L+\Ker(f))$ is a lattice in $\Im(f)$ which is not open.
By extending we obtain a lattice $M \subset W$ such that $M\cap \Im(f) = \widebar{f}(L+\Ker(f))$. 
But $f^{-1}(M) = L + \Ker(f)$ is an open lattice of $V$. 
Therefore $M$ must be an open lattice by the definition of the topology on $W$.
This yields a contradiction.
\end{proof}

\begin{lemma}[{Snake lemma for quasi-abelian categories \cite{KopylovKuzminov00KerCokerSeqSemiAbCat}\footnote{Note that the notion of ``semiabelian'' categories used in \cite{KopylovKuzminov00KerCokerSeqSemiAbCat} agrees with the one of ``quasi-abelian'' categories from \cite{Schneiders98QuasiAbCat}, cf.\ \cite[p.\ 511]{KopylovKuzminov00KerCokerSeqSemiAbCat}.}
}]\label{Lemma A1 - Snake lemma}	
Consider the commutative diagram
\begin{equation*}
	\begin{tikzcd}
		0 \ar[r] &V' \ar[r] \ar[d, "\alpha"] &V \ar[r] \ar[d, "\beta"] &V'' \ar[r] \ar[d,"\gamma"] & 0 \\
		0 \ar[r] &W' \ar[r]					&W \ar[r] 					&W'' \ar[r] 			& 0
	\end{tikzcd}
\end{equation*}
of continuous homomorphisms between locally convex $K$-vector spaces with strictly exact rows.
Then the induced $\Ker$-$\Coker$-sequence
\begin{equation}\label{Eq A1 - Ker-Coker-sequence}
	0 \lra \Ker(\alpha) \overset{\varepsilon}{\lra} \Ker(\beta) \overset{\zeta}{\lra} \Ker(\gamma) \overset{\delta}{\lra} \Coker(\alpha) \overset{\tau}{\lra} \Coker(\beta) \overset{\theta}{\lra} \Coker(\gamma) \lra 0 
\end{equation}
of continuous homomorphisms is exact in $\Ker(\alpha)$, $\Ker(\beta)$, $\Coker(\beta)$, and $\Coker(\gamma)$.
Furthermore, $\varepsilon$ and $\theta$ are strict.
We moreover have:
\begin{altenumeratelevel2}
	\item
	If $\beta$ is strict, then \eqref{Eq A1 - Ker-Coker-sequence} is exact in $\Ker(\gamma)$ and $\Coker(\alpha)$, and $\delta$ is strict.
	\item
	If $\alpha$ is strict, then \eqref{Eq A1 - Ker-Coker-sequence} is exact in $\Ker(\beta)$ and $\Ker(\gamma)$, and $\zeta$ is strict.
	\item
	If $\gamma$ is strict, then \eqref{Eq A1 - Ker-Coker-sequence} is exact in $\Coker(\alpha)$ and $\Coker(\beta)$, and $\tau$ is strict.
\end{altenumeratelevel2}
In particular, if all three $\alpha$, $\beta$, and $\gamma$ are strict, then \eqref{Eq A1 - Ker-Coker-sequence} is a strictly exact sequence.
\end{lemma}

\vspace{3ex}

\numberwithin{theorem}{subsection}

%Defining command for writing the indices $a_{...,...}^{...,...}$ and $gcd(p,m)=1$ to change the style there quickly
\newcommand{\upindex}[1]{{(#1)}}
\newcommand{\downindex}[1]{{(#1)}}
\newcommand{\ndivp}[1]{{p \hspace{1pt}\nmid\hspace{1pt} #1}}
\newcommand{\Nwithoutp}{{\BN_{p'}}}

\section{Continuous and Locally Analytic Characters}\label{Sect - Continuous and Locally Analytic Characters}

In this appendix we consider continuous and locally analytic characters, i.e.\ one-di\-men\-sion\-al representations, of the multiplicative group of a non-archimedean local field of positive characteristic.

However, let us first recall the situation for a $p$-adic field, i.e.\ a finite extension $K$ of $\BQ_p$.
Let $\BC_p$ be the completion of an algebraic closure of $\BQ_p$.
Fixing a uniformizer $\unif$ in the ring of integers $\CO_K$, there is an isomorphism of topological groups \cite[II.\ Satz 5.3]{Neukirch92AlgZahlenTh}
\begin{equation}\label{Eq A2 - Decomposition of units of local field}
	K\unts \cong \unif^\BZ \times \mu_{q-1} \times U_K^{(1)} .
\end{equation}
Here $\mu_{q-1}$ denotes the group of $(q-1)$-st roots of unity of $K$, where $q$ is the number of elements of the residue field $\CO_K/(\unif)$, and
\begin{equation*} 
	U_K^{(n)} \defeq \big\{ x \in \CO_K \,\big\vert\, x \equiv 1 \mod (\unif^n) \big\} \subset \CO_K\unts \quad\text{, for $n\in \BN$.}
\end{equation*}
Because $\mu_{q-1}$ and $\unif^\BZ$ both are discrete groups, it suffices to focus on $U_K^{(1)}$ when considering continuous characters of $K\unts$ with values in $\BC_p\unts$.
For $n > \frac{e}{p-1}$ with $e$ being the index of ramification of $K/\BQ_p$, the logarithm and exponential functions afford an isomorphism of topological groups between $U_K^{(n)}$ and the additive subgroup $(\unif^n)$ of $\CO_K$ \cite[II.\ Satz 5.5]{Neukirch92AlgZahlenTh}.
Moreover, every $U^{(n)}_K$ is of finite index in $U^{(1)}_K$.

It follows from Mahler's theorem \cite[Thm.\ 13.1]{SchneiderTeitelbaum04ContLocAnRepThLectHangzhou} that every continuous additive character $\psi \colon (\unif) \ra \BC_p\unts$ is of the form $\psi(a) = z_1^{a_1} \cdots z_d^{a_d}$, for $z_i \in \BC_p$ with $\abs{z_i-1}<1$. 
Here we write $d\defeq[K:\BQ_p]$ and $a=a_1 e_1 +\ldots + a_d e_d$ in some $\BZ_p$-basis $e_1,\ldots,e_d$ of $\CO_K$ \cite[Example 16.2]{SchneiderTeitelbaum04ContLocAnRepThLectHangzhou}.
By Amice's theorem \cite[Thm.\ 13.2]{SchneiderTeitelbaum04ContLocAnRepThLectHangzhou} such a character $\psi$ is locally $\BQ_p$-analytic, i.e.\ it is a locally analytic function when its source is considered as a locally $\BQ_p$-analytic Lie group.
Moreover, it is locally $K$-analytic if and only if its differential $d_0\psi \colon K \ra \BC_p$ is not only $\BQ_p$-linear but even $K$-linear \cite[Prop.\ 16.3]{SchneiderTeitelbaum04ContLocAnRepThLectHangzhou}.

As the logarithm and exponential functions are locally $K$-analytic, it follows that every continuous character $\chi \colon K\unts \ra \BC_p\unts$ is locally $\BQ_p$-analytic, and even locally $K$-analytic if the differential $d_1\chi$ is $K$-linear.
Furthermore, in the latter case there exists $c\in \BC_p$ such that $\chi(z) = z^c \defeq \exp \big(c \log(z) \big)$ on $U^{(n)}_K$, for $n > \frac{e}{p-1}$.
%For such a character, we have $\chi(z)= \psi (\log(z))$, for some continuous additive character $\psi$ which necessarily has $K$-linear differential. Hence $\psi$ is locally $K$-analytic and given by $\psi(a)=x^a$, for some $x \in \BC_p$ ($\psi$ is of the form $\psi= z_1^{a_1}\cdots z_d^{a_d}$ from above, consider $d_0\psi$ and see that $K$-linearity implies that $z_1 = \ldots = z_d$). Hence $\chi(z) = \psi (\log(z)) = x^{\log(z)} = z^{\log(x)}$, and set $c \defeq \log(x)$.

\subsection{Continuous Characters}

We now turn towards the situation of a local non-archime\-de\-an field $K$ of positive characteristic $p$, say $K= \BF_q \llrrparen{t}$, for $q = p^r$.
Of course, the only additive character of $K$ or subgroups thereof with values in a field of equal characteristic $p$ is the trivial one.
%For such $\psi$, $\psi(a)^p = \psi(pa) = \psi(0)= 1$ implies $0 = \psi(a)^p - 1 = (\psi(a) - 1)^p$, hence $\psi(a) = 1$, for all $a \in K$.
For continuous multiplicative characters of $K\unts$, the same decomposition \eqref{Eq A2 - Decomposition of units of local field} holds and again reduces the task to studying the continuous characters of $U^{(1)}_K$.

To this end, consider more generally abelian metric groups $G$ and $H$ with $G$ locally compact and second countable.
Then the set $\Hom_\cont (G,H)$ of continuous group homomorphisms can be endowed with the compact-open topology whose open sets are given by
\begin{equation*}
	\CU(C,U) \defeq \big\{ f \in \Hom_\cont (G,H) \,\big\vert\, f(C) \subset U \big\} \quad\text{, for $C\subset G$ compact and $U\subset H$ open.}
\end{equation*}
With the pointwise group operation $(f+f')(g)\defeq f(g)+f'(g)$, $\Hom_\cont (G,H)$ becomes a metrizable topological group itself \cite[Prop.\ 6.5.2]{Markley10TopGrpsIntro}.
If moreover $G$ is compact, this topology coincides with the one of uniform convergence and a metric is given by \cite[p.\ 238]{Markley10TopGrpsIntro}
\begin{equation*}
	d(f,f') \defeq \sup_{g \in G} d_H \big(f(g),f'(g) \big) \quad\text{, for $f,f'\in \Hom_\cont (G,H)$.}
\end{equation*}
In the following, we will always view $\Hom_\cont (G,H)$ as a metric group this way.

Let $K\subset L$ be a finite field extension with absolute value $\abs{\blank}$ on $L$ which extends the one of $K$.
Let $f$ and $e$ be the inertia degree respectively the ramification index of $L$ over $K$, i.e.\ $L \cong \BF_{q^f} \mkern-3mu \big(\mkern-3mu \big( t^\frac{1}{e} \big) \mkern-3mu \big) $.
Then the metric on $\Hom_\cont \big(U_K^{(1)},U_L^{(1)} \big)$ is given by
\begin{equation*}
	d(\chi,\chi') \defeq \sup_{z \in U_K^{(1)}} \abs{\chi(z) - \chi'(z)} \quad\text{, for $\chi,\chi'\in \Hom_\cont \big(U_K^{(1)}, U_L^{(1)} \big)$.}
\end{equation*}

We need the following description \cite[II.\ Satz 5.7 (ii)]{Neukirch92AlgZahlenTh} (which in turn reproduces \cite[Prop.\ 2.8]{Iwasawa86LocClassFieldTh}):
Let $\omega_1,\ldots,\omega_r$ be an $\BF_p$-basis of $\BF_q$ and define
\begin{equation*}
	E_K \defeq \prod_\ndivp{m} \prod_{i=1}^r \BZ_p
\end{equation*}
where the first product is taken over all positive integers $m$ which are not divisible by $p$.
Then the map
\begin{equation}\label{Eq A2 - Isomorphism for one-units}
	E_K \overset{\cong}{\lra} U_K^{(1)} \,, \quad \big(a_\downindex{{m,i}} \big)_\downindex{{m,i}} \lto \prod_\ndivp{m} \prod_{i=1}^r \left( 1 + \omega_i t^m \right)^{a_\downindex{{m,i}}} ,
\end{equation}
is a well-defined isomorphism of topological groups.

Moreover, consider an $\BF_p$-basis $\omega_{1,1},\ldots,\omega_{r,f}$ of $\BF_{q^f}$ such that $\omega_{i,1} = \omega_{i}$, for all $i=1,\ldots,r$.
Analogously, we have the isomorphism
\begin{align*}
	E_L \defeq \prod_\ndivp{n} \prod_{j,k=1}^{r,f} \BZ_p \overset{\cong}{\lra} U_L^{(1)} \,, \quad \big( b_\downindex{{n,j,k}} \big) \lto \prod_\ndivp{n}\prod_{j,k=1}^{r,f} \left( 1 + \omega_{j,k} t^\frac{n}{e} \right)^{b_\downindex{{n,j,k}}}
\end{align*}
of topological groups.
Let $s,e'\in \BN$ such that $e=p^s e'$ and $p \nmid e'$.
Under the above isomorphisms, the canonical inclusion $U_K^{(1)} \hookrightarrow U_L^{(1)}$ corresponds to the embedding
\begin{equation*}
	E_K \lhook\joinrel\longrightarrow E_L \,,\quad \big( a_\downindex{{m,i}} \big) \lto \big( b_\downindex{{n,j,k}} \big) \quad\text{, where $b_\downindex{{n,j,k}}= \begin{cases}	p^s a_\downindex{m,j} &, \text{if $n=e'm$ and $k=1$,} \\
			0	&, \text{else.}					\end{cases}$}							
\end{equation*}

\begin{proposition}\label{Prop A2 - Continuous characters in equal characteristic}
	There is an isomorphism of topological groups
	\begin{equation}\label{Eq A2 - Description of continuous characters in equal characteristic}
		\Hom_\cont \big(U_K^{(1)},U_L^{(1)} \big) \overset{\cong}{\lra} \prod_\ndivp{n} \prod_{j,k=1}^{r,f}  c_0(\Nwithoutp,\BZ_p^r) \,,\quad
		\chi \lto \Big( \ul{a}_\downindex{{n,j,k}} = \big( a_\downindex{{n,j,k}}^\upindex{{m,i}} \big) \Big) ,
	\end{equation}
	given by
	\[ \chi(1+\omega_i t^m) = \prod_\ndivp{n}\prod_{j,k=1}^{r,f} (1+\omega_{j,k} t^{\frac{n}{e}} )^{a_\downindex{{n,j,k}}^\upindex{m,i}} . \]
	Here
	\begin{equation*}
		c_0(\Nwithoutp,\BZ_p^r) \defeq \left\{ \ul{a} = \big( a^\upindex{{m,1}},\ldots,a^\upindex{{m,r}} \big)_\ndivp{m} \subset \BZ_p^r \middle{|} \max_{i=1}^r \big\lvert a^\upindex{{m,i}} \big\rvert \ra 0 \text{ as $m\ra \infty$} \right\}.
	\end{equation*}
	carries the structure of a metric group by addition of sequences and the supremum-norm.
\end{proposition}
\begin{proof}
	First we use the description of $U_L^{(1)}$ as a countable product of copies of $\BZ_p$ to reduce to determining $\Hom_\cont \big(U_K^{(1)}, \BZ_p \big)$.
	This is done by the following probably well-known lemma for which we did not find a reference in the literature.

	\begin{lemma}\label{Lemma A2 - Hom to product}
		Let $G$ be a locally compact, second countable, abelian metric group, and let $H_i$, $i\in I$, be abelian metric groups, for a countable index set $I$.
		Then the map
		\begin{align*}
			\alpha \colon \Hom_\cont \bigg( G, \prod_{i\in I} H_i \bigg) &\lra \prod_{i\in I} \Hom_\cont (G, H_i) \\
			\chi &\lto (\pr_i \circ \chi)_{i\in I}
		\end{align*}
		is an isomorphism of topological groups.
	\end{lemma}
	\begin{proof}
		Consider the homomorphisms
		\[\alpha_j \colon \Hom_\cont \bigg( G, \prod_{i\in I} H_i \bigg) \lra \Hom_\cont (G, H_j) \,,\quad \chi \lto \pr_j \circ \chi ,\]
		for $j \in I$.
		For $C \subset G$ compact and $U_j\subset H_j$ open, we have 
		\begin{equation*} 
			\alpha_j^{-1} \big( \CU(C,U_j) \big) = \CU\bigg(C, U_j \times \prod_{i \in I\setminus\{j\}}H_i \bigg) . 
		\end{equation*}
		Therefore the $\alpha_j$ are continuous.
		As $\alpha$ is induced from the $\alpha_j$ by the universal property of the product, $\alpha$ is a continuous group homomorphism.

		An inverse $\beta$ to $\alpha$ on the level of group homomorphisms is given by mapping a collection $(\chi_i \colon G \ra H_i)_{i\in I}$ of homomorphisms to the homomorphism $\prod_{i\in I} \chi_i \colon G \ra \prod_{i\in I} H_i $ induced by the universal property of the product.
		Moreover, this inverse is continuous:
		Let $C \subset G$ be compact and let $U_i \subset H_i$ be open with $U_i = H_i$, for almost all $i \in I$.
		Then we have
		\begin{align*}
			\beta^{-1} \Bigg( \CU\bigg( C, \prod_{i\in I} U_i \bigg) \Bigg) = \prod_{i\in I} \CU(C, U_i) .
		\end{align*}
	\end{proof}

	For $m\in \BN$ with $\ndivp{m}$, and $i\in \{1,\ldots,r\}$, let $\mathbbm{1}_\downindex{m,i} \in E_K$ denote the element with
	\begin{equation*}
		\pr_\downindex{n,j} \big(\mathbbm{1}_\downindex{{m,i}} \big) = \begin{cases}	1&, \text{if $m=n$ and $i=j$,} \\
			0	&, \text{else.}					\end{cases}
	\end{equation*}
	Via the topological isomorphism \eqref{Eq A2 - Isomorphism for one-units}, the following description of $ \Hom_\cont (E_K, \BZ_p)$ finishes the proof of \Cref{Prop A2 - Continuous characters in equal characteristic}.

	\begin{lemma}
		There is an isomorphism of topological groups
		\begin{align}\label{Eq A2 - Description of characters to p-adic integers}
			\Hom_\cont (E_K, \BZ_p) &\overset{\cong}{\lra} c_0(\Nwithoutp, \BZ_p^r) \\
			\chi &\lto \Big( \chi\big(\mathbbm{1}_\downindex{m,1}\big), \ldots ,\chi\big(\mathbbm{1}_\downindex{m,r}\big) \Big)_\ndivp{m} \nonumber \\
			\bigg[ a= \big(a_\downindex{m,i} \big) \mto \sum_\ndivp{m}\sum_{i=1}^r a_\downindex{{m,i}} \lambda^\upindex{{m,i}} \bigg] &\longmapsfrom \Big(  \lambda^\upindex{{m,1}},\ldots, \lambda^\upindex{{m,r}}  \Big)_\ndivp{m} \nonumber ,
		\end{align}
		(cf.\ \cite[Prop.\ 3.5 and 3.6]{DavisWan14LFunctOfpAdicChar} where the above map is shown to be a bijection).
	\end{lemma}
	\begin{proof}
		For $N \geq 0$, let
		\begin{equation*}
			E_{K,N} \defeq \prod_{\substack{\ndivp{m} \\ m< N}}  p^N \BZ_p^r \times \prod_{\substack{\ndivp{m} \\ m\geq N}}  \BZ_p^r.
		\end{equation*}
		The subsets $E_{K,N}$ form a system of fundamental open neighbourhoods of $0 \in E_K$, and satisfy $\mathbbm{1}_\downindex{m} \defeq (\mathbbm{1}_\downindex{m,1},\ldots,\mathbbm{1}_\downindex{m,r}) \in E_{K,N}$, for all $m\geq N$.
		Therefore, for $\chi \in \Hom_\cont (E_K, \BZ_p)$, the continuity of $\chi$ implies that the sequence $\big( \chi(\mathbbm{1}_\downindex{m})\big)_\ndivp{m} = \big( \chi(\mathbbm{1}_\downindex{{m,1}}), \ldots ,\chi(\mathbbm{1}_\downindex{m,r}) \big)_\ndivp{m}$ tends to $0$ in $\BZ_p^r$ when $m \ra \infty$.

		On the other hand, for a zero sequence $\big( \lambda^{(m)} \big)_\ndivp{m} \defeq \big( \lambda^\upindex{{m,1}},\ldots, \lambda^\upindex{{m,r}} \big)_\ndivp{m}$ of the right hand side, clearly $a_\downindex{{m,i}}\lambda^\upindex{{m,i}}$, for $a_\downindex{{m,i}} \in \BZ_p$, is summable.
		It follows that \eqref{Eq A2 - Description of characters to p-adic integers} is an isomorphism of (abstract) groups.

		It remains to show that \eqref{Eq A2 - Description of characters to p-adic integers} is a homeomorphism.
		To do so, we consider a sequence $(\chi_k)_{k\in \BN}$ in $\Hom_\cont (E_K,\BZ_p)$.
		Then convergence of $(\chi_k)_{k\in \BN}$ to $0$ is equivalent to 
		\begin{equation*}
			\sup_{a \in E_K} \abs{\chi_k(a) } \lra 0 \quad\text{as $k \ra \infty$.}
		\end{equation*}
		When this is the case, it clearly follows that
		\begin{equation*}
			\sup_{\ndivp{m}} \big\lvert \chi_k \big(\mathbbm{1}_\downindex{m}\big)  \big\rvert \lra 0 \quad\text{as $k\ra \infty$,}
		\end{equation*}
		i.e.\ that $\big( \big(\chi_k (\mathbbm{1}_\downindex{m}) \big)_{\ndivp{m}} \big)_{k\in \BN}$ converges to $0$ in $c_0 (\Nwithoutp , \BZ_p^r)$.

		Conversely, let $\big( \lambda_k^\upindex{m} \big)_\ndivp{m}$ be a zero sequence of elements of the right hand side corresponding to $\chi_k$, for $k\in \BN$, and assume that $\big( \lambda_k^\upindex{m} \big)_\ndivp{m}$ converges to $0$ in $c_0(\Nwithoutp,\BZ_p)$.
		Then
		\begin{align*}
			\sup_{a \in E_K} \lvert \chi_k(a)  \rvert 
			&= \sup_{a\in E_K} \bigg\lvert \sum_\ndivp{m} \sum_{i=1}^r a_\downindex{m,i}  \lambda_k^\upindex{m,i}  \bigg\rvert \\			
			&\leq \sup_{a \in E_K} \sup_{\substack{\ndivp{m} \\ i=1,\ldots,r}} \lvert a_\downindex{m,i} \rvert \big\lvert \lambda_k^\upindex{m,i} \big\rvert 
			\leq \sup_{\ndivp{m}} \big\lvert \lambda_k^\upindex{m} \big\rvert \lra 0 \quad\text{as $k \ra \infty$.}
		\end{align*}
		Hence $(\chi_k)_{k\in \BN}$ converges to $0$ in $\Hom_\cont (E_K, \BZ_p)$.
	\end{proof}

\end{proof}

\subsection{Locally Analytic Characters}

In contrast to the case of a $p$-adic field, for a local field of positive characteristic, there are significantly less locally analytic characters than continuous characters when compared in a reasonable way.
Furthermore these locally analytic characters behave more rigidly than their $p$-adic counterparts.
The reason for this is the presence of the Frobenius endomorphism $z \mto z^p$ on $K$. 

Let $A$ be a commutative unital $K$-Banach algebra which is an integral domain.
We now consider locally $K$-analytic characters of $U_K^{(1)}$ with values in $A\unts$, i.e.\ group homomorphisms $\chi \colon U_K^{(1)} \ra A\unts$ such that $U_K^{(1)} \xrightarrow{\chi} A\unts \hookrightarrow A$ is locally $K$-analytic in the sense of \Cref{Def 1 - Locally analytic maps between locally analytic manifolds} (i).
We recall the following (see \cite{Jeong11CalcPosChar} for example):

\begin{lemma}[Lucas's theorem]\label{Lemma A2 - Lucas's Theorem}
	Let $m$ and $n$ be non-negative integers and $p$ a prime.
	Then
	\[ \binom{m}{n} \equiv \prod_{i=0}^{k} \binom{m_i}{n_i}  \mod (p) , \]
	where $m = \sum_{i=0}^k m_i p^i$ and $n = \sum_{i=0}^k n_i p^i$ are the $p$-adic expansions, and we use the convention that $\binom{a}{b} = 0$ if $a < b$, for $a,b \in \BN_0$.
\end{lemma}

In particular, we can canonically extend the definition of the binomial coefficient $\binom{c}{n}$ modulo $(p)$ to $c\in \BZ_p$ and $n \in \BZ$ via 
\[\binom{c}{n} \defeq  \binom{c_0}{n_0} \cdots \binom{c_k}{n_k} \mod (p) \]
where $c= \sum_{i=0}^{\infty} c_i p^i$ and $n = \sum_{i=0}^k n_i p^i $ are the $p$-adic expansions.

Furthermore, we will use hyperderivatives which were originally introduced by Hasse and Teichmüller, and whose properties are recollected in \cite[\S 2]{Jeong11CalcPosChar}:
Let $f(z)=\sum_{n=0}^{\infty} a_n (z-z_0)^n$ be a formal power series with values in $A$ and centred at some $z_0 \in K$.
Then the \textit{$k$-th hyperderivative of $f$} is defined as the formal power series
\[D^{(k)} f (z) = \sum_{n=k}^{\infty} \binom{n}{k} \, a_n \,(z-z_0)^{n-k}  .\]
If $f$ is strictly convergent on $B_r^1(z_0) \subset K$, for some $r>0$, then $D^{(k)}f$ is strictly convergent with the same radius of convergence around $z_0$.
Taking $k$-th hyperderivatives is $K$-linear.
We will also use a special instance of the chain rule \cite[\S 2]{Jeong11CalcPosChar} (or \cite{Hasse36TheoDiffFunktKoerpBelChar}) for power series with values in $K$:
\[D^{(1)}(f\circ g) = \big( ( D^{(1)}f ) \circ g \big)\cdot D^{(1)} g  . \]

\begin{theorem}\label{Thm A2 - Locally analytic characters in equal characteristic}
	Let $\chi \colon U_K^{(1)} \ra A\unts$ be a locally $K$-analytic character.
	Then $\chi$ factors over the inclusion $U_K^{(1)} \hookrightarrow A\unts$, and there exists $c\in \BZ_p$ such that $\chi = \chi_c$ where
	\begin{equation}\label{Eq A2 - Form of multiplicative locally analytic characters}
		\chi_c (z) = z^c \defeq \sum_{n=0}^\infty \binom{c}{n} (z-1)^n \quad\text{, for all $z\in U_K^{(1)}$.}
	\end{equation}
	Moreover, the values of all $p^i$-th hyperderivatives of $\chi$ at $1$ are in fact contained in $\BF_p \subset A$ and $c$ is uniquely determined by $c_i \equiv D^{(p^i)} \chi(1) \mod (p)$, for the $p$-adic expansion $c=\sum_{i= 0}^\infty c_i p^i$.
\end{theorem}
\begin{proof}
	First note that every function $\chi_c$ as in \eqref{Eq A2 - Form of multiplicative locally analytic characters} defines a (locally) $K$-analytic character from $U_K^{(1)}$ to itself.
	Indeed, for $c\in \BN$, the equation $\chi_c(zw)=\chi_c(z)\chi_c(w)$ follows from the identity of formal power series in $\BZ\llrrbracket{z,w}$.
	Approximation of $c\in \BZ_p$ then yields 
	\begin{equation*} 
		\chi_c(zw) \equiv \chi_c(z) \chi_c(w) \mod (t^n) \quad\text{, for all $n \in \BN$, $z,w \in U_K^{(1)}$,}
	\end{equation*}
	where $t$ is a uniformizer of $\CO_K = \BF_q \llrrbracket{t}$.
	This shows the sought functional equation.

	Now, let $\chi \colon U_K^{(1)} \ra A\unts$ be a locally $K$-analytic character.
	Let $N \geq 1$ such that on $U_K^{(N)} = 1 + t^N \CO_K$ the character $\chi$ is given by the strictly convergent power series
	\[\chi(z)= \sum_{n=0}^{\infty} a_n (z-1)^n \quad\text{, with $a_n \in A$.} \]
	It follows from $\chi(z^p) = \chi(z)^p$ that
	\[\sum_{n=0}^{\infty} a_n (z^p-1)^n = \sum_{n=0}^{\infty} a_n^p (z^p -1)^n  \quad\text{, for all $z\in U_K^{(N)}$.} \]
	Hence by the identity theorem \ref{Prop 1 - Identity theorem for power series}, we have $a_n = a_n^p$.
	Since $A$ is an integral domain by assumption, this implies that $a_n$ already is contained in $\BF_p \subset A$.
	As $\abs{a_n}=1$, for all $n\geq 0$, the power series $\sum_{n=0}^{\infty} a_n (z-1)^n$ strictly converges on the whole of $U_K^{(1)}$.
	For $z\in U_K^{(1)}$, let $j\geq 0$ such that $z^{p^j} \in U_K^{(N)}$.
	Then 
	\[\chi(z)^{p^j} = \chi(z^{p^j}) = \sum_{n=0}^{\infty} a_n (z^{p^j} - 1)^n = \bigg( \sum_{n=0}^{\infty} a_n (z-1)^n \bigg)^{p^j} . \]
	Hence it follows that $\chi(z) = \sum_{n=0}^{\infty} a_n (z-1)^n$, for all $z \in U_K^{(1)}$, using that $A$ is an integral domain of characteristic $p$.

	Now consider, for all $w\in U_K^{(1)}$, the chain rule applied to the hyperderivative with respect to $z$: 
	\begin{align*}
		\chi(w) \, D^{(1)} \chi (z) = D^{(1)} \big( \chi(w)\chi(z) \big) 	
		=D^{(1)} \chi(wz)
		=D^{(1)}\big( \chi \circ (w\blank)\big) (z) 
		=D^{(1)}\chi(wz) \cdot w			 .
	\end{align*}
	Setting $z=1$ and noting that $D^{(1)} \chi (1)=a_1$, we conclude that
	\[\chi(w)\, a_1 = w \, D^{(1)} \chi (w) = (w-1)D^{(1)}\chi(w)+ D^{(1)}\chi(w) \, .\]
	By the identity theorem for power series \ref{Prop 1 - Identity theorem for power series}, as $w\in U_K^{(1)}$ was arbitrary, this is an identity between power series.
	It follows that, for all $n\geq 1$,
	\begin{align}\label{Eq A2 - Coefficients of character 1}
		a_1 a_n = \binom{n+1}{1} a_{n+1} + \binom{n}{1} a_n  .
	\end{align}
	For all $l \geq 0$, we want to show by induction on $j$ that
	\begin{align}\label{Eq A2 - Coefficients of character 2}
		a_{lp +j} = \binom{a_1}{j} a_{lp}  \quad\text{, for $j=0,\ldots,p-1$,} 
	\end{align}
	where we identify $a_1\in \BF_p$ with its representative in $\{0,\ldots,p-1\} \subset \BZ$.
	For $j=0$, \eqref{Eq A2 - Coefficients of character 2} holds trivially as $\binom{a_1}{0}=1$.
	Now assume \eqref{Eq A2 - Coefficients of character 2} holds for $j\in \{0,\ldots,p-2\}$.
	Then \eqref{Eq A2 - Coefficients of character 1}, for $n=lp + j$, together with the induction hypothesis gives
	\[(j+1)a_{lp+j+1} = \binom{lp +j+1}{1} a_{lp + j +1} = \left(a_1 -\binom{lp+j}{1}\right) a_{lp+j} = (a_1 - j) \binom{a_1}{j} a_{lp} \, . \]
	Hence 
	\[a_{lp+j+1} = \frac{a_1-j}{j+1} \binom{a_1}{j} a_{lp} = \binom{a_1}{j+1} a_{lp} \]
	showing \eqref{Eq A2 - Coefficients of character 2} for $j+1$.

	Observe that
	\[z^{a_1} = \sum_{j=0}^{p-1} \binom{a_1}{j} (z-1)^j  ,\]
	and therefore
	\begin{align*}
		\chi(z) &= \sum_{l=0}^{\infty} \bigg( \sum_{j=0}^{p-1} a_{lp+j} (z-1)^j \bigg) (z-1)^{lp} \\
		&= \sum_{l=0}^{\infty} \bigg( \sum_{j=0}^{p-1} \binom{a_1}{j} (z-1)^{j} \bigg) a_{lp} (z-1)^{lp} 
		= z^{a_1} \sum_{l=0}^{\infty} a_{lp} (z^p - 1)^l  .
	\end{align*}
	We define, for $z\in U_K^{(1)}$ and $i\geq 0$,
	\[\chi^{(i)} (z) \defeq \sum_{n=0}^{\infty} a^{(i)}_{n} (z-1)^n \quad\text{, where $a^{(i)}_n \defeq a_{n p^i}$.} \]
	Then $\chi(z) = z^{a_1} \chi^{(1)} (z^p)$. 
	Using the injectivity of the Frobenius endomorphism this implies that $\chi^{(1)}$ is a locally $K$-analytic character on $U_K^{(1)}$, too.
	Recursively it follows that $\chi^{(i)}$ is a locally $K$-analytic character on $U_K^{(1)}$, for all $i\geq 0$.
	Moreover, \eqref{Eq A2 - Coefficients of character 2} applied to $\chi^{(i)}$ gives, for all $l\geq 0$,
	\begin{align}\label{Eq A2 - Coefficients of character 3}
		a_{l p^{i+1} + j p^i} = a^{(i)}_{lp+j} = \binom{a_{1}^{(i)}}{j} a_{lp}^{(i)} = \binom{a_{p^i}}{j} a_{lp^{i+1}} \quad\text{, for $j=0,\ldots,p-1$.}
	\end{align}
	
	To finish the proof we show that, for all $n\geq 0$, $a_n = \prod_{i=0}^k \binom{a_{p^i}}{n_i} $ where $n=\sum_{i=0}^{k} n_i p^i$ is the $p$-adic expansion of $n$.
	Indeed, fix $n\in \BN$ with such a $p$-adic expansion.
	Then \eqref{Eq A2 - Coefficients of character 3}, for $i=0,\ldots,k$, gives
	\begin{align*}
		a_n &= a_{\left( \sum_{i=1}^k n_i p^{i-1}\right) p + n_0} = \binom{a_1}{n_0}\,a_{\sum_{i=1}^k n_i p^i } \\
		&= \binom{a_1}{n_0} \, a_{\left( \sum_{i=2}^k n_i p^{i-2}\right)p^2 + n_1 p} = \binom{a_1}{n_0} \binom{a_p}{n_1} \, a_{ \sum_{i=2}^k n_i p^{i} } \\
		&= \ldots = \binom{a_1}{n_0}\cdots\binom{a_{p^k}}{n_k} .
	\end{align*}
	Lastly note that $ a_{p^i}=D^{(p^i)}\chi(1) $, for all $i \geq 0$.
\end{proof}

\begin{corollary}\label{Cor A2 - Ring of locally analytic characters in equal characteristic}
	The locally analytic characters $\End_\mathrm{la} \big(U_K^{(1)} \big) \subset \End_\cont \big(U_K^{(1)} \big)$ constitute a closed subring with multiplication given by composition. 
	Moreover, with the induced subspace topology, the description in \eqref{Eq A2 - Form of multiplicative locally analytic characters} yields an isomorphism
	\begin{equation*}
		\BZ_p \overset{\cong}{\lra} \End_\mathrm{la}\big(U_K^{(1)} \big) \,,\quad c \lto \chi_c ,
	\end{equation*}
	of topological rings, and the above embedding $\End_\mathrm{la} \big(U_K^{(1)} \big) \subset \End_\cont \big(U_K^{(1)} \big)$ corresponds to 
	\begin{align}\label{Eq A2 - Embedding for locally analytic characters}
		\mathrm{diag}\colon \BZ_p \lra \prod_\ndivp{n} \prod_{j=1}^{r}  c_0(\Nwithoutp,\BZ_p^r) \,,\quad c &\lto \mathrm{diag}(c) \defeq \Big(\ul{c}_\downindex{{n,j}}= \big(c_\downindex{{n,j}}^\upindex{{m,i}} \big) \Big) , \\ 
		&\quad\text{where $c_\downindex{{n,j}}^\upindex{{m,i}} \defeq \begin{cases}	c&, \text{if $m=n$ and $i=j$,} \\
				0	&, \text{else,}					\end{cases}$} \nonumber
	\end{align}
	under this isomorphism and the identification \eqref{Eq A2 - Description of continuous characters in equal characteristic}.
\end{corollary}
\begin{proof}
	One readily computes that the image $\mathrm{diag}(\BZ_p)$ is a closed subgroup.
	To show that the injective homomorphism \eqref{Eq A2 - Embedding for locally analytic characters} of additive groups is a topological embedding, we show that a sequence $(c_k)_{k\in \BN} \subset \BZ_p$ converges to $0$ if and only if $\big( \mathrm{diag}(c_k) \big)_{k\in \BN}$ converges to $0$.
	Indeed, the latter is equivalent to convergence of $\big( (\ul{c_k})_\downindex{n,j} \big)_{k\in \BN}$ to $0$ in $c_0 (\Nwithoutp, \BZ_p)$, for all $n \in \BN$ with $\ndivp{n}$ and $j=1,\ldots,r$.
	But for the supremum-norm $\abs{\blank}_\infty$ of $c_0 (\Nwithoutp, \BZ_p)$ we compute
	\[ \lvert (\ul{c_k})_\downindex{n,j} \rvert_\infty \defeq \sup_{\ndivp{m}} \big\lvert (\ul{c_k})_\downindex{n,j}^\upindex{m,i} \big\rvert = \abs{c_k} . \]
	
	The only assertion left to verify is that the isomorphism $\End_\mathrm{la} \big(U_K^{(1)} \big) \cong \BZ_p$ is compatible with the multiplication, i.e.\ that $\chi_c \circ \chi_d = \chi_{cd}$, for $c,d \in \BZ_p$.
	For $c,d \in \BN$, this holds because $\chi_c(z)$ is the usual exponentiation $z^c$ in this case.
	The general case then follows by approximation and continuity.	
\end{proof}

\sloppy
\AtNextBibliography{\small}
\printbibliography

\end{document}